\documentclass[english]{article}
\usepackage{geometry}
\usepackage[unicode=true,
bookmarks=false,
breaklinks=false,pdfborder={0 0 1},colorlinks=false]
{hyperref}
\hypersetup{
	colorlinks,citecolor=blue,filecolor=blue,linkcolor=blue,urlcolor=blue}
\usepackage{mathtools}
\usepackage{empheq}
\usepackage{amsfonts}
\usepackage{amsgen,amsmath,amstext,amsbsy,amsopn,amssymb,subfigure,amsthm,stmaryrd}
\usepackage{bm}
\usepackage{bbm}
\usepackage{comment}
\usepackage[dvips]{graphicx}
\usepackage{enumitem}
\usepackage{natbib}
\usepackage{booktabs}
\usepackage{blkarray}
\usepackage[linesnumbered,ruled,vlined]{algorithm2e}
\usepackage{url}
\usepackage{longtable}
\usepackage{mathtools}
\usepackage{multirow}

\usepackage[usenames,dvipsnames]{xcolor}

\usepackage{mathtools}
\usepackage{empheq}
\usepackage{amsfonts}

\usepackage[dvips]{graphicx}
\usepackage{enumitem}
\usepackage{natbib}
\usepackage{booktabs}
\usepackage{blkarray}
\usepackage{algpseudocode}
\usepackage{url}
\usepackage{longtable}
\usepackage{mathtools}
\usepackage{multirow}

\definecolor{yxc}{RGB}{255,0,0}

\allowdisplaybreaks

\newtheorem{thm}{Theorem}

\newtheorem{lemma}{Lemma}

\newtheorem{Definition}{Definition}
\newtheorem{rmk}{Remark}

\newtheorem{assump}{Assumption}

\allowdisplaybreaks

\newcommand{\be}{\begin{equation}}
	\newcommand{\ee}{\end{equation}}
\newcommand{\bea}{\begin{eqnarray}}
	\newcommand{\eea}{\end{eqnarray}}
\newcommand{\beas}{\begin{eqnarray*}}
	\newcommand{\eeas}{\end{eqnarray*}}

\newcommand{\bbR}{\mathbb{R}}

\newcommand{\F}{{\mathrm{F}}}

\newcommand{\argmin}{\mathop{\rm arg\min}}
\newcommand{\argmax}{\mathop{\rm arg\max}}

\newcommand{\bbP}{\mathbb{P}}
\newcommand{\bbE}{\mathbb{E}}

\newcommand{\Algoname}{{\sf High\text{-}order HeteroClustering}}
\newcommand{\Algoabbrev}{{\sf HHC}}
\newcommand{\Algosubspace}{{\sf Thresholded~Deflated\text{-}HeteroPCA}}

\setdescription{font = \normalfont}

\makeatletter
\newcommand*{\rom}[1]{\expandafter\@slowromancap\romannumeral #1@}
\makeatother

\title{Heteroskedastic Tensor Clustering}

\author{Yuchen Zhou\thanks{Department of Statistics and Data Science, Wharton School, University of Pennsylvania, Philadelphia, PA 19104, USA.} 	\and
	Yuxin Chen\footnotemark[1] \thanks{Department of Electrical and Systems Engineering, University of Pennsylvania, Philadelphia, PA 19104, USA.} 
}

\date{\today}

\begin{document}
	\maketitle
	\begin{abstract}
		Tensor clustering, which seeks to extract underlying cluster structures from noisy tensor observations, has gained increasing attention. 
		One extensively studied  model for tensor clustering is the tensor block model, 
		which postulates the existence of clustering structures along each mode and has found broad applications in areas like multi-tissue gene expression analysis and multilayer network analysis. However, currently available  computationally feasible methods for tensor clustering either are limited to handling i.i.d.~sub-Gaussian noise or suffer from suboptimal statistical performance, 
		which restrains their utility in applications that have to deal with heteroskedastic data and/or low signal-to-noise-ratio (SNR).

		To overcome these challenges, we propose a two-stage method, named \Algoname~(\Algoabbrev), 
		which starts by performing tensor subspace estimation via a novel spectral algorithm called \Algosubspace, followed by approximate $k$-means to obtain cluster nodes. 
		Encouragingly, our algorithm provably achieves exact clustering as long as the SNR exceeds the computational limit (ignoring logarithmic factors); 
		here, the SNR refers to the ratio of the pairwise disparity between nodes to the noise level, and the computational limit indicates the lowest SNR that enables exact clustering with polynomial runtime. Comprehensive simulation and real-data experiments suggest that our algorithm outperforms existing algorithms across various settings, delivering more reliable clustering performance.
	\end{abstract}

	\tableofcontents

	\section{Introduction}

The past few years have witnessed a surge of interest in tensor data analysis across various domains, including recommendation systems \citep{bi2018multilayer, nasiri2014fuzzy}, neuroimaging \citep{wozniak2007neurocognitive, zhou2013tensor}, computational imaging \citep{li2010tensor, zhang2020denoising}, medical imaging \citep{fu20163d}, signal processing \citep{cichocki2015tensor,sidiropoulos2017tensor}, among other things. Compared to vectors and matrices, tensors (or multiway data arrays) offer the ability to characterize complex interrelations and interactions across multiple dimensions, allowing one to simultaneously capture the effects brought about by multiple factors. 
The growing prevalence of tensor data has sparked  in-depth statistical research --- from both methodological and theoretical perspectives --- into various tensor estimation and learning problems (see, e.g., \cite{zhou2013tensor,richard2014statistical,yuan2016tensor,xia2021statistically,cai2022nonconvex,liu2020tensor,bi2021tensors,han2022optimal,deng2023correlation}).

Within this body of research, one important problem  that has garnered increasing attention is tensor clustering, 
which aims to extract the underlying cluster structures inherent in the observed tensor data.
A model of this kind that has received widespread adoption is the {\em tensor block model} \citep{wang2019multiway,chi2020provable,han2022exact}. 
Concretely, suppose that the observed data takes the form of an order-three tensor $\bm{\mathcal{Y}} \in \bbR^{n_1 \times n_2 \times n_3}$, drawn from the following data generating mechanism. 
\begin{itemize}
	\item {\em Cluster structure.}
In each mode $1\leq i \leq 3$, 
the $n_i$ indices (or nodes) are divided into $k_i$ clusters, and the cluster memberships of these nodes are encoded by a {\em cluster assignment vector} $\bm{z}_i^\star=[z_{i,j}^\star]_{1\leq j\leq n_i} \in [k_i]^{n_i}$ such that
\begin{equation}
	z_{i,j}^\star = \ell ~~\text{if the }j\text{th node falls within cluster }\ell \qquad (1\leq j\leq n_i).
\end{equation}
Here and throughout, we denote by $[d] = \{1, \dots, d\}$ for any positive integer $d$.

	\item {\em Observations.}
		The entries of the observed tensor $\bm{\mathcal{Y}}=[Y_{i,j,\ell}]$ obey
\begin{align}\label{eq:model}
	Y_{i, j, \ell} = S^\star_{z_{1, i}^\star, z_{2, j}^\star, z_{3, \ell}^\star} + E_{i, j, \ell},\qquad \forall (i, j, \ell) \in [n_1] \times [n_2] \times [n_3].
\end{align}
		Here, $S^\star_{j_1, j_2, j_3}$ is the $(j_1, j_2, j_3)$-th entry of an underlying core tensor $\bm{\mathcal{S}}^\star \in \bbR^{k_1 \times k_2 \times k_3}$ (which often has much lower dimension than $\bm{\mathcal{Y}}$), whereas the $E_{i, j, \ell}$'s represent independent zero-mean noise contaminating the measurements. 
		Alternatively, we can rewrite this model in the tensor form as
		\begin{align}\label{model:additive}
			\bm{\mathcal{Y}} = \bm{\mathcal{X}}^{\star} + \bm{\mathcal{E}} \in \bbR^{n_1 \times n_2 \times n_3},
		\end{align}
		where $\bm{\mathcal{E}}=[E_{i,j,\ell}\big]_{(i, j, \ell) \in [n_1] \times [n_2] \times [n_3]}$ stands for a noise tensor, and
		the underlying tensor 
		\begin{equation}
			\bm{\mathcal{X}}^{\star} = \big[S^\star_{z_{1, i}^\star, z_{2, j}^\star, z_{3, \ell}^\star}\big]_{(i, j, \ell) \in [n_1] \times [n_2] \times [n_3] }
		\end{equation}
		exhibits block --- and hence low-rank --- structures. 
		Crucially, for any set of indices $(i,j,\ell)\in [n_1] \times [n_2] \times [n_3]$, 
		the mean of the observed entry $Y_{i, j, \ell}$ is determined by the cluster membership of $(i,j,\ell)$, 
		making it possible to retrieve the cluster assignment information from the observed tensor as long as the noise level is not overly large.


	\item {\em Goal.} The aim is to reconstruct the underlying cluster structure along each mode --- namely, recovering each $\bm{z}_i^\star$ $(1\leq i\leq 3)$ --- on the basis of the observation $\bm{\mathcal{Y}}$. 

\end{itemize}
Notably, this tensor block model finds a diverse range of applications. For instance, in multi-tissue gene expression analysis \citep{wang2019three,wang2019multiway,han2022exact}, the expression levels of  numerous genes are measured from various tissues across multiple individuals, and there could be natural group structures for genes,  tissues, and individuals, respectively, which can be captured by the above model. Another instance arises from multilayer network analysis \citep{lei2020consistent}, wherein multiple (directed or undirected) graphs with identical vertices are gathered from various scenarios or experiments, inherently forming a tensor. A task stemming from this kind of data is identifying the clustering structures among the vertices and across different layers based on their connectivity patterns.

While numerous clustering algorithms have been studied in the literature, directly applying traditional clustering methods, such as $k$-means, to (the unfoldings of) the tensor data $\bm{\mathcal{Y}}$ may fail to capture the inherent tensor structures and therefore lead to unsatisfactory results. 
To overcome this issue, \cite{han2022exact} proposed a spectral clustering method called {\sf High-order Spectral Clustering (HSC)}, which starts by projecting the tensor data onto their estimated top singular subspaces along each mode, followed by an approximate $k$-means procedure to cluster nodes.  Informally speaking, the singular subspace estimation procedure adopted by {\sf HSC} directly calculates the left singular subspaces of the unfoldings of $\bm{\mathcal{Y}}$ along each mode --- which we shall refer to as a vanilla SVD-based approach in the sequel.   
To further improve the spectral estimates,  \cite{han2022exact} also came up with an algorithm called  {\sf High-order Lloyd Algorithm (HLloyd)} 
to iteratively refine the block membership estimates. When the noise tensor $\bm{\mathcal{E}}$ has i.i.d.~sub-Gaussian noise entries,  
{\sf HSC} (resp.~{\sf HSC} followed by {\sf HLloyd}) provably achieves consistent (resp.~exact) clustering results while accommodating a near-optimal range of signal-to-noise ratio (SNR) conditions (among polynomial-time algorithms) \citep{han2022exact}.




However, the {\sf HSC}  algorithm and the intriguing theory developed by \citet{han2022exact} fall short of accommodating  heteroskedastic data, 
a common scenario in practice where variances of noise entries vary across locations.  
It has now been widely recognized that the vanilla SVD-based approach mentioned above could generate highly sub-optimal subspace estimates in the face of 
heteroskedastic noise \citep{zhang2022heteroskedastic,cai2021subspace,zhou2023deflated}; as a consequence, {\sf HSC}, which is initialized based on this approach, becomes statistically sub-optimal. 
This issue severely hinders the performance of {\sf HSC} in, say, a broad array of applications with discrete-valued observations --- including multi-tissue gene expression data analysis and multilayer network data analysis --- which often have to deal with heterogeneous data. To the best of our knowledge, no computationally efficient algorithm has been shown to achieve consistent estimation --- not to mention exact recovery --- of the underlying cluster structure under the widest possible SNR conditions.


\subsection{Main contributions}

Aimed at addressing the challenges resulting from heteroskedastic data, 
this paper proposes a new tensor clustering algorithm called  \Algoname~(\Algoabbrev). 
The key innovation compared to {\sf HSC} lies in the development of a new paradigm for  estimating the top singular subspaces of the unfolded tensor, 
in the hope of tackling heteroskedasticity. In a nutshell, the proposed \Algoabbrev~algorithm encompasses two stages: 
\begin{itemize}
	\item[1.] \emph{Subspace estimation.} 
		This stage seeks to estimate the column subspaces of the unfoldings of $\bm{\mathcal{X}}^\star$ along each mode. 
		Inspired by a spectral algorithm {\sf Deflated-HeteroPCA} that proves effective in the face of heteroskedastic noise \citep{zhou2023deflated}, 
		we propose a new variant, called  \Algosubspace, that combines {\sf Deflated-HeteroPCA} with a {\em data-driven} thresholding procedure. 
		Our procedure only attempts to estimate the ``useful'' part of the column subspaces --- that is, the subspace associated with reasonably large singular values --- 
		of the unfolded tensors, 
		which plays a crucial role in achieving statistical guarantees that are independent of the magnitude of the smallest singular value of $\bm{\mathcal{S}}^\star$.

	

	\item[2.] \emph{Approximate $k$-means.} Armed with the above subspace estimates, 
		the second stage projects the unfolding of $\bm{\mathcal{Y}}$ onto the estimated subspace for denoising purposes, 
		followed by an approximate  $k$-means algorithm to cluster nodes. 
\end{itemize}
Encouragingly, 
the proposed \Algoabbrev~algorithm allows for exact clustering as long as a certain ``necessary'' SNR condition holds (up to logarithmic factors), 
where the SNR is captured by the ratio of certain ``pairwise'' difference between nodes to the noise level.  
Here, a ``necessary" SNR condition refers to a condition that is essential to ensure that the cluster assignment vectors $\bm{z}_i^\star$ can be exactly recovered in polynomial time. Empirically, we conduct simulation experiments and find that \Algoabbrev~can reliably estimate the cluster structures, and that \Algoabbrev~combined with {\sf HLloyd} \citep{han2022exact} enables enhanced numerical performance. We also apply our method (\Algoabbrev~+ {\sf HLloyd}) and {\sf HSC + HLloyd} to the flight route network data, 
in which our method leads to better  clustering results. It is noteworthy that: 
while the current paper focuses on three-way tensor for simplicity of presentation, both our algorithm and the proof can be straightforwardly extended to accommodate general higher-order tensors.

\paragraph{Paper organization.} The rest of this article is organized as follows. 
Section~\ref{sec:setting} formulates the mathematical model and introduces the key assumptions, 
while Section~\ref{sec:algorithm} presents the proposed algorithm. 
The theoretical guarantees for our algorithm are provided in Section~\ref{sec:theory}, 
with the analysis deferred to the appendix. 
Numerical performance on both synthetic and real data is reported in Section~\ref{sec:simulation}.

\subsection{Notation}\label{subsection:notation}
Throughout the paper, we denote $[n] \coloneqq \{1, \dots, n\}$ for any integer $n>0$. 
We often use  bold capital letters (e.g., $\bm{X}, \bm{Y}, \bm{Z}$) and  bold lowercase letters (e.g., $\bm{x}, \bm{y}, \bm{z}$) to denote matrices and vectors, respectively, 
and employ boldface calligraphic letters (e.g., $\bm{\mathcal{X}}$, $\bm{\mathcal{Y}}$, $\bm{\mathcal{Z}}$) to represent tensors. 
For any matrix $\bm{X} \in \bbR^{n_1 \times n_2}$,  we let $\lambda_{i}(\bm{X})$ and $\sigma_{i}(\bm{X})$ denote the $i$-th largest eigenvalue (in magnitude) and the $i$-th largest singular value of $\bm{A}$, respectively. Define $\|\cdot\|_{\F}$ for Frobenious norm and $\|\cdot\|$ for spectral norm. We denote by $\bm{A}_{i,:}$ and $\bm{A}_{:,j}$ the $i$-th column and the $j$-th row of a matrix $\bm{A}$, respectively, 
and define its $\ell_{2,\infty}$ norm as $\|\bm{A}\|_{2,\infty} \coloneqq \max_{i \in [n_1]}\|\bm{A}_{i,:}\|_2$. Let $\mathcal{O}^{n, r} \coloneqq \{\bm{U} \in \bbR^{n \times r}: \bm{U}^\top\bm{U} = \bm{I}_r\}$ denote the set containing all $n$-by-$r$ matrices with orthonormal columns. We use $\mathcal{P}_{\sf diag}$ to represent the projection that keeps all diagonal entries and zeros out all non-diagonal entries, and  define $\mathcal{P}_{\sf off\text{-}diag}(\bm{M}) \coloneqq \bm{M} - \mathcal{P}_{\sf diag}(\bm{M})$ for any $\bm{M} \in \bbR^{n \times n}$. For any vector $\bm{a} = [a_i]_{1\leq i\leq n}$, we denote by ${\sf diag}(\bm{a})$ the diagonal matrix whose $(i,i)$-th entry is $a_{i}$. 
We let $C, c, C_0, c_0, \dots$ denote absolute constants whose values may change from line to line.

For any two matrices $\bm{A} \in \mathbb{R}^{m \times n}$ and $\bm{B} \in \mathbb{R}^{p \times q}$, we define the Kronecker product of them as
\begin{align*}
	\bm{A} \otimes \bm{B} \coloneqq \begin{bmatrix}
		a_{11}\bm{B}& \cdots& a_{1n}\bm{B}\\
		\vdots & \ddots &\vdots\\
		a_{m1}\bm{B} & \cdots & a_{mn}\bm{B}
	\end{bmatrix}.
\end{align*}
For any tensor $\bm{\mathcal{G}} \in \bbR^{r_1 \times r_2 \times r_3}$ and any matrix $\bm{V}_1 \in \bbR^{n_1 \times r_1}$, the multi-linear product $\times_1$ is defined as
\begin{align*}
	\bm{\mathcal{G}} \times_1 \bm{V}_1 = \bigg(\sum_{j=1}^{r_1}G_{j, i_2, i_3}V_{i_1, j}\bigg)_{i_1 \in [n_1], i_2 \in [r_2], i_3 \in [r_3]}.
\end{align*} 
We can also define $\times_2$ and $\times_3$ analogously. 
For any tensor $\bm{\mathcal{X}} \in \bbR^{n_1 \times n_2 \times n_3}$ and $1\leq j \leq 3$, let $\mathcal{M}_j(\bm{\mathcal{X}}) \in \bbR^{n_j \times (n_1n_2n_3/n_j)}$ denote the $j$-th matricization of $\bm{\mathcal{X}}$ such that 
\begin{align*}
	\big[\mathcal{M}_1\left(\bm{\mathcal{X}}\right)\big]_{i_1, i_2 + n_2\left(i_3 - 1\right)} = \big[\mathcal{M}_2\left(\bm{\mathcal{X}}\right)\big]_{i_2, i_3 + n_3\left(i_1 - 1\right)} = \big[\mathcal{M}_3\left(\bm{\mathcal{X}}\right)\big]_{i_3, i_1 + n_1\left(i_2 - 1\right)} = X_{i_1, i_2, i_3}
\end{align*} 
%
for all $(i_1, i_2, i_3) \in [n_1] \times [n_2] \times [n_3]$.
We further define the Frobenious norm of a tensor $\bm{\mathcal{X}} \in \bbR^{n_1 \times n_2 \times n_3}$ as
\begin{align*}
	\left\|\bm{\mathcal{X}}\right\|_{\F} = \Bigg(\sum_{i=1}^{n_1}\sum_{j=1}^{n_2}\sum_{k=1}^{n_3}X_{i,j,k}^2\Bigg)^{1/2}.
\end{align*}

In addition, we say that $f(x) \lesssim g(x)$ or $f(x) = O(g(x))$ if $|f(x)| \leq Cg(x)$ for some constant $C > 0$; we let $f(x) \gtrsim g(x)$ denote $f(x) \geq C|g(x)|$ for some constant $C > 0$; we say $f(x) \asymp g(x)$ or $f(x) = \Omega(g(x))$ if $f(x) \lesssim g(x)$ and $f(x) \gtrsim g(x)$ hold; we use the notation $f(x) \ll g(x)$ to represent that $f(n_1, n_2) \leq cg(n_1, n_2)$  holds for some sufficiently small constant $c > 0$, and we say $f(n_1, n_2) \gg g(n_1, n_2)$ if $g(n_1, n_2) \ll f(n_1, n_2)$. In addition, we use $f(n_1, n_2) = o(g(n_1, n_2))$ to indicate that $f(n_1, n_2)/g(n_1, n_2) \to 0$ as $\min\{n_1, n_2\} \to \infty$. For any $a, b \in \bbR$, let $a \wedge b \coloneqq \min\{a, b\}$ and $a \vee b \coloneqq \max\{a, b\}$.
%
Moreover, denote by $\Phi$  the set of all permutations $\phi: [k] \to [k]$. For any $\bm{z}=[z_j]_{1\leq j\leq n}, \overline{\bm{z}}  =[\widehat{z}_j]_{1\leq j\leq n} \in [k]^n$, we define the misclassification rate as follows:
\begin{align}
	{\sf MCR}\left(\bm{z}, \widehat{\bm{z}}\right) \coloneqq \inf_{\phi \in \Phi}\frac{1}{n}\sum_{j =1}^{n}\mathbbm{1}\big\{\widehat{z}_{j} \neq \phi\left(z_{j}\right)\big\}.
	\label{eq:defn-MCR}
\end{align}
Also, for any $\phi \in \Phi$, we use the notation $\widehat{\bm{z}} = \phi(\bm{z})$ to mean that
\begin{align*}
	\widehat{z}_j = \phi\left(z_j\right),~\quad~\forall j \in [n]. 
\end{align*}
Clearly, ${\sf MCR}\left(\bm{z}, \widehat{\bm{z}}\right) = 0$ holds  if and only if $\widehat{\bm{z}} = \phi(\bm{z})$ for some $\phi \in \Phi$.

	\section{Problem formulation}
\label{sec:setting}

\subsection{Models}  
Recall that the vectors $\bm{z}_i^\star$'s encode the cluster assignment. 
The tensor block model~\eqref{eq:model} can be equivalently written as 
\begin{align}\label{eq:model_equivalence}
	\bm{\mathcal{Y}} = \bm{\mathcal{X}}^\star + \bm{\mathcal{E}} \in \bbR^{n_1 \times n_2 \times n_3}, 
\end{align}
where  the low-rank tensor $\bm{\mathcal{X}}^\star$ can be decomposed as 
$\bm{\mathcal{X}}^\star = \bm{\mathcal{S}}^\star \times_1 \bm{M}_1^\star \times_2 \bm{M}_2^\star \times_3 \bm{M}_3^\star$. 
Here, $\bm{\mathcal{S}}^\star \in \bbR^{k_1 \times k_2 \times k_3}$ 
stands for the core tensor,  and $\bm{M}_i^\star \in \{0, 1\}^{n_i \times k_i}$ ($1\leq i\leq 3$) represents a membership matrix satisfying 
\begin{align}\label{eq:membership_matrix}
	\left(\bm{M}_i^\star\right)_{j, \ell} = \begin{cases}
		1, \quad & \text{if}~z_{i, j}^\star = \ell,\\
		0, \quad & \text{else}.
	\end{cases}
\end{align} 
%
The goal is to recover the cluster assignment vectors $\{\bm{z}_i^\star\}$, or equivalently, the membership matrices $\{\bm{M}_i^\star\}$, based on the observation $\bm{\mathcal{Y}}$.

We would like to immediately single out two special cases of the above model, 
which represent two distinctive types of noise models of important practical value (see more discussions in \cite{han2022exact}). 
%
%
\begin{itemize}
	\item[1.] \emph{Sub-Gaussian tensor block models.} 
	In this scenario, the entries of the noise tensor $\bm{\mathcal{E}}$ are zero-mean sub-Gaussian random variables 
	generated independently, which capture, say, random contamination during the data collection process for measuring $\bm{\mathcal{X}}^\star$.  
	Importantly, we allow the noise to be heteroskedastic --- namely, the variance of the noise components can be location-varying --- 
	a remarkable extension of the one studied in \cite{han2022exact} (recall that the noise is assumed to be i.i.d.~therein).  
	
	\item[2.] \emph{Stochastic tensor block models.} 
	Consider another scenario where each entry of the core tensor $\bm{\mathcal{S}}^\star$ falls within $[0, 1]$, 
	and the observed entries $\{Y_{i, j, \ell}\}$ are independent Bernoulli random variables satisfying
	\begin{align}\label{model:stochastic_tensor_block}
		Y_{i, j, \ell} = \begin{cases}
			1,~\quad~&\text{with probability}~S^\star_{z_{1, i}^\star, z_{2, j}^\star, z_{3, \ell}^\star},\\
			0,~\quad~&\text{with probability}~1 - S^\star_{z_{1, i}^\star, z_{2, j}^\star, z_{3, \ell}^\star}.
		\end{cases}
	\end{align}
	This scenario can be understood as a generalization of the classical bipartite stochastic block model (e.g., \citet{florescu2016spectral}).   
	Informally, each binary variable $Y_{i, j, \ell}$ encodes whether there is a hyper-edge connecting the vertices $(i,j,\ell)\in[n_1]\times [n_2]\times [n_3]$, 
	and the probability that such a hyper-edge is present is determined  by the clusters they belong to. 
	Clearly, this noise model is, in general, heteroskedastic.  
\end{itemize}
We seek to develop a suite of theory and algorithms that can readily accommodate these two important scenarios.

\subsection{Assumptions and definitions} 
Next, let us introduce a couple of definitions and assumptions that will be used throughout this paper. Before proceeding, we find it helpful to define
\begin{align}\label{def:n_k}
	n = \max_{1 \leq i \leq 3}n_i,~\quad~\text{and}~\quad~k = \max_{1 \leq i \leq 3}k_i,
\end{align}
and denote
\begin{align}\label{def:variance_max}
	\omega_{i_1,i_2,i_3}^2 \coloneqq \bbE\left[E_{i_1, i_2, i_3}^2\right]~\quad~\text{and}~\quad~\omega_{\sf max}^2 \coloneqq \max_{i_1 \in [n_1], i_2 \in [n_2], i_3 \in [n_3]}\omega_{i,j,k}^2.
\end{align}

We start by imposing the following assumption on the noise tensor $\bm{\mathcal{E}}$. 
\begin{assump}\label{assump:noise}
	Suppose that the noise components satisfy the following conditions: 
	\begin{itemize}
		\item[1.] The $E_{i_1,i_2,i_3}$'s are independent and zero-mean;
		\item[2.] For every $(i_1,i_2,i_3)$, one has $\bbP(|E_{i_1,i_2,i_3}| > B) \leq n^{-24}$ for some quantity $B$ satisfying
		$$B \leq C_{\sf b}\omega_{\sf max}\frac{\left(n_1n_2n_3\right)^{1/4}}{\log n},$$
		where  $C_{\sf b} > 0$ is some universal constant.
	\end{itemize}
\end{assump}
\begin{rmk}
	Here, the tail condition $\bbP(|E_{i_1,i_2,i_3}| > B) \leq n^{-24}$ can be relaxed to $\bbP(|E_{i_1,i_2,i_3}| > B) \leq n^{-c}$ for any constant $c \geq 4$. 
	We choose the exponent 24 to streamline the presentation of the proof a little bit. 
\end{rmk}
%

Notably, Assumption~\ref{assump:noise} is very mild and accommodates a broad range of scenarios of interest. 
For example, all $\omega_{\sf max}$-sub-Gaussian random variables (see, e.g., \cite{vershynin2018high}) 
satisfy Condition 2 of Assumption \ref{assump:noise} with $B \asymp \omega_{\sf max}\sqrt{\log n}$;   
centered Poisson random variables also easily satisfy this condition \citep{boucheron2013concentration,zhang2020non}. 
In addition, the aforementioned stochastic tensor block model \eqref{model:stochastic_tensor_block} obeys Assumption \ref{assump:noise} as long as the following conditions hold:
\begin{equation}
	\frac{2\log^2 n}{C_{\sf b}^2(n_1n_2n_3)^{1/2}} \leq S_{i_1, i_2, i_3}^\star \leq 1 - \frac{2\log^2 n}{C_{\sf b}^2(n_1n_2n_3)^{1/2}},~\quad~\forall (i_1, i_2, i_3) \in [k_1] \times [k_2] \times [k_3].
	\label{eq:S-i123-range}
\end{equation}
Recognizing that $(n_1n_2n_3)^{1/2} \gg \log^2 n$, we see that the validity of Condition~\eqref{eq:S-i123-range} is guaranteed 
as long as the entries of the center tensor $\bm{\mathcal{S}}^\star$ are not extremely close to $0$ or $1$.  

In addition, we introduce the following parameter that reveals cluster size information. 
\begin{Definition}[Balance of cluster sizes]\label{assump:cluster_size}
	Let $\beta \leq 1$ denote the largest quantity such that
	%
	%
	\begin{align}\label{def:beta}
		\big|\big\{j \in [n_i]: \left(\bm{z}_i^\star\right)_j = \ell\big\}\big| \geq \beta n_i/k_i, \qquad 
		1\leq i\leq 3, \quad 1\leq \ell \leq k_i.
	\end{align}
\end{Definition}
\noindent 
In words, the parameter $\beta$ measures how balanced these cluster sizes are, with a larger $\beta$ indicating more balanced cluster sizes; 
for instance, $\beta=1$ corresponds to the scenario where all clusters are of the same size.

Another quantity that plays an important role in our theory is concerned with the separation condition. 
\begin{Definition}[Separation]
	\label{defn-separation}
	For any $1\leq i\leq 3$, define
	\begin{align}
		\Delta_i^2 \coloneqq \min_{1 \leq j_1 \neq j_2 \leq k_i}\Big\| \big[ \mathcal{M}_i\left(\mathcal{S}^\star\right) \big]_{j_1,:} - \big[\mathcal{M}_i\left(\mathcal{S}^\star\right)\big]_{j_2,:}\Big\|_2^2
		\label{eq:defn-Delta-i}
	\end{align}
	to be the minimum distance between the rows of $\mathcal{M}_i(\mathcal{S})$, the $i$-th matrizication of the core tensor. 
	We can also define
	\begin{align}
		\Delta_{\sf min}^2 \coloneqq \min\left\{\Delta_1^2, \Delta_2^2, \Delta_3^2\right\}.
		\label{eq:defn-Delta-min}
	\end{align}
\end{Definition}
\noindent 
In a nutshell, the above separation between clusters captures the ``signal strength,'' 
which determines the degree of noise variability that can be tolerated without compromising the feasibility of exact clustering. 
In contrast to generic low-rank tensor estimation problems where the signal strength is typically represented by the least singular value of $\bm{\mathcal{S}}^\star$, 
the above separation condition is more natural in capturing the differentiability between two different clusters. 
Noteworthily,  having a desirable separation condition does not necessarily imply that the least singular value of $\bm{\mathcal{S}}^\star$ is sufficiently large.

Armed with the above separation metrics, we can readily introduce the following quantity to quantify the ``signal-to-noise ratio'':
\begin{align}\label{def:snr}
	\mathsf{SNR} \coloneqq \Delta_{\sf min}/\omega_{\sf max}.
\end{align}
An ideal tensor clustering algorithm would allow for exact clustering for the widest possible range of SNRs.


	\section{Algorithm: \Algoname}\label{sec:algorithm}

In this section, we introduce the proposed procedure for tensor clustering.  
Akin to other spectral-method-based tensor clustering schemes, our algorithm begins by performing subspace estimation with the aid of matricization, 
followed by an application of the (approximate) $k$-means algorithm to estimate clustering assignment.


\subsection{Stage 1: subspace estimation via \Algosubspace}

First of all, we would like to estimate the ``important" column subspaces of
$\bm{\mathcal{X}}$ after matricization along each dimension, 
namely, the important column subspace of $\bm{X}_i = \mathcal{M}_i\left(\bm{\mathcal{X}}\right)$ for each $1\leq i\leq 3$. 
It is noteworthy that: it might not be necessary to estimate the entire rank-$k_i$ column subspace for $\mathcal{M}_i\left(\bm{\mathcal{X}}\right)$, 
given that those singular vectors corresponding to overly small singular values might only exert a negligible impact on the final clustering outcome. 
Instead, for each $1\leq i \leq 3$, 
it often suffices to find a suitable estimator $\bm{U}_i \in \mathcal{O}^{n_i, r_i}$, for some $r_i \leq k_i$, 
that can reliably estimate the subspace formed by the singular vectors associated with large enough singular values of $\bm{X}_i$. 
As alluded to previously, however, the vanilla SVD-based approach (i.e., directly computing the SVD of $\bm{X}_i$) might result in unsatisfactory subspace estimation results when the noise is heteroskedastic. 
This motivates us to develop a more sophisticated algorithm, inspired by our recent work \cite{zhou2023deflated}.


\paragraph{Review:~{\sf Deflated-HeteroPCA}.} 
The recently proposed {\sf Deflated-HeteroPCA} algorithm is particularly effective in subspace estimation in the face of heteroskedastic noise \citep{zhou2023deflated}, 
which we briefly review here. 
%
%
Let $\bm{Y}_i = \mathcal{M}_i\left(\bm{\mathcal{Y}}\right)$ denote the $i$-th matricization of $\bm{Y}_i$. 
Starting from a diagonal-deleted gram matrix $\bm{G}_0 = \mathcal{P}_{\sf off\text{-}diag}\left(\bm{Y}_i\bm{Y}_i^\top\right)$, the main idea of {\sf Deflated-HeteroPCA} is to sequentially choose ranks $0 < r_1 < \cdots < r_{k_{\sf max}} = r$ that divide the eigenvalues of $\bm{X}_i^\star\bm{X}_i^{\star\top}$ into ``well-conditioned" and sufficiently separated subblocks, 
and progressively improve the estimation accuracy. 
Informally, we sequentially incorporate new subblocks into consideration and invoke the {\sf HeteroPCA} algorithm \citep{zhang2022heteroskedastic} to gradually improve 
the estimation accuracy of both the column subspace and the diagonal entries of $\bm{X}_i^\star\bm{X}_i^{\star\top}$. 
%
%
As proven in \cite{zhou2023deflated}, {\sf Deflated-HeteroPCA} enjoys theoretical guarantees that are condition-number-free and accommodate the widest possible range of SNRs, 
all of which are appealing for the tensor clustering application. 
%
%
However, existing theory of {\sf Deflated-HeteroPCA} requires the least singular value of $\bm{X}_i^\star$ to exceed a certain level (depending on the noise variance), 
which might oftentimes be unnecessary for clustering applications. 
In fact, even in the presence of a large separation metric \eqref{eq:defn-Delta-min}, 
we cannot preclude the possibility of  $\bm{X}_i^\star$ having a (nearly) zero singular value, 
thus limiting the utility of {\sf Deflated-HeteroPCA}.

	
	\paragraph{Proposed procedure: \Algosubspace.}To address the aforementioned issue, we incorporate a thresholding procedure into {\sf Deflated-HeteroPCA} in order to make sure we only include important subblocks. More specifically, suppose that in the $k$-th round, we choose rank $r_k$ and perform {\sf HeteroPCA} with rank $r_k$ and initialization $\bm{G}_{k-1}$ to obtain $\bm{G}_k$, 
	where both $\bm{G}_{k-1}$ and $\bm{G}_{k}$ are intermediate estimates of the gram matrix $\bm{X}_i^\star\bm{X}_i^{\star\top}$. 
	We then decide whether to proceed to the $(k+1)$-th round based on whether the condition $\sigma_{r_k+1}(\bm{G}_k) > \tau$ is met for some pre-determined threshold $\tau$. With a properly chosen $\tau$, we can extract sufficient information needed for clustering. 
	The details of \Algosubspace~can be found in Algorithm~\ref{algorithm:sequential_heteroPCA}. 
	A theoretically-guided procedure for selecting the tuning parameter $\tau$ is deferred to Section~\ref{sec:tau_selection}. 

\begin{algorithm}[t]
	\caption{\Algosubspace($\bm{Y}, r, \tau, \{t_j\}_{j \geq 1}$)} \label{algorithm:sequential_heteroPCA}
	\DontPrintSemicolon
	\textbf{input:} data matrix $\bm{Y}$, rank $r$, threshold $\tau$, maximum numbers of iterations $\{t_j\}_{j\geq 1}$.\\
	\textbf{initialization:} $j = 0, r_0 = 0, \bm{G}_0 = \mathcal{P}_{\sf off\text{-}diag}\left(\bm{Y}\bm{Y}^\top\right)$.\\
	{\color{blue}\tcc{sequentially invoke HeteroPCA until  eigenvalues fall below the threshold.}}
	\While{$r_j < r$ {\rm and} $\sigma_{r_j+1}(\bm{G}_j) > \tau$}{
		$j \leftarrow j+1$.\\
		compute $r_j = \mathsf{Rank Selection}(\bm{G}_{j-1}, r, r_{j-1})$.\\
		$\left(\bm{G}_j, \bm{U}_j\right) = ${\sf HeteroPCA}$\left(\bm{G}_{j-1}, r_j, t_j\right)$.		
	}
	\textbf{output:} subspace estimate $\bm{U} = \bm{U}_j$.
\end{algorithm}

\begin{algorithm}[t]
	\caption{{\sf RankSelection($\bm{G}_{j-1}$, $r$, $r_{j-1}$)}} \label{algorithm:rank_selection}
	\DontPrintSemicolon 
	\textbf{input:} rank $r$, selected rank $r_{j-1}$, matrix $\bm{G}_{j-1}$.\\
	{\color{blue}\tcc{identify a subblock of eigenvalues that are well-conditioned and well-separated from the remaining eigenvalues.}}
	\textbf{output} rank 
	\vspace{-1em}
	\begin{align}
		r_j = \begin{cases}
			\max\mathcal{R}_j, &\text{ if } \mathcal{R}_j \neq \emptyset,\\
			r, & \text{ otherwise},\end{cases}\label{eq:rank_selection}
	\end{align}
	where 
	\begin{align}
		\mathcal{R}_j \coloneqq \bigg\{r': r_{j-1} < r' \leq r,~ \frac{ \sigma_{r_{j-1}+1}\left(\bm{G}_{j-1}\right) }{\sigma_{r'}\left(\bm{G}_{j-1}\right) } \leq 4\notag~\text{and}~ \sigma_{r'}\left(\bm{G}_{j-1}\right) - \sigma_{r'+1}\left(\bm{G}_{j-1}\right) \geq \frac{1}{r}\sigma_{r'}\left(\bm{G}_{j-1}\right)\bigg\}.
	\end{align}
\end{algorithm}

\begin{algorithm}[t]
	\caption{{\sf HeteroPCA($\bm{G}_{\mathsf{in}}$, $r$, $t_{\sf max}$)} ~\citep{zhang2022heteroskedastic}} \label{algorithm:heteroPCA}
	\DontPrintSemicolon
	\textbf{input:} symmetric matrix $\bm{G}_{\mathsf{in}}$, rank $r$, number of iterations $t_{\sf max}$.\\
	\textbf{initialization:} $\bm{G}^{0} = \bm{G}_{\mathsf{in}}$.\\
	\For{$t = 0, 1, \dots, t_{\sf max}$} 
	{
		$\bm{U}^{t}\bm{\Lambda}^{t}\bm{U}^{t\top}$ $\,\leftarrow\,$  rank-$r$ leading eigendecompostion of $\bm{G}^t$.\\
		$\bm{G}^{t+1} = \mathcal{P}_{\sf off\text{-}diag}\left(\bm{G}^{t}\right) + \mathcal{P}_{\sf diag}\left(\bm{U}^{t}\bm{\Lambda}^{t}\bm{U}^{t\top}\right)$.\\
	}
	\textbf{output:} matrix estimate $\bm{G} = \bm{G}^{t_{\sf max}}$ and subspace estimate $\bm{U} = \bm{U}^{t_{\sf max}}$.    
\end{algorithm}

\subsection{Stage 2: approximate $k$-means}

Having obtained the subspace estimates $\bm{U}_j, 1\leq j\leq 3$, we would like to extract clustering information based on these subspace estimates as well as the observation $\bm{\mathcal{Y}}$. 
More specifically, let us construct the following matrix
\begin{equation}
	\widehat{\bm{B}}_i = \bm{U}_i\bm{U}_i^\top\mathcal{M}_i({\bm{\mathcal{Y}}})\big(\bm{U}_{i+2} \otimes \bm{U}_{i+1}\big) \in \bbR^{n_i \times (r_1r_2r_3/r_i)}
\end{equation}
for each $1\leq i\leq 3$, 
and we propose to 
apply the (approximate) $k$-means algorithm on the rows of $\widehat{\bm{B}}_i $ to estimate the cluster assignment vectors. 
Here, the indices $i+1$ and $i+2$ are computed module 3. 
This procedure is equivalent to applying (approximate) $k$-means on the rows of the $i$-th matricization of the tensor estimate $\widehat{\bm{\mathcal{Y}}} = \bm{\mathcal{Y}} \times_1 \bm{U}_1\bm{U}_1^\top \times_2 \bm{U}_2\bm{U}_2^\top \times_3 \bm{U}_3\bm{U}_3^\top$, but operates upon matrices with significantly reduced sizes. 
Recognizing that performing exact $k$-means can be computationally intractable, 
we instead employ an $M$-approximate $k$-means approach for some $M > 1$. To be more specific, we find a cluster assignment vector estimate $\widehat{z}_i$ and centroids $\{\widehat{\bm{b}}_j^{(i)}\}$ that satisfy \eqref{ineq:k_means_approximate} for each $1\leq i\leq 3$. 
This relaxed version of $k$-means can be efficiently solved by a number of algorithms. For example, 
for $M = O(\log k)$, this problem can be solved with running time $O(nk^3)$ by $k$-means++ \citep{bahmani2012scalable,arthur2007k}\footnote{More generally, the time complexity of $k$-means++ for $n$ points in $\mathbb{R}^d$ is $O(nkd)$. Here, the dimension $d = r_1r_2r_3/r_i \leq k^2$ and thus the time complexity does not exceed $O(nk^3)$.}. 
The requirement \eqref{ineq:minimum_label} ensures the ``optimality" of the cluster estimates given the centroids.

\subsection{Full procedure} 

The full procedure of the proposed \Algoname~(\Algoabbrev) is summarized in Algorithm~\ref{algorithm:k_means}. Since both \Algosubspace~and approximate $k$-means are polynomial-time algorithms, our proposed method is computationally efficient. Empirically, we recommend using the {\sf high-order Lloyd Algorithm (HLloyd)} \citep{han2022exact} to further refine the clustering results obtained by \Algoabbrev, 
where the description of this refinement procedure is deferred to Algorithm~\ref{algorithm:HLloyd} in Appendix~\ref{sec:procedure_HLloyd}. As will be demonstrated in Section~\ref{sec:simulation}, the combined application of our algorithm with {\sf HLloyd} can yield superior empirical performance when compared to other methods. 


\begin{algorithm}[t]
	\caption{\Algoname~(\Algoabbrev)} \label{algorithm:k_means}
	\textbf{input:} observed tensor $\bm{\mathcal{Y}}$, numbers of clusters $k_1, k_2, k_3$, numbers of iterations $\{t_{i,j}\}_{1\leq i\leq 3, j \geq 1}$, thresholds $\tau_1, \tau_2, \tau_3$, relaxation factor $M > 1$.\\
	{\color{blue}\tcc{Stage 1: subspace estimation}}
	\textbf{subspace estimation:} for each $1\leq i\leq 3$, compute $\bm{U}_i \in \mathcal{O}^{n_i, k_i}$ as follows
	$$\bm{U}_i = \begin{cases}
		\Algosubspace\big(\mathcal{M}_i({\bm{\mathcal{Y}}}), k_i, \tau_i, \{t_{i,j}\}_{j \geq 1}\big), \qquad &k_i \geq 2,\\
		(1, \dots, 1)^\top/\sqrt{n_i}, &k_i = 1,
	\end{cases} $$ 
	and set $\bm{U}_4=\bm{U}_1$ and $\bm{U}_5=\bm{U}_2$ for convenience. \\
	{\color{blue}\tcc{Stage 2: approximate $k$-means}}
	\For{$i = 1, 2, 3$}{
    	compute $\widehat{\bm{B}}_i = \bm{U}_i\bm{U}_i^\top\mathcal{M}_i({\bm{\mathcal{Y}}})\big(\bm{U}_{i+2} \otimes \bm{U}_{i+1}\big) \in \bbR^{n_i \times (r_1r_2r_3/r_i)}.$\\
	perform approximate $k$-means on the rows of $\widehat{\bm{B}}_i$, i.e., find a cluster assignment vector estimate $\widehat{\bm{z}}_i \in [k_i]^{n_i}$ and center estimates $\widehat{\bm{b}}^{(i)}_l \in \bbR^{r_1r_2r_3/r_i}, l \in [k_i]$ such that 
    	\begin{subequations}
    		\begin{align}
    			\sum_{j=1}^{n_i}\left\|\big(\widehat{\bm{B}}_i\big)_{j,:}^\top - \widehat{\bm{b}}^{(i)}_{\widehat{z}_{i,j}}\right\|_2^2 &\leq M\min_{\bm{b}_1, \dots, \bm{b}_{k_i} \in \bbR^{r_1r_2r_3/r_i}\atop \bm{z}_i \in [k_i]^{n_i}}\sum_{j=1}^{n_i}\left\|\big(\widehat{\bm{B}}_i\big)_{j,:}^\top - \bm{b}_{z_{i,j}}\right\|_2^2,\label{ineq:k_means_approximate}\\
    				\widehat{z}_{i,j} &\in \argmin_{\ell \in [k_i]}\left\|\big(\widehat{\bm{B}}_i\big)_{j,:}^\top - \widehat{\bm{b}}^{(i)}_{\ell}\right\|_2,
				\qquad \forall j\in[n_i].\label{ineq:minimum_label}
    		\end{align}
    	\end{subequations}
}
\textbf{output:} estimates $\widehat{\bm{z}}_1, \widehat{\bm{z}}_2, \widehat{\bm{z}}_3$ of cluster assignment vectors.
\end{algorithm}

     \section{Main theory}\label{sec:theory}

In this section, we develop theoretical performance guarantees for our algorithm proposed in Section~\ref{sec:algorithm}. 
Before proceeding, we find it convenient to introduce the following additional notation for each $1\leq i\leq 3$: 
\begin{itemize}
	\item $\sigma_{i,j}^\star$: the $j$-th largest singular value of $\mathcal{M}_i(\bm{\mathcal{S}}^\star)$ (i.e., the $i$-th matricization of $\bm{\mathcal{S}}^\star$); 
	\item $\{r_{i,j}\}_{j \geq 1}$: the ranks selected in Algorithm~\ref{algorithm:sequential_heteroPCA} with the input matrix $\bm{Y} = \mathcal{M}_i(\bm{\mathcal{Y}})$, the rank $r = k_i$, the threshold $\tau=\tau_i$, and the numbers of iterations $\{t_{i,j}\}_{j\geq 1}$; 
	\item $r_{i, j_{\sf max}^i}$: the largest rank selected by Algorithm~\ref{algorithm:sequential_heteroPCA} with the above inputs.

\end{itemize}
%


 \subsection{Theoretical guarantees for exact clustering}
 The first theorem below demonstrates that \Algoabbrev~achieves intriguing exact recovery guarantees, as long as the tuning parameters $\{\tau_i\}_{1\leq i\leq 3}$ are suitably selected. 
 \begin{thm}\label{thm:cluster_tensor_simplified}
	 Assume that $k_i \lesssim 1$ for each $1\leq i\leq 3$, $\beta \asymp 1$, and
 	\begin{subequations}
 		\begin{align}
 			n_1n_2n_3 &\geq c_1n^2, \label{ineq:dimension_assumption_simplified}\\
 			c_{\tau}\left(n_1n_2n_3\right)^{1/2}\log^2 n \leq \tau_i/\omega_{\sf max}^2 &\leq C_{\tau}\left(n_1n_2n_3\right)^{1/2}\log^2 n,~\quad~\forall 1\leq i \leq 3,\label{ineq:threshold_tensor_simplified}\\
 			\mathsf{SNR}~= \Delta_{\sf min}/\omega_{\sf max} &\geq C_1\sqrt{M}\left(n_1n_2n_3\right)^{-1/4}\log n,\label{ineq:signal_to_noise_ratio_simplified}
 		\end{align}
 	\end{subequations}
 where $C_1, c_1, C_{\tau}$ and $c_{\tau}$ are some large enough positive constants. 
	Suppose Assumption \ref{assump:noise} holds. 
 	If the numbers of iterations satisfy
 	\begin{subequations}
 		\begin{align}
 			t_{i,j} &\geq \log\left(C\frac{\sigma_{i, r_{i, j-1}+1}^{\star2}}{\sigma_{i, r_{i, j}+1}^{\star2}}\right),~\quad~1 \leq j \leq j_{\sf max}^i - 1,\label{iter1_simple}\\
 			t_{i,j_{\sf max}^i} &\geq \log\left(n^3\frac{\sigma_{i, r_{i, j_{\sf max}^i-1}+1}^{\star2}}{\omega_{\sf max}^2}\right),\label{iter2_simple}
 		\end{align}
 	\end{subequations}
	 for all $1\leq i \leq 3$ with $C > 0$ some large enough constant, then with probability exceeding $1 - O(n^{-10})$, the misclassification rate (cf.~\eqref{eq:defn-MCR}) of the outputs $\{\widehat{\bm{z}}_i\}$ returned by Algorithm \ref{algorithm:k_means} satisfy
 	\begin{align*}
 		\mathsf{MCR}\big(\widehat{\bm{z}}_i, \bm{z}_i^\star\big) = 0,~\quad~\forall 1\leq i\leq 3.
 	\end{align*}
 \end{thm}
 In words, this theorem asserts that Algorithm \ref{algorithm:k_means} enables exact clustering under the assumptions imposed above.  
A more general version of Theorem~\ref{thm:cluster_tensor_simplified} --- which allows $k_i$ and $\beta$ grows with $n$ --- as well as its proof can be found in Section~\ref{sec:proof_thm_cluster_tensor_simplified}.

Let us now take a moment to discuss the conditions assumed in Theorem~\ref{thm:cluster_tensor_simplified}. 
Condition~\eqref{ineq:dimension_assumption_simplified} assumes that the dimensions of the observed tensor are not extremely unbalanced; 
for instance,  it holds in the scenario where $n_1 \lesssim n_2n_3$, $n_2 \lesssim n_1n_3$ and $n_3 \lesssim n_1n_2$. 
In the regime where $n_1 \asymp n_2 \asymp n_3 \asymp n$  and $M \asymp \log k \asymp 1$, the signal-to-noise ratio condition~\eqref{ineq:signal_to_noise_ratio_simplified} simplifies to 
\begin{equation}
	\mathsf{SNR} \gtrsim n^{-3/4}\log n. \label{eq:SNR-simpler}
\end{equation}
Interestingly, this condition \eqref{eq:SNR-simpler} is almost necessary among polynomial-time algorithms; as shown in \citet[Theorem 7]{han2022exact}, if $\mathsf{SNR}~= n^{\gamma}$ for any $\gamma < -3/4$, then there exists no polynomial-time algorithm that can exactly recover the cluster assignment vectors. Furthermore, combining Theorem \ref{thm:cluster_tensor_simplified} with \citet[Theorem~2]{han2022exact}, we know that \Algoabbrev + {\sf HLloyd} and \Algoabbrev~can achieve the same theoretical guarantees in terms of the exact cluster recovery. In other words, applying {\sf HLloyd} to further refinement would not degrade the theoretical performance at all. As we will illustrate in Section \ref{sec:simulation}, \Algoabbrev~combined with {\sf HLloyd} might sometimes achieve improved empirical results compared to \Algoabbrev~on its own.

\paragraph{Comparisons with {\sf HSC} and {\sf HSC} + {\sf HLloyd}.} To highlight the advantages of our algorithm and theory, we make comparisons with the state-of-the-art prior work \cite{han2022exact}, 
which proposed the {\sf HSC} algorithm and its combination with a follow-up {\sf HLloyd} procedure. 
Firstly, in stark contrast to {\sf HSC} and {\sf HSC} + {\sf HLloyd} --- which assumes {\em identical variances} of the noise entries in order to guarantee their desired theoretical results \citet[Theorems 3 and 4]{han2022exact} --- our algorithm \Algoabbrev~is able to handle heteroskedastic noise efficiently without compromising the applicable range of SNRs. Secondly, in comparison with \citet[Theorems 3 and 4]{han2022exact} that assume sub-Gaussian noise, our assumption is more mild and can accommodate a wider range of applications, including those with binary or count data outcomes. In addition, \Algoabbrev~can exactly recover the cluster assignment vectors without the aid of {\sf HLloyd}, a feature that stands in contrast to {\sf HSC}.

\paragraph{Other prior results.} In addition to \cite{han2022exact}, the tensor block model has been studied in several other past works. Focusing on sub-Gaussian noise, \cite{wang2019multiway} characterized the misclassification rate and the tensor estimation error for the the least-square estimator; this estimator, however, is computationally intractable. \cite{chi2020provable} proposed a convex method and provided theoretical guarantees for the tensor estimation error,  but they did not establish misclassification-rate-based theory for their proposed method. \cite{agterberg2022estimating} further investigated a more general model, called the tensor mixed-membership block model, in the presence of sub-Gaussian noise. 
When applied to Model \eqref{eq:model} with $k \asymp 1$ and $\beta \asymp 1$, their signal-to-noise ratio condition becomes $$\Delta/\omega_{\sf max} \gtrsim \kappa^2\frac{n\sqrt{\log n}}{(n_1n_2n_3)^{1/2}(\min_{i}n_i)^{1/4}},$$ where $\kappa$ is the condition number of the tensor $\mathcal{S}^\star$ satisfying $\kappa \lesssim (\min_{i}n_i)^{1/8}$. In comparison, Theorem \ref{thm:cluster_tensor_simplified} does not require any assumptions on $\kappa$ and our signal-to-noise ratio condition is less stringent.

\paragraph{A glimpse of proof highlights.} To prove that \Algoabbrev~alone is enough to achieve exact clustering, 
a crucial step lies in carefully controlling the magnitude of $\|(\bm{I} - \bm{U}_i\bm{U}_i^\top)\bm{X}_i^\star\|_{2,\infty}$. Here, $\bm{X}_i^\star = \mathcal{M}_i\left(\bm{\mathcal{X}}^\star\right)$ is the $i$-th matricization of $\bm{\mathcal{X}}^\star$, and $\bm{U}_i$ is the subspace estimator. Unlike the subspace/matrix estimation problems considered in the literature, we aim to derive sharp and condition-number-free $\ell_{2,\infty}$ guarantees for $\left(\bm{I} - \bm{U}_i\bm{U}_i^\top\right)\bm{X}_i^\star$ without imposing any restrictions on the condition number and the least singular value of $\bm{X}_i^\star$. In this context, existing techniques for deriving $\ell_{2,\infty}$ guarantees are inadequate for reaching our target bound; 
for instance, existing leave-one-out analyses \citep{zhong2018near, chen2021spectral, ma2020implicit} require the condition number to not be overly large, 
whereas the subspace representation theorem \citep{xia2021normal,zhou2023deflated} relies on assumptions on the least singular value.  

To establish the desired performance guarantees, we develop a new technique (i.e., Lemma \ref{lm:space_estimate_expansion}) that allows for effective control of $\|(\bm{I} - \bm{U}_i\bm{U}_i^\top)\bm{X}_i^\star\|_{2,\infty}$ via bounding an infinite sum of $\ell_{2, \infty}$ norms of polynomials of the error matrix $\bm{E}_i = \mathcal{M}_i(\bm{\mathcal{E}})$, alongside other terms that can be easily bounded. Using a strategy akin to, but more intricate than, the one used in \cite{zhou2023deflated}, we are able to control those $\ell_{2, \infty}$ norms of the error polynomials and, in turn, achieve the desired guarantees.

\subsection{Data-driven selection of the thresholds $\{\tau_i\}$}\label{sec:tau_selection}

As shown in Theorem \ref{thm:cluster_tensor_simplified}, \Algoabbrev~can successfully recover the cluster assignment vectors of interest if the tuning parameters $\{\tau_i\}$ satisfy \eqref{ineq:threshold_tensor_simplified}. However, the maximum variance $\omega_{\sf max}^2$ is usually unknown {\em a priori} and, therefore, we need to carefully choose $\tau_i$. In what follows, we discuss how to select these tuning parameters.

Without loss of generality, we assume $n_1 \leq n_2 \leq n_3$ and  let
\begin{align}\label{def:sigma_hat}
	\widehat{\omega} = \sigma_{k_1 + 1}\big(\mathcal{M}_1\left(\bm{\mathcal{Y}}\right)\big)/\sqrt{n_2n_3},
\end{align}
which can be used to estimate the order of the noise level. Then the threshold $\tau_i$ can be chosen as follows: 
\begin{align}\label{def:tuning_selection}
	\tau_i = \tau = \overline{C}_{\tau}\left(n_1n_2n_3\right)^{1/2}\widehat{\omega}^2\log^2 n = \overline{C}_{\tau}\sqrt{\frac{n_1}{n_2n_3}}\sigma_{k_1 + 1}^2\big(\mathcal{M}_1\left(\bm{\mathcal{Y}}\right)\big)\log^2 n,~\qquad~\forall 1\leq i \leq 3, 
\end{align}
where $\overline{C}_{\tau}>0$ is some sufficiently large constant. The following theorem asserts that the $\tau_i$'s computed in \eqref{def:tuning_selection} satisfy the desired property \eqref{ineq:threshold_tensor_simplified}.
\begin{thm}\label{thm:tuning_selection}
	Suppose that Assumption \ref{assump:noise} holds, and $\min\{n_2/k_2, n_3/k_3\} \geq C\sqrt{\log n}$ holds for some sufficiently large constant $C > 0$. Assume that either of the following conditions is satisfied:
	\begin{itemize}
		\item[1.] For all $i \in [n_1]$, there exists some numerical constant $c > 0$ such that
		\begin{align}\label{assump:variance_same_order}
			\sum_{j = 1}^{n_2}\sum_{\ell=1}^{n_3}\omega_{i,j,\ell}^2 \geq cn_2n_3\omega_{\sf max}^2~\quad~\text{where}~\quad~\omega_{i,j,\ell}^2 = {\sf Var}\left[E_{i,j,\ell}\right];
		\end{align}
	    \item[2.] The observation model is the stochastic tensor block model \eqref{model:stochastic_tensor_block}, with the numbers of clusters satisfying $k_i \lesssim 1$ for each $1\leq i \leq 3$ and the balance parameter obeying $\beta \asymp 1$.
	\end{itemize}
	Then with probability exceeding $1 - O(n^{-10})$, the thresholds $\{\tau_i\}$ defined in \eqref{def:tuning_selection} satisfy \eqref{ineq:threshold_tensor_simplified}.
\end{thm}

\begin{rmk}
	Here, Condition \eqref{assump:variance_same_order} posits that the average variance for each row of $\mathcal{M}_1\left(\bm{\mathcal{E}}\right)$ is on the same order as $\omega_{\sf max}^2$. This condition is met when the noise is not excessively spiky. For example, any noise tensor $\mathcal{E}$ with variances $\omega_{i,j,k}^2 \asymp \omega_{\sf max}^2$ satisfies this condition.
\end{rmk}

The proof of Theorem \ref{thm:tuning_selection} can be found in Section \ref{proof:tuning_selection}. Putting Theorem \ref{thm:cluster_tensor_simplified} and Theorem \ref{thm:tuning_selection} together, we arrive at the following result:
\begin{thm}\label{thm:final}
	Suppose that Assumption \ref{assump:noise} holds, $k_i \lesssim 1$ for every $1\leq i \leq 3$, and $\beta \asymp 1$. Assume that either the following conditions is satisfied:
	\begin{itemize}
		\item[1.] Condition \eqref{assump:variance_same_order} holds;
		\item[2.] The observation model is the stochastic tensor block model \eqref{model:stochastic_tensor_block}.
	\end{itemize}
    We further assume that 
    \begin{align*}
    	n_1n_2n_3 &\geq c_1n^2, \\
    	\mathsf{SNR}~= \Delta_{\sf min}/\sigma_{\sf max} &\geq C_1\sqrt{M}\left(n_1n_2n_3\right)^{-1/4}\log n
    \end{align*}
    for some large enough constants $C_1, c_1 > 0$.
	If we choose the tuning parameter $\tau$ as in \eqref{def:tuning_selection} and the numbers of iterations satisfy \eqref{iter1_simple} and \eqref{iter2_simple}, then with probability exceeding $1 - O(n^{-10})$, \Algoabbrev~achieves exact clustering, i.e., the misclassification rate (cf.~\eqref{eq:defn-MCR}) obeys
    \begin{align*}
    	\mathsf{MCR}\big(\widehat{\bm{z}}_i, \bm{z}_i^\star\big) = 0,~\quad~\forall 1\leq i\leq 3.
    \end{align*}
\end{thm}

Theorem \ref{thm:final} shows that our data-driven procedure can still achieve the exact clustering under the same signal-to-noise ratio condition as in Theorem \ref{thm:cluster_tensor_simplified}, provided that the noise condition \eqref{assump:variance_same_order} is satisfied. If the model of interest is the stochastic tensor block model, no extra assumptions on the noise are needed to justify the validity of the data-driven choices of $\{\tau_i\}$.

    \section{Empirical studies}\label{sec:simulation}

In this section, we conduct a series of numerical experiments to evaluate the practical effectiveness of the proposed algorithms: \Algoabbrev, and \Algoabbrev~+ {\sf HLloyd}. 
Throughout this section, the thresholds are chosen in a data-driven manner as
\begin{align}\label{eq:tuning_selection_empirical}
	\tau_i \equiv \tau = 1.1\left(n_1n_2n_3\right)^{1/2}\widehat{\omega}^2 = 1.1\sqrt{\frac{n_1}{n_2n_3}}\sigma_{k_1 + 1}^2\big(\mathcal{M}_1\left(\bm{\mathcal{Y}}\right)\big),
\end{align}
where $\widehat{\omega}$ is defined in \eqref{def:sigma_hat}.

\subsection{Experiments on synthetic data}
First, we carry out numerical experiments on synthetic data to corroborate the efficacy of \Algoabbrev~and \Algoabbrev~+ {\sf HLloyd}. Following the settings in \cite{han2022exact}, we set the dimensions to be $n_1 = n_2 = n_3 = n$ and the numbers of clusters as $k_1 = k_2 = k_3 = k$, and let the cluster sizes be balanced. The following four methods are considered: (1) {\sf HSC}: the high-order spectral clustering algorithm proposed in \cite{han2022exact}; (2) {\sf HSC + HLloyd}: the procedure that uses {\sf HSC} to obtain initial clustering results, followed by a 10-iteration high-order Lloyd algorithm \citep{han2022exact} for refinement; (3) \Algoabbrev: the method proposed in Algorithm \ref{algorithm:k_means} with the numbers of iterations $t_{i,j} = 10$; (4) \Algoabbrev~+ {\sf HLloyd}: the procedure that employs \Algoabbrev~(where $t_{i,j} = 10$) as initial cluster assignment vector estimators and then applies {\sf HLloyd} with the iteration number $t = 10$ to compute the final clustering results. 
To evaluate the clustering performance,  we calculate, for each method, the empirical clustering error rate (CER), which is one minus the adjusted random index \citep{milligan1986study}. A lower CER indicates a better clustering result. Specifically, an exact recovery of clustering is achieved when CER equals $0$. All results are averaged over 100 independent replicates.

\paragraph{Sub-Gaussian tensor block models.}
Let us begin by considering Model \eqref{eq:model_equivalence} with Gaussian noise. We fix the dimensions $n_1 = n_2 = n_3 = n \in \{100, 150\}$, generate a random tensor $\overline{\bm{\mathcal{S}}} \in \mathcal{R}^{k, k, k}$ with independent entries $\overline{S}_{i_1, i_2, i_3} \sim \mathcal{N}(0,1)$ and the core tensor $\bm{\mathcal{S}}^\star$ is obtained by rescaling $\overline{\bm{\mathcal{S}}}$ such that $\Delta_{\sf min} = 40n^{-\delta}$ (so that {\sf SNR} decreases as $\delta$ increases). We randomly generate the cluster assignment vectors $\bm{z}_i \in [k]^n$, $1\leq i \leq 3$. We generate three vectors $\bm{\alpha}, \bm{\beta}, \bm{\gamma}$ such that $\{\alpha_i\}, \{\beta_j\}, \{\gamma_k\}$ are independently and uniformly drawn from $[0,2]$. The entries of the noise tensor $\bm{\mathcal{E}} \in \bbR^{n \times n \times n}$ are generated independently with $E_{i,j,k} \sim \mathcal{N}(0, \alpha_i^2\beta_j^2\gamma_k^2)$. For each method, we report the averaged CER and the percentage of exact recovery for cluster assignment vectors. The results for $n = 100$ and $n = 150$ are illustrated in Figures~\ref{figure:sub-Gaussian1} and \ref{figure:sub-Gaussian2}, respectively. As can be seen, \Algoabbrev~and \Algoabbrev~+ {\sf HLloyd} achieve much smaller CER compared with {\sf HSC} and {\sf HSC + HLloyd}. In terms of the percentage of exact recovery, \Algoabbrev~+ {\sf HLloyd} achieves the best performance among all these four methods.
\begin{figure}[t]
	\centering
	\begin{minipage}[t]{\linewidth}
		\centering
		\subfigure[$k=3$, averaged CER]{
			\includegraphics[width=0.32\linewidth]{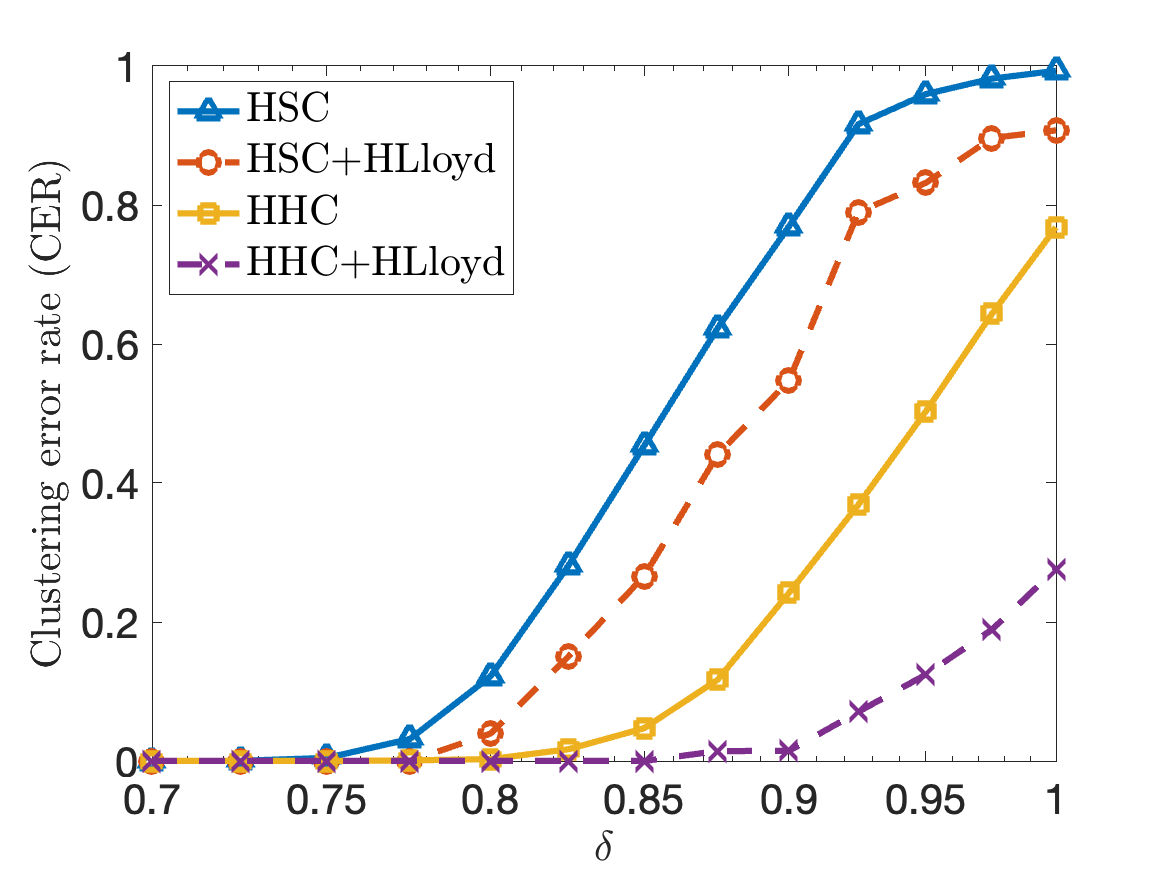}}
		\subfigure[$k=3$, percentage of exact recovery]{
			\includegraphics[width=0.32\linewidth]{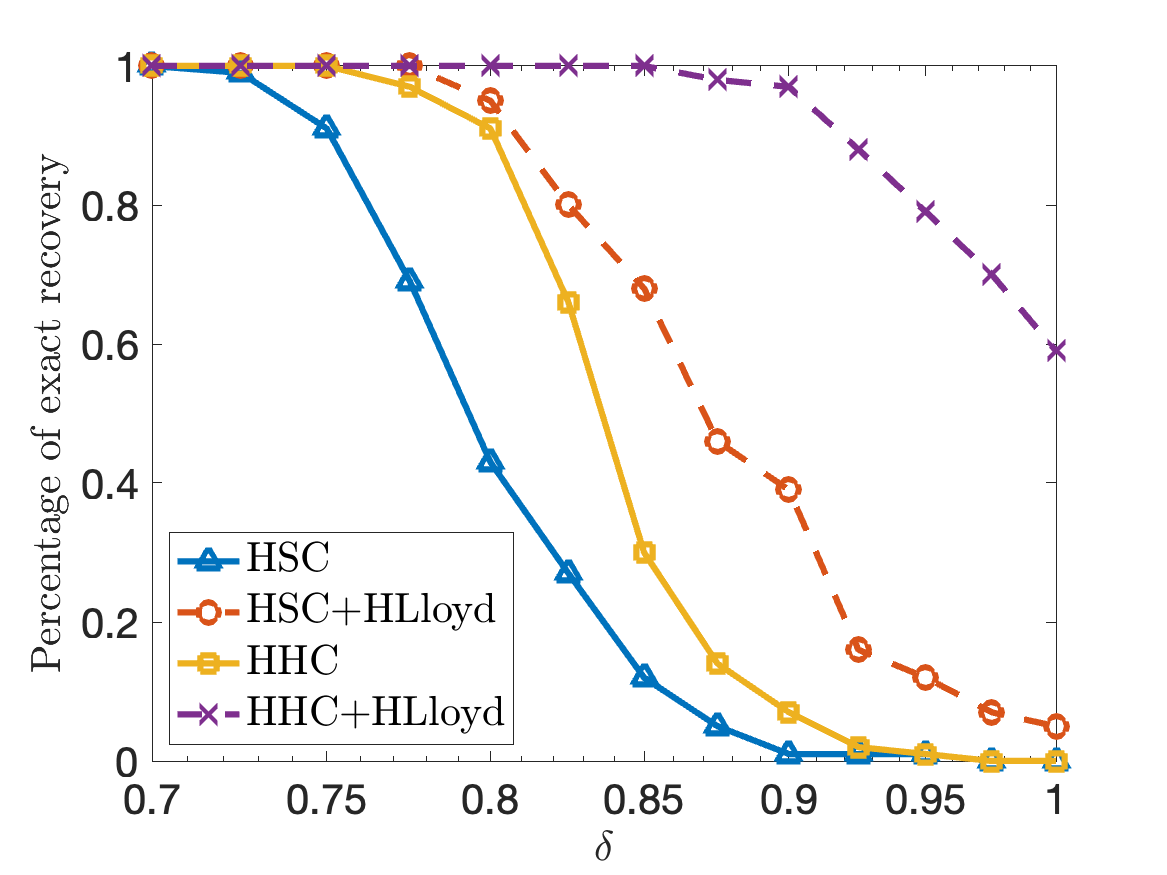}}\\
		\subfigure[$k = 5$, averaged CER]{
			\includegraphics[width=0.32\linewidth]{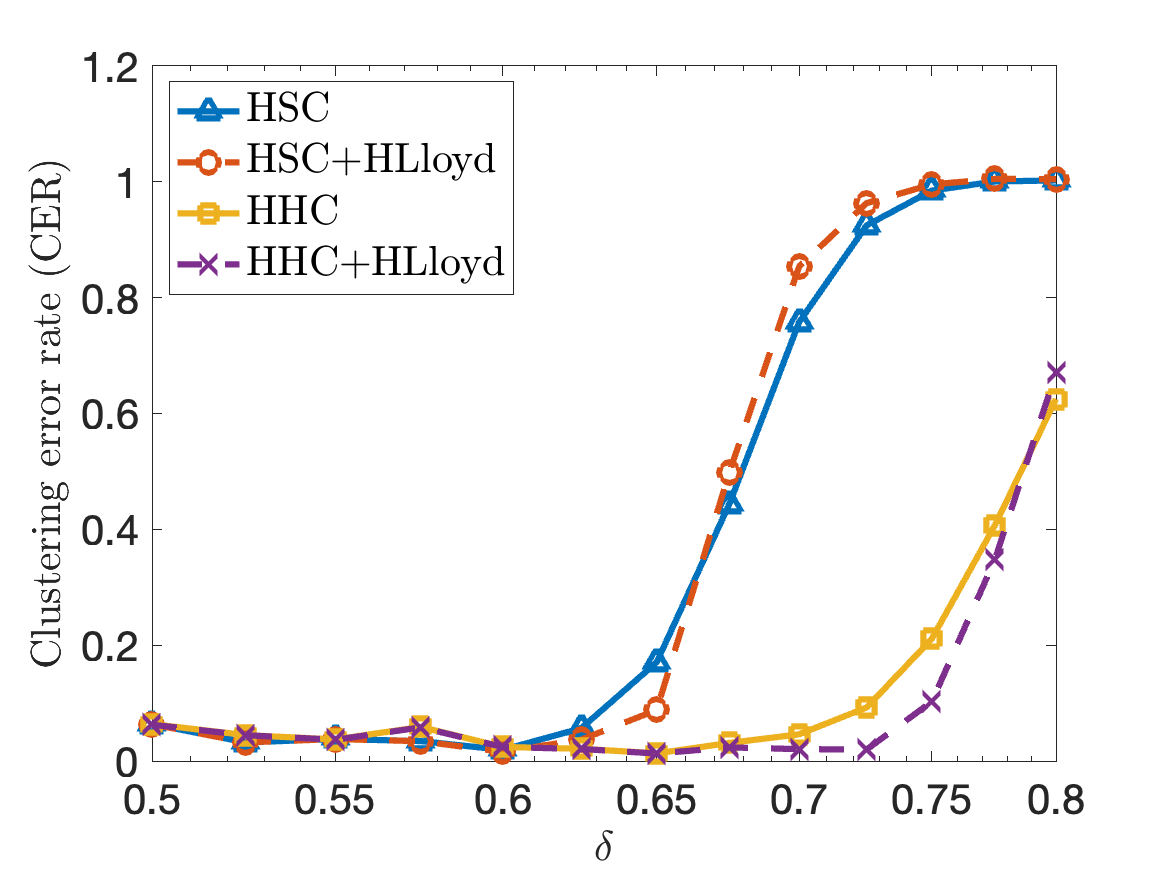}}
		\subfigure[$k = 5$, percentage of exact recovery]{
			\includegraphics[width=0.32\linewidth]{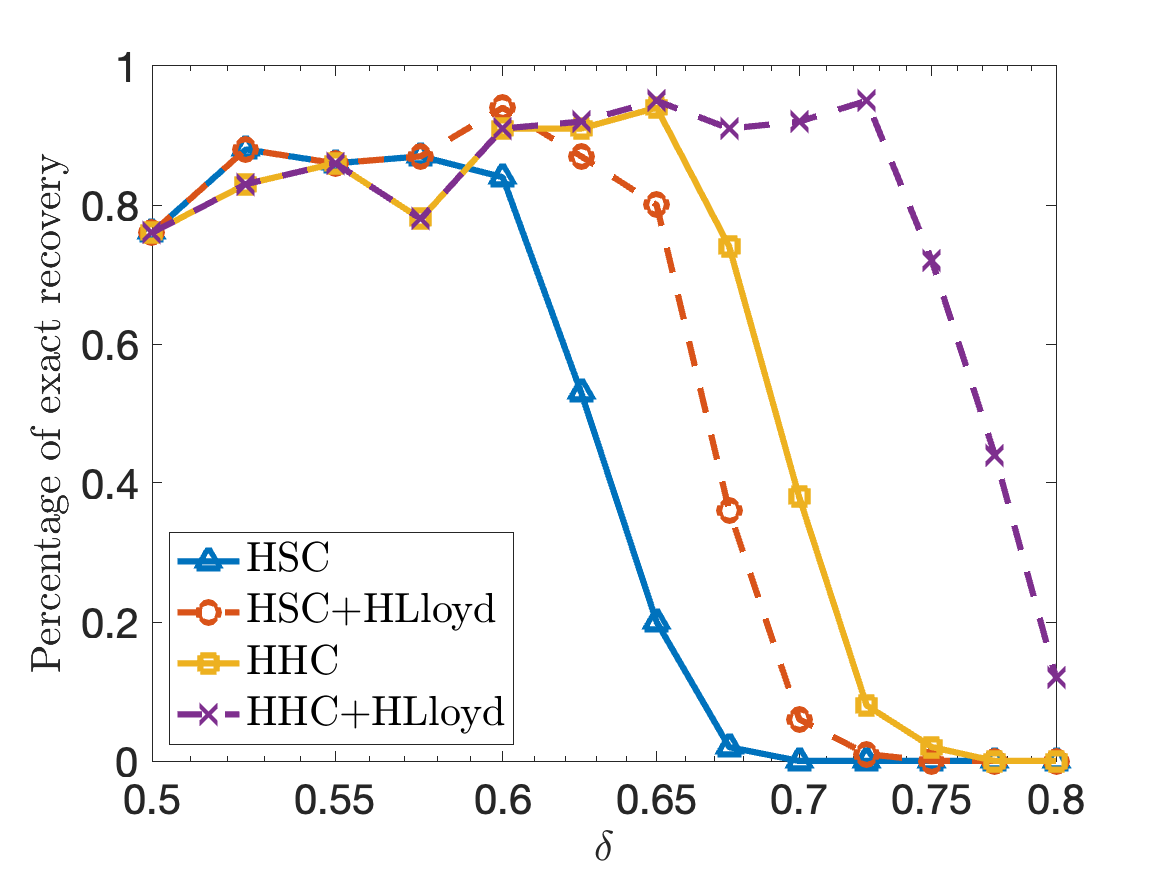}}
	\end{minipage}\caption{Averaged CER and percentage of exact community recovery for {\sf HSC}, {\sf HSC + HLloyd}, \Algoabbrev~and \Algoabbrev~+ {\sf HLloyd} under the sub-Gaussian tensor block models with $n = 100$. Here, $\Delta_{\sf min} = 40n^{-\delta}$.}
\label{figure:sub-Gaussian1}
\end{figure}

\begin{figure}[t]
	\centering
	\begin{minipage}[t]{\linewidth}
		\centering
		\subfigure[$k=3$, averaged CER]{
			\includegraphics[width=0.32\linewidth]{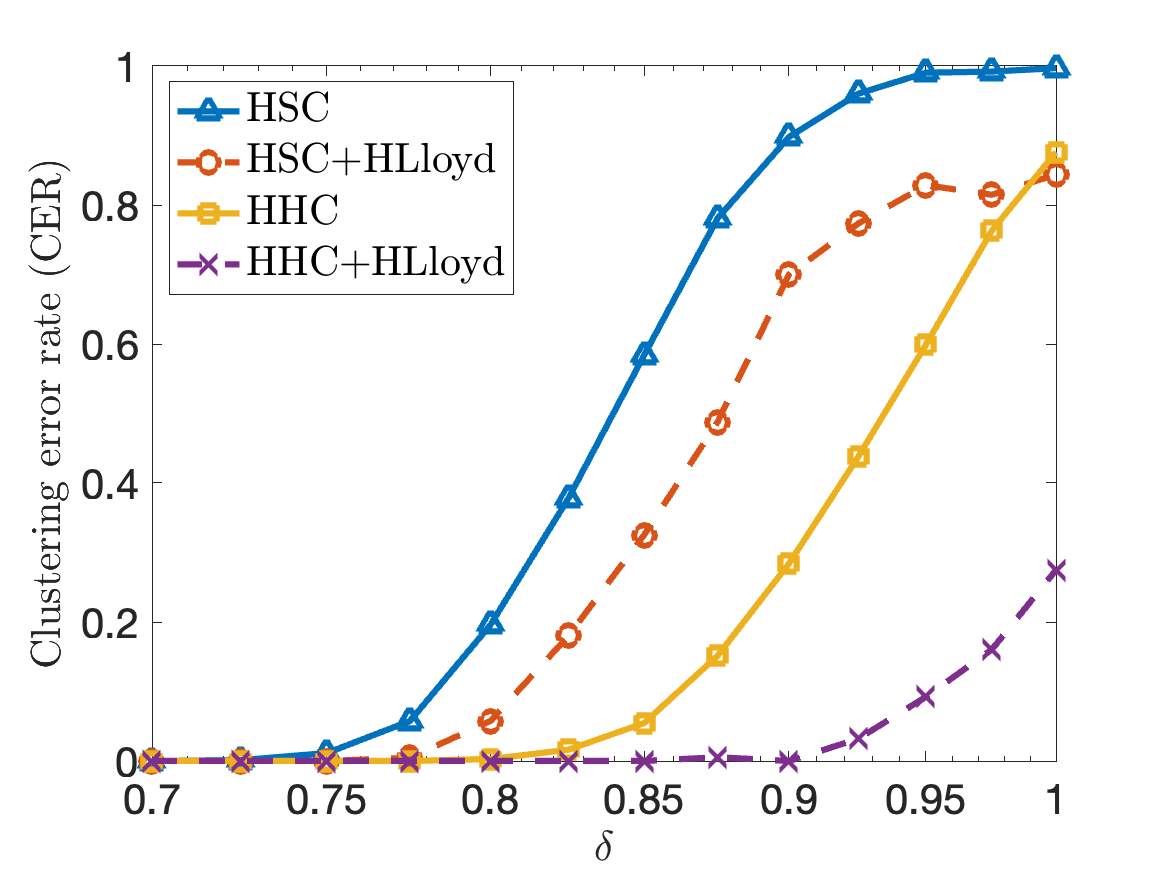}}
		\subfigure[$k=3$, percentage of exact recovery]{
			\includegraphics[width=0.32\linewidth]{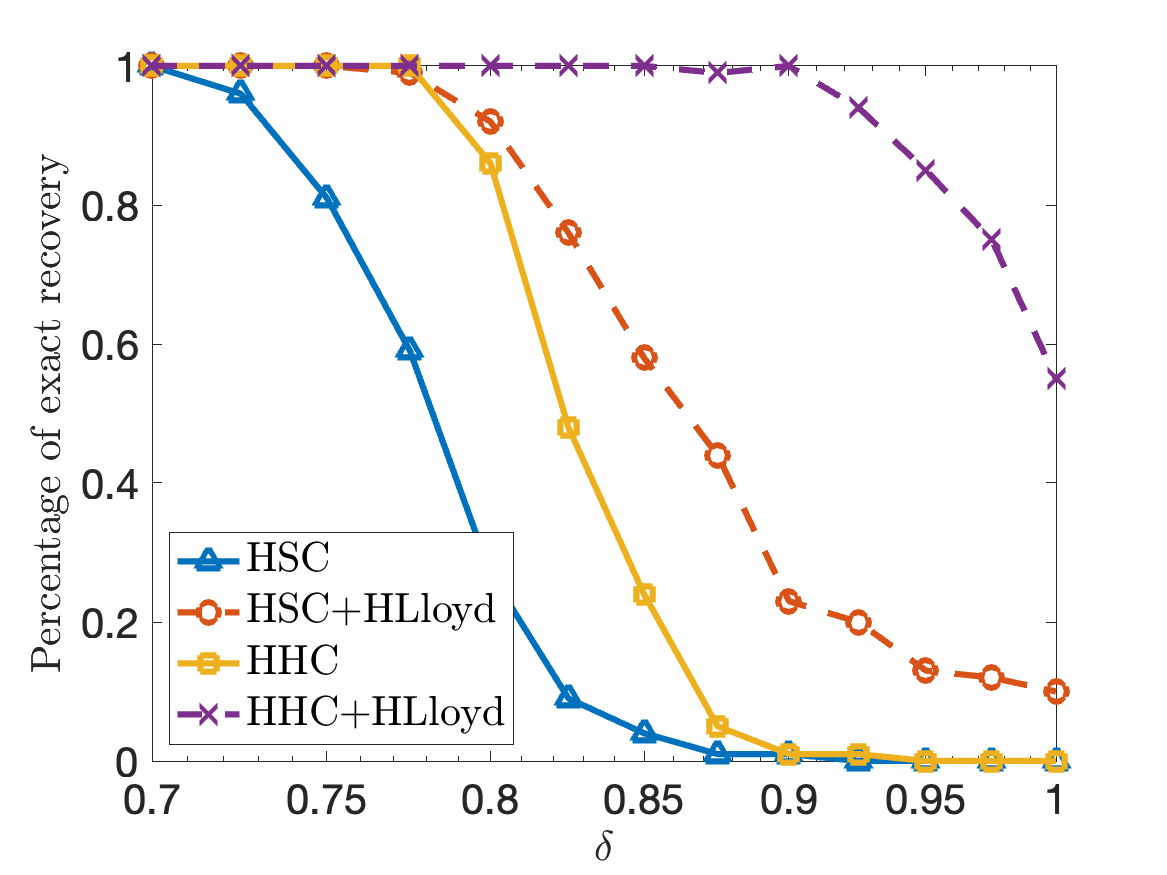}}\\
		\subfigure[$k = 5$, averaged CER]{
			\includegraphics[width=0.32\linewidth]{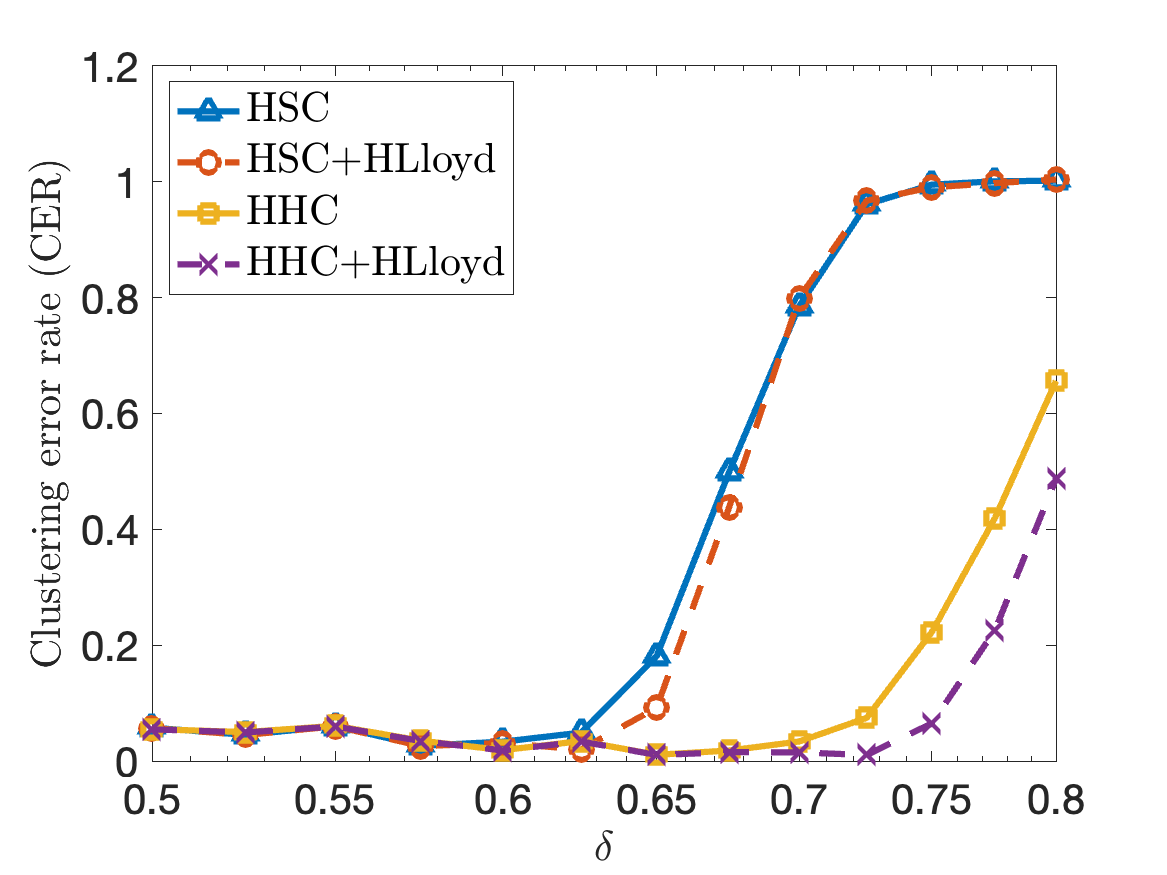}}
		\subfigure[$k = 5$, percentage of exact recovery]{
			\includegraphics[width=0.32\linewidth]{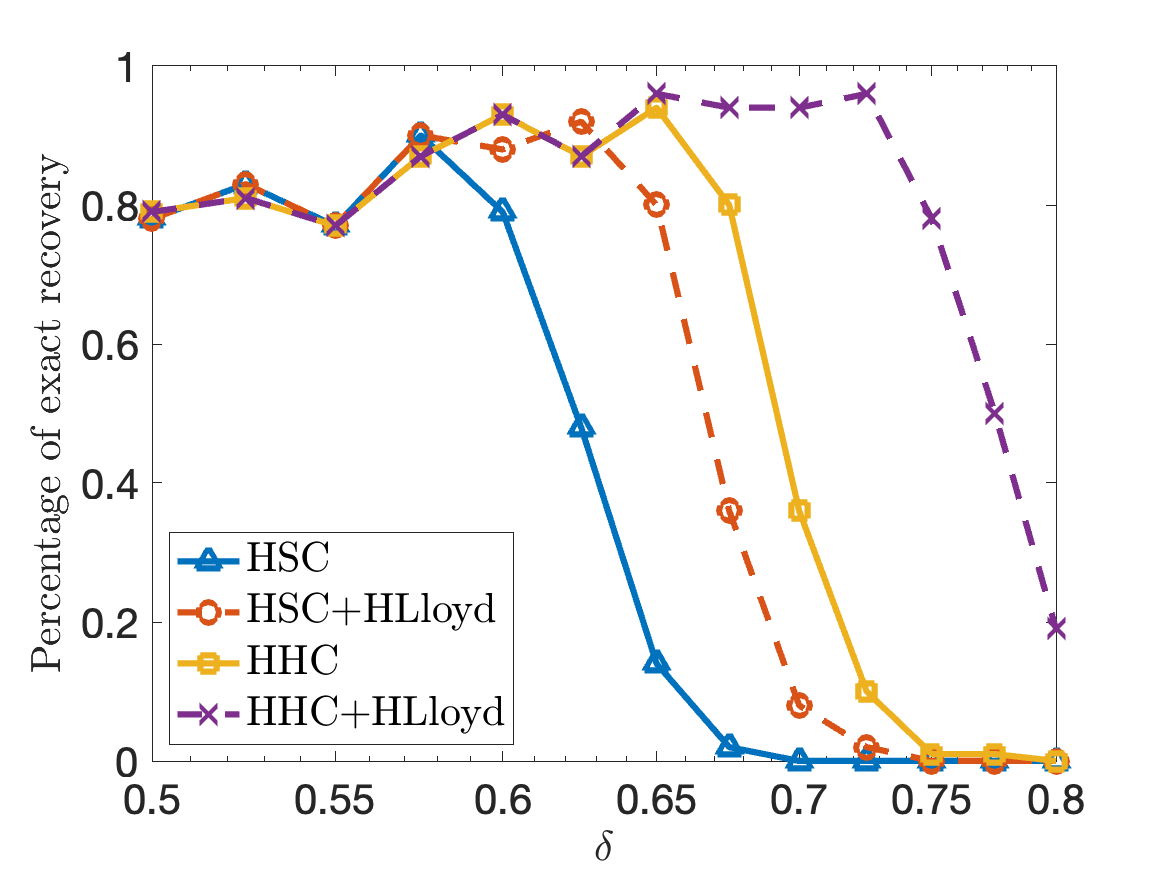}}
	\end{minipage}\caption{Averaged CER and percentage of exact community recovery for {\sf HSC}, {\sf HSC + HLloyd}, \Algoabbrev~and \Algoabbrev~+ {\sf HLloyd} under the sub-Gaussian tensor block models with $n = 150$. Here, $\Delta_{\sf min} = 40n^{-\delta}$.}
\label{figure:sub-Gaussian2}
\end{figure}
\paragraph{Stochastic tensor block models.} Next, we study the stochastic block model \eqref{model:stochastic_tensor_block}. We choose the core tensor $\bm{\mathcal{S}}^\star \in \mathbb{R}^{k \times k \times k}$ satisfying
\begin{align}\label{simulation:STBM_core_tensor}
	S_{i_1,i_2,i_3}^\star = \begin{cases}
		10a\cdot n^{-3/2}\left(1 - \frac{i_1 - 1}{2(k - 1)}\right),~\quad~&i_1 = i_2 = i_3,\\
		0.1a\cdot n^{-3/2},~\quad~&\text{otherwise}.
	\end{cases}
\end{align}
Here, $a$ is a scalar. When $a$ is not too large, {\sf SNR} increases with $a$. The empirical results of the above four methods for $n = 100$ and $n = 150$ are displayed in Figures \ref{figure:stochastic1} and \ref{figure:stochastic2}. From these plots, one sees that \Algoabbrev~and \Algoabbrev~+ {\sf HLloyd} achieve more accurate clustering results, and \Algoabbrev~+ {\sf HLloyd} outperforms all other methods  in achieving the highest percentage of the exact recovery for the cluster assignment vectors. 
\begin{figure}[t]
	\centering
	\begin{minipage}[t]{\linewidth}
		\centering
		\subfigure[$k=3$, averaged CER]{
			\includegraphics[width=0.32\linewidth]{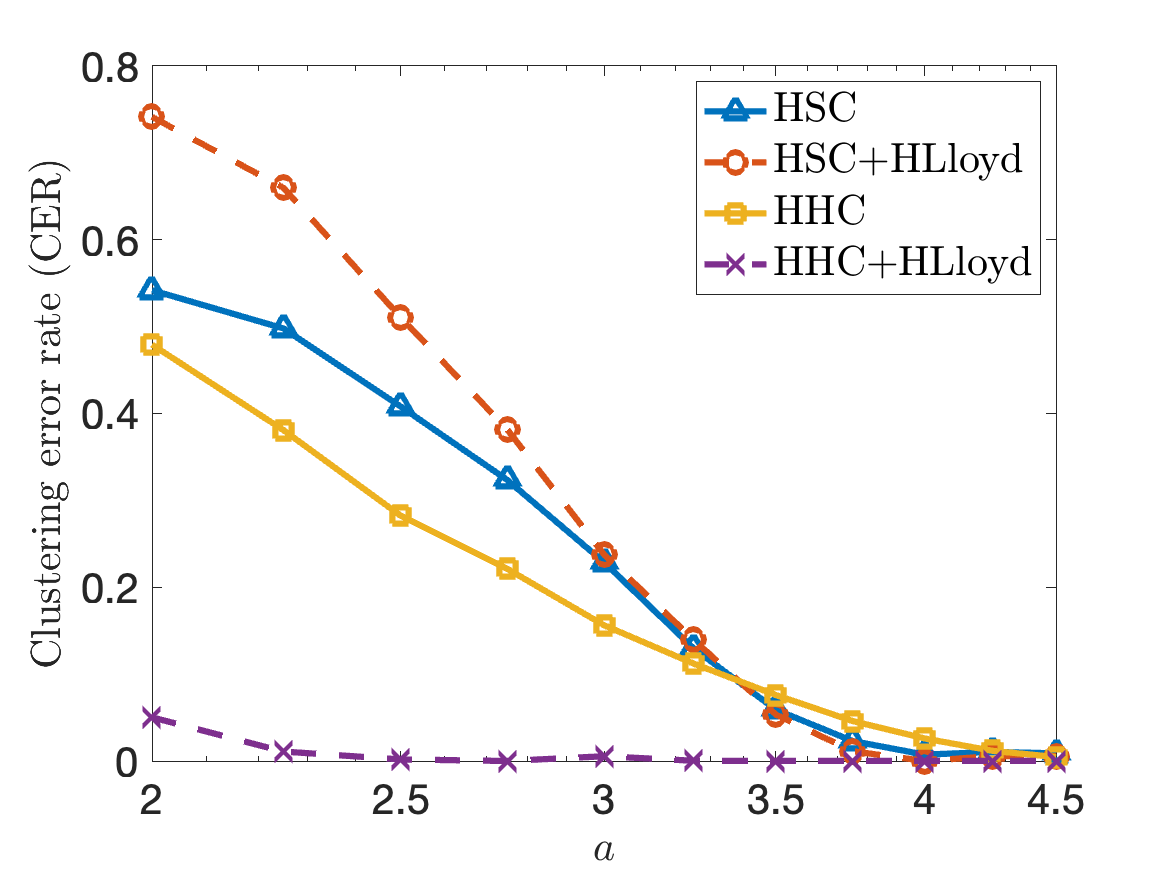}}
		\subfigure[$k=3$, percentage of exact recovery]{
			\includegraphics[width=0.32\linewidth]{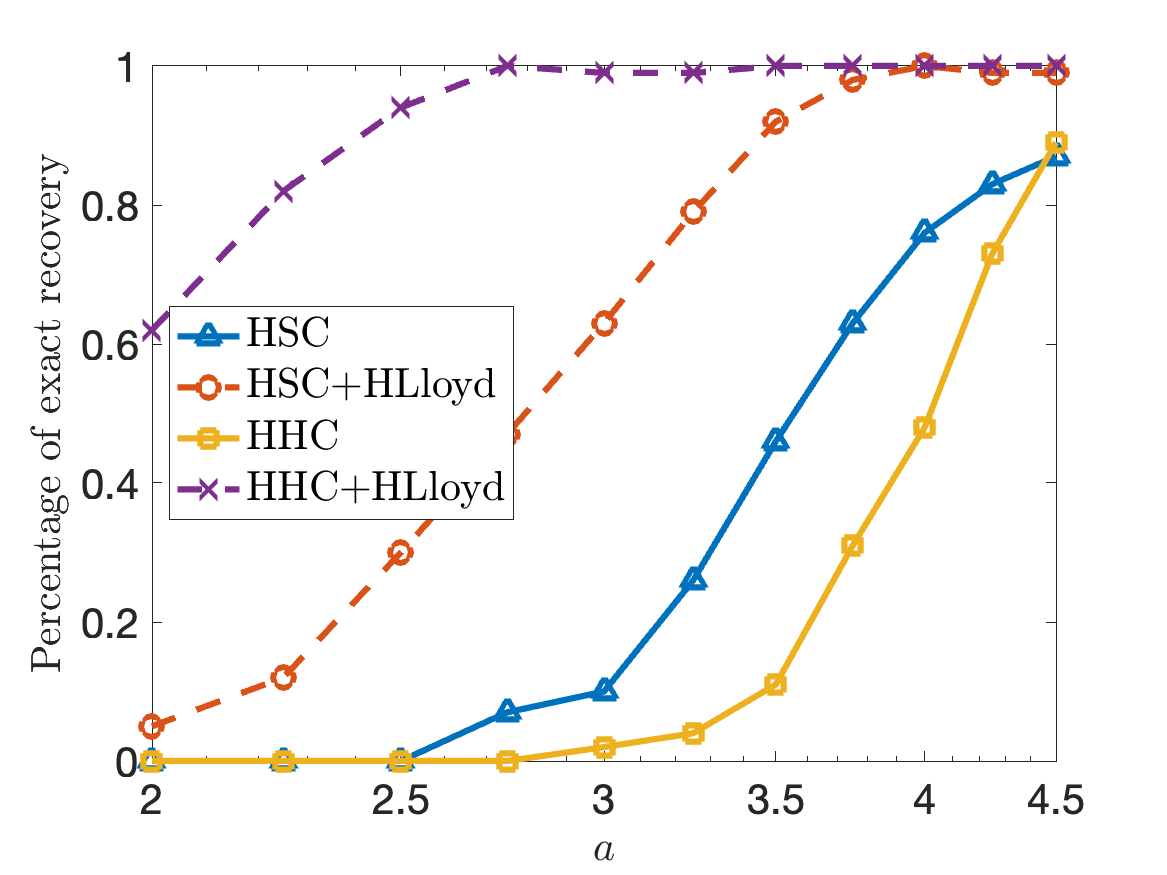}}\\
		\subfigure[$k = 4$, averaged CER]{
			\includegraphics[width=0.32\linewidth]{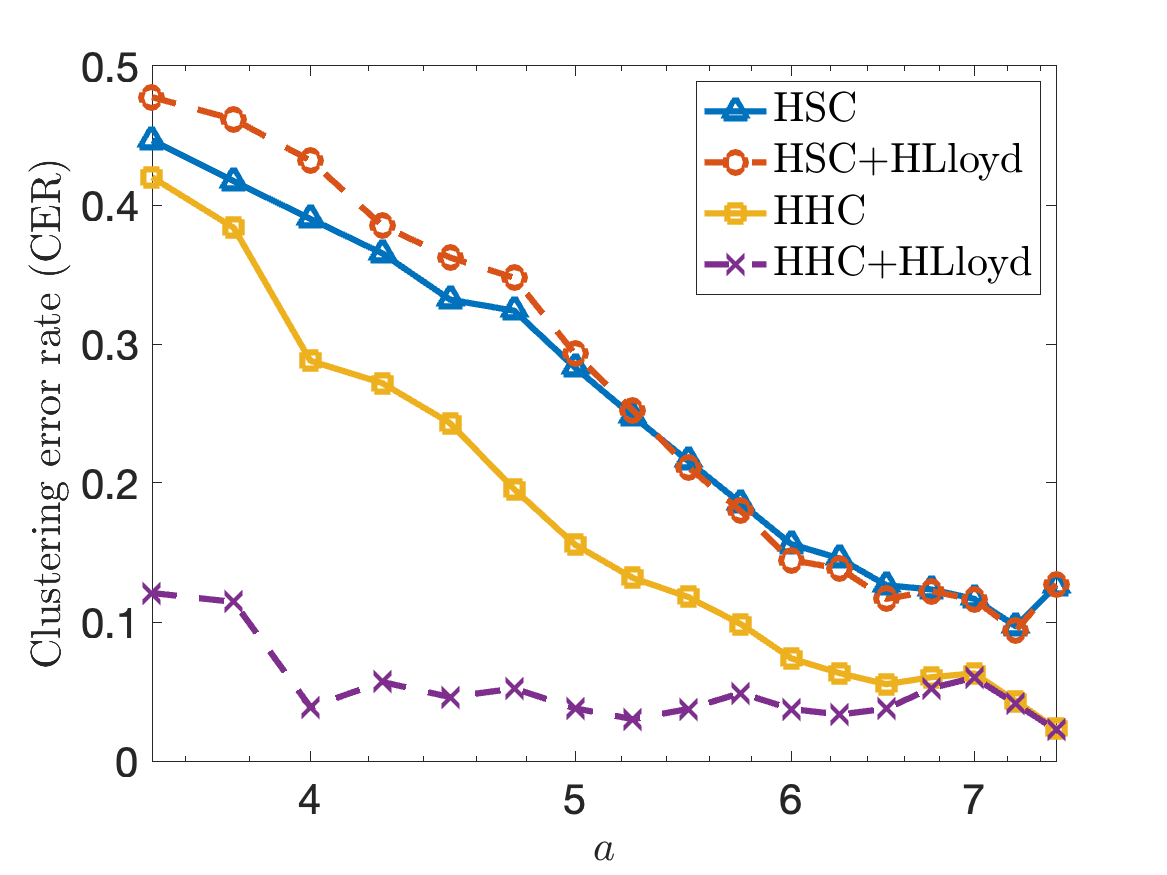}}
		\subfigure[$k = 4$, percentage of exact recovery]{
			\includegraphics[width=0.32\linewidth]{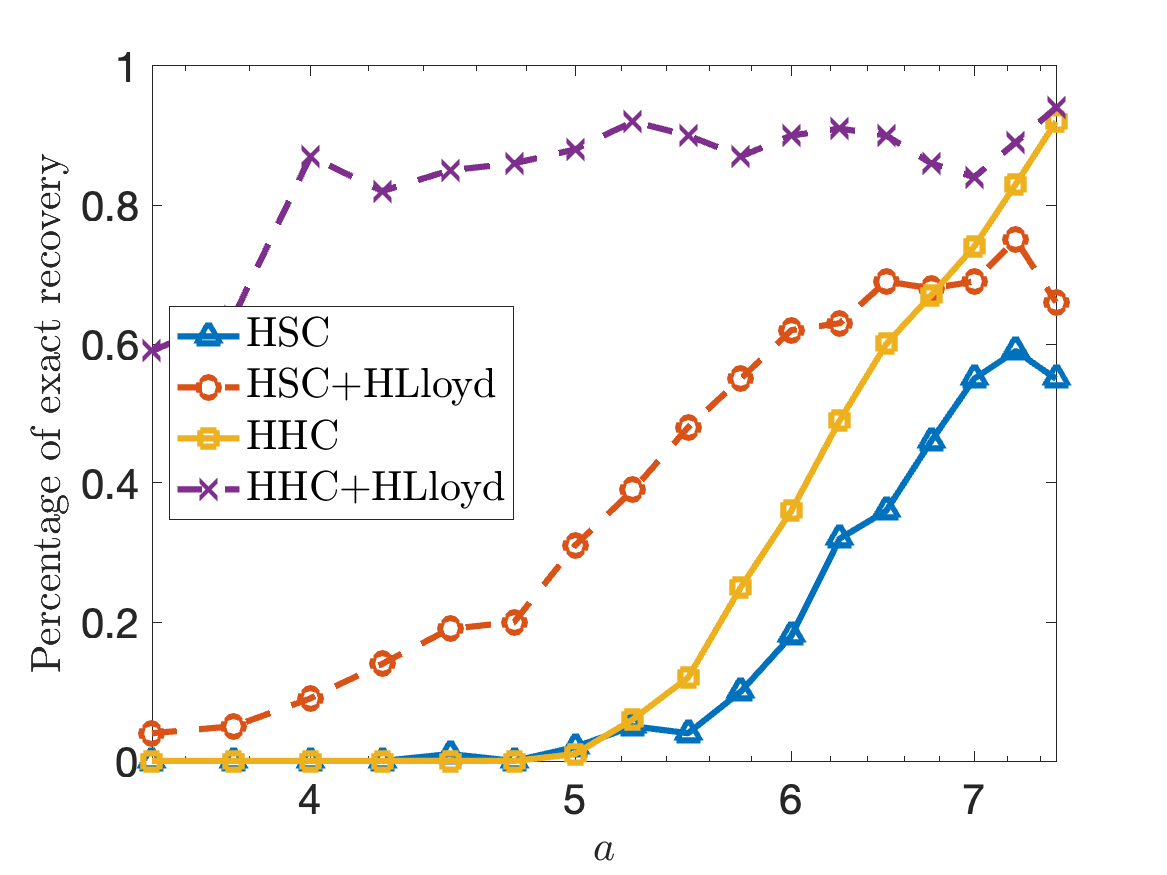}}
	\end{minipage}
    \caption{Averaged CER and percentage of exact community recovery for {\sf HSC}, {\sf HSC + HLloyd}, \Algoabbrev~and \Algoabbrev~+ {\sf HLloyd} under the Stochastic tensor block models with $n = 100$. Here, the quantity $a$ satisfies \eqref{simulation:STBM_core_tensor}.}
    \label{figure:stochastic1}
\end{figure}

\begin{figure}[t]
	\centering
	\begin{minipage}[t]{\linewidth}
		\centering
		\subfigure[$k=3$, averaged CER]{
			\includegraphics[width=0.32\linewidth]{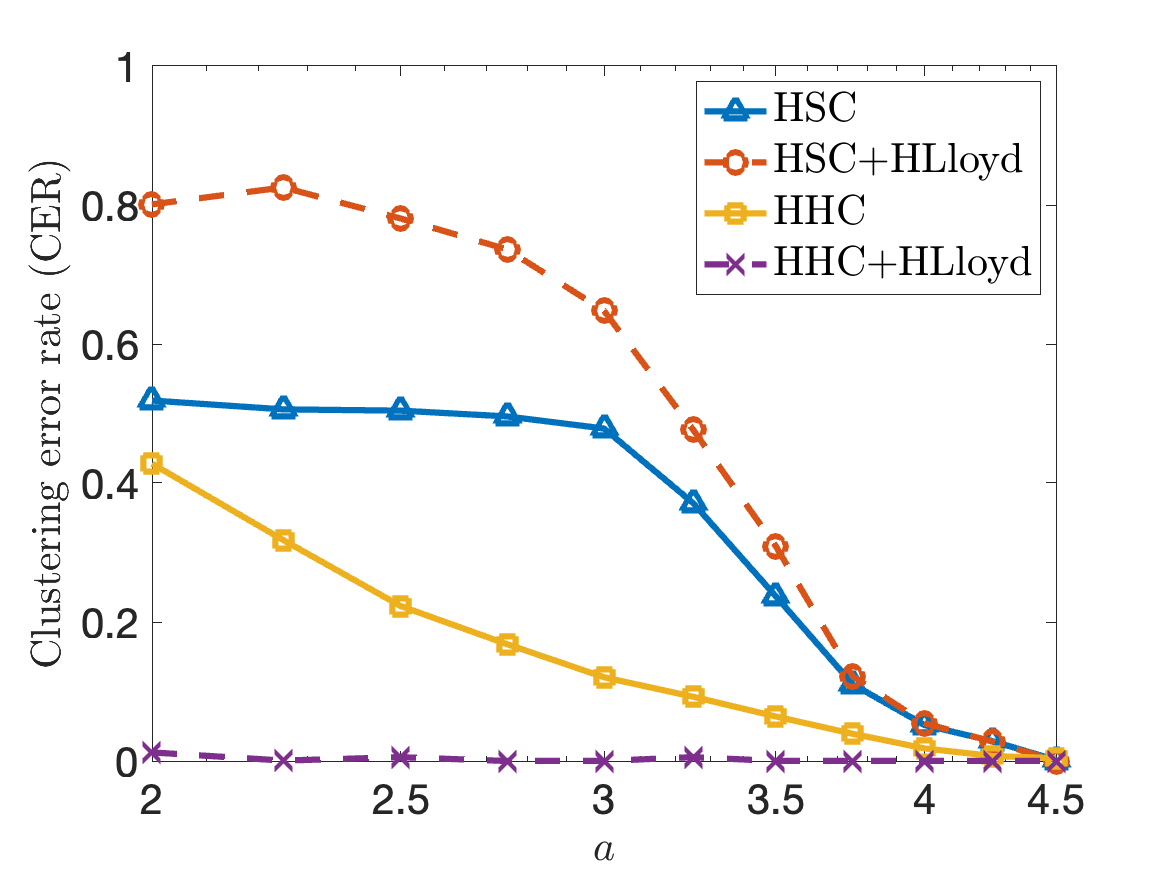}}
		\subfigure[$k=3$, percentage of exact recovery]{
			\includegraphics[width=0.32\linewidth]{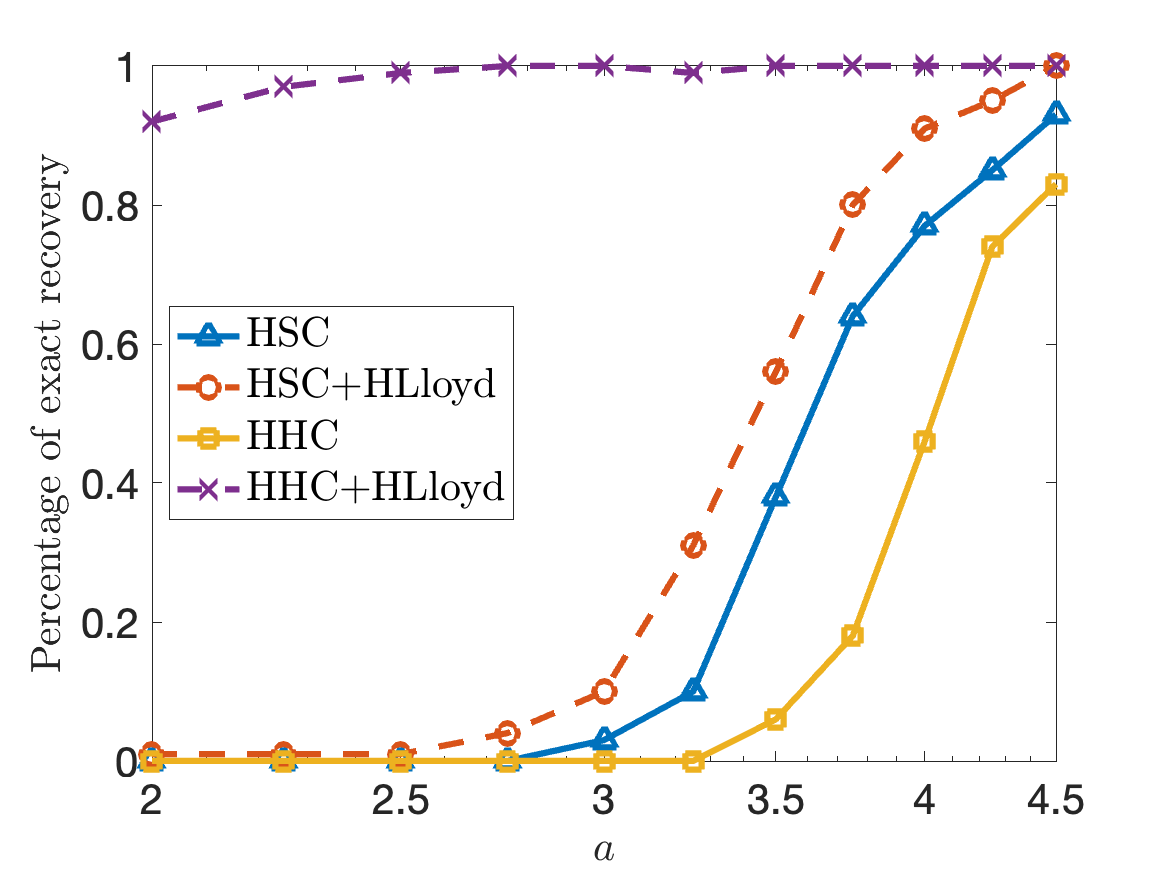}}\\
		\subfigure[$k = 4$, averaged CER]{
			\includegraphics[width=0.32\linewidth]{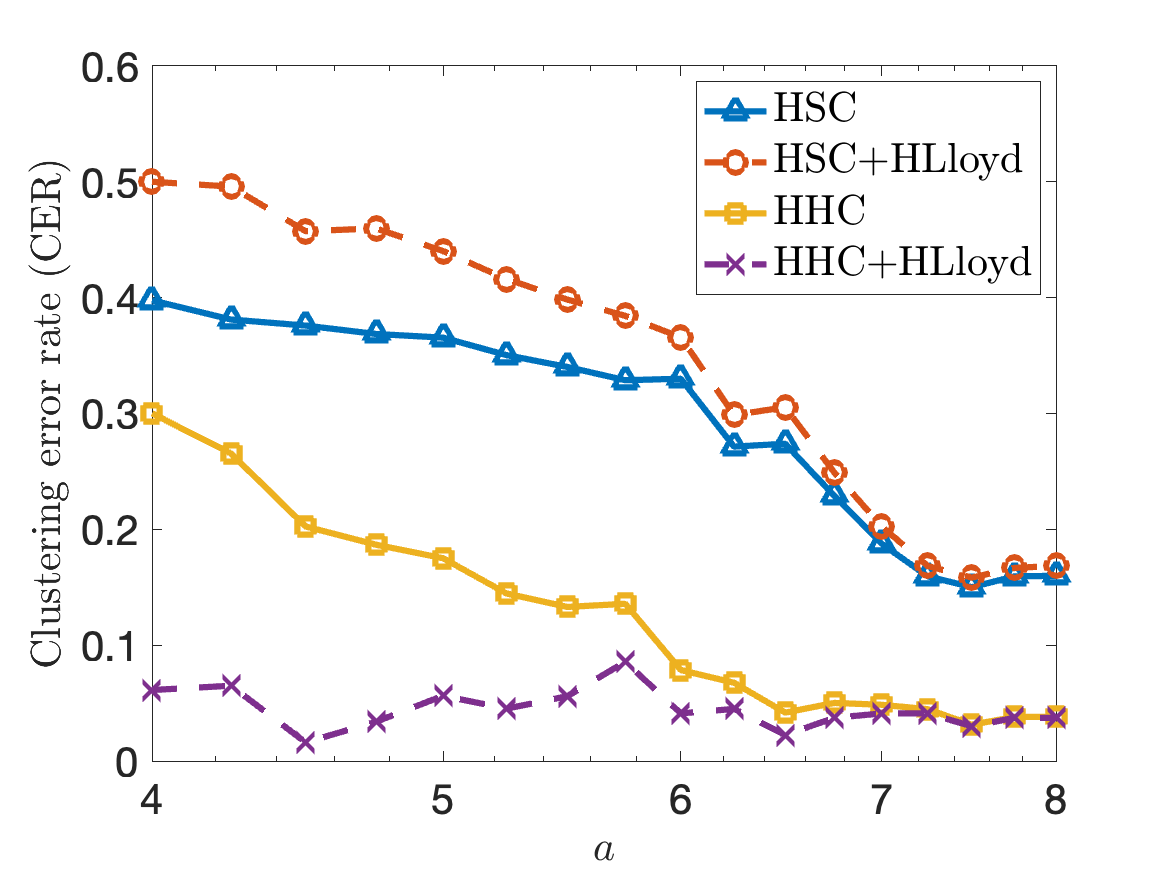}}
		\subfigure[$k = 4$, percentage of exact recovery]{
			\includegraphics[width=0.32\linewidth]{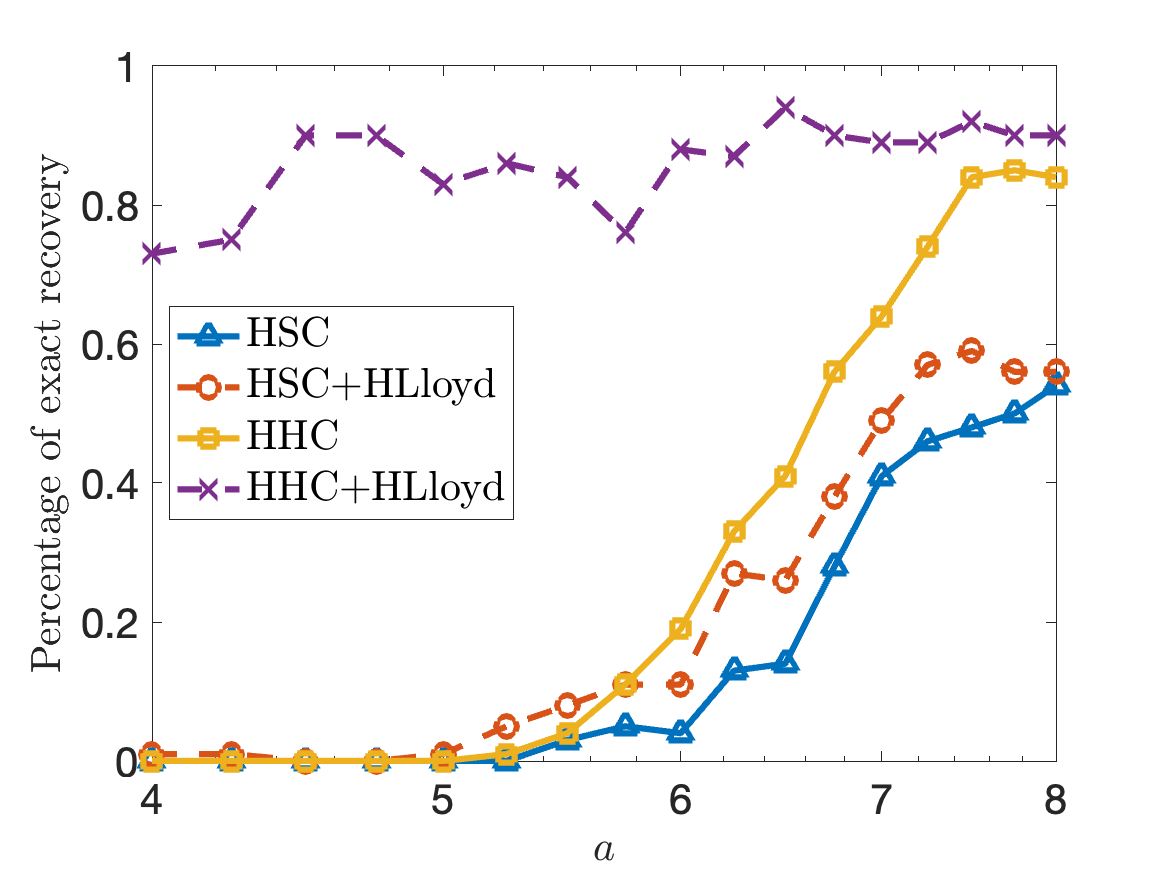}}
	\end{minipage}
    \caption{Averaged CER and percentage of exact community recovery for {\sf HSC}, {\sf HSC + HLloyd}, \Algoabbrev~and \Algoabbrev~+ {\sf HLloyd} under the Stochastic tensor block models with $n = 150$. Here, the quantity $a$ satisfies \eqref{simulation:STBM_core_tensor}.}
    \label{figure:stochastic2}
\end{figure}

\subsection{Real data analysis and real-data-inspired simulation studies}

\paragraph{Real data example: the flight route network.} We now turn attention to the flight route network data studied in \cite{han2022exact}. In adherence to their setup, we also take into account the top 50 airports based on the number of flight routes.\footnote{The original database at \url{https://openflights.org/data.html\#route}. Here, we use the processed data provided at \url{https://github.com/Rungang/HLloyd/blob/master/experiment/flight_route.RData}.} This results in a $39 \times 50 \times 50$ tensor $\bm{\mathcal{Y}}$ with binary entries, where the first mode represent airlines, and the remaining two modes represent airports. The entries of the tensor $\bm{\mathcal{Y}}$ satisfies
\begin{align}
	Y_{i,j,k} = \begin{cases}
		1,\qquad&\text{if airline}~i~\text{operates a flight route from airport}~j~\text{to airport}~k,\\
		0,\qquad&\text{otherwise}.
	\end{cases}
\end{align}
We select the clustering sizes based on the Bayesian information criterion (BIC) as described in \cite{wang2019multiway, han2022exact}. This criterion suggests the numbers of clusters $(k_1, k_2, k_3) = (5, 5, 5)$. 
 We apply \Algoabbrev~+ {\sf HLloyd} and {\sf HSC + HLloyd} to the data $\bm{\mathcal{Y}}$, with results summarized in Tables \ref{table:airline_cluster}-\ref{table:airport_cluster_rungang}.\footnote{For each method, we run 100 independent replicates and choose the result that occurs most frequently.} Tables~\ref{table:airline_cluster} and \ref{table:airport_cluster} reveal that \Algoabbrev~+ {\sf HLloyd} produces reasonable clustering results, effectively grouping airlines/airports from China, Europe, and the United States. A comparison of Tables \ref{table:airline_cluster} and \ref{table:airline_cluster_rungang} indicates that \Algoabbrev~+ {\sf HLloyd} outperforms {\sf HSC + HLloyd} in clustering European and US airlines. For instance, Cluster 2 in both tables shows \Algoabbrev~+ {\sf HLloyd} grouping three US airlines together, whereas {\sf HSC + HLloyd} includes only two (AA and US); in Cluster 3, \Algoabbrev~+ {\sf HLloyd} groups three European airlines, but {\sf HSC + HLloyd} gives a mixture of US and European airlines. For airport clustering, our results in Table \ref{table:airport_cluster} appear more reasonable than those for {\sf HSC + HLloyd} in Table \ref{table:airport_cluster_rungang}. Notably, with regards to Cluster 3 in both tables, \Algoabbrev~+ {\sf HLloyd} identifies a cluster of airports from four major European cities along with ATL (a hub). In contrast, {\sf HSC + HLloyd} groups only CDG (France) and ATL (USA) together.

\begin{table}[t]
		\begin{center}
			\scalebox{1}{\begin{tabular}{c|c}
					& Airlines\\
					\hline
					Cluster 1 & CA, MU, CZ, HU, 3U, ZH (China)\\
					\hline 
					Cluster 2 & AA, UA, US (USA)\\
					\hline
					Cluster 3 & AF, AZ, KL (Europe)\\
					\hline
					Cluster 4 & BA, AY, IB (Europe), DL (USA)\\
					\hline
					Cluster 5 & SU, AB, AI, AM, NH, AC, AS, FL, DE, ET, etc. (Mixture)
			\end{tabular}}
		\end{center}
	 \caption{Airline clustering results using \Algoabbrev~+ {\sf HLloyd}.}
	 \label{table:airline_cluster}
\end{table}

\begin{table}[t]
	\begin{center}
		\scalebox{1}{\begin{tabular}{c|c}
				& Airlines\\
				\hline
				Cluster 1 & CA, MU, CZ, HU, 3U, ZH (China)\\
				\hline 
				Cluster 2 & AA, US (USA)\\
				\hline
				Cluster 3 & AF, AZ, KL (Europe), DL (USA)\\
				\hline
				Cluster 4 & BA, AY, IB (Europe), UA (USA)\\
				\hline
				Cluster 5 & SU, AB, AI, AM, NH, AC, AS, FL, DE, ET, etc. (Mixture)
		\end{tabular}}
	\end{center}
	\caption{Airline clustering results for {\sf HSC + HLloyd}.}
	\label{table:airline_cluster_rungang}
\end{table}

\begin{table}[t]
	\begin{center}
		\scalebox{1}{\begin{tabular}{c|c}
				& Airlines\\
				\hline
				Cluster 1 & BRU, DUS, MUC, MAN, LGW, AMS, BCN, VIE, etc. (Mixture)\\
				\hline 
				Cluster 2 &  LAX, MIA, DFW, PHL, JFK, ORD, CLT (USA)\\
				\hline
				\multirow{2}{*}{Cluster 3} & Europe: LHR (London), MAD (Madrid), CDG (Paris), FCO (Rome)\\
				& USA: ATL (Atlanta)\\
				\hline
				Cluster 4 & PEK, CAN, XIY, KMG, HGH, CKG, CTU, PVG (China)\\
				\hline 
				\multirow{2}{*}{Cluster 5} & PHX, SFO, EWR, IAH, DEN, LAS (USA)\\
				& YYZ (Canada), FRA (Germany), MEX (Mexico)
		\end{tabular}}
	\end{center}
	\caption{Airport clustering results using \Algoabbrev~+ {\sf HLloyd}.}
	\label{table:airport_cluster}
\end{table}

\begin{table}[t]
	\begin{center}
		\scalebox{1}{\begin{tabular}{c|c}
				& Airlines\\
				\hline
				Cluster 1 & BRU, MUC, LGW, AMS, BCN, VIE, ZRH, DXB, etc. (Mixture)\\
				\hline 
				Cluster 2 & LHR (UK), MIA, DFW, PHL, JFK, ORD, CLT (USA)\\
				\hline
				Cluster 3 & CDG (France), ATL (USA)\\
				\hline
				Cluster 4 & PEK, CAN, XIY, KMG, HGH, CKG, CTU, PVG (China)\\
				\hline
				\multirow{2}{*}{Cluster 5} &  YYZ, FRA, DUS, MAN, MAD, FCO (Europe), MEX (Mexico)\\
				& PHX, SFO, LAX, EWR, IAH, DEN, LAS (USA)
		\end{tabular}}
	\end{center}
	\caption{Airport clustering results for {\sf HSC + HLloyd}.}
	\label{table:airport_cluster_rungang}
\end{table}

\paragraph{Real-data-inspired numerical studies.} While \Algoabbrev~+ {\sf HLloyd} appears to yield more reasonable real data results, a challenge arises due to the absence of a known ground truth for validation. To draw a more convincing conclusion, we adopt real-data-inspired numerical studies to establish a quantitative comparison between \Algoabbrev~+ {\sf HLloyd} and {\sf HSC + HLloyd}. Recall that in the real data example, \Algoabbrev~+ {\sf HLloyd} 
gives us the following estimates: the centroid tensor $\widehat{\bm{\mathcal{S}}}^{\Algoabbrev + {\sf HLloyd}} \in [0, 1]^{5 \times 5 \times 5}$ and the cluster assignment vector estimates $\widehat{\bm{z}}_i^{\Algoabbrev + {\sf HLloyd}}$. In contrast, {\sf HSC + HLloyd} provides $\widehat{\bm{\mathcal{S}}}^{{\sf HSC + HLloyd}} \in [0, 1]^{5 \times 5 \times 5}$ and $\widehat{\bm{z}}_i^{{\sf HSC + HLloyd}}$. We then generate stochastic tensor block models, setting the truth $\bm{\mathcal{S}}^\star = \widehat{\bm{\mathcal{S}}}^{\Algoabbrev + {\sf HLloyd}}$ and $\bm{z}_i^\star = \widehat{\bm{z}}_i^{\Algoabbrev + {\sf HLloyd}}$ for the first scenario, and $\bm{\mathcal{S}}^\star = \widehat{\bm{\mathcal{S}}}^{{\sf HSC + HLloyd}}$ and $\bm{z}_i^\star = \widehat{\bm{z}}_i^{{\sf HSC + HLloyd}}$) for the second. We apply the four methods --- \Algoabbrev, \Algoabbrev~+ {\sf HLloyd}, {\sf HSC} and {\sf HSC + HLloyd} --- to the generated data. The results are averaged over 100 Monte Carlo runs and are reported in Table \ref{table:real_data_inspired1} and Table \ref{table:real_data_inspired2}, respectively. From these results, it becomes evident that under the model with $\bm{\mathcal{S}}^\star = \widehat{\bm{\mathcal{S}}}^{\Algoabbrev + {\sf HLloyd}}$ and $\bm{z}_i^\star = \widehat{\bm{z}}_i^{\Algoabbrev + {\sf HLloyd}}$, \Algoabbrev~+ {\sf HLloyd} outperforms in terms of estimation error and recovery rate. For the second model, while all four methods demonstrate comparable exact recovery percentages, \Algoabbrev~and \Algoabbrev~+ {\sf HLloyd} have noticeably smaller estimation errors. This means that even for data that best fits {\sf HSC + HLloyd}, our methods can achieve better clustering performance. 
\begin{table}[t]
	\begin{center}
		\scalebox{1}{\begin{tabular}{c|c|c|c}
				& error mean & standard deviation & recovery rate\\
				\hline 
				{\sf HSC}  & 0.0225 & 0.0269 & 0.36\\
				\hline
				{\sf HSC + HLloyd} & 0.0129 & 0.0275  & 0.69\\
				\hline
				\Algoabbrev & 0.0181 & 0.0472 & 0.6\\
				\hline
				\Algoabbrev~+ {\sf HLloyd} & 0.0115 & 0.0453 & 0.83
		\end{tabular}}
	\end{center}
	\caption{Real-data-inspired numerical experiments: $\bm{\mathcal{S}}^\star = \widehat{\bm{\mathcal{S}}}^{\Algoabbrev + {\sf HLloyd}}$ and $\bm{z}_i^\star = \widehat{\bm{z}}_i^{\Algoabbrev + {\sf HLloyd}}$.}
	\label{table:real_data_inspired1}
\end{table}

\begin{table}[t]
	\begin{center}
		\scalebox{1}{\begin{tabular}{c|c|c|c}
				& error mean & standard deviation & recovery rate\\
				\hline 
				{\sf HSC}  & 0.0273 & 0.0942  & 0.89\\
				\hline
				{\sf HSC + HLloyd} &0.0311 & 0.0959    & 0.84\\
				\hline
				\Algoabbrev & 0.0120 &0.0386   & 0.85\\
				\hline
				\Algoabbrev~+ {\sf HLloyd} & 0.0124 & 0.0419  & 0.88
		\end{tabular}}
	\end{center}
	\caption{Real-data-inspired nemerical experiments: $\bm{\mathcal{S}}^\star = \widehat{\bm{\mathcal{S}}}^{{\sf HSC + HLloyd}}$ and $\bm{z}_i^\star = \widehat{\bm{z}}_i^{{\sf HSC + HLloyd}}$.}
	\label{table:real_data_inspired2}
\end{table}

    \section{Related work}\label{sec:related_work}

The tensor clustering problem considered in the current paper is closely related to several classical clustering problems, 
which we briefly review here. Among the most commonly studied clustering models are stochastic and censored block models \citep{holland1983stochastic,rohe2011spectral,mossel2014belief,lei2015consistency,abbe2015exact,mossel2015reconstruction,hajek2016achieving,hajek2016achievingE,cai2015robust,abbe2015community,chin2015stochastic,zhang2016minimax,florescu2016spectral,chen2016community,guedon2016community,abbe2017community,gao2017achieving,deshpande2017asymptotic,amini2018semidefinite,li2021convex,cai2021subspace}, 
synchronization \citep{singer2011angular,javanmard2016phase,bandeira2017tightness,chen2018projected,zhong2018near,gao2021exact,li2022non,li2023approximate,celentano2023local}, and (sub-)Gaussian mixture models \citep{lu2016statistical,cai2018rate,li2020birds,chen2021cutoff,ndaoud2022sharp,abbe2022lp,han2023eigen}. 
For all these models, spectral clustering algorithms have emerged as a powerful paradigm 
and have achieved both theoretical and empirical success \citep{von2007tutorial,kannan2009spectral,abbe2017community,chen2021spectral}. 
Sharp and intriguing statistical guarantees have recently been derived for spectral clustering \citep{lei2015consistency,chin2015stochastic,abbe2020entrywise,loffler2021optimality,zhang2022leave,zhang2023fundamental}. However, direct applications of these methods and theories to the tensor block model might yield highly sub-optimal signal-to-noise ratio conditions, 
while in the meantime introducing unnecessary assumptions on the condition number. 

Turning to the tensor block model, \cite{wang2019multiway} investigated the theoretical properties for the MLE estimator, which is, however, computationally intractable. \cite{chi2020provable} considered a convex procedure and derived its theoretical guarantees concerning the tensor estimation error. Under the i.i.d.~noise setting, \cite{han2022exact} proved the existence of a statistical-computational gap for this problem and proposed a polynomial-time algorithm that can achieve exact clustering if the signal-to-noise ratio exceeds the computational limit (ignoring logarithmic factors). However, in the presence of heteroskedastic noise, these methods fall short of statistical efficiency. Going beyond this model, \cite{hu2023multiway} considered degree-corrected tensor block models and \cite{agterberg2022estimating} studied mixed-membership tensor block models. However, the methods and theoretical results in these two papers either lean on i.i.d.~noise assumptions or require the underlying tensor to be well-conditioned, both of which can be relaxed using our approach. 
In addition to the model considered in this paper, a couple of other tensor clustering problems have been proposed and studied in the literature; see, e.g., \citet{jegelka2009approximation,sun2019dynamic,wu2019essential,lyu2022optimal,mai2022doubly}.

Our work is also closely related to the tensor PCA models \citep{richard2014statistical,hopkins2015tensor,anandkumar2017homotopy,zhang2018tensor,arous2019landscape,han2022optimal,cai2022nonconvex,cai2022uncertainty,xia2022inference,zhou2022optimal}, which aim to estimate the true tensor or the associated subspaces based on noisy observations. To accomplish this task, a commonly used strategy is to apply spectral methods \citep{chen2021spectral} to obtain initial subspace estimates, followed by further refinement steps \citep{de2000best,zhang2018tensor,han2022optimal,tong2022scaling,cai2022nonconvex}. Some popular initialization methods include the vanilla SVD-based approach \citep{cai2018rate,zhang2018tensor}, diagonal-deleted/reweighted PCA \citep{lounici2014high,florescu2016spectral,montanari2018spectral,cai2021subspace,cai2022nonconvex} and {\sf HeteroPCA} \citep{zhang2022heteroskedastic,yan2021inference,han2022optimal}. However, in contrast to the tensor PCA models, the tensor block models studied in the present paper do not impose any assumptions on the least singular value or on the singular gaps of the true tensor. Therefore, directly applying these tensor PCA methods may not yield subspace estimates with the desired statistical accuracy.

Recently, it has been shown that sharp $\ell_{2,\infty}$ or $\ell_{\infty}$ guarantees for singular subspaces play a pivotal role for proving that spectral clustering (with or without the help of $k$-means) can achieve exact recovery or optimal mis-clustering rates for many clustering problems \citep{abbe2020entrywise,cai2021subspace,abbe2022lp,zhang2023fundamental}. To derive such subspace estimation guarantees, a powerful and perhaps the most popular tool is the leave-one-out analysis \citep{zhong2018near,ma2020implicit,chen2019spectral,abbe2020entrywise,lei2019unified,chen2020noisy,chen2019inference,chen2021bridging,cai2021subspace,chen2021convex,cai2022nonconvex,abbe2022lp,yan2021inference,ling2022near,ke2022optimal,zhang2022leave,yang2022optimal}. However, the results obtained using the leave-one-out analysis are often sub-optimal with respect to the condition number of the truth. This can lead to unsatisfactory results under the tensor block models, especially since there is no assumption made on the singular value of the underlying tensor $\bm{\mathcal{S}}^\star$.


    \section{Discussion}

In this paper, we have studied the tensor clustering problem in the presence of heteroskedastic noise. 
To better deal with heteroskedastic noise and improve statistical performance, 
we have proposed a novel method called \Algoname~(\Algoabbrev), which first employs \Algosubspace~to obtain subspace estimates and then applies approximate $k$-means for clustering.  The proposed method provably achieves exact clustering for a wide range of signal-to-noise ratio conditions that are essentially unimprovable among polynomial-time algorithms. Empirically, we have evaluated the numerical performance of \Algoabbrev, and \Algoabbrev~followed by the high-order Lloyd algorithm ({\sf HLloyd}, \cite{han2022exact}), on both synthetic and real data. Both of these two methods achieve low empirical mis-clustering rates, with \Algoabbrev~+ {\sf HLloyd} outperforming other existing methods proposed in the literature.

Moving beyond, there are numerous future directions that are worth investigating. For instance, thus far our theory has focused primarily on exact clustering; 
it remains unclear whether our algorithm can achieve optimal mis-classification rates when only partial recovery is feasible. 
In addition, our signal-to-noise ratio condition might be sub-optimal if the clusters are highly-unbalanced, a scenario where the balance parameter $\beta$ is exceedingly small. It would be interesting to investigate the plausibility of further improvement under such imbalanced settings. Furthermore, it would be worthwhile to explore the feasibility of extending our paradigm to tackle the tensor mixed-membership block model \citep{agterberg2022estimating}, with the aim of achieving optimal statistical performance without being affected by the condition number of the true tensor.

\section*{Acknowledgements}

Y.~Chen is supported in part by the Alfred P.~Sloan Research Fellowship,   
and the NSF grants CCF-1907661, DMS-2014279, IIS-2218713 and IIS-2218773.

\appendix
    \section{Procedure of {\sf High-order Lloyd Algorithm (HLloyd)}}\label{sec:procedure_HLloyd}
This section provides a formal description of the procedure of {\sf High-order Lloyd Algorithm (HLloyd)} proposed by \citep{han2022exact}; see Algorithm \ref{algorithm:HLloyd}.

\begin{algorithm}[h]
	\caption{{\sf High-order Lloyd Algorithm (HLloyd)} \citep{han2022exact}} \label{algorithm:HLloyd}
	\textbf{input:} observed tensor $\bm{\mathcal{Y}}$, numbers of clusters $k_1, k_2, k_3$, initial cluster assignment vector estimates $\{\widehat{\bm{z}}_\ell^{(0)}\}_{1 \leq \ell \leq 3}$, number of iterations $T$.\\
	\For{$t = 0, \dots, T - 1$}{\textbf{block mean update:} calculate $\widehat{\bm{\mathcal{S}}}^{(t)} \in \bbR^{k_1 \times k_2 \times k_3}$ such that
		\begin{align*}
			\widehat{\bm{S}}_{i_1, i_2, i_3}^{(t)} =~\text{Average}\left(\Big\{\bm{\mathcal{Y}}_{j_1, j_2, j_3}: \widehat{z}_{\ell, j_\ell}^{(t)} = i_\ell, \forall \ell \in [3]\Big\}\right),~\quad~\forall i_{\ell} \in [k_\ell], \ell \in [3].
		\end{align*}\\
		calculate $\widehat{\bm{\mathcal{B}}}_1^{(t)} \in \bbR^{n_1 \times k_2 \times k_3}, \widehat{\bm{\mathcal{B}}}_2^{(t)} \in \bbR^{k_1 \times n_2 \times k_3}, \widehat{\bm{\mathcal{B}}}_3^{(t)} \in \bbR^{k_1 \times k_2 \times n_3}$ such that
		\begin{align*}
			\Big(\widehat{\bm{\mathcal{B}}}_1^{(t)}\Big)_{j_1, i_2, i_3} &= \text{Average}\left(\Big\{\bm{\mathcal{Y}}_{j_1, j_2, j_3}: \widehat{z}_{\ell, j_\ell}^{(t)} = i_\ell, \ell = 2, 3\Big\}\right),~\quad~\forall j_1 \in [n_1], i_2 \in [k_2], i_3 \in [k_3],\\
			\Big(\widehat{\bm{\mathcal{B}}}_2^{(t)}\Big)_{i_1, j_2, i_3} &= \text{Average}\left(\Big\{\bm{\mathcal{Y}}_{j_1, j_2, j_3}: \widehat{z}_{\ell, j_\ell}^{(t)} = i_\ell, \ell = 1, 3\Big\}\right),~\quad~\forall i_1 \in [k_1], j_2 \in [n_2], i_3 \in [k_3],\\
			\Big(\widehat{\bm{\mathcal{B}}}_3^{(t)}\Big)_{i_1, i_2, j_3} &= \text{Average}\left(\Big\{\bm{\mathcal{Y}}_{j_1, j_2, j_3}: \widehat{z}_{\ell, j_\ell}^{(t)} = i_\ell, \ell = 1, 2\Big\}\right),~\quad~\forall i_1 \in [k_1], i_2 \in [k_2], j_3 \in [n_3].
		\end{align*}
		\textbf{cluster update:} calculate cluster assignment vector estimates $\{\widehat{\bm{z}}_{i}^{(t+1)}\}_{i \in [3]}$:
		\begin{align}
			\widehat{z}_{i,j}^{(t+1)} \in \argmin_{\ell \in [k_i]}\Big\|\Big(\mathcal{M}_i\big(\widehat{\bm{\mathcal{B}}}_i^{(t)}\big)\Big)_{j, :} - \Big(\mathcal{M}_i\big(\widehat{\bm{\mathcal{S}}}^{(t)}\big)\Big)_{\ell, :}\Big\|_2,\quad~\forall i \in [3], j \in [n_i].
		\end{align}
	}
	\textbf{output:} cluster assignment vector estimates $\widehat{\bm{S}}= \widehat{\bm{S}}^{(T - 1)}, \widehat{\bm{z}}_1 = \widehat{\bm{z}}_1^{(T)}, \widehat{\bm{z}}_2 = \widehat{\bm{z}}_2^{(T)}, \widehat{\bm{z}}_3 = \widehat{\bm{z}}_2^{(T)}$.
\end{algorithm}

\section{Proof of Theorem \ref{thm:cluster_tensor_simplified}}\label{sec:proof_thm_cluster_tensor_simplified}

In this section, we present the proof of our main result Theorem~\ref{thm:cluster_tensor_simplified}, 
by establishing a more general version as follows. Here and throughout, we define $k_{-i} = k_1k_2k_3/k_i$ and $n_{-i} = n_1n_2n_3/n_i$ for $i \in [3]$. 
\begin{thm}\label{thm:cluster_tensor}
	Suppose that Assumption \eqref{assump:noise} holds, and assume that for all $1\leq i\leq 3$, 
	\begin{subequations}
		\begin{align}
			n_1n_2n_3 &\gtrsim k^4n^2, \label{ineq:dimension_assumption}\\
			n_i &\geq \frac{c_1k^4}{\beta^2},\label{ineq:dimension_assumption2}\\
			k_i &\gtrsim k_{-i},\label{ineq:cluster_requirement}\\
			c_{\tau}k_i^2\left(n_1n_2n_3\right)^{1/2}\log^2 n &\leq \tau_i/\omega_{\sf max}^2 \leq C_{\tau}k_i^2\left(n_1n_2n_3\right)^{1/2}\log^2 n,\label{ineq:threshold_tensor}\\
			\frac{\Delta_{\sf min}}{\omega_{\sf max}} &\geq C_1\sqrt{M}\left(\frac{k^{9/2}}{\beta^{5/2}}\left(n_1n_2n_3\right)^{-1/4}\log n + \frac{k^{9}}{\beta^{5}}\left(n_1n_2n_3/n\right)^{-1/2}\sqrt{\log n}\right),\label{ineq:signal_to_noise_ratio}
		\end{align}
	\end{subequations}
	where $C_1, c_1, C_{\tau}$ and $c_{\tau}$ are some large enough constants.
	If we choose the numbers of iterations to obey
	\begin{subequations}
		\begin{align}
			t_{i,j} &\geq \log\left(C\frac{k^3}{\beta^3}\frac{\sigma_{i, r_{i, j-1}+1}^{\star2}}{\sigma_{i, r_{i, j}+1}^{\star2}}\right),~\quad~1 \leq j \leq j_{\sf max}^i - 1\label{iter1}\\
			t_{i,j_{\sf max}^i} &\geq \log\left(Cn^3\frac{\sigma_{i, r_{i, j_{\sf max}^i-1}+1}^{\star2}}{\omega_{\sf max}^2}\right)\label{iter2}
		\end{align}
	\end{subequations}
	for all $1\leq i \leq 3$, then with probability exceeding $1 - O(n^{-10})$, 
	the misclassification rate (cf.~\eqref{eq:defn-MCR}) of the outputs $\{\widehat{\bm{z}}_i\}$ returned by Algorithm \ref{algorithm:k_means} satisfy 
 	\begin{align*}
 		\mathsf{MCR}\big(\widehat{\bm{z}}_i, \bm{z}_i^\star\big) = 0,~\quad~\forall 1\leq i\leq 3.
 	\end{align*}
\end{thm}

The rest of this section is dedicated to proving Theorem~\ref{thm:cluster_tensor}.

\subsection{Several key results under the matrix setting}
To begin with, we first consider the following model: suppose we observe
\begin{align}\label{model:matrix}
	\bm{Y} = \bm{X}^\star + \bm{E} \in \bbR^{m_1 \times m_2},
\end{align}
where the noise matrix $\bm{E}$ has independent and zero-mean entries, and $\bm{X}^\star$ is a matrix with rank not exceeding $r$ and admits the following SVD decomposition:
\begin{align}\label{eq:signal_matrix}
	\bm{X}^\star = \bm{U}^\star\bm{\Sigma}^\star\bm{V}^{\star\top} = \sum_{i=1}^{r}\sigma_i^\star\bm{u}_i^\star\bm{v}_i^{\star\top},
\end{align}
where $\sigma_1^\star \geq \cdots \geq \sigma_r^\star \geq 0$ are the singular values of $\bm{X}$, $\bm{U}^\star = [\bm{u}_1^\star, \dots, \bm{u}_r^\star] \in \mathcal{O}^{m_1, r}$ (resp.~$\bm{V} = [\bm{v}_1^\star, \dots, \bm{v}_r^\star] \in \mathcal{O}^{m_2, r}$) is the column (resp.~row) subspace of $\bm{X}^\star$, and $\bm{\Sigma}^\star = {\sf diag}\left(\sigma_1^\star, \dots, \sigma_r^\star\right)$.
In addition, we define the incoherence parameter
\begin{align}
	(\text{Incoherence parameter})~\qquad~\mu = \mu\left(\bm{X}^\star\right) := \max\left\{\frac{m_1}{r}\max_{i \in [m_1]}\left\|\bm{U}_{i,:}^\star\right\|_2^2, \frac{m_2}{r}\max_{j \in [m_2]}\left\|\bm{V}_{j,:}^\star\right\|_2^2\right\}.
\end{align}
For notational convenience, we also define
\begin{align}\label{eq:dimension}
	m \coloneqq \max\{m_1, m_2\}~\qquad~\text{and}~\qquad~\sigma^\star_{r+1} = 0.
\end{align}
Furthermore, we impose the noise assumption on the noise matrix $\bm{E}$:
\begin{assump}\label{assump:noise_matrix}
	Suppose that the following conditions on the noise matrix $\bm{E}$ hold:
	\begin{itemize}
		\item[1.] The $E_{i,j}$'s, the entries of $\bm{E}$, are independently generated and satisfy $\bbE[E_{i,j}] = 0$;
		\item[2.] $\bbP\left(\left|E_{i,j}\right| > B\right) \leq m^{-12}$, where $B$ is some quantity satisfying 
		\begin{align*}
			B \leq C_{\sf b}\omega_{\sf max}\frac{\min\big\{(m_1m_2)^{1/4}, \sqrt{m_2}\big\}}{\log m}.
		\end{align*}
	\end{itemize}
\end{assump}
One can immediately find that $\mathcal{M}_i\left(\bm{\mathcal{E}}\right)$ obeys the conditions in Assumption \ref{assump:noise_matrix} with dimension $m_1 = n_i$ and $m_2 = n_{-i}$. 
Moreover, we define
\begin{align}\label{def:M}
	\bm{M} = \left(\bm{U}^\star\bm{\Sigma}^\star + \bm{E}\bm{V}^\star\right)\left(\bm{U}^\star\bm{\Sigma}^\star + \bm{E}\bm{V}^\star\right)^\top
\end{align}
and
\begin{align}\label{def:oracle_matrix}
	\bm{M}^{\sf oracle} = \bm{M} +  \mathcal{P}_{\sf off\text{-}diag}\left(\bm{E}\bm{E}^\top - \bm{E}\bm{V}^\star\bm{V}^{\star\top}\bm{E}^\top\right) = \mathcal{P}_{\sf off\text{-}diag}\left(\bm{Y}\bm{Y}^\top\right) + \mathcal{P}_{\sf diag}\left(\bm{M}\right).
\end{align}
Let $\bm{U}^{\sf oracle} \in \mathcal{O}^{m_1, r}$ denote the leading-$r$ eigenvector of $\bm{M}^{\sf oracle}$. The following theorem shows that $\bm{U}_{:,1: r'}^{\sf oracle}$ and $\bm{U}^\star_{:, 1: r'}$ are reasonably close if there is a sufficiently large gap between $\sigma^\star_{r'}$ and $\sigma^\star_{r'+1}$; the proof is deferred to Section~\ref{sec:proof_theorem_oracle_two_to_infty}.
\begin{thm}\label{thm:oracle_two_to_infty}
	Suppose that $r \geq 2$, Assumption \ref{assump:noise_matrix} holds and
	\begin{subequations}
		\begin{align}
			\sigma_{1}^\star/\omega_{\sf max} &\geq 2C_0r\big[(m_1m_2)^{1/4} + rm_1^{1/2}\big]\log m\label{ineq:snr_matrix}\\
			\mu &\leq c_0\frac{m_1}{r^3}\label{ineq:incoherence}
		\end{align}
	\end{subequations}
	hold for some sufficiently large (resp.~small) constant $C_0 > 0$ (resp.~$c_0 > 0$). 
	
	(a) The set defined below
	\begin{align}\label{def:A}
		\mathcal{A} = \left\{j: 1 \leq j \leq r, \sigma_j^\star \geq \frac{4r}{4r - 1}\sigma_{j+1}^\star \vee C_0r\big[(m_1m_2)^{1/4} + rm_1^{1/2}\big]\omega_{\sf max}\log m\right\}
	\end{align}
	is non-empty.
	
	(b) With probability exceeding $1 - O(n^{-10})$, for all $r' \in \mathcal{A}$, we have
	\begin{subequations}
		\begin{align}
			\left\|\bm{U}_{:,1: r'}^{\sf oracle}\bm{U}_{:,1: r'}^{\sf oracle\top} - \widetilde{\bm{U}}_{:,1: r'}\widetilde{\bm{U}}_{:,1: r'}^\top\right\|_{2,\infty} &\lesssim \sqrt{\frac{\mu r^3}{m_1}}\left(\frac{r^2\sqrt{m_1}\omega_{\sf max}\log m}{\sigma_{r'}^\star} + \frac{r^2 \sqrt{m_1m_2}\omega_{\sf max}^2\log^2 m}{\sigma_{r'}^{\star2}}\right),\label{ineq:oracle_tilde}\\
			\big\|\bm{U}_{:,1: r'}^{\sf oracle}\bm{U}_{:,1: r'}^{\sf oracle\top} - \bm{U}_{:,1: r'}^{\star}\bm{U}_{:,1: r'}^{\star\top}\big\|_{2,\infty} &\lesssim \sqrt{\frac{\mu r^3}{m_1}}\left(\frac{r^2\sqrt{m_1}\omega_{\sf max}\log m}{\sigma_{r'}^\star} + \frac{r^2 \sqrt{m_1m_2}\omega_{\sf max}^2\log^2 m}{\sigma_{r'}^{\star2}}\right).\label{ineq:oracle}
		\end{align}
	\end{subequations}
Here, $\widetilde{\bm{U}}_{:,1: r'}$ is the leading rank-$r'$ left singular space of $\bm{U}_{:, 1:\overline{r}}^\star\bm{\Sigma}_{1:\overline{r}, 1:\overline{r}}^\star + \bm{E}\bm{V}_{:, 1:\overline{r}}^\star$ with $\overline{r} = \max\mathcal{A}$.
\end{thm}
With the aid of Theorem \ref{thm:oracle_two_to_infty}, we are able to develop an upper bound on $\left\|\left(\bm{I}_{m_1} - \bm{U}\bm{U}^\top\right)\bm{X}^\star\right\|_{2, \infty}$ if the threshold $\tau$ is properly chosen, as asserted by the following theorem.
\begin{thm}\label{thm:residual}
	Suppose that $r \geq 2$, Assumption \ref{assump:noise_matrix} holds, and
	\begin{subequations}
		\begin{align}
			c_{\tau}r^2\big[(m_1m_2)^{1/2} + r^2m_1\big]\log^2 m &\leq \tau/\omega_{\sf max}^2 \leq C_{\tau}r^2\big[(m_1m_2)^{1/2} + r^2m_1\big]\log^2 m\label{ineq:threshold}\\
			\sigma_{1}^\star/\omega_{\sf max} &\geq C_1r\big[(m_1m_2)^{1/4} + rm_1^{1/2}\big]\log m\label{ineq:snr_matrix1}\\
			\mu &\leq c_1\frac{m_1}{r^3}\label{ineq:incoherence1}
		\end{align}
	\end{subequations}
	hold for some sufficiently large (resp.~small) constant $C_1, C_{\tau}, c_{\tau} > 0$ satisfying $C_1^2/2 > C_{\tau} > c_{\tau}$ (resp.~$c_1 > 0$). If the numbers of iterations obey
	\begin{subequations}
		\begin{align}
			t_k &> \log\left(C\frac{\sigma_{r_{k-1}}^{\star2}}{\sigma_{r_{k}}^{\star2}}\right),~\qquad~1 \leq k < k_{\sf max}\label{iter1_matrix}\\
			t_{k_{\sf max}} &> \log\left(C\frac{\sigma_{r_{k_{\sf max}-1} + 1}^{\star2}}{\omega_{\sf max}^2}\right)\label{iter2_matrix}
		\end{align}
	\end{subequations}
for some sufficiently large constants $C > 0$, then with probability exceeding $1 - O(m^{-10})$, the output of Algorithm \ref{algorithm:sequential_heteroPCA} satisfies
\begin{subequations}
	\begin{align}
		\big\|\bm{U}\bm{U}^\top - \bm{U}_{:, 1:r_{k_{\sf max}}}^\star\bm{U}_{:, 1:r_{k_{\sf max}}}^{\star\top}\big\| &\lesssim \sqrt{\frac{\mu r^3}{m_1}},\label{ineq:two_to_infty_crude}\\
		\big\|\big(\bm{U}\bm{U}^\top - \bm{U}_{:, 1:r_{k_{\sf max}}}^\star\bm{U}_{:, 1:r_{k_{\sf max}}}^{\star\top}\big)\bm{X}^\star\big\|_{2,\infty} &\lesssim \sqrt{\frac{\mu r^3}{m_1}}\left(r^2\sqrt{m_1}\omega_{\sf max}\log m + r(m_1m_2)^{1/4}\omega_{\sf max}\log m\right),\label{ineq:two_to_infty_residual_matrix1}\\
		\left\|\left(\bm{I}_{m_1} - \bm{U}\bm{U}^\top\right)\bm{X}^\star\right\|_{2, \infty} &\lesssim \sqrt{\frac{\mu r^3}{m_1}}\left(r^2\sqrt{m_1}\omega_{\sf max}\log m + r(m_1m_2)^{1/4}\omega_{\sf max}\log m\right)\label{ineq:two_to_infty_residual_matrix2}.
	\end{align}
\end{subequations}
Here, $r_0 = 0. r_1, \dots, r_{k_{\sf max}}$ are the ranks selected in Algorithm \ref{algorithm:sequential_heteroPCA} and $k_{\sf max}$ satisfies $r_{k_{\sf max}} = r$ or $\sigma_{r_{k_{\sf max}}+1}(\bm{G}_{k_{\sf max}}) \leq \tau$.
\end{thm}
The proof of Theorem \ref{thm:residual} can be found in Section \ref{proof:thm_residual}. With Theorem \ref{thm:residual} in hand, we are now  positioned to prove Theorem \ref{thm:cluster_tensor}. The proof consists of four steps, to be detailed next. 
\subsection{Main steps for proving Theorem \ref{thm:cluster_tensor}}
\paragraph{Step 1: verifying \eqref{ineq:threshold} - \eqref{ineq:incoherence1}.} To apply Theorem \ref{thm:residual}, one needs to verify the conditions \eqref{ineq:threshold} - \eqref{ineq:incoherence1} for $\mathcal{M}_1(\bm{\mathcal{Y}}) = \mathcal{M}_1(\bm{\mathcal{X}}) + \mathcal{M}_i(\bm{\mathcal{E}})$ with dimensions $m_1 = n_i$, $m_2 = n_{-i}$ and the rank $r = k_i$.
\paragraph{Step 1.1: verifying \eqref{ineq:threshold}.} Noting that $n_1n_2n_3 \geq k^4n^2$, we have $(n_1n_2n_3)^{1/2} + k_i^2n_i \asymp (n_1n_2n_3)^{1/2}$. Then we know from \eqref{ineq:threshold_tensor} that \eqref{ineq:threshold} is valid.
\paragraph{Step 1.2: verifying \eqref{ineq:incoherence1}.} For notational convenience, we let 
\begin{align*}
	\overline{\bm{M}}_i^\star = \bm{M}_i^\star\left(\bm{M}_i^{\star\top}\bm{M}_i^{\star}\right)^{-1/2} \in \mathcal{O}^{n_i, k_i}, ~\quad~\forall i \in [3]
\end{align*}
and
\begin{align*}
	\overline{\bm{\mathcal{S}}}^\star = \bm{\mathcal{S}}^\star \times_1 \left(\bm{M}_1^{\star\top}\bm{M}_1^{\star}\right)^{1/2} \times_2 \left(\bm{M}_2^{\star\top}\bm{M}_2^{\star}\right)^{1/2} \times_3 \left(\bm{M}_3^{\star\top}\bm{M}_3^{\star}\right)^{1/2}.
\end{align*}
In addition, for any $\bm{U} \in \mathcal{O}^{n, r}$, we define the projection matrix 
\begin{align}\label{def:subspace_projection}
	\mathcal{P}_{\bm{U}} = \bm{U}\bm{U}^\top.
\end{align}
Recognizing that
\begin{align*}
    \mathcal{M}_i(\bm{\mathcal{X}}^\star) = \bm{M}_i^\star\mathcal{M}_i\left(\bm{\mathcal{S}}^\star\right)\left(\bm{M}_{i+2}^\star \otimes \bm{M}_{i+1}^\star\right)^\top = \overline{\bm{M}}_i^\star\mathcal{M}_i\big(\overline{\bm{\mathcal{S}}}^\star\big)\big(\overline{\bm{M}}_{i+2}^\star \otimes \overline{\bm{M}}_{i+1}^\star\big)^\top,
\end{align*}
We let $\bm{U}_{\bm{X}_i^\star}\bm{\Sigma}_{\bm{X}_i^\star}\bm{V}_{\bm{X}_i^\star}^{\top}$ denote the SVD of $\bm{X}_i^\star := \mathcal{M}_i(\bm{\mathcal{X}}^\star)$, where $\bm{\Sigma}_{\bm{X}_i^\star} = {\sf diag}(\sigma_{1}\left(\bm{X}_i^\star\right), \dots, \sigma_{k_i}\left(\bm{X}_i^\star\right))$ is a diagonal matrix containing all singular values of $\bm{X}_i^\star$ (in decreasing order), and $\bm{U}_{\bm{X}_i^\star} \in \mathcal{O}^{n_i, k_i}$ (resp.~$\bm{V}_{\bm{X}_i^\star} \in \mathcal{O}^{n_{-i}, k_i}$) denotes the left (resp.~right) singular subspace of $\bm{X}_i^\star$ and satisfies 
\begin{align}\label{eq16}
	\bm{U}_{\bm{X}_i^\star} = \overline{\bm{M}}_i^\star\bm{A}_i~\quad~(\text{resp.}~\bm{V}_{\bm{X}_i^\star} = \big(\overline{\bm{M}}_{i+2}^\star \otimes \overline{\bm{M}}_{i+1}^\star\big)\bm{B}_i)
\end{align}
for some $\bm{A}_i \in \mathcal{O}^{k_i, k_i}$ and $\bm{B} \in \mathcal{O}^{k_{-i}, k_i}$. Then it follows immediately that
\begin{align}\label{ineq155}
	\left\|\bm{U}_{\bm{X}_i^\star}\right\|_{2,\infty} \leq \big\|\overline{\bm{M}}_i^\star\big\|_{2,\infty} \leq \left\|\bm{M}_i^\star\right\|_{2,\infty}\left\|\sigma_{k_i}\left(\bm{M}_i^{\star\top}\bm{M}_i^{\star}\right)\right\|^{-1/2} \leq 1\cdot \sqrt{\frac{k_i}{\beta n_i}}.
\end{align}
Here, the second inequality comes from the fact that  $\bm{M}_i^{\star\top}\bm{M}_i^{\star}$ is a diagonal matrix and its diagonal entries
\begin{align}\label{ineq154}
	\left(\bm{M}_i^{\star\top}\bm{M}_i^{\star}\right)_{\ell, \ell} = \big|\big\{j \in [n_i]: z_{i, j}^\star = \ell\big\}\big| \geq \beta n_i/k_i.
\end{align}
Similarly, one can bound $\|\bm{V}_{\bm{X}_i^\star}\|_{2,\infty}$ as follows:
\begin{align}\label{ineq156}
	\left\|\bm{V}_{\bm{X}_i^\star}\right\|_{2,\infty} \leq \big\|\overline{\bm{M}}_{i+1}^\star\big\|_{2,\infty}\big\|\overline{\bm{M}}_{i+2}^\star\big\|_{2,\infty} \leq \sqrt{\frac{k_{i+1}}{\beta n_{i+1}}}\sqrt{\frac{k_{i+2}}{\beta n_{i+2}}} = \sqrt{\frac{\left(\frac{k_{-i}}{\beta^2k_i}\right)k_i}{n_{-i}}}.
\end{align}
We then arrive at
\begin{align}\label{ineq157}
	\mu_i = \max\left\{\frac{1}{\beta}, \frac{k_{-i}}{\beta^2k_i}\right\} \stackrel{\eqref{ineq:dimension_assumption2}}{\leq} c_1\frac{n_i}{k_i^3}.
\end{align}

\paragraph{Step 1.3: verifying \eqref{ineq:snr_matrix1}.} Next, let us validate Condition \eqref{ineq:snr_matrix1}. Note that for any $1 \leq j_1 \neq j_2 \leq k_i$,
\begin{align}\label{ineq158}
	\Delta_i^2 &\leq \big\|\mathcal{M}_i\left(\mathcal{S}^\star\right)_{j_1,:} - \mathcal{M}_i\left(\mathcal{S}^\star\right)_{j_2,:}\big\|_2^2\notag\\
	&= \left\|\left(\bm{e}_{j_1} - \bm{e}_{j_2}\right)^\top\mathcal{M}_i\left(\mathcal{S}^\star\right)\right\|_2^2\notag\\
	&\leq \left\|\mathcal{M}_i\left(\mathcal{S}^\star\right)\right\|^2\left\|\bm{e}_{j_1} - \bm{e}_{j_2}\right\|_2^2\notag\\
	&= 2\left\|\mathcal{M}_i\left(\mathcal{S}^\star\right)\right\|^2,
\end{align}
where $e_{j}$ is the $j$-th canonical basis of $\bbR^{k_i}$. Furthermore, recognizing that $\bm{M}_i^{\star\top}\bm{M}_i^{\star}$ is a diagonal matrix with diagonal entries satisfying \eqref{ineq154}, we know that
\begin{align}\label{ineq159}
	\sigma_{k_i}\left(\bm{M}_i\right) \geq \sqrt{\frac{\beta n_i}{k_i}}, 
\end{align}
and consequently
\begin{align}\label{ineq160}
	\left\|\bm{X}_i^\star\right\| = \big\|\bm{M}_i^\star\mathcal{M}_i\left(\bm{\mathcal{S}}^\star\right)\left(\bm{M}_{i+2}^\star \otimes \bm{M}_{i+1}^\star\right)^\top\big\| \geq \left\|\mathcal{M}_i\left(\bm{\mathcal{S}}^\star\right)\right\|\prod_{i=1}^{3}\sigma_{k_i}\left(\bm{M}_i^\star\right) \geq \sqrt{\frac{\beta^3n_1n_2n_3}{k_1k_2k_3}}\left\|\mathcal{M}_i\left(\bm{\mathcal{S}}^\star\right)\right\|.
\end{align}
Combining \eqref{ineq158}, \eqref{ineq160} and the assumption $\Delta_i /\sigma_{\sf max} \gg k^{5/2}\left(n_1n_2n_3\right)^{-1/4}\log m/\beta^{3/2}$, one has
\begin{align}
	\left\|\bm{X}_i^\star\right\| \geq \sqrt{\frac{\beta^3n_1n_2n_3}{2k_1k_2k_3}}\Delta_i \gg k_i\left(n_1n_2n_3\right)^{1/4}\log n \asymp k_i\left[\left(n_1n_2n_3\right)^{1/4} + k_in_i^{1/2}\right]\log n,
\end{align}
and thus condition \eqref{ineq:snr_matrix1} holds. Here, the last inequality holds since $n_1n_2n_3 \geq k^4n^2$.

\paragraph{Step 2: bounding $\|\bm{\mathcal{Y}} \times_1 \mathcal{P}_{\bm{U}_1} \times_2 \mathcal{P}_{\bm{U}_2} \times_3 \mathcal{P}_{\bm{U}_3} - \bm{\mathcal{X}}^\star\|_{\rm F}^2$.} We define 
\begin{align*}
	\widehat{\bm{\mathcal{X}}} = \bm{\mathcal{Y}} \times_1 \mathcal{P}_{\bm{U}_1} \times_2 \mathcal{P}_{\bm{U}_2} \times_3 \mathcal{P}_{\bm{U}_3}.
\end{align*}
Then the triangle inequality and the fact $\|\mathcal{P}_{\bm{U}_i}\| = 1$ lead to the following upper bound: 
\begin{align}\label{ineq161}
	\big\|\widehat{\bm{\mathcal{X}}} - \bm{\mathcal{X}}^\star\big\|_{\rm F}^2 &= \left\|\bm{\mathcal{Y}} \times_1 \mathcal{P}_{\bm{U}_1} \times_2 \mathcal{P}_{\bm{U}_2} \times_3 \mathcal{P}_{\bm{U}_3} - \bm{\mathcal{X}}^\star\right\|_{\rm 
	F}^2\notag\\
&\leq 2\left(\left\|\bm{\mathcal{X}}^\star  \times_1 \mathcal{P}_{\bm{U}_1} \times_2 \mathcal{P}_{\bm{U}_2} \times_3 \mathcal{P}_{\bm{U}_3} - \bm{\mathcal{X}}^\star\right\|_{\rm F}^2 + \left\|\bm{\mathcal{E}}  \times_1 \mathcal{P}_{\bm{U}_1} \times_2 \mathcal{P}_{\bm{U}_2} \times_3 \mathcal{P}_{\bm{U}_3}\right\|_{\rm F}^2\right)\notag\\
&\leq 6\big(\left\|\bm{\mathcal{X}}^\star \times_1 \left(\bm{I}_{n_1} - \mathcal{P}_{\bm{U}_1}\right) \times_2 \mathcal{P}_{\bm{U}_2} \times_3 \mathcal{P}_{\bm{U}_3}\right\|_{\rm F}^2 + \left\|\bm{\mathcal{X}}^\star \times_2 \left(\bm{I}_{n_2} - \mathcal{P}_{\bm{U}_2}\right) \times_3 \mathcal{P}_{\bm{U}_3}\right\|_{\rm F}^2\notag\\&\hspace{1cm} + \left\|\bm{\mathcal{X}}^\star \times_3 \left(\bm{I}_{n_3} - \mathcal{P}_{\bm{U}_3}\right)\right\|_{\rm F}^2\big) + 2\left\|\bm{\mathcal{E}}  \times_1 \mathcal{P}_{\bm{U}_1} \times_2 \mathcal{P}_{\bm{U}_2} \times_3 \mathcal{P}_{\bm{U}_3}\right\|_{\rm F}^2\notag\\
&\leq 6\left(\left\|\left(\bm{I}_{n_1} - \bm{U}_1\bm{U}_1^\top\right)\bm{X}_1^\star\right\|_{\rm F}^2 + \left\|\left(\bm{I}_{n_2} - \bm{U}_2\bm{U}_2^\top\right)\bm{X}_2^\star\right\|_{\rm F}^2 + \left\|\left(\bm{I}_{n_3} - \bm{U}_3\bm{U}_3^\top\right)\bm{X}_3^\star\right\|_{\rm F}^2\right)\notag\\
&\quad + 2\left\|\bm{\mathcal{E}}  \times_1 \mathcal{P}_{\bm{U}_1} \times_2 \mathcal{P}_{\bm{U}_2} \times_3 \mathcal{P}_{\bm{U}_3}\right\|_{\rm F}^2.
\end{align}
Recognizing that for any $1 \leq \ell \leq k_i$, we have
\begin{align*}
	\sigma_{\ell}\left(\mathcal{M}_i\left(\bm{\mathcal{X}}^\star\right)\right) = \sigma_{\ell}\left(\bm{M}_i^\star\mathcal{M}_i\left(\bm{\mathcal{S}}^\star\right)\left(\bm{M}_{i+2}^\star \otimes \bm{M}_{i+1}^\star\right)^\top\right) \geq \prod_{i=1}^{3}\sigma_{k_i}\left(\bm{M}_i^\star\right)\sigma_{\ell}\left(\mathcal{M}_i\left(\bm{\mathcal{S}}^\star\right)\right) \stackrel{\eqref{ineq159}}{\geq} \sqrt{\frac{\beta^3n_1n_2n_3}{k_1k_2k_3}}\sigma_{\ell}\left(\mathcal{M}_i\left(\bm{\mathcal{S}}^\star\right)\right)
\end{align*}
and
\begin{align*}
	\sigma_{\ell}\left(\mathcal{M}_i\left(\bm{\mathcal{X}}^\star\right)\right) \leq \prod_{i=1}^{3}\left\|\bm{M}_i^\star\right\|\sigma_{\ell}\left(\mathcal{M}_i\left(\bm{\mathcal{S}}^\star\right)\right) \leq \sqrt{n_1n_2n_3}\sigma_{\ell}\left(\mathcal{M}_i\left(\bm{\mathcal{S}}^\star\right)\right).
\end{align*}
In view of Theorem \ref{thm:residual}, by choosing the numbers of iterations as in \eqref{iter1} and \eqref{iter2}, 
we have 
\begin{align}\label{ineq162}
	\left\|\left(\bm{I}_{n_1} - \bm{U}_1\bm{U}_1^\top\right)\bm{X}_1^\star\right\|_{\rm F}^2 &\leq n_1\left\|\left(\bm{I}_{n_1} - \bm{U}_1\bm{U}_1^\top\right)\bm{X}_1^\star\right\|_{2,\infty}^2\notag\\ &\lesssim n_1\cdot \frac{\mu_1k_1^3}{n_1}\left(k_1^2\sqrt{n_1}\omega_{\sf max}\log n + k_1\left(n_1n_2n_3\right)^{1/4}\omega_{\sf max}\log n\right)^2\notag\\
	&\stackrel{\eqref{ineq:dimension_assumption}~\text{and}~\eqref{ineq157}}{\lesssim} \frac{k^6}{\beta^2}\left(n_1n_2n_3\right)^{1/2}\omega_{\sf max}^2\log^2 n 
\end{align}
with probability at least $1 - O(n^{-10})$. 
Similarly, one has, with probability exceeding $1 - O(n^{-10})$,
\begin{subequations}
	\begin{align}
		\left\|\left(\bm{I}_{n_2} - \bm{U}_2\bm{U}_2^\top\right)\bm{X}_2^\star\right\|_{\rm F}^2 \lesssim \frac{k^6}{\beta^2}\left(n_1n_2n_3\right)^{1/2}\omega_{\sf max}^2\log^2 n,\label{ineq170a}\\
		\left\|\left(\bm{I}_{n_2} - \bm{U}_2\bm{U}_2^\top\right)\bm{X}_2^\star\right\|_{\rm F}^2 \lesssim \frac{k^6}{\beta^2}\left(n_1n_2n_3\right)^{1/2}\omega_{\sf max}^2\log^2 n.\label{ineq170b}
	\end{align}
\end{subequations}
Moreover, we learn from Theorem \ref{thm:residual} that with probability at least $1 - O(n^{-10})$,
\begin{align}\label{ineq:incoherence_estimate}
	\left\|\bm{U}_i\right\|_{2,\infty} = \left\|\bm{U}_i\bm{U}_i^\top\right\|_{2,\infty} \leq \left\|\bm{U}_i\bm{U}_i^\top - \bm{U}^\star\bm{U}^{\star}\right\|_{2,\infty} + \left\|\bm{U}^\star\right\|_{2,\infty} \leq 2\sqrt{\frac{\mu_ik_i^3}{n_i}},~\quad~\forall i \in [3].
\end{align}
Applying Lemma \ref{lm:noise} yields that with probability at least $1 - O(n^{-10})$,
\begin{align}\label{ineq171}
	\left\|\bm{\mathcal{E}}  \times_1 \mathcal{P}_{\bm{U}_1} \times_2 \mathcal{P}_{\bm{U}_2} \times_3 \mathcal{P}_{\bm{U}_3}\right\|_{\rm F}^2 &\leq k_1\left\|\bm{U}_1^\top\mathcal{M}_1\left(\bm{\mathcal{E}}\right)\left(\bm{U}_3 \otimes \bm{U}_2\right)\right\|^2\notag\\ &\lesssim kn\left(\mu_1k_1^2\right)\left(\mu_2k_2^2\right)\left(\mu_3k_3^2\right)k^3\omega_{\sf max}^2\log n\notag\\ &\stackrel{\eqref{ineq157}}{\lesssim} \frac{k^{13}}{\beta^6}n\omega_{\sf max}^2\log n.
\end{align}
Putting \eqref{ineq161} - \eqref{ineq171} together, we obtain 
\begin{align}\label{ineq172}
	\big\|\widehat{\bm{\mathcal{X}}} - \bm{\mathcal{X}}^\star\big\|_{\rm F}^2 \lesssim \frac{k^6}{\beta^2}\left(n_1n_2n_3\right)^{1/2}\omega_{\sf max}^2\log^2 n + \frac{k^{13}}{\beta^6}n\omega_{\sf max}^2\log n
\end{align}
with probability exceeding $1 - O(n^{-10})$. 
\paragraph{Step 3: deriving estimation accuracy of the center estimates.} We let
\begin{align}\label{def:1}
    \widehat{\bm{\theta}}_\ell^{(i)} = \big(\bm{U}_{i+2} \otimes \bm{U}_{i+1}\big)\widehat{\bm{b}}_{\ell}^{(i)} \in \bbR^{n_{-i}},~\quad~\forall i \in [3], \ell \in [k_i],
\end{align}
and also define
\begin{align}\label{def:2}
	\bm{\theta}_\ell^{(i)\star} = \left(\mathcal{M}_i\left(\bm{\mathcal{X}}^\star\right)\right)_{j, :}^\top \in \bbR^{n_{-i}},~\quad~\forall i \in [3], \ell \in [k_i].
\end{align}
Here, $j \in [n_i]$ is any index satisfying $z_{i, j}^\star = \ell$ and the $\widehat{\bm{b}}_\ell^{(i)}$'s are the center estimates satisfying \eqref{ineq:k_means_approximate}. Recalling that $\widehat{\bm{B}}_i = \bm{U}_i\bm{U}_i^\top\mathcal{M}_i({\bm{\mathcal{Y}}})(\bm{U}_{i+2} \otimes \bm{U}_{i+1})$, we have
\begin{align*}
	\mathcal{M}_i\big(\widehat{\bm{\mathcal{X}}}\big) = \widehat{\bm{B}}_i\big(\bm{U}_{i+2}^\top \otimes \bm{U}_{i+1}^\top\big),~\quad~\forall i \in [3].
\end{align*}
This allows one to show that
\begin{align}
	\sum_{j=1}^{n_i}\left\|\big(\mathcal{M}_i\big(\widehat{\bm{\mathcal{X}}}\big)\big)_{j:}^\top - \widehat{\bm{\theta}}_{\widehat{z}_{i,j}}^{(i)}\right\|_2^2 &= \sum_{j=1}^{n_i}\left\|\big(\bm{U}_{i+2} \otimes \bm{U}_{i+1}\big)\left(\big(\widehat{\bm{B}}_i\big)_{j:}^\top - \widehat{\bm{b}}^{(i)}_{\widehat{z}_{i,j}}\right)\right\|_2^2\notag\\
	&= \sum_{j=1}^{n_i}\left\|\big(\widehat{\bm{B}}_i\big)_{j:}^\top - \widehat{\bm{b}}^{(i)}_{\widehat{z}_{i,j}}\right\|_2^2\notag\\ &\stackrel{\eqref{ineq:k_means_approximate}}{\leq} M\min_{\bm{b}_1, \dots, \bm{b}_{k_i} \in \bbR^{r_1r_2r_3/r_i}\atop \bm{z}_i \in [k_i]^{n_i}}\sum_{j=1}^{n_i}\left\|\big(\widehat{\bm{B}}_i\big)_{j:}^\top - \bm{b}_{z_{i,j}}\right\|_2^2\notag\\
	&= M\min_{\bm{\theta}_1, \dots, \bm{\theta}_{k_i} \in \bbR^{n_{-i}}\atop \bm{z}_i \in [k_i]^{n_i}}\sum_{j=1}^{n_i}\left\|\big(\mathcal{M}_i\big(\widehat{\bm{\mathcal{X}}}\big)\big)_{j:}^\top - \bm{\theta}_{z_{i,j}}\right\|_2^2\notag\\
	&\leq M\sum_{j=1}^{n_i}\left\|\big(\mathcal{M}_i\big(\widehat{\bm{\mathcal{X}}}\big)\big)_{j:}^\top - \bm{\theta}^{(i)\star}_{z_{i,j}^\star}\right\|_2^2\notag\\
	&= M\sum_{j=1}^{n_i}\left\|\big(\mathcal{M}_i\big(\widehat{\bm{\mathcal{X}}}\big)\big)_{j:}^\top - \left(\mathcal{M}_i\left(\bm{\mathcal{X}}^\star\right)\right)_{j:}^\top\right\|_2^2\notag\\
	&= M\big\|\widehat{\bm{\mathcal{X}}} - \bm{\mathcal{X}}^\star\big\|_{\rm F}^2.\label{ineq173}
\end{align}
Here, the fourth line makes use of \citet[Eqn. (38)]{han2022exact}. As a result, we have
\begin{align}\label{ineq176}
	\sum_{j=1}^{n_i}\left\|\widehat{\bm{\theta}}_{\widehat{z}_{i,j}}^{(i)} - \bm{\theta}^{(i)\star}_{z_{i,j}^\star}\right\|_2^2 &\leq 2\left(\sum_{j=1}^{n_i}\left\|\big(\mathcal{M}_i\big(\widehat{\bm{\mathcal{X}}}\big)\big)_{j:}^\top - \widehat{\bm{\theta}}_{\widehat{z}_{i,j}}^{(i)}\right\|_2^2 + \sum_{j=1}^{n_i}\left\|\big(\mathcal{M}_i\big(\widehat{\bm{\mathcal{X}}}\big)\big)_{j:}^\top - \bm{\theta}^{(i)\star}_{z_{i,j}^\star}\right\|_2^2\right)\notag\\
	&\leq 2\left(M\big\|\widehat{\bm{\mathcal{X}}} - \bm{\mathcal{X}}^\star\big\|_{\rm F}^2 + \big\|\widehat{\bm{\mathcal{X}}} - \bm{\mathcal{X}}^\star\big\|_{\rm F}^2\right)\notag\\
	&\leq 4M\big\|\widehat{\bm{\mathcal{X}}} - \bm{\mathcal{X}}^\star\big\|_{\rm F}^2.
\end{align}
Noting that $\bm{M}_i^{\star\top}\bm{M}_i^{\star}$ is a diagonal matrix and the diagonal entries
\begin{align*}
	\left(\bm{M}_i^{\star\top}\bm{M}_i^{\star}\right)_{\ell, \ell} = \sum_{j=1}^{n_i}\left(\bm{M}_i^\star\right)_{j,\ell}^2 = \left\{j \in [n_i]: z_{i, j}^\star = \ell\right\} \geq \beta n_i/k_i,
\end{align*}
we have
\begin{align}\label{ineq174}
	\sigma_{k_i}\left(\bm{M}_i^{\star}\right) \geq \sqrt{\beta n_i/k_i},~\quad~\forall i \in [3].
\end{align}
As a consequence, for all $\ell_1 \neq \ell_2 \in [k_i]$, we can derive
\begin{align}\label{ineq175}
	\left\|\bm{\theta}_{\ell_1}^{(i)\star} - \bm{\theta}_{\ell_2}^{(i)\star}\right\| &= \big\|\left(\mathcal{M}_i\left(\bm{\mathcal{X}}\right)\right)_{j_1, :} - \left(\mathcal{M}_i\left(\bm{\mathcal{X}}\right)\right)_{j_2, :}\big\|\notag\\
	&= \left\|\big(\left(\mathcal{M}_i\left(\bm{\mathcal{S}}\right)\right)_{\ell_1, :} - \left(\mathcal{M}_i\left(\bm{\mathcal{S}}\right)\right)_{\ell_2, :}\big)\left(\bm{M}_{i+2}^\star \otimes \bm{M}_{i+1}^\star\right)\right\|\notag\\
	&\geq \big\|\left(\mathcal{M}_i\left(\bm{\mathcal{S}}\right)\right)_{\ell_1, :} - \left(\mathcal{M}_i\left(\bm{\mathcal{S}}\right)\right)_{\ell_2, :}\big\|\sigma_{k_{i+1}}\left(\bm{M}_{i+1}^{\star}\right)\sigma_{k_{i+2}}\left(\bm{M}_{i+2}^{\star}\right)\notag\\
	&\geq \beta\sqrt{\frac{n_{-i}}{k_{-i}}}\Delta_i,
\end{align}
where $j_1$ (resp.~$j_2$) is any index such that $z_{i, j_1}^\star = \ell_1$ (resp.~$z_{i, j_2}^\star = \ell_2$). We define
\begin{align}\label{def:S_i}
	\mathcal{S}_i := \left\{j \in [n_i]: \left\|\widehat{\bm{\theta}}_{\widehat{z}_{i,j}}^{(i)} - \bm{\theta}^{(i)\star}_{z_{i, j}^\star}\right\|_2 \geq \frac{\beta}{2}\sqrt{\frac{n_{-i}}{k_{-i}}}\Delta_i\right\}.
\end{align}
In view of \eqref{ineq172} and \eqref{ineq176}, with probability exceeding $1 - O(n^{-10})$,
\begin{align}\label{ineq177}
    \left|\mathcal{S}_i\right| &\leq \frac{\sum_{j=1}^{n_i}\left\|\widehat{\bm{\theta}}_{\widehat{z}_{i,j}}^{(i)} - \bm{\theta}^{(i)\star}_{z_{i, j}^\star}\right\|_2^2}{\left(\frac{\beta}{2}\sqrt{\frac{n_{-i}}{k_{-i}}}\Delta_i\right)^2}\notag\\&\leq \frac{4M\cdot C_6\left(\frac{k^6}{\beta^2}\left(n_1n_2n_3\right)^{1/2}\omega_{\sf max}^2\log^2 n + \frac{k^{13}}{\beta^6}n\omega_{\sf max}^2\log n\right)}{\left(\frac{\beta}{2}\sqrt{\frac{n_{-i}}{k_{-i}}}\Delta_i\right)^2}\notag\\
    &\leq \frac{\beta}{2}\frac{n_i}{k_i},
\end{align}
provided that 
$$\Delta_i \geq C_1\sqrt{M}\left(\frac{k^{9/2}}{\beta^{5/2}}\left(n_1n_2n_3\right)^{-1/4}\omega_{\sf max}\log n + \frac{k^{8}}{\beta^{9/2}}\left(n_1n_2n_3/n\right)^{-1/2}\omega_{\sf max}\sqrt{\log n}\right).$$ 
For each $1\leq i\leq 3$ and $\ell \in [k_i]$, denote by $\mathcal{N}_{i,\ell}$ the following set:
\begin{align}
	\mathcal{N}_{i,\ell} := \left\{j \in [n_i]: z_{i, j}^\star = \ell, j \in \mathcal{S}^c\right\}.
\end{align}
Then we can verify that with probability exceeding $1 - O(n^{-10})$, the following two important properties of the $\mathcal{N}_{i,\ell}$'s hold:
\begin{itemize}
	\item[1]. The $\mathcal{N}_{i,\ell}$'s are nonempty:
	\begin{align}\label{ineq178}
		\left|\mathcal{N}_{i,\ell}\right| \geq \big|\big\{j \in [n_i]: z_{i, j}^\star = \ell\big\}\big| - \left|\mathcal{S}_i\right| \stackrel{\eqref{def:beta}~\text{and}~\eqref{ineq177}}{\geq} \beta\frac{n_i}{k_i} - \frac{\beta}{2}\frac{n_i}{k_i} = \frac{\beta}{2}\frac{n_i}{k_i} > 0.
	\end{align}
    \item[2.] For any $i \in [3]$, the sets $\{\widehat{z}_{i,j}: j \in \mathcal{N}_{i, \ell}\}, \ell \in [k_i]$ are disjoint: for all $\ell_1 \neq \ell_2 \in [k_i], j_1 \in \mathcal{N}_{i, \ell_1}, j_2 \in \mathcal{N}_{i, \ell_2}$, 
    \begin{align*}
    	\left\|\widehat{\bm{\theta}}^{(i)}_{\widehat{z}_{i,j_1}} - \widehat{\bm{\theta}}^{(i)}_{\widehat{z}_{i,j_2}}\right\|_2 &\geq \left\|{\bm{\theta}}^{(i)\star}_{z_{i, j_1}^\star} - {\bm{\theta}}^{(i)\star}_{z_{i, j_2}^\star}\right\|_2 - \left\|\widehat{\bm{\theta}}^{(i)}_{\widehat{z}_{i,j_1}} - {\bm{\theta}}^{(i)\star}_{z_{i, j_1}^\star}\right\|_2 - \left\|\widehat{\bm{\theta}}^{(i)}_{\widehat{z}_{i,j_2}} - {\bm{\theta}}^{(i)\star}_{z_{i, j_2}^\star}\right\|_2\\
    	&= \left\|{\bm{\theta}}^{(i)\star}_{\ell_1} - {\bm{\theta}}^{(i)\star}_{\ell_2}\right\|_2 - \left\|\widehat{\bm{\theta}}^{(i)}_{\widehat{z}_{i,j_1}} - {\bm{\theta}}^{(i)\star}_{z_{i, j_1}^\star}\right\|_2 - \left\|\widehat{\bm{\theta}}^{(i)}_{\widehat{z}_{i,j_2}} - {\bm{\theta}}^{(i)\star}_{z_{i, j_2}^\star}\right\|_2\\
    	&\stackrel{\eqref{ineq175}}{>} \beta\sqrt{\frac{n_{-i}}{k_{-i}}}\Delta_i - \frac{\beta}{2}\sqrt{\frac{n_{-i}}{k_{-i}}}\Delta_i - \frac{\beta}{2}\sqrt{\frac{n_{-i}}{k_{-i}}}\Delta_i = 0,
    \end{align*}
which implies $\widehat{z}_{i,j_1} \neq \widehat{z}_{i,j_2}$ and further tells us $\{\widehat{z}_{i,j}: j \in \mathcal{N}_{i, \ell_1}\} \cap \{\widehat{z}_{i,j}: j \in \mathcal{N}_{i, \ell_2}\} = \emptyset$.
\end{itemize}
Therefore, with probability at least $1 - O(n^{-10})$, for any $i \in [3]$, there exists a permutation $\phi_i: [k_i] \to [k_i]$ such that
\begin{align}\label{permutation}
	\widehat{z}_{i,j} = \phi_i\left(z_{i, j}^\star\right),~\quad~\forall j \in \mathcal{S}_i^c, i \in [3].
\end{align}
In view of \eqref{ineq177}, with probability at least $1 - O(n^{-10})$,  one has
\begin{align}\label{ineq:loss}
	{\sf MCR}\left(\widehat{\bm{z}}_i, \bm{z}_i^\star\right) &\leq \frac{1}{n_i}\left|\left\{j \in [n_i]: \widehat{z}_{i,j} \neq \phi_i\left(z_{i, j}^\star\right)\right\}\right| \leq \frac{1}{n_i}\left|\mathcal{S}_i\right|\notag\\
	&\leq \frac{4M\cdot C_6\left(\frac{k^6}{\beta^2}\left(n_1n_2n_3\right)^{1/2}\omega_{\sf max}^2\log^2 n + \frac{k^{13}}{\beta^6}n\omega_{\sf max}^2\log n\right)}{n_i\left(\frac{\beta}{2}\sqrt{\frac{n_{-i}}{k_{-i}}}\Delta_i\right)^2}
\end{align}
for all $i \in [3]$, and   
\begin{align}\label{ineq:theta_difference}
	\left\|\widehat{\bm{\theta}}_{\phi_i\left(\ell\right)}^{(i)} - \bm{\theta}^{(i)\star}_{\ell}\right\|_2^2 &\leq \frac{\sum_{j=1}^{n_i}\left\|\widehat{\bm{\theta}}_{\widehat{z}_{i,j}}^{(i)} - \bm{\theta}^{(i)\star}_{z_{i, j}^\star}\right\|_2^2}{\left|\left\{j \in [n_i]: \widehat{z}_{i,j} = \phi_i\left(\ell\right), z_{i, j}^\star = \ell\right\}\right|}\notag\\
	&\leq \frac{\sum_{j=1}^{n_i}\left\|\widehat{\bm{\theta}}_{\widehat{z}_{i,j}}^{(i)} - \bm{\theta}^{(i)\star}_{z_{i, j}^\star}\right\|_2^2}{\left|\left\{j \in [n_i]: z_{i, j}^\star = \ell, j \in \mathcal{S}^c\right\}\right|}\notag\\
	&= \frac{\sum_{j=1}^{n_i}\left\|\widehat{\bm{\theta}}_{\widehat{z}_{i,j}}^{(i)} - \bm{\theta}^{(i)\star}_{z_{i, j}^\star}\right\|_2^2}{\left|\mathcal{N}_{i,\ell}\right|}\notag\\
	&\leq \frac{4M\cdot C\left(\frac{k^6}{\beta^2}\left(n_1n_2n_3\right)^{1/2}\omega_{\sf max}^2\log^2 n + \frac{k^{13}}{\beta^6}n\omega_{\sf max}^2\log n\right)}{\frac{\beta n_i}{2k_i}}\notag\\
	&\leq \frac{C_7M\left(\frac{k^7}{\beta^3}\left(n_1n_2n_3\right)^{1/2}\omega_{\sf max}^2\log^2 n + \frac{k^{14}}{\beta^7}n\omega_{\sf max}^2\log n\right)}{n_i}
\end{align}
for all $\ell \in [k_i]$, 
provided that $$\Delta_i \geq C_1\sqrt{M}\left(\frac{k^{9/2}}{\beta^{5/2}}\left(n_1n_2n_3\right)^{-1/4}\omega_{\sf max}\log n + \frac{k^{8}}{\beta^{9/2}}\left(n_1n_2n_3/n\right)^{-1/2}\omega_{\sf max}\sqrt{\log n}\right).$$ 
Here, the fourth line of \eqref{ineq:theta_difference} makes use of \eqref{ineq172}, \eqref{ineq176} and \eqref{ineq178}.

\paragraph{Step 4: proving $\mathsf{MCR}\big(\widehat{\bm{z}}_i, \bm{z}_i^\star\big) = 0$.} Finally, we would like to show that with probability exceeding $1 - O(n^{-10})$, $\widehat{\bm{z}}_i = \phi_i(\bm{z}_i^\star)$ for all $i \in [3]$. In view of Theorem \ref{thm:residual}, Lemma \ref{lm:noise}, \eqref{ineq157} and \eqref{ineq:incoherence_estimate}, with probability at least $1 - O(n^{-10})$, 
\begin{align}\label{ineq179}
	\left\|\left(\bm{I}_{n_i} - \bm{U}_i\bm{U}_i^\top\right)\bm{X}_i^\star\right\|_{2,\infty} &\lesssim \sqrt{\frac{\mu_ik_i^3}{n_i}}\left(k_i^2\sqrt{n_1}\omega_{\sf max}\log n + k_i\left(n_1n_2n_3\right)^{1/4}\omega_{\sf max}\log n\right)\notag\\
	&\stackrel{\eqref{ineq157}}{\lesssim} \frac{k^3}{\beta}\frac{\left(n_1n_2n_3\right)^{1/4}}{n_i^{1/2}}\omega_{\sf max}\log n
\end{align}
and
\begin{align}\label{ineq180}
	\left\|\bm{U}_i\bm{U}_i^\top\bm{E}_i\left(\bm{U}_{i+2}\bm{U}_{i+2}^\top \otimes \bm{U}_{i+1}\bm{U}_{i+1}^\top\right)\right\|_{2,\infty} &\leq \left\|\bm{U}_i\right\|_{2,\infty}\left\|\bm{U}_i^\top\bm{E}_i\left(\bm{U}_{i+2} \otimes \bm{U}_{i+1}\right)\right\|\notag\\
	&\stackrel{\eqref{ineq:incoherence_estimate}~\text{and}~\text{Lemma}~\ref{lm:noise}}{\lesssim} \sqrt{\frac{\mu_ik_i^3}{n_i}}\omega_{\sf max}\sqrt{n\left(\mu_1k_1^2\right)\left(\mu_2k_2^2\right)\left(\mu_3k_3^2\right)k^3\log n}\notag\\
	&\stackrel{\eqref{ineq157}}{\lesssim} \sqrt{\frac{k^4}{\beta^2n_i}}\frac{k^6}{\beta^3}\omega_{\sf max}\sqrt{n\log n}\notag\\
	&\leq \frac{k^8}{\beta^4}\omega_{\sf max}\sqrt{\frac{n\log n}{n_i}}
\end{align}
holds for all $i \in [3]$. 
Here, $\bm{X}_i^\star  = \mathcal{M}_i(\bm{\mathcal{X}}^\star)$ and $\bm{E}_i  = \mathcal{M}_i(\bm{\mathcal{E}})$. By virtue of \eqref{ineq:minimum_label}, we know that for any $i \in [3]$,
\begin{align}\label{ineq183}
	\left\{j \in [n_1]: \widehat{z}_{i,j} \neq \phi_i\left(z_{i, j}^\star\right)\right\} \subseteq \left\{j \in [n_1]: \exists \ell \in [k_i]~\text{s.t.}~\left\|\big(\widehat{\bm{B}}_i\big)_{j:}^\top - \widehat{\bm{b}}^{(i)}_{\ell}\right\|_2 \leq \left\|\big(\widehat{\bm{B}}_i\big)_{j:}^\top - \widehat{\bm{b}}^{(i)}_{\phi_i(z_{i, j}^\star)}\right\|_2\right\}.
\end{align}
For any fixed $\ell \neq \phi_i(z_{i, j}^\star) \in [k_i]$, recalling that $\widehat{\bm{\theta}}_\ell^{(i)} = \big(\bm{U}_{i+2} \otimes \bm{U}_{i+1}\big)\widehat{\bm{b}}_{\ell}^{(i)}$, one has
\begin{align}\label{ineq182}
	&\mathbbm{1}\left\{\left\|\big(\widehat{\bm{B}}_i\big)_{j:}^\top - \widehat{\bm{b}}^{(i)}_{\ell}\right\|_2 \leq \left\|\big(\widehat{\bm{B}}_i\big)_{j:}^\top - \widehat{\bm{b}}^{(i)}_{\phi_i(z_{i, j}^\star)}\right\|_2\right\}\notag\\ &\quad = \mathbbm{1}\left\{\left\|\big(\widehat{\bm{B}}_i\big)_{j:}^\top - \widehat{\bm{b}}^{(i)}_{\ell}\right\|_2 + \left\|\big(\widehat{\bm{B}}_i\big)_{j:}^\top - \widehat{\bm{b}}^{(i)}_{\phi_i(z_{i, j}^\star)}\right\|_2 \leq 2\left\|\big(\widehat{\bm{B}}_i\big)_{j:}^\top - \widehat{\bm{b}}^{(i)}_{\phi_i(z_{i, j}^\star)}\right\|_2\right\}\notag\\
	&\quad \leq \mathbbm{1}\left\{\left\|\widehat{\bm{b}}^{(i)}_{\ell} - \widehat{\bm{b}}^{(i)}_{\phi_i(z_{i, j}^\star)}\right\|_2 \leq 2\left\|\big(\widehat{\bm{B}}_i\big)_{j:}^\top - \widehat{\bm{b}}^{(i)}_{\phi_i(z_{i, j}^\star)}\right\|_2\right\}\notag\\
	&\quad \leq \mathbbm{1}\left\{\left\|\widehat{\bm{\theta}}^{(i)}_{\ell} - \widehat{\bm{\theta}}^{(i)}_{\phi_i(z_{i, j}^\star)}\right\|_2 \leq 2\left\|\big(\widehat{\bm{B}}_i\big)_{j:}^\top - \widehat{\bm{b}}^{(i)}_{\phi_i(z_{i, j}^\star)}\right\|_2\right\}\notag\\
	&\quad = \mathbbm{1}\left\{\left\|\widehat{\bm{\theta}}^{(i)}_{\ell} - \widehat{\bm{\theta}}^{(i)}_{\phi_i(z_{i, j}^\star)}\right\|_2 \leq 2\left\|\left(\bm{U}_i\right)_{j,:}\bm{U}_i^\top\bm{Y}_i\left(\bm{U}_{i+2} \otimes \bm{U}_{i+1}\right) - \widehat{\bm{b}}^{(i)\top}_{\phi_i(z_{i, j}^\star)}\right\|_2\right\},
\end{align}
where $\bm{Y}_i = \mathcal{M}_i(\bm{\mathcal{Y}}) = \bm{X}_i^\star + \bm{E}_i$. Note that 
\begin{align*}
	&\left\|\left(\bm{U}_i\right)_{j,:}\bm{U}_i^\top\bm{Y}_i\left(\bm{U}_{i+2} \otimes \bm{U}_{i+1}\right) - \widehat{\bm{b}}^{(i)\top}_{\phi_i(z_{i, j}^\star)}\right\|_2\\
	&\quad \leq \left\|\left(\bm{U}_i\right)_{j,:}\bm{U}_i^\top\bm{E}_i\left(\bm{U}_{i+2} \otimes \bm{U}_{i+1}\right)\right\|_2 + \left\|\left(\bm{U}_i\right)_{j,:}\bm{U}_i^\top\bm{X}_i^\star\left(\bm{U}_{i+2} \otimes \bm{U}_{i+1}\right) - \widehat{\bm{b}}^{(i)\top}_{\phi_i(z_{i, j}^\star)}\right\|_2\\
	&\quad \leq \left\|\bm{U}_i\bm{U}_i^\top\bm{E}_i\left(\bm{U}_{i+2} \otimes \bm{U}_{i+1}\right)\right\|_{2,\infty} + \left\|\left(\bm{U}_i\right)_{j,:}\bm{U}_i^\top\bm{X}_i^\star\left(\bm{U}_{i+2} \otimes \bm{U}_{i+1}\right) - \widehat{\bm{\theta}}^{(i)\top}_{\phi_i(z_{i, j}^\star)}\left(\bm{U}_{i+2} \otimes \bm{U}_{i+1}\right)\right\|_2\\
	&\quad \leq \left\|\bm{U}_i\bm{U}_i^\top\bm{E}_i\left(\bm{U}_{i+2} \otimes \bm{U}_{i+1}\right)\right\|_{2,\infty} + \left\|\left(\bm{U}_i\right)_{j,:}\bm{U}_i^\top\bm{X}_i^\star - \widehat{\bm{\theta}}^{(i)\top}_{\phi_i(z_{i, j}^\star)}\right\|_2\\
	&\quad \leq \left\|\bm{U}_i\bm{U}_i^\top\bm{E}_i\left(\bm{U}_{i+2} \otimes \bm{U}_{i+1}\right)\right\|_{2,\infty} + \left\|\left(\bm{U}_i\right)_{j,:}\bm{U}_i^\top\bm{X}_i^\star - \left(\bm{X}_i^\star\right)_{j,:}\right\|_2 + \left\|\left(\bm{X}_i^\star\right)_{j,:} - \widehat{\bm{\theta}}^{(i)\top}_{\phi_i(z_{i, j}^\star)}\right\|_2\\
	&\quad \leq \left\|\bm{U}_i\bm{U}_i^\top\bm{E}_i\left(\bm{U}_{i+2} \otimes \bm{U}_{i+1}\right)\right\|_{2,\infty} + \left\|\left(\bm{I}_{n_i} - \bm{U}_i\bm{U}_i^\top\right)\bm{X}_i^\star\right\|_{2,\infty} + \left\|\left(\bm{X}_i^\star\right)_{j,:} - \widehat{\bm{\theta}}^{(i)\top}_{\phi_i(z_{i, j}^\star)}\right\|_2\\
	&\quad = \left\|\bm{U}_i\bm{U}_i^\top\bm{E}_i\left(\bm{U}_{i+2} \otimes \bm{U}_{i+1}\right)\right\|_{2,\infty} + \left\|\left(\bm{I}_{n_i} - \bm{U}_i\bm{U}_i^\top\right)\bm{X}_i^\star\right\|_{2,\infty} + \left\|\widehat{\bm{\theta}}^{(i)}_{\phi_i(z_{i, j}^\star)} - \bm{\theta}^{(i)\star}_{z_{i, j}^\star}\right\|_2\\
	&\quad \leq \left\|\bm{U}_i\bm{U}_i^\top\bm{E}_i\left(\bm{U}_{i+2} \otimes \bm{U}_{i+1}\right)\right\|_{2,\infty} + \left\|\left(\bm{I}_{n_i} - \bm{U}_i\bm{U}_i^\top\right)\bm{X}_i^\star\right\|_{2,\infty} + \sup_{a \in [k_i]}\left\|\widehat{\bm{\theta}}^{(i)}_{\phi_i(a)} - \bm{\theta}^{(i)\star}_{a}\right\|_2\\
\end{align*}
and
\begin{align*}
	\left\|\widehat{\bm{\theta}}^{(i)}_{\ell} - \widehat{\bm{\theta}}^{(i)}_{\phi_i(z_{i, j}^\star)}\right\|_2 &\geq \left\|\bm{\theta}^{(i)\star}_{\phi_i^{-1}(\ell)} - \bm{\theta}^{(i)\star}_{z_{i, j}^\star}\right\|_2 - \left\|\widehat{\bm{\theta}}^{(i)}_{\ell} - \bm{\theta}^{(i)\star}_{\phi_i^{-1}(\ell)}\right\|_2 - \left\|\widehat{\bm{\theta}}^{(i)}_{\phi_i(z_{i, j}^\star)} - \bm{\theta}^{(i)\star}_{z_{i, j}^\star}\right\|_2\\
	&\stackrel{\eqref{ineq175}}{\geq} \beta\sqrt{\frac{n_{-i}}{k_{-i}}}\Delta_i - 2\sup_{a \in [k_i]}\left\|\widehat{\bm{\theta}}^{(i)}_{\phi_i(a)} - \bm{\theta}^{(i)\star}_{a}\right\|_2.
\end{align*}
Putting the previous two inequalities, \eqref{ineq:theta_difference}, \eqref{ineq179} and \eqref{ineq180} together yields: with probability at least $1 - O(n^{-10})$, 
\begin{align}\label{ineq181}
	&\left\|\widehat{\bm{\theta}}^{(i)}_{\ell} - \widehat{\bm{\theta}}^{(i)}_{\phi_i(z_{i, j}^\star)}\right\|_2 - 2\left\|\left(\bm{U}_i\right)_{j,:}\bm{U}_i^\top\bm{Y}_i\left(\bm{U}_{i+2} \otimes \bm{U}_{i+1}\right) - \widehat{\bm{b}}^{(i)\top}_{\phi_i(z_{i, j}^\star)}\right\|_2\notag\\
	&\quad \geq \beta\sqrt{\frac{n_{-i}}{k_{-i}}}\Delta_i - 2\sup_{a \in [k_i]}\left\|\widehat{\bm{\theta}}^{(i)}_{\phi_i(a)} - \bm{\theta}^{(i)\star}_{a}\right\|_2\notag\\
	&\qquad - 2\left(\left\|\bm{U}_i\bm{U}_i^\top\bm{E}_i\left(\bm{U}_{i+2} \otimes \bm{U}_{i+1}\right)\right\|_{2,\infty} + \left\|\left(\bm{I}_{n_i} - \bm{U}_i\bm{U}_i^\top\right)\bm{X}_i^\star\right\|_{2,\infty} + \sup_{a \in [k_i]}\left\|\widehat{\bm{\theta}}^{(i)}_{\phi_i(a)} - \bm{\theta}^{(i)\star}_{a}\right\|_2\right)\notag\\
	&\quad \geq \beta\sqrt{\frac{n_{-i}}{k_{-i}}}\Delta_i - 4\sqrt{\frac{C_7M\left(\frac{k^7}{\beta^3}\left(n_1n_2n_3\right)^{1/2}\omega_{\sf max}^2\log^2 n + \frac{k^{14}}{\beta^7}n\omega_{\sf max}^2\log n\right)}{n_i}}\notag\\
	&\qquad - C_8\frac{k^3}{\beta}\frac{\left(n_1n_2n_3\right)^{1/4}}{n_i^{1/2}}\omega_{\sf max}\log n - C_8\frac{k^8}{\beta^4}\omega_{\sf max}\sqrt{\frac{n\log n}{n_i}}\notag\\
	&\quad > 0
\end{align}
holds for all $\ell \in [k_i]$, 
provided that $$\Delta_i/\omega_{\sf max} \geq C_1\sqrt{M}\left(\frac{k^{9/2}}{\beta^{5/2}}\left(n_1n_2n_3\right)^{-1/4}\log n + \frac{k^{9}}{\beta^{5}}\left(n_1n_2n_3/n\right)^{-1/2}\sqrt{\log n}\right).$$

Combining \eqref{ineq183}, \eqref{ineq182} and \eqref{ineq181}, we arrive at
\begin{align*}
	\mathsf{MCR}\big(\widehat{\bm{z}}_i, \bm{z}_i^\star\big) = 0,~\quad~\forall i \in [3]
\end{align*}
with probability exceeding $1 - O(n^{-10})$.

    \section{Proof of Theorem \ref{thm:oracle_two_to_infty}}\label{sec:proof_theorem_oracle_two_to_infty}
\paragraph{Part (a): proving $\mathcal{A} \neq \emptyset$.} Let
\begin{align}\label{def:bar_r}
	\overline{r} = \begin{cases}
		\max\mathcal{A}, \quad & \text{if } \mathcal{A} \neq \emptyset;\\
		0, \quad & \text{otherwise}.
	\end{cases},
\end{align}
We claim that 
\begin{align}\label{ineq1}
	\sigma_{\overline{r}+1}^\star \leq 2C_0r\big[(m_1m_2)^{1/4} + rm_1^{1/2}\big]\omega_{\sf max}\log m.
\end{align}
In fact, if $\sigma_{\overline{r}+1}^\star  = 0$, then \eqref{ineq1} clearly holds. If $\sigma_{\overline{r}+1}^\star  > 0$, then we must have $\overline{r} < r$. Let
\begin{align*}
	i = \min\left\{j: \overline{r}+1 \leq j \leq r, \sigma_j^\star \geq \frac{4r}{4r - 1}\sigma_{j+1}^\star\right\}.
\end{align*}
Note that such $i$ does exist as the largest $j \leq r$ satisfying $\sigma_{j}^\star > 0$ must obey $\sigma_j^\star \geq \frac{4r}{4r - 1}\sigma_{j+1}^\star$. The definition of $\overline{r}$ immediately tells us that $\sigma_i^\star \leq C_0r[(m_1m_2)^{1/4} + rm_1^{1/2}]\log m$ and consequently one has
\begin{align*}
	\sigma_{\overline{r}+1}^\star = \sigma_i^\star\prod_{j=\overline{r}+1}^{i-1}\frac{\sigma_j^\star}{\sigma_{j+1}^\star} \leq \sigma_i^\star\left(\frac{4r}{4r-1}\right)^{i-\overline{r} - 1} \leq \sigma_i^\star\cdot \left(1 + \frac{1}{3r}\right)^r \leq 2C_0r\big[(m_1m_2)^{1/4} + rm_1^{1/2}\big]\omega_{\sf max}\log m.
\end{align*}
The first inequality holds due to the definition of $i$. This combined with the assumption on $\sigma_1^\star$ reveals that $\overline{r} \neq 0$, i.e., $\mathcal{A} \neq \emptyset$ and
\begin{align}\label{def:bar_r_new}
	\overline{r} = \max\mathcal{A}.
\end{align}

\paragraph{Part (b): proving \eqref{ineq:oracle}.} The rest of the proof is devoted to proving \eqref{ineq:oracle}. Letting
\begin{subequations}
\begin{align}\label{def:U_1}
	\bm{U}^{\star(1)} &= \left[\bm{u}_1^\star, \dots, \bm{u}_{\overline{r}}^\star\right],~\qquad~\bm{\Sigma}^{\star(1)} = {\sf diag}\left(\sigma_1^\star, \dots, \sigma_{\overline{r}}^\star\right),~\qquad~\bm{V}^{\star(1)} = \left[\bm{v}_1^\star, \dots, \bm{v}_{\overline{r}}^\star\right],\\
	\bm{U}^{\star(2)} &= \left[\bm{u}_{\overline{r}+1}^\star, \dots, \bm{u}_{r}^\star\right],~\qquad~\bm{\Sigma}^{\star(2)} = {\sf diag}\left(\sigma_{\overline{r}+1}^\star, \dots, \sigma_{r}^\star\right),~\qquad~\bm{V}^{\star(2)} = \left[\bm{v}_{\overline{r}+1}^\star, \dots, \bm{v}_{r}^\star\right],
\end{align}
\end{subequations}
we can derive 
\begin{align}
	\bm{U}^\star = \big[\bm{U}^{\star(1)}\ \bm{U}^{\star(2)}\big],~\qquad~\bm{V}^\star = \big[\bm{V}^{\star(1)}\ \bm{V}^{\star(2)}\big],~\qquad~\bm{\Sigma}^\star = \begin{bmatrix}
		\bm{\Sigma}^{\star(1)}  &\bm{0}\\
		\bm{0}  &\bm{\Sigma}^{\star(2)}
	\end{bmatrix}.
\end{align}
Let the SVD of $\bm{U}^{\star(1)}\bm{\Sigma}^{\star(1)} + \bm{E}\bm{V}^{\star(1)}$ be denoted by 
\begin{align}\label{svd1}
	\widetilde{\bm{U}}^{(1)}\widetilde{\bm{\Sigma}}^{(1)}\widetilde{\bm{W}}^{(1)\top} = \bm{U}^{\star(1)}\bm{\Sigma}^{\star(1)} + \bm{E}\bm{V}^{\star(1)}.
\end{align}
Here, $\widetilde{\bm{U}}^{(1)} \in \mathcal{O}^{n_1, \overline{r}}$, $\widetilde{\bm{\Sigma}}^{(1)} = {\sf diag}(\widetilde{\sigma}_1, \dots, \widetilde{\sigma}_{\overline{r}})$ where $\widetilde{\sigma}_1 \geq \cdots \geq \widetilde{\sigma}_{\overline{r}} \geq 0$ , $\widetilde{\bm{W}}^{(1)} \in \mathcal{O}^{\overline{r},\overline{r}}$. Then one has
\begin{align}\label{svd2}
	\big(\bm{U}^{\star(1)}\bm{\Sigma}^{\star(1)} + \bm{E}\bm{V}^{\star(1)}\big)\big(\bm{U}^{\star(1)}\bm{\Sigma}^{\star(1)} + \bm{E}\bm{V}^{\star(1)}\big)^\top = \widetilde{\bm{U}}^{(1)}\big(\widetilde{\bm{\Sigma}}^{(1)}\big)^2\widetilde{\bm{U}}^{(1)\top}.
\end{align}
We can then write $\bm{M}^{\sf oracle}$ as follows:
\begin{align}\label{eq:decomposition_M_oracle}
	\bm{M}^{\sf oracle} 
	&= \left(\bm{X}^\star + \bm{E}\right)\bm{V}^\star\bm{V}^{\star\top}\left(\bm{X}^\star + \bm{E}\right)^\top + \mathcal{P}_{\sf off\text{-}diag}\left(\bm{E}\bm{E}^\top - \bm{E}\bm{V}^\star\bm{V}^{\star\top}\bm{E}^\top\right)\notag\\
	&= \left(\bm{X}^\star + \bm{E}\right)\bm{V}^{\star(1)}\bm{V}^{\star(1)\top}\left(\bm{X}^\star + \bm{E}\right)^\top + \left(\bm{X}^\star + \bm{E}\right)\bm{V}^{\star(2)}\bm{V}^{\star(2)\top}\left(\bm{X}^\star + \bm{E}\right)^\top\notag\\
	&\quad + \mathcal{P}_{\sf off\text{-}diag}\left(\bm{E}\bm{E}^\top - \bm{E}\bm{V}^\star\bm{V}^{\star\top}\bm{E}^\top\right)\notag\\
	&= \underbrace{\big(\bm{U}^{\star(1)}\bm{\Sigma}^{\star(1)} + \bm{E}\bm{V}^{\star(1)}\big)\big(\bm{U}^{\star(1)}\bm{\Sigma}^{\star(1)} + \bm{E}\bm{V}^{\star(1)}\big)^\top + \mathcal{P}_{(\widetilde{\bm{U}}^{(1)})_{\perp}}\bm{U}^{\star(2)}\big(\bm{\Sigma}^{\star(2)}\big)^2\bm{U}^{\star(2)\top}\mathcal{P}_{(\widetilde{\bm{U}}^{(1)})_{\perp}}}_{=: \widetilde{\bm{M}}}\notag\\
	&\quad + \underbrace{\bm{U}^{\star(2)}\bm{\Sigma}^{\star(2)}\bm{V}^{\star(2)\top}\bm{E}^\top + \bm{E}\bm{V}^{\star(2)}\bm{\Sigma}^{\star(2)}\bm{U}^{\star(2)\top} + \bm{E}\bm{V}^{\star(2)}\bm{V}^{\star(2)\top}\bm{E}^\top}_{=:\bm{Z}_1}\notag\\
	&\quad + \underbrace{\mathcal{P}_{\widetilde{\bm{U}}^{(1)}}\bm{U}^{\star(2)}\big(\bm{\Sigma}^{\star(2)}\big)^2\bm{U}^{\star(2)\top}\mathcal{P}_{\big(\widetilde{\bm{U}}^{(1)}\big)_{\perp}} + \bm{U}^{\star(2)}\big(\bm{\Sigma}^{\star(2)}\big)^2\bm{U}^{\star(2)\top}\mathcal{P}_{\widetilde{\bm{U}}^{(1)}}}_{=:\bm{Z}_2}\notag\\
	&\quad + \underbrace{\mathcal{P}_{\sf off\text{-}diag}\left(\bm{E}\bm{E}^\top - \bm{E}\bm{V}^\star\bm{V}^{\star\top}\bm{E}^\top\right)}_{=:\bm{Z}_3}.
\end{align}
For convenience, we shall also let 
\begin{align}\label{eq:Z}
	\bm{Z} = \bm{Z}_1 + \bm{Z}_2 + \bm{Z}_3.
\end{align}

\subsection{Several key lemmas}
Before proceeding, we first introduce the following lemma, which allows us to bound an infinite sum of $\ell_{2,\infty}$ norms of perturbation matrix polynomials instead of bounding $\left\|\bm{U}_{:,1: r'}^{\sf oracle}\bm{U}_{:,1: r'}^{\sf oracle\top} - \widetilde{\bm{U}}_{:,1: r'}\widetilde{\bm{U}}_{:,1: r'}^{\top}\right\|_{2,\infty}$ directly.
\begin{lemma}\label{lm:space_estimate_expansion}
	Suppose that $\bm{M} = \overline{\bm{M}} + \bm{Z} \in \bbR^{n \times n}$, where $\overline{\bm{M}}$ and $\bm{Z}$ are both symmetric matrices. 
	Assume that $\overline{\bm{M}}$  is a matrix with rank not exceeding $r$ and has eigenvalues $\overline\lambda_1 \geq \dots \geq \overline\lambda_{r} \geq 0 $  
	and rank-$r$ leading eigenspace $\overline{\bm{U}} = [\overline{\bm{u}}_1,\ \dots,\ \overline{\bm{u}}_r]$ (so that $\overline{\bm{u}}_i$ is the eigenvector associated with $\overline{\lambda}_i$). 
	If there exists some $r_1$ obeying $1 \leq r_1 \leq r$ and
	\begin{align}
		\overline\lambda_{r_1} - \overline\lambda_{r_1 + 1} > 2\|\bm{Z}\|,
	\end{align}
	then it holds that
\begin{subequations}
		\begin{align}
		\big\|\overline{\bm{U}}_1\overline{\bm{U}}_1^\top - \bm{U}_1\bm{U}_1^\top\big\|_{2,\infty} &\leq \frac{8}{\pi}\sum_{k \geq 1}\left(\frac{2}{\overline{\lambda}_{r_1} - \overline{\lambda}_{r_1+1}}\right)^{k}\sum_{0 \leq j_1, \dots, j_{k+1} \leq r\atop \left(j_1, ..., j_{k+1}\right) \neq \bm{0}}\left\|\overline{\bm{P}}_{j_1}\bm{Z}\overline{\bm{P}}_{j_2}\bm{Z}\cdots\bm{Z}\overline{\bm{P}}_{j_{k+1}}\right\|_{2,\infty},\label{ineq:space_estimate_expansion}\\
		\big\|\big(\overline{\bm{U}}_1\overline{\bm{U}}_1^\top - \bm{U}_1\bm{U}_1^\top\big)\overline{\bm{M}}\big\|_{2,\infty} &\leq \frac{40}{\pi}\sum_{k \geq 1}\overline{\lambda}_{r_1}\left(\frac{2}{\overline{\lambda}_{r_1} - \overline{\lambda}_{r_1+1}}\right)^{k}\sum_{0 \leq j_1, \dots, j_{k+1} \leq r\atop \left(j_1, ..., j_{k+1}\right) \neq \bm{0}}\left\|\overline{\bm{P}}_{j_1}\bm{Z}\overline{\bm{P}}_{j_2}\bm{Z}\cdots\bm{Z}\overline{\bm{P}}_{j_{k+1}}\right\|_{2,\infty}. \label{ineq:residual_expansion}
	\end{align} 
\end{subequations}
	Here, $\overline{\bm{U}}_1$ and $\bm{U}_1$ denote the rank-$r_1$ leading eigen-subspace of $\overline{\bm{M}}$ and $\bm{M}$, respectively;  
	and we denote $\overline{\bm{P}}_j = \overline{\bm{u}}_j\overline{\bm{u}}_j^\top$ for any $1 \leq j \leq r$ 
	and $\overline{\bm{P}}_0 = \overline{\bm{U}}_{\perp}\overline{\bm{U}}_{\perp}^\top$. 
\end{lemma}
The proof of Lemma \ref{lm:space_estimate_expansion} can be found in Section \ref{proof:lm_space_estimate_expansion}. In addition, the following lemmas deliver sharp $\ell_{2,\infty}$ guarantees for some polynomials of the noise matrix.
\begin{lemma}[\cite{zhou2023deflated}, Lemma 2]\label{lm:power_V}
	Suppose that Assumption \ref{assump:noise_matrix} holds. Then we have, with probability at least $1 - O(m^{-10})$,
	\begin{align}\label{ineq:power_V}
		\left\|\left[\mathcal{P}_{\sf off\text{-}diag}\left(\bm{E}\bm{E}^\top\right)\right]^k\bm{E}\bm{V}^\star\right\|_{2,\infty} \leq C_3\sqrt{\mu r}\left(C_3\left(\sqrt{m_1m_2} + m_1\right)\omega_{\sf max}^2\log^2 m\right)^k\omega_{\sf max}\log m
	\end{align}
    for all $0 \leq k \leq \log n$, where $C_3$ is some sufficiently large constant.
\end{lemma}
\begin{lemma}[\cite{zhou2023deflated}, Lemma 3]\label{lm:power_U}
	Suppose that Assumption \ref{assump:noise_matrix} holds. Then we have, with probability at least $1 - O(m^{-10})$,
	\begin{align}\label{ineq:power_U}
		\left\|\left[\mathcal{P}_{\sf off\text{-}diag}\left(\bm{E}\bm{E}^\top\right)\right]^k\bm{U}^\star\right\|_{2,\infty} \leq C_3\sqrt{\frac{\mu r}{n_1}}\left(C_3\left(\sqrt{m_1m_2} + m_1\right)\omega_{\sf max}^2\log^2 m\right)^k
	\end{align}
	for all $0 \leq k \leq \log n$, where $C_3$ is some sufficiently large constant.
\end{lemma}
The following lemma, which was also established in \cite{zhou2023deflated}, provides some helpful consequences on the eigenvalue perturbation, the size of some perturbation matrix, and some incoherence properties of $\widetilde{\bm{U}}^{(1)}$.
\begin{lemma}[\cite{zhou2023deflated}, Lemma 4]\label{lm:event_collection}
	Instate the assumptions in Theorem \ref{thm:oracle_two_to_infty}. Then there exist some large enough constant $C_5 > 0$ such that with probability exceeding $1 - O(m^{-10})$,
	\begin{subequations}
		\begin{align}
			\big|\widetilde{\sigma}_i - \sigma_i^\star\big| &\leq \big\|\bm{E}\bm{V}^{\star(1)}\big\| \leq \left\|\bm{E}\bm{V}^\star\right\| \leq \sqrt{C}_5\sqrt{m_1}\omega_{\sf max}\log m,~\qquad~\forall i \in [\overline{r}],\label{ineq2a}\\
			\left\|\mathcal{P}_{\sf off\text{-}diag}\left(\bm{E}\bm{E}^\top - \bm{E}\bm{V}^\star\bm{V}^{\star\top}\bm{E}^\top\right)\right\| &\leq 3C_5\left(\sqrt{m_1m_2} + m_1\right)\omega_{\sf max}^2\log^2 m \label{ineq2b} \\
			\big\|\bm{U}^{\star(1)}\bm{U}^{\star(1)\top}\widetilde{\bm{U}}^{(1)} - \widetilde{\bm{U}}^{(1)}\big\|_{2, \infty} &\leq \frac{4C_5\sqrt{\mu r}\omega_{\sf max}\log m}{\sigma_r^\star}\leq \sqrt{\frac{\mu r}{m_1}},\label{ineq2c}\\
			\big\|\widetilde{\bm{U}}^{(1)}\big\|_{2, \infty} &\leq 2\sqrt{\frac{\mu r}{m_1}}.\label{ineq2d}
		\end{align}
	\end{subequations}
\end{lemma}
Finally, the following lemma develops $\ell_{2,\infty}$ bounds on the polynomials of the perturbation matrix $\bm{Z}_3$, which will play a key role in the subsequent proof.
\begin{lemma}\label{lm:noise_product}
	Suppose that Assumption \ref{assump:noise_matrix} holds. Let
	\begin{align}\label{eq:event1}
		\mathcal{E} = \{\eqref{ineq:power_V} \text{ and }\eqref{ineq:power_U} \text{ hold for }0 \leq k \leq \log n\} \cap \{\eqref{ineq2a}, \eqref{ineq2b}, \eqref{ineq2c}\text{ and }\eqref{ineq2d}\text{ hold}\}.
	\end{align}
Then there exists some large enough constant $C_2, C_3 > 0$ (independent of $C_0$) such that under $\mathcal{E}$, for any $0 \leq i \leq \log n$, one has
\begin{subequations}
	\begin{align}
		\left\|\bm{Z}_3^i\bm{U}^\star\right\|_{2, \infty} &\leq  3C_3\sqrt{\frac{\mu r}{m_1}}\left(C_3\left(\sqrt{m_1m_2} + m_1\right)\omega_{\sf max}^2\log^2 m\right)^{i},\label{ineq7a}\\
		\left\|\bm{Z}_3^i\bm{E}\bm{V}^\star\right\|_{2, \infty} &\leq 3C_3\sqrt{\mu r}\left(C_3\left(\sqrt{m_1m_2} + m_1\right)\omega_{\sf max}^2\log^2 m\right)^{i}\omega_{\sf max}\log m,\label{ineq7b}\\
		\big\|\bm{Z}_3^i\widetilde{\bm{U}}^{(1)}\big\|_{2, \infty} &\leq 4C_3\sqrt{\frac{\mu r}{m_1}}\left(C_3\left(\sqrt{m_1m_2} + m_1\right)\omega_{\sf max}^2\log^2 m\right)^i,\label{ineq7c}\\
		\left\|\bm{Z}_3^i\bm{Z}_1\right\|_{2,\infty} &\leq C_2\sqrt{\mu r}\left(\sigma_{\overline{r} + 1}^\star + \sqrt{m}_1\omega_{\sf max}\log m\right)\left(C_3\left(\sqrt{m_1m_2} + m_1\right)\omega_{\sf max}^2\log^2 m\right)^{i}\omega_{\sf max}\log m,\label{ineq7d}\\
		\left\|\bm{Z}_3^i\bm{Z}_2\right\|_{2,\infty} &\leq C_2\sqrt{\mu r}\left(C_3\left(\sqrt{m_1m_2} + m_1\right)\omega_{\sf max}^2\log^2 m\right)^{i}\omega_{\sf max}\sigma_{\overline{r}+1}^\star\log m.\label{ineq7e}
	\end{align}
\end{subequations}
\end{lemma}
The proof of Lemma \ref{lm:noise_product} is postponed to Section \ref{proof:lm_noise_product}. The union bound taken together with Lemma \ref{lm:power_V}, Lemma \ref{lm:power_U} and Lemma \ref{lm:event_collection} shows that
\begin{align}\label{ineq:event_probability}
	\bbP\left(\mathcal{E}\right) \geq 1 - O\left(n^{-10}\right).
\end{align}
In the rest of the proof, we assume that $\mathcal{E}$ occurs unless otherwise noted.

\subsection{Main steps for proving \eqref{ineq:oracle}}
\paragraph{Step 1: bounding $\|\widetilde{\bm{U}}^{(2)}\|_2$.} We start with controlling $\|\widetilde{\bm{U}}^{(2)}\|_2$.
Combining \eqref{ineq2a}, \cite[Lemma 2.5]{chen2021spectral} and Wedin's $\sin\bm{\Theta}$ theorem, one has
\begin{align}\label{ineq6}
	\big\|\widetilde{\bm{U}}^{(1)\top}\bm{U}^{\star(2)}\big\| &\leq \big\|\widetilde{\bm{U}}^{(1)\top}\big(\bm{U}^{\star(1)}\big)_{\perp}\big\| = \big\|\widetilde{\bm{U}}^{(1)}\widetilde{\bm{U}}^{(1)\top} - \bm{U}^{\star(1)}\bm{U}^{\star(1)\top}\big\|\notag\\ &\leq \frac{2\big\|\bm{E}\bm{V}^{\star(1)}\big\|}{\sigma_{\overline{r}}^\star} \leq \frac{2\sqrt{C_5}\sqrt{m}_1\omega_{\sf max}\log m}{\sigma_{\overline{r}}^\star} \ll \frac{1}{2},
\end{align}
where the first inequality makes use of the fact that $\bm{U}_1^{\star\top}\bm{U}_2^\star = \bm{0}$. Note that the rank of $\widetilde{\bm{M}}$ is at most $r$. We also denote the eigendecomposition of $\mathcal{P}_{(\widetilde{\bm{U}}^{(1)})_{\perp}}\bm{U}^{\star(2)}\big(\bm{\Sigma}^{\star(2)}\big)^2\bm{U}^{\star(2)\top}\mathcal{P}_{(\widetilde{\bm{U}}^{(1)})_{\perp}}$ by
\begin{align}\label{eq:eigen_2}
	\widetilde{\bm{U}}^{(2)}\big(\widetilde{\bm{\Sigma}}^{(2)}\big)^2\widetilde{\bm{U}}^{(2)} = \mathcal{P}_{(\widetilde{\bm{U}}^{(1)})_{\perp}}\bm{U}^{\star(2)}\big(\bm{\Sigma}^{\star(2)}\big)^2\bm{U}^{\star(2)\top}\mathcal{P}_{(\widetilde{\bm{U}}^{(1)})_{\perp}},
\end{align}
where $\widetilde{\bm{U}}^{(2)} \in \mathcal{O}^{m_1, r - \overline{r}}$ and $\widetilde{\bm{\Sigma}}^{(1)} = {\sf diag}(\widetilde{\sigma}_{\overline{r} + 1}, \dots, \widetilde{\sigma}_r)$ with $\widetilde{\sigma}_{\overline{r} + 1} \geq \cdots \geq \widetilde{\sigma}_r \geq 0$. Recognizing that $\widetilde{\bm{U}}^{(1)\top}\widetilde{\bm{U}}^{(2)} = \bm{0}$, we see that the eigendecomposition of $\widetilde{\bm{M}}$ can be written as
\begin{align}\label{eq:eigen_tilde}
	\widetilde{\bm{M}} = \widetilde{\bm{U}}\widetilde{\bm{\Lambda}}\widetilde{\bm{U}}^\top,
\end{align}
where
\begin{align}\label{eq1}
	\widetilde{\bm{U}} = \big[\widetilde{\bm{U}}^{(1)}\ \widetilde{\bm{U}}^{(2)}\big],~\qquad~and~\qquad~\widetilde{\bm{\Lambda}} = {\sf diag}\big(\widetilde{\sigma}_1^2, \dots, \widetilde{\sigma}_{r}^2\big) = \begin{bmatrix}
		\big(\widetilde{\bm{\Sigma}}^{(1)}\big)^2  &\bm{0}\\
		\bm{0}  &\big(\widetilde{\bm{\Sigma}}^{(2)}\big)^2
	\end{bmatrix}.
\end{align}
In addition, one observes that
\begin{align}\label{eq2}
	\sigma_{r-\overline{r}}\big(\mathcal{P}_{(\widetilde{\bm{U}}^{(1)})_{\perp}}\bm{U}^{\star(2)}\big) &= \sigma_{r-\overline{r}}\left(\big(\widetilde{\bm{U}}^{(1)}\big)_{\perp}^\top\bm{U}^{\star(2)}\right) = \min_{\bm{a} \in \bbR^{r - \overline{r}}: \left\|\bm{a}\right\|_2 = 1}\left\|\big(\widetilde{\bm{U}}^{(1)}\big)_{\perp}^\top\bm{U}^{\star(2)}\bm{a}\right\|_2\notag\\ &= \sqrt{\min_{\bm{a} \in \bbR^{r - \overline{r}}: \left\|\bm{a}\right\|_2 = 1}\left(\big\|\bm{U}^{\star(2)}\bm{a}\big\|_2^2 - \big\|\widetilde{\bm{U}}^{(1)\top}\bm{U}^{\star(2)}\bm{a}\big\|_2^2\right)}\notag\\
	&= \sqrt{1 - \max_{\bm{a} \in \bbR^{r - \overline{r}}: \left\|\bm{a}\right\|_2 = 1}\big\|\widetilde{\bm{U}}^{(1)\top}\bm{U}^{\star(2)}\bm{a}\big\|_2^2}\notag\\
	&= \sqrt{1 - \big\|\widetilde{\bm{U}}^{(1)\top}\bm{U}^{\star(2)}\big\|^2}\notag\\
	&\geq \sqrt{1 - \frac{1}{Cr^2}}\notag\\
	&\geq \sqrt{1 - \left(\frac{1}{2}\right)^2} = \frac{\sqrt{3}}{2}.
\end{align} 
The last line holds due to \eqref{ineq6}. Noting that $\widetilde{\bm{U}}^{(2)}$ is also the column subspace of $\mathcal{P}_{(\widetilde{\bm{U}}^{(1)})_{\perp}}\bm{U}^{\star(2)} \in \bbR^{m_1, r - \overline{r}}$ and combining \eqref{ineq2d}, \eqref{ineq6} and \eqref{eq2}, one reaches
\begin{align}\label{ineq:incoherence_tilde_U_2}
	\big\|\widetilde{\bm{U}}^{(2)}\big\|_{2, \infty} &\leq \big\|\mathcal{P}_{(\widetilde{\bm{U}}^{(1)})_{\perp}}\bm{U}^{\star(2)}\big\|_{2, \infty}\sigma_{r-\overline{r}}^{-1}\big(\mathcal{P}_{(\widetilde{\bm{U}}^{(1)})_{\perp}}\bm{U}^{\star(2)}\big)\notag\\
	&\leq \frac{2}{\sqrt{3}}\left(\big\|\bm{U}^{\star(2)}\big\|_{2, \infty} + \big\|\mathcal{P}_{\widetilde{\bm{U}}^{(1)}}\bm{U}^{\star(2)}\big\|_{2, \infty}\right)\notag\\
	&\leq \frac{2}{\sqrt{3}}\left(\sqrt{\frac{\mu r}{m_1}} + \big\|\widetilde{\bm{U}}^{(1)}\big\|_{2,\infty} \big\|\widetilde{\bm{U}}^{(1)\top}\bm{U}^{\star(2)}\big\|\right)\notag\\
	&\leq \frac{2}{\sqrt{3}}\left(\sqrt{\frac{\mu r}{m_1}} + 2\sqrt{\frac{\mu r}{m_1}}\cdot \frac{1}{2}\right)\notag\\
	&\leq 2\sqrt{\frac{\mu r}{m_1}}.
\end{align}
\paragraph{Step 2: bounding $\widetilde{\sigma}_{r'}^2 - \widetilde{\sigma}_{r'+1}^2$ and $\|\bm{Z}\|$.} Recall that $\lambda_{i}(\widetilde{M}) = \widetilde{\sigma}_{i}^2$ for $i \in [r]$. To apply Lemma \ref{lm:space_estimate_expansion}, one needs to check the condition 
\begin{align}\label{ineq27}
	\widetilde{\sigma}_{r'}^2 - \widetilde{\sigma}_{r'+1}^2 > 2\|\bm{Z}\|,~\qquad~\forall r' \in \mathcal{A}.
\end{align}
It is seen from the definition of $\widetilde{\sigma}_{\overline{r}+1}$ that
\begin{align}\label{ineq25}
	\widetilde{\sigma}_{\overline{r}+1} \leq \big\|\bm{\Sigma}^{\star(2)}\big\| = \sigma_{\overline{r}+1}^\star.
\end{align}
Further, \eqref{eq:eigen_2} and \eqref{eq2} taken together imply that
\begin{align}\label{ineq127}
	\widetilde{\sigma}_{\overline{r}+1}^2 &\geq \sigma_{r-\overline{r}}^2\big(\mathcal{P}_{(\widetilde{\bm{U}}^{(1)})_{\perp}}\bm{U}^{\star(2)}\big)\sigma_{\overline{r}+1}^{\star2} \geq \left(1 - \frac{2\sqrt{C_5}\sqrt{m}_1\omega_{\sf max}\log m}{\sigma_{\overline{r}}^\star}\right)^2\sigma_{\overline{r}+1}^{\star2} \geq \left(1 - \frac{1}{Cr^2}\right)^2\sigma_{\overline{r}+1}^{\star2}
\end{align}
for some large constant $C > 0$.
By virtue of \eqref{ineq25} and \eqref{ineq127}, one has
\begin{align}\label{ineq128}
	\max\left\{\left(1 - \frac{1}{Cr^2}\right)\sigma_{\overline{r}+1}^\star, \sigma_{\overline{r}+1}^{\star} - 2\sqrt{C_5}\sqrt{m}_1\omega_{\sf max}\log m\right\} \leq \widetilde{\sigma}_{\overline{r}+1} \leq \sigma_{\overline{r}+1}^\star.
\end{align}
Putting \eqref{ineq2a}, \eqref{ineq25} and the fact $\sigma_{r'}^\star \geq \sigma_{\overline{r}}^\star \geq C_0r[(m_1m_2)^{1/4} + m_1^{1/2}]\log m$ together yields
\begin{align*}
	\widetilde{\sigma}_{r'} - \widetilde{\sigma}_{r'+1} &\geq \sigma_{r'}^\star - \sigma_{r'+1}^\star - 2\left\|\bm{E}\bm{V}^\star\right\| \\
	&\geq \sigma_{r'}^\star - \sigma_{r'+1}^\star - 2\sqrt{C_5}\sqrt{m_1}\omega_{\sf max}\log m\\
	&\geq \sigma_{r'}^\star - \sigma_{r'+1}^\star - \frac{\sigma_{r'}^\star}{Cr}\\
	&\geq \frac{1}{2}\left(\sigma_{r'}^\star - \sigma_{r'+1}^\star\right) + \frac{1}{2}\left(\sigma_{r'}^\star - \frac{4r-1}{4r}\sigma_{r'}^\star\right)  - \frac{\sigma_{r'}^\star}{Cr}\\
	&\geq \frac{1}{2}\left(\sigma_{r'}^\star - \sigma_{r'+1}^\star\right) \geq \frac{\sigma_{r'}^\star}{8r}.
\end{align*}
Here, the penultimate and the last lines hold due to the fact $r' \in \mathcal{A}$. In addition, we observe that
\begin{align*}
	\widetilde{\sigma}_{r'} + \widetilde{\sigma}_{r'+1} \geq \sigma_{r'}^\star + \sigma_{r'+1}^\star - 2\left\|\bm{E}\bm{V}^\star\right\| \geq \sigma_{r'}^\star + \sigma_{r'+1}^\star - \frac{\sigma_{r'}^\star}{Cr} \geq \frac{1}{2}\left(\sigma_{r'}^\star + \sigma_{r'+1}^\star\right).
\end{align*}
Combining the previous two inequalities leads to
\begin{align}\label{ineq26}
	\widetilde{\sigma}_{r'}^2 - \widetilde{\sigma}_{r'+1}^2 \geq \frac{1}{4}\left(\sigma_{r'}^{\star2} - \sigma_{r'+1}^{\star2}\right) \geq \frac{\sigma_{r'}^{\star2}}{16r}.
\end{align}
Now, we move on to control $\|\bm{Z}\|$. In view of \eqref{ineq2a} and \eqref{ineq6}, we have
\begin{align}\label{ineq:spectral_Z1}
	\left\|\bm{Z}_1\right\| \leq 2\big\|\bm{\Sigma}^{\star(2)}\big\|\left\|\bm{E}\bm{V}^\star\right\| + \left\|\bm{E}\bm{V}^\star\right\|^2 \leq 2\sqrt{C_5}\sqrt{m_1}\omega_{\sf max}\log m\cdot \sigma_{\overline{r}+1}^{\star} + C_5m_1\omega_{\sf max}^2\log^2 m
\end{align}
and
\begin{align}\label{ineq:spectral_Z2}
	\left\|\bm{Z}_2\right\| \leq 2\big\|\widetilde{\bm{U}}_1^\top\bm{U}_2^{\star(2)}\big\|\big\|\bm{\Sigma}^{\star(2)}\big\|^2 \lesssim \frac{\sqrt{m}_1\omega_{\sf max}\log m}{\sigma_{\overline{r}}^\star}\sigma_{\overline{r}+1}^{\star2} \leq \sqrt{m_1}\omega_{\sf max}\log m\cdot \sigma_{\overline{r}+1}^{\star}.
\end{align}
Combining \eqref{ineq:spectral_Z1}, \eqref{ineq:spectral_Z2} and \eqref{ineq2b}, we arrive at
\begin{align}\label{ineq:spectral_Z}
	\left\|\bm{Z}\right\| \lesssim \sqrt{m_1}\omega_{\sf max}\log m\cdot \sigma_{\overline{r}+1}^{\star} + \left(\sqrt{m_1m_2} + m_1\right)\omega_{\sf max}^2\log^2 m \ll \frac{\sigma_{r'}^{\star2}}{16r} \leq \widetilde{\sigma}_{r'}^2 - \widetilde{\sigma}_{r'+1}^2,
\end{align}
which validates \eqref{ineq27}. Here, the second inequality uses the facts $\sigma_{r'}^\star \geq \sigma_{\overline{r}}^\star \geq C_0r[(m_1m_2)^{1/4} + rm_1^{1/2}]\log m$. 
By virtue of Lemma \ref{lm:space_estimate_expansion} and \eqref{ineq26}, we have
\begin{align}\label{ineq32}
	&\left\|\bm{U}_{:,1: r'}^{\sf oracle}\bm{U}_{:,1: r'}^{\sf oracle\top} - \widetilde{\bm{U}}_{:,1: r'}\widetilde{\bm{U}}_{:,1: r'}^\top\right\|_{2,\infty}\notag\\ &\quad\leq \frac{8}{\pi}\sum_{k \geq 1}\frac{2^{k}}{\left(\widetilde{\sigma}_{r'}^2 - \widetilde{\sigma}_{r'+1}^2\right)^k}\sum_{0 \leq j_1, \dots, j_{k+1} \leq r\atop \left(j_1, ..., j_{k+1}\right)^{\top} \neq \bm{0}_{k+1}}\big\|\widetilde{\bm{P}}_{j_1}\bm{Z}\widetilde{\bm{P}}_{j_2}\bm{Z}\cdots\bm{Z}\widetilde{\bm{P}}_{j_{k+1}}\big\|_{2,\infty}\notag\\
	&\quad\leq \frac{8}{\pi}\sum_{k \geq 1}\left(\frac{8}{\sigma_{r'}^{\star2}-\sigma_{r'+1}^{\star2}}\right)^k\sum_{0 \leq j_1, \dots, j_{k+1} \leq r\atop \left(j_1, ..., j_{k+1}\right)^{\top} \neq \bm{0}_{k+1}}\big\|\widetilde{\bm{P}}_{j_1}\bm{Z}\widetilde{\bm{P}}_{j_2}\bm{Z}\cdots\bm{Z}\widetilde{\bm{P}}_{j_{k+1}}\big\|_{2,\infty},
\end{align}
Here, for any $1 \leq j \leq r$, $\widetilde{\bm{P}}_{j} = \widetilde{\bm{u}}_j\widetilde{\bm{u}}_j^\top$ and $\widetilde{\bm{P}}_{0} = \widetilde{\bm{U}}_{\perp}\widetilde{\bm{U}}_{\perp}^\top$. 

To bound $\|\bm{U}_{:,1: r'}^{\sf oracle}\bm{U}_{:,1: r'}^{\sf oracle\top} - \widetilde{\bm{U}}_{:,1: r'}\widetilde{\bm{U}}_{:,1: r'}^\top\|_{2,\infty}$, we will bound each single term $\|\widetilde{\bm{P}}_{j_1}\bm{Z}\widetilde{\bm{P}}_{j_2}\bm{Z}\cdots\bm{Z}\widetilde{\bm{P}}_{j_{k+1}}\|_{2,\infty}$ for $1 \leq k \leq \log n$, and show that the total contribution of the remaining terms is small.
\paragraph{Step 3: bounding $\|\widetilde{\bm{P}}_{j_1}\bm{Z}\widetilde{\bm{P}}_{j_2}\bm{Z}\cdots\bm{Z}\widetilde{\bm{P}}_{j_{k+1}}\|_{2,\infty}$ for small $k$.} For any $1 \leq k \leq \log n$ and $(j_1, \dots, j_{k+1}) \in \{0, 1, \dots, r\}^{k+1}\backslash \bm{0}$, let $\ell$ denote the the smallest $i$ such that $j_i \neq 0$.

\paragraph{Step 3.1: bounding $\|\widetilde{\bm{P}}_{j_1}\bm{Z}\widetilde{\bm{P}}_{j_2}\bm{Z}\cdots\bm{Z}\widetilde{\bm{P}}_{j_{k+1}}\|_{2,\infty}$ when $\ell = 1$.} If $\ell=1$, then \eqref{ineq2d} and \eqref{ineq:incoherence_tilde_U_2} taken collectively show that
\begin{align}\label{ineq101}
	\left\|\widetilde{\bm{u}}_{j_1}\right\|_{\infty} \leq \max\big\{\big\|\widetilde{\bm{U}}_1\big\|_{2,\infty}, \big\|\widetilde{\bm{U}}_2\big\|_{2,\infty}\big\} \leq 2\sqrt{\frac{\mu r}{m_1}}.
\end{align}
Inequality \eqref{ineq:spectral_Z} taken together with \eqref{ineq101} leads to
\begin{align}\label{ineq30}
	\big\|\widetilde{\bm{P}}_{j_1}\bm{Z}\widetilde{\bm{P}}_{j_2}\bm{Z}\cdots\bm{Z}\widetilde{\bm{P}}_{j_{k+1}}\big\|_{2,\infty} &= \big\|\widetilde{\bm{u}}_{j_1}\widetilde{\bm{u}}_{j_1}^\top\bm{Z}\widetilde{\bm{P}}_{j_2}\bm{Z}\cdots\bm{Z}\widetilde{\bm{P}}_{j_{k+1}}\big\|_{2,\infty}\notag\\
	&\leq \left\|\widetilde{\bm{u}}_{j_1}\right\|_{\infty}\big\|\widetilde{\bm{u}}_{j_1}^\top\bm{Z}\widetilde{\bm{P}}_{j_2}\bm{Z}\cdots\bm{Z}\widetilde{\bm{P}}_{j_{k+1}}\big\|\notag\\
	&\leq 2\sqrt{\frac{\mu r}{m_1}}\left\|\bm{Z}\right\|^{k}\notag\\
	&\leq 2\sqrt{\frac{\mu r}{m_1}}\left(C_2\left(\sqrt{m_1}\omega_{\sf max}\log m\cdot \sigma_{\overline{r}+1}^{\star} + \left(\sqrt{m_1m_2} + m_1\right)\omega_{\sf max}^2\log^2 m\right)\right)^k.
\end{align}

\paragraph{Step 3.2: bounding $\|\bm{Z}^i\widetilde{\bm{U}}\|_{2,\infty}$.} Turning to $\ell \geq 2$, we see from the triangle inequality that
\begin{align}\label{ineq29}
	&\big\|\widetilde{\bm{P}}_{j_1}\bm{Z}\widetilde{\bm{P}}_{j_2}\bm{Z}\cdots\bm{Z}\widetilde{\bm{P}}_{j_{k+1}}\big\|_{2,\infty} \leq \big\|\bm{Z}^{\ell-1}\widetilde{\bm{P}}_{j_\ell}\bm{Z}\cdots\bm{Z}\widetilde{\bm{P}}_{j_{k+1}}\big\|_{2,\infty} + \sum_{i=1}^{\ell-1}\left\|\bm{Z}^{i-1}\bm{P}_{\widetilde{\bm{U}}}\bm{Z}\widetilde{\bm{P}}_{j_{i+1}}\bm{Z}\cdots\bm{Z}\widetilde{\bm{P}}_{j_{k+1}}\right\|_{2,\infty}.
\end{align}
To bound the right-hand side of \eqref{ineq29}, it is helpful to bound $\|\bm{Z}^i\widetilde{\bm{U}}\|_{2,\infty}$ first. 
\paragraph{Step 3.2.1: bounding $\|\bm{Z}^i\widetilde{\bm{U}}^{(1)}\|_{2,\infty}$.} Recognizing that for any matrices $\bm{A}, \bm{B} \in \bbR^{m_1 \times m_1}$, we see that the following equation holds:
\begin{align*}
	\left(\bm{A} + \bm{B}\right)^{i} = \bm{B}^i + \sum_{j = 0}^{i-1}\bm{B}^j\bm{A}\left(\bm{A} + \bm{B}\right)^{i-j-1}.
\end{align*}
This allows one to derive 
\begin{align}\label{eq10}
	\bm{Z}^i\widetilde{\bm{U}}^{(1)} &= \left(\bm{Z}_1 + \bm{Z}_2 + \bm{Z}_3\right)^i\widetilde{\bm{U}}^{(1)}\notag\\
	&= \bm{Z}_3^i\widetilde{\bm{U}}^{(1)} + \sum_{j = 0}^{i-1}\bm{Z}_3^j\bm{Z}_1\bm{Z}^{i-j-1}\widetilde{\bm{U}}^{(1)} + \sum_{j = 0}^{i-1}\bm{Z}_3^j\bm{Z}_2\bm{Z}^{i-j-1}\widetilde{\bm{U}}^{(1)}.
\end{align}
By virtue of Lemma \ref{lm:noise_product} and \eqref{ineq:spectral_Z}, one has
\begin{align}\label{ineq102}
	\big\|\bm{Z}^i\widetilde{\bm{U}}^{(1)}\big\|_{2,\infty} &\leq \big\|\bm{Z}_3^i\widetilde{\bm{U}}^{(1)}\big\|_{2,\infty} + \sum_{j = 0}^{i-1}\big\|\bm{Z}_3^j\left(\bm{Z}_1 + \bm{Z}_2\right)\bm{Z}^{i-j-1}\widetilde{\bm{U}}^{(1)}\big\|_{2,\infty}\notag\\
	&\leq \big\|\bm{Z}_3^i\widetilde{\bm{U}}^{(1)}\big\|_{2,\infty} + \sum_{j = 0}^{i-1}\big(\big\|\bm{Z}_3^j\bm{Z}_1\big\|_{2,\infty} + \big\|\bm{Z}_3^j\bm{Z}_2\big\|_{2,\infty} \big)\big\|\bm{Z}^{i-j-1}\widetilde{\bm{U}}^{(1)}\big\|\notag\\
	&\leq \big\|\bm{Z}_3^i\widetilde{\bm{U}}^{(1)}\big\|_{2,\infty} + \sum_{j = 0}^{i-1}\big(\big\|\bm{Z}_3^j\bm{Z}_1\big\|_{2,\infty} + \big\|\bm{Z}_3^j\bm{Z}_2\big\|_{2,\infty} \big)\big\|\bm{Z}\big\|^{i-j-1}\notag\\
	&\leq 4C_3\sqrt{\frac{\mu r}{m_1}}\left(C_3\left(\sqrt{m_1m_2} + m_1\right)\omega_{\sf max}^2\log^2 m\right)^i\notag\\&\quad + \sum_{j=0}^{i-1}C_2\sqrt{\mu r}\left(\sigma_{\overline{r} + 1}^\star + \sqrt{m}_1\omega_{\sf max}\log m\right)\left(C_3\left(\sqrt{m_1m_2} + m_1\right)\omega_{\sf max}^2\log^2 m\right)^{j}\omega_{\sf max}\log m\cdot\notag\\
	&\hspace{1.2cm}\left(C_2\left(\sqrt{m_1}\omega_{\sf max}\log m\cdot \sigma_{\overline{r}+1}^{\star} + \left(\sqrt{m_1m_2} + m_1\right)\omega_{\sf max}^2\log^2 m\right)\right)^{i-j-1}\notag\\
	&\leq 4C_3\sqrt{\frac{\mu r}{m_1}}\left(C_2\left(\sqrt{m_1}\omega_{\sf max}\log m\cdot \sigma_{\overline{r}+1}^{\star} + \left(\sqrt{m_1m_2} + m_1\right)\omega_{\sf max}^2\log^2 m\right)\right)^{i}\notag\\
	&\quad + \sqrt{\frac{\mu r}{m_1}}\left(C_2\left(\sqrt{m_1}\omega_{\sf max}\log m\cdot \sigma_{\overline{r}+1}^{\star} + \left(\sqrt{m_1m_2} + m_1\right)\omega_{\sf max}^2\log^2 m\right)\right)^{i}\sum_{j=0}^{i-1}\frac{1}{2^j}\notag\\
	&\leq C_2\sqrt{\frac{\mu r}{m_1}}\left(C_2\left(\sqrt{m_1}\omega_{\sf max}\log m\cdot \sigma_{\overline{r}+1}^{\star} + \left(\sqrt{m_1m_2} + m_1\right)\omega_{\sf max}^2\log^2 m\right)\right)^{i},
\end{align}
provided that $C_2 \geq 4C_3 + 1$.
\paragraph{Step 3.2.2: bounding $\|\bm{Z}^i\widetilde{\bm{U}}^{(2)}\|_{2,\infty}$.}
Note that $\widetilde{\bm{U}}^{(2)}$ is also the column subspace of $\mathcal{P}_{(\widetilde{\bm{U}}^{(1)})_{\perp}}\bm{U}^{\star(2)} \in \bbR^{m_1, r - \overline{r}}$. In view of Lemma \ref{lm:noise_product}, \eqref{ineq6}, \eqref{eq2} and \eqref{ineq102}, we arrive at
\begin{align}\label{ineq103}
	\big\|\bm{Z}^i\widetilde{\bm{U}}^{(2)}\big\|_{2,\infty} &\leq \big\|\bm{Z}^i\mathcal{P}_{(\widetilde{\bm{U}}^{(1)})_{\perp}}\bm{U}^{\star(2)}\big\|_{2,\infty} \sigma_{r-\overline{r}}^{-1}\big(\mathcal{P}_{(\widetilde{\bm{U}}^{(1)})_{\perp}}\bm{U}^{\star(2)}\big)\notag\\
	&\leq \frac{2}{\sqrt{3}}\left(\big\|\bm{Z}^i\bm{U}^{\star(2)}\big\|_{2,\infty} + \big\|\bm{Z}^i\mathcal{P}_{\widetilde{\bm{U}}^{(1)}}\bm{U}^{\star(2)}\big\|_{2,\infty}\right)\notag\\
	&\leq \frac{2}{\sqrt{3}}\left(\big\|\bm{Z}^i\bm{U}^{\star(2)}\big\|_{2,\infty} + \big\|\bm{Z}^i\widetilde{\bm{U}}^{(1)}\big\|_{2,\infty}\big\|\widetilde{\bm{U}}^{(1)\top}\bm{U}^{\star(2)}\big\|\right)\notag\\
	&\leq \frac{2}{\sqrt{3}}\bigg(3C_3\sqrt{\frac{\mu r}{m_1}}\left(C_3\left(\sqrt{m_1m_2} + m_1\right)\omega_{\sf max}^2\log^2 m\right)^{i}\notag\\&\hspace{1.2cm} + C_2\sqrt{\frac{\mu r}{m_1}}\left(C_2\left(\sqrt{m_1}\omega_{\sf max}\log m\cdot \sigma_{\overline{r}+1}^{\star} + \left(\sqrt{m_1m_2} + m_1\right)\omega_{\sf max}^2\log^2 m\right)\right)^{i}\cdot\frac{1}{2}\bigg)\notag\\
	&\leq C_2\sqrt{\frac{\mu r}{m_1}}\left(C_2\left(\sqrt{m_1}\omega_{\sf max}\log m\cdot \sigma_{\overline{r}+1}^{\star} + \left(\sqrt{m_1m_2} + m_1\right)\omega_{\sf max}^2\log^2 m\right)\right)^{i}.
\end{align}

Putting \eqref{ineq102} and \eqref{ineq103} together and recognizing that $\widetilde{\bm{U}} = [\widetilde{\bm{U}}^{(1)}\ \widetilde{\bm{U}}^{(2)}]$,
we conclude that
\begin{align}\label{ineq28}
	\big\|\bm{Z}^i\widetilde{\bm{U}}\big\|_{2,\infty} &\leq \big\|\bm{Z}^i\widetilde{\bm{U}}^{(1)}\big\|_{2,\infty} + \big\|\bm{Z}^i\widetilde{\bm{U}}^{(2)}\big\|_{2,\infty}\notag\\ &\leq 2C_2\sqrt{\frac{\mu r}{m_1}}\left(C_2\left(\sqrt{m_1}\omega_{\sf max}\log m\cdot \sigma_{\overline{r}+1}^{\star} + \left(\sqrt{m_1m_2} + m_1\right)\omega_{\sf max}^2\log^2 m\right)\right)^{i}.
\end{align}

\paragraph{Step 4: bounding $\|\widetilde{\bm{P}}_{j_1}\bm{Z}\widetilde{\bm{P}}_{j_2}\bm{Z}\cdots\bm{Z}\widetilde{\bm{P}}_{j_{k+1}}\|_{2,\infty}$ when $\ell > 1$.} Plugging \eqref{ineq:spectral_Z} and \eqref{ineq28} into \eqref{ineq29} yields that, for $\ell \geq 2$, 
\begin{align*}
	&\big\|\widetilde{\bm{P}}_{j_1}\bm{Z}\widetilde{\bm{P}}_{j_2}\bm{Z}\cdots\bm{Z}\widetilde{\bm{P}}_{j_{k+1}}\big\|_{2,\infty}\\
	&\quad \leq \left\|\bm{Z}^{\ell-1}\widetilde{\bm{u}}_{j_\ell}\right\|_{2,\infty}\big\|\widetilde{\bm{u}}_{j_\ell}^\top\bm{Z}\cdots\bm{Z}\widetilde{\bm{P}}_{j_{k+1}}\big\|_{2,\infty} + \sum_{i=1}^{\ell-1}\big\|\bm{Z}^{i-1}\widetilde{\bm{U}}\big\|_{2,\infty}\big\|\widetilde{\bm{U}}^\top\bm{Z}\widetilde{\bm{P}}_{j_{i+1}}\bm{Z}\cdots\bm{Z}\widetilde{\bm{P}}_{j_{k+1}}\big\|_{2,\infty}\\
	&\quad \leq \big\|\bm{Z}^{\ell-1}\widetilde{\bm{U}}\big\|_{2,\infty}\left\|\bm{Z}\right\|^{k-\ell+1} + \sum_{i=1}^{\ell-1}\big\|\bm{Z}^{i-1}\widetilde{\bm{U}}\big\|_{2,\infty}\left\|\bm{Z}\right\|^{k-i+1}\\
	&\quad \leq 2C_2\sqrt{\frac{\mu r}{m_1}}\left(C_2\left(\sqrt{m_1}\omega_{\sf max}\log m\cdot \sigma_{\overline{r}+1}^{\star} + \left(\sqrt{m_1m_2} + m_1\right)\omega_{\sf max}^2\log^2 m\right)\right)^{\ell-1}\cdot\\
	&\qquad \left(C_2\left(\sqrt{m_1}\omega_{\sf max}\log m\cdot \sigma_{\overline{r}+1}^{\star} + \left(\sqrt{m_1m_2} + m_1\right)\omega_{\sf max}^2\log^2 m\right)\right)^{k-\ell+1}\\
	&\qquad + \sum_{i=1}^{\ell-1}2C_2\sqrt{\frac{\mu r}{m_1}}\left(C_2\left(\sqrt{m_1}\omega_{\sf max}\log m\cdot \sigma_{\overline{r}+1}^{\star} + \left(\sqrt{m_1m_2} + m_1\right)\omega_{\sf max}^2\log^2 m\right)\right)^{i-1}\cdot\\
	&\hspace{1.6cm}\left(C_2\left(\sqrt{m_1}\omega_{\sf max}\log m\cdot \sigma_{\overline{r}+1}^{\star} + \left(\sqrt{m_1m_2} + m_1\right)\omega_{\sf max}^2\log^2 m\right)\right)^{k-i+1}\\
	&\quad= 2C_2\sqrt{\frac{\mu r}{m_1}}\left(C_2\left(\sqrt{m_1}\omega_{\sf max}\log m\cdot \sigma_{\overline{r}+1}^{\star} + \left(\sqrt{m_1m_2} + m_1\right)\omega_{\sf max}^2\log^2 m\right)\right)^{k}\cdot\ell\\
	&\quad \leq 2C_2\sqrt{\frac{\mu r}{m_1}}\left(2C_2\left(\sqrt{m_1}\omega_{\sf max}\log m\cdot \sigma_{\overline{r}+1}^{\star} + \left(\sqrt{m_1m_2} + m_1\right)\omega_{\sf max}^2\log^2 m\right)\right)^{k}.
\end{align*}
The last inequality comes from $\ell \leq k+1 \leq 2^k$.
Combining the previous inequality and \eqref{ineq30} reveals that: for any $1 \leq k \leq \log n$ and $(j_1, \dots, j_{k+1}) \in \{0, 1, \dots, r\}^{k+1}\backslash \bm{0}$, it holds that
\begin{align}\label{ineq31}
	&\big\|\widetilde{\bm{P}}_{j_1}\bm{Z}\widetilde{\bm{P}}_{j_2}\bm{Z}\cdots\bm{Z}\widetilde{\bm{P}}_{j_{k+1}}\big\|_{2,\infty}\notag\\ &\hspace{1cm}\leq 2C_2\sqrt{\frac{\mu r}{m_1}}\left(2C_2\left(\sqrt{m_1}\omega_{\sf max}\log m\cdot \sigma_{\overline{r}+1}^{\star} + \left(\sqrt{m_1m_2} + m_1\right)\omega_{\sf max}^2\log^2 m\right)\right)^{k}.
\end{align}
\paragraph{Step 5: bounding $\|\bm{U}_{:,1: r'}^{\sf oracle}\bm{U}_{:,1: r'}^{\sf oracle\top} - \widetilde{\bm{U}}_{:,1: r'}\widetilde{\bm{U}}_{:,1: r'}^\top\|_{2,\infty}$.} Now, we are ready to bound $\|\bm{U}_{:,1: r'}^{\sf oracle}\bm{U}_{:,1: r'}^{\sf oracle\top} - \widetilde{\bm{U}}_{:,1: r'}\widetilde{\bm{U}}_{:,1: r'}^\top\|_{2,\infty}$. As a consequence of \eqref{ineq31}, for any $1 \leq k \leq \log n$, one has
\begin{align}\label{ineq104}
	&\left(\frac{8}{\sigma_{r'}^{\star2}-\sigma_{r'+1}^{\star2}}\right)^k\sum_{0 \leq j_1, \dots, j_{k+1} \leq r\atop \left(j_1, ..., j_{k+1}\right)^{\top} \neq \bm{0}_{k+1}}\big\|\widetilde{\bm{P}}_{j_1}\bm{Z}\widetilde{\bm{P}}_{j_2}\bm{Z}\cdots\bm{Z}\widetilde{\bm{P}}_{j_{k+1}}\big\|_{2,\infty}\notag\\ &\quad\leq \left(\frac{8}{\sigma_{r'}^{\star2}-\sigma_{r'+1}^{\star2}}\right)^k\cdot (r+1)^{k+1}\cdot  2C_2\sqrt{\frac{\mu r}{m_1}}\left(2C_2\left(\sqrt{m_1}\omega_{\sf max}\log m\cdot \sigma_{\overline{r}+1}^{\star} + \left(\sqrt{m_1m_2} + m_1\right)\omega_{\sf max}^2\log^2 m\right)\right)^{k}\notag\\
	&\quad\leq 4C_2\sqrt{\frac{\mu r^3}{m_1}}\left(\frac{32C_2r\left(\sqrt{m_1}\omega_{\sf max}\log m\cdot \sigma_{\overline{r}+1}^{\star} + \left(\sqrt{m_1m_2} + m_1\right)\omega_{\sf max}^2\log^2 m\right)}{\sigma_{r'}^{\star2}-\sigma_{r'+1}^{\star2}}\right)^k.
\end{align}
Recalling that $\sigma_{r'}^\star \geq \sigma_{\overline{r}}^\star \geq C_0r[(m_1m_2)^{1/4} + rm_1^{1/2}]\log m$, we know from \eqref{ineq26} that there exists some large constant $C > 0$ such that
\begin{align}\label{ineq105}
	&\left(\frac{32C_2r\left(\sqrt{m_1}\omega_{\sf max}\log m\cdot \sigma_{\overline{r}+1}^{\star} + \left(\sqrt{m_1m_2} + m_1\right)\omega_{\sf max}^2\log^2 m\right)}{\sigma_{r'}^{\star2}-\sigma_{r'+1}^{\star2}}\right)^{k-1}\notag\\
	&\hspace{1cm}\leq \left(\frac{32C_2r\left(\sqrt{m_1}\omega_{\sf max}\log m\cdot \sigma_{\overline{r}+1}^{\star} + \left(\sqrt{m_1m_2} + m_1\right)\omega_{\sf max}^2\log^2 m\right)}{\sigma_{r'}^{\star2}/(4r)}\right)^{k-1}\notag\\
	&\hspace{1cm}\leq \left(\frac{1}{C^2}\right)^{k-1} \leq \frac{1}{C^k}.
\end{align}
For any $k \ge \lfloor\log m\rfloor + 1$, in view of \eqref{ineq:spectral_Z} and the the previous inequality, we have
\begin{align}\label{ineq106}
	&\left(\frac{8}{\sigma_{r'}^{\star2}-\sigma_{r'+1}^{\star2}}\right)^k\sum_{0 \leq j_1, \dots, j_{k+1} \leq r\atop \left(j_1, ..., j_{k+1}\right)^{\top} \neq \bm{0}_{k+1}}\big\|\widetilde{\bm{P}}_{j_1}\bm{Z}\widetilde{\bm{P}}_{j_2}\bm{Z}\cdots\bm{Z}\widetilde{\bm{P}}_{j_{k+1}}\big\|_{2,\infty}\notag\\
	&\quad \leq \left(\frac{8}{\sigma_{r'}^{\star2}-\sigma_{r'+1}^{\star2}}\right)^k\cdot (r+1)^{k+1}\left\|\bm{Z}\right\|^k\notag\\
	&\quad\leq 2r\cdot \left(\frac{16rC_2\left(\sqrt{m_1}\omega_{\sf max}\log m\cdot \sigma_{\overline{r}+1}^{\star} + \left(\sqrt{m_1m_2} + m_1\right)\omega_{\sf max}^2\log^2 m\right)}{\sigma_{r'}^{\star2}-\sigma_{r'+1}^{\star2}}\right)^{k}\notag\\
	&\quad \leq \frac{2r}{C^k}\cdot\frac{16rC_2\left(\sqrt{m_1}\omega_{\sf max}\log m\cdot \sigma_{\overline{r}+1}^{\star} + \left(\sqrt{m_1m_2} + m_1\right)\omega_{\sf max}^2\log^2 m\right)}{\sigma_{r'}^{\star2}-\sigma_{r'+1}^{\star2}}.
\end{align}
Combining \eqref{ineq32}, \eqref{ineq104}, \eqref{ineq105} and \eqref{ineq106}, one can obtain
\begin{align}\label{ineq112}
	&\left\|\bm{U}_{:,1: r'}^{\sf oracle}\bm{U}_{:,1: r'}^{\sf oracle\top} - \widetilde{\bm{U}}_{:,1: r'}\widetilde{\bm{U}}_{:,1: r'}^\top\right\|_{2,\infty}\notag\\
	&\hspace{1cm}\leq \sum_{1 \leq k \leq \log m}4C_2\sqrt{\frac{\mu r^3}{m_1}}\left(\frac{32C_2r\left(\sqrt{m_1}\omega_{\sf max}\log m\cdot \sigma_{\overline{r}+1}^{\star} + \left(\sqrt{m_1m_2} + m_1\right)\omega_{\sf max}^2\log^2 m\right)}{\sigma_{r'}^{\star2}-\sigma_{r'+1}^{\star2}}\right)^k\notag\\
	&\hspace{1.3cm} + \sum_{k \geq \lfloor\log m\rfloor + 1}\frac{2r}{C^k}\cdot\frac{16C_2r\left(\sqrt{m_1}\omega_{\sf max}\log m\cdot \sigma_{\overline{r}+1}^{\star} + \left(\sqrt{m_1m_2} + m_1\right)\omega_{\sf max}^2\log^2 m\right)}{\sigma_{r'}^{\star2}-\sigma_{r'+1}^{\star2}}\notag\\
	&\hspace{1cm}\lesssim \sqrt{\frac{\mu r^3}{m_1}}\frac{r\left(\sqrt{m_1}\omega_{\sf max}\log m\cdot \sigma_{\overline{r}+1}^{\star} + \left(\sqrt{m_1m_2} + m_1\right)\omega_{\sf max}^2\log^2 m\right)}{\sigma_{r'}^{\star2}-\sigma_{r'+1}^{\star2}}\notag\\
	&\hspace{1cm}\lesssim \sqrt{\frac{\mu r^3}{m_1}}\left(\frac{r^2\sqrt{m_1}\omega_{\sf max}\log m}{\sigma_{r'}^\star} + \frac{r^2 \left(\sqrt{m_1m_2} + m_1\right)\omega_{\sf max}^2\log^2 m}{\sigma_{r'}^{\star2}}\right)\notag\\
	&\hspace{1cm}\asymp \sqrt{\frac{\mu r^3}{m_1}}\left(\frac{r^2\sqrt{m_1}\omega_{\sf max}\log m}{\sigma_{r'}^\star} + \frac{r^2 \sqrt{m_1m_2}\omega_{\sf max}^2\log^2 m}{\sigma_{r'}^{\star2}}\right).
\end{align}
Here, the second last line holds due to \eqref{ineq26} and the the last line makes use of the inequality
\begin{align*}
	\frac{r^2 m_1\omega_{\sf max}^2\log^2 m}{\sigma_{r'}^{\star2}} = r^2\left(\frac{\sqrt{m_1}\omega_{\sf max}\log m}{\sigma_{r'}^\star}\right)^2 \lesssim r^2\frac{\sqrt{m_1}\omega_{\sf max}\log m}{\sigma_{r'}^\star},
\end{align*}
provided that $\sigma_{r'}^\star \gtrsim \sqrt{m_1}\omega_{\sf max}\log m$.
\paragraph{Step 6: bounding $\|\widetilde{\bm{U}}_{:,1: r'}\widetilde{\bm{U}}_{:,1: r'}^\top - \bm{U}_{:,1: r'}^\star\bm{U}_{:,1: r'}^{\star\top}\|_{2,\infty}$.} To control $\|\bm{U}_{:,1: r'}^{\sf oracle}\bm{U}_{:,1: r'}^{\sf oracle\top} - \bm{U}_{:,1: r'}^\star\bm{U}_{:,1: r'}^{\star\top}\|_{2,\infty}$, one still needs to bound $\|\widetilde{\bm{U}}_{:,1: r'}\widetilde{\bm{U}}_{:,1: r'}^\top - \bm{U}_{:,1: r'}^\star\bm{U}_{:,1: r'}^{\star\top}\|_{2,\infty}$. Recall that $\widetilde{\bm{U}}^{(1)}\widetilde{\bm{\Sigma}}^{(1)}\widetilde{\bm{W}}^{(1)\top}$ is the SVD of $\bm{U}^{\star(1)}\bm{\Sigma}^{\star(1)} + \bm{E}\bm{V}^{\star(1)} = \bm{U}_{:,1: \overline{r}}^\star\bm{\Sigma}^{\star}_{1:\overline{r}, 1:\overline{r}} + \bm{E}\bm{V}_{:,1: \overline{r}}$ and $\widetilde{\bm{U}}_{:,1: r'}$ (resp.~$\bm{U}_{:,1: r'}^\star$) is the matrix containing the first $r'$ columns of $\widetilde{\bm{U}}^{(1)}$ (resp.~$\bm{U}^{\star(1)}$). We make the observation that
\begin{align}\label{ineq110}
	&\mathcal{P}_{\left(\widetilde{\bm{U}}_{:,1: r'}\right)_{\perp}}\bm{U}_{:,1: r'}^\star\notag\\ &\quad= \mathcal{P}_{\left(\widetilde{\bm{U}}_{:,1: r'}\right)_{\perp}}\left(\bm{U}_{:,1: r'}^\star\bm{\Sigma}^{\star}_{1:r', 1:\overline{r}}\right)\left(\bm{I}_{r'}\ \bm{0}_{r' \times (\overline{r} - r')}\right)^\top\left(\bm{\Sigma}^{\star}_{1:r', 1:r'}\right)^{-1}\notag\\
	&\quad= \mathcal{P}_{\left(\widetilde{\bm{U}}_{:,1: r'}\right)_{\perp}}\left(\widetilde{\bm{U}}^{(1)}\widetilde{\bm{\Sigma}}^{(1)}\widetilde{\bm{W}}^{(1)\top} - \bm{E}\bm{V}^{\star(1)} - \bm{U}_{:, r'+1: \overline{r}}^\star\bm{\Sigma}^{\star}_{r'+1: \overline{r}, 1:\overline{r}}\right)\left(\bm{I}_{r'}\ \bm{0}_{r' \times (\overline{r} - r')}\right)^\top\left(\bm{\Sigma}^{\star}_{1:r', 1:r'}\right)^{-1}\notag\\
	&\quad= \mathcal{P}_{\left(\widetilde{\bm{U}}_{:,1: r'}\right)_{\perp}}\left(\widetilde{\bm{U}}_{:,1: r'}^{(1)}\widetilde{\bm{\Sigma}}_{1:r', 1:r'}^{(1)}\widetilde{\bm{W}}_{:,1: r'}^{(1)\top} + \widetilde{\bm{U}}_{:,r'+1: \overline{r}}^{(1)}\widetilde{\bm{\Sigma}}_{r'+1: \overline{r}, r'+1: \overline{r}}^{(1)}\widetilde{\bm{W}}_{:,r'+1: \overline{r}}^{(1)\top} - \bm{E}\bm{V}^{\star(1)}\right)\left(\bm{I}_{r'}\ \bm{0}_{r' \times (\overline{r} - r')}\right)^\top\notag\\&\qquad \cdot\left(\bm{\Sigma}^{\star}_{1:r', 1:r'}\right)^{-1}\notag\\
	&\quad= \widetilde{\bm{U}}_{:,r'+1: \overline{r}}^{(1)}\widetilde{\bm{\Sigma}}_{r'+1: \overline{r}, r'+1: \overline{r}}^{(1)}\widetilde{\bm{W}}_{:,r'+1: \overline{r}}^{(1)\top}\left(\bm{I}_{r'}\ \bm{0}_{r' \times (\overline{r} - r')}\right)^\top\left(\bm{\Sigma}^{\star}_{1:r', 1:r'}\right)^{-1}\notag\\&\qquad - \mathcal{P}_{\left(\widetilde{\bm{U}}_{:,1: r'}\right)_{\perp}}\bm{E}\bm{V}^{\star(1)}\left(\bm{I}_{r'}\ \bm{0}_{r' \times (\overline{r} - r')}\right)^\top\left(\bm{\Sigma}^{\star}_{1:r', 1:r'}\right)^{-1},
\end{align}
where the second identity is valid since $$\bm{\Sigma}^{\star}_{r'+1: \overline{r}, 1:\overline{r}}\left(\bm{I}_{r'}\ \bm{0}_{r' \times (\overline{r} - r')}\right)^\top = \left(\bm{0}_{(\overline{r} - r') \times r'}\ \bm{\Sigma}^{\star}_{r'+1: \overline{r}, r'+1:\overline{r}}\right)\left(\bm{I}_{r'}\ \bm{0}_{r' \times (\overline{r} - r')}\right)^\top = \bm{0}_{(\overline{r} - r') \times r'}$$ and the last line comes from $\mathcal{P}_{(\widetilde{\bm{U}}_{:,1: r'})_{\perp}}\widetilde{\bm{U}}_{:,1: r'}^{(1)} = \bm{0}$ and $\mathcal{P}_{(\widetilde{\bm{U}}_{:,1: r'})_{\perp}}\widetilde{\bm{U}}_{:,r'+1: \overline{r}}^{(1)} = \widetilde{\bm{U}}_{:,r'+1: \overline{r}}^{(1)}$. Note that $\widetilde{\bm{W}}_{:,1: r'}^{(1)}$ (resp.~$\left(\bm{I}_{r'}\ \bm{0}\right)^\top$) is the leading $r'$ right singular space of $\bm{U}^{\star(1)}\bm{\Sigma}^{\star(1)} + \bm{E}\bm{V}^{\star(1)}$ (resp.~$\bm{U}^{\star(1)}\bm{\Sigma}^{\star(1)}$) and
\begin{align*}
	\sigma_{r'}^\star - \sigma_{r'+1}^\star \geq \frac{1}{4r}\sigma_{r'}^\star \geq \frac{C_0}{4}\big[(m_1m_2)^{1/4} + rm_1^{1/2}\big]\omega_{\sf max}\log m \gg r\sqrt{m}_1\omega_{\sf max}\log m.
\end{align*}
\citet[Lemma 2.6, Eqn.~(2.26a)]{chen2021spectral} and Lemma \ref{lm:event_collection} taken together imply that
\begin{subequations}
	\begin{align}
		\left\|\widetilde{\bm{W}}_{:,r'+1: \overline{r}}^{(1)\top}\left(\bm{I}_{r'}\ \bm{0}_{r' \times (\overline{r} - r')}\right)^\top\right\| &= \left\|\big(\widetilde{\bm{W}}_{:,1:r'}\big)_{\perp}^{(1)\top}\left(\bm{I}_{r'}\ \bm{0}_{r' \times (\overline{r} - r')}\right)^\top\right\| \lesssim \frac{\left\|\bm{E}\bm{V}^{\star(1)}\right\|}{\sigma_{r'}^\star - \sigma_{r'+1}^\star} \lesssim \frac{r\sqrt{m}_1\omega_{\sf max}\log m}{\sigma_{r'}^\star},\label{ineq109a}\\
		\left\|\big(\bm{U}_{:,1: r'}^\star\big)_{\perp}^\top\widetilde{\bm{U}}_{:,1: r'}\right\| &\lesssim \frac{\left\|\bm{E}\bm{V}^{\star(1)}\right\|}{\sigma_{r'}^\star - \sigma_{r'+1}^\star} \lesssim \frac{r\sqrt{m}_1\omega_{\sf max}\log m}{\sigma_{r'}^\star}.\label{ineq109b}
	\end{align}
\end{subequations}
Moreover, combining \eqref{ineq2a} and the assumption $\sigma_{\overline{r}}^\star \geq C_0r[(m_1m_2)^{1/4} + rm_1^{1/2}]\omega_{\sf max}\log m$ gives
\begin{align}\label{ineq4}
	\widetilde{\sigma}_{i} \leq \sigma_{i}^\star + \left\|\bm{E}\bm{V}^\star\right\| \leq 2\sigma_{i}^\star,~\qquad~\forall i \in \left[\overline{r}\right].
\end{align}
Inequality \eqref{ineq2d} combined with \eqref{ineq109a} and \eqref{ineq4} gives
\begin{align}\label{ineq107}
	&\left\|\widetilde{\bm{U}}_{:,r'+1: \overline{r}}^{(1)}\widetilde{\bm{\Sigma}}_{r'+1: \overline{r}, r'+1: \overline{r}}^{(1)}\widetilde{\bm{W}}_{:,r'+1: \overline{r}}^{(1)\top}\left(\bm{I}_{r'}\ \bm{0}_{r' \times (\overline{r} - r')}\right)^\top\left(\bm{\Sigma}^{\star}_{1:r', 1:r'}\right)^{-1}\right\|_{2,\infty}\notag\\
	&\quad \leq \big\|\widetilde{\bm{U}}^{(1)}\big\|_{2,\infty}\big\|\widetilde{\bm{\Sigma}}_{r'+1: \overline{r}, r'+1: \overline{r}}^{(1)}\big\|\left\|\widetilde{\bm{W}}_{:,r'+1: \overline{r}}^{(1)\top}\left(\bm{I}_{r'}\ \bm{0}_{r' \times (\overline{r} - r')}\right)^\top\right\|\left\|\left(\bm{\Sigma}^{\star}_{1:r', 1:r'}\right)^{-1}\right\|\notag\\
	&\quad \lesssim 2\sqrt{\frac{\mu r}{m_1}}\cdot \widetilde{\sigma}_{r'+1}\cdot \frac{r\sqrt{m}_1\omega_{\sf max}\log m}{\sigma_{r'}^\star}\cdot \frac{1}{\sigma_{r'}^\star}\notag\\
	&\quad \lesssim 2\sqrt{\frac{\mu r}{m_1}}\cdot 2\sigma_{r'+1}^\star\cdot \frac{r\sqrt{m}_1\omega_{\sf max}\log m}{\sigma_{r'}^\star}\cdot \frac{1}{\sigma_{r'}^\star}\notag\\
	&\quad \lesssim \sqrt{\frac{\mu r}{m_1}}\frac{r\sqrt{m}_1\omega_{\sf max}\log m}{\sigma_{r'}^\star}.
\end{align}
In addition, Lemma \ref{lm:power_V} and \eqref{ineq2d} taken together imply that
\begin{align}\label{ineq108}
	&\left\|\mathcal{P}_{\left(\widetilde{\bm{U}}_{:,1: r'}\right)_{\perp}}\bm{E}\bm{V}^{\star(1)}\left(\bm{I}_{r'}\ \bm{0}_{r' \times (\overline{r} - r')}\right)^\top\left(\bm{\Sigma}^{\star}_{1:r', 1:r'}\right)^{-1}\right\|_{2,\infty}\notag\\
	&\quad \leq \frac{1}{\sigma_{r'}^\star}\left\|\mathcal{P}_{\left(\widetilde{\bm{U}}_{:,1: r'}\right)_{\perp}}\bm{E}\bm{V}^{\star(1)}\right\|_{2,\infty}\notag\\
	&\quad \leq \frac{1}{\sigma_{r'}^\star}\left(\big\|\bm{E}\bm{V}^{\star(1)}\big\|_{2,\infty} + \left\|\mathcal{P}_{\widetilde{\bm{U}}_{:,1:r'}}\bm{E}\bm{V}^{\star(1)}\right\|_{2,\infty}\right)\notag\\
	&\quad \leq \frac{1}{\sigma_{r'}^\star}\left(\big\|\bm{E}\bm{V}^{\star(1)}\big\|_{2,\infty} + \big\|\widetilde{\bm{U}}_{:,1:r'}\big\|_{2,\infty}\big\|\widetilde{\bm{U}}_{:,1:r'}\big\|\big\|\bm{E}\bm{V}^{\star(1)}\big\|\right)\notag\\
	&\quad \lesssim \frac{1}{\sigma_{r'}^\star}\left(\sqrt{\mu r}\omega_{\sf max}\log n + 2\sqrt{\frac{\mu r}{m_1}}\cdot\sqrt{m_1}\omega_{\sf max}\log n\right)\notag\\
	&\quad \asymp \sqrt{\frac{\mu r}{m_1}}\frac{\sqrt{m}_1\omega_{\sf max}\log m}{\sigma_{r'}^\star}.
\end{align}
Taking \eqref{ineq110}, \eqref{ineq107} and \eqref{ineq108} together gives
\begin{align*}
	\left\|\mathcal{P}_{\left(\widetilde{\bm{U}}_{:,1: r'}\right)_{\perp}}\bm{U}_{:,1: r'}^\star\right\|_{2,\infty} \lesssim \sqrt{\frac{\mu r}{m_1}}\frac{r\sqrt{m}_1\omega_{\sf max}\log m}{\sigma_{r'}^\star}.
\end{align*}
The previous inequality taken together with \eqref{ineq109b} and \eqref{ineq2d} reveals that
\begin{align}\label{ineq111}
	&\big\|\widetilde{\bm{U}}_{:,1: r'}\widetilde{\bm{U}}_{:,1: r'}^\top - \bm{U}_{:,1: r'}^\star\bm{U}_{:,1: r'}^{\star\top}\big\|_{2,\infty}\notag\\ &\hspace{1cm}\leq \big\|\big(\bm{U}_{:,1: r'}^\star - \widetilde{\bm{U}}_{:,1: r'}\widetilde{\bm{U}}_{:,1: r'}^\top\bm{U}_{:,1: r'}^\star\big)\bm{U}_{:,1: r'}^{\star\top}\big\|_{2,\infty} + \big\|\widetilde{\bm{U}}_{:,1: r'}\widetilde{\bm{U}}_{:,1: r'}^\top\bm{U}_{:,1: r'}^\star\bm{U}_{:,1: r'}^{\star\top} - \widetilde{\bm{U}}_{:,1: r'}\widetilde{\bm{U}}_{:,1: r'}^\top\big\|_{2,\infty}\notag\\
	&\hspace{1cm}\leq \left\|\left(\mathcal{P}_{\left(\widetilde{\bm{U}}_{:,1: r'}\right)_{\perp}}\bm{U}_{:,1: r'}^\star\right)\bm{U}_{:,1: r'}^{\star\top}\right\|_{2,\infty} + \big\|\widetilde{\bm{U}}_{:,1: r'}\widetilde{\bm{U}}_{:,1: r'}^\top\left(\bm{U}_{:,1: r'}^\star\right)_{\perp}\left(\bm{U}_{:,1: r'}^\star\right)_{\perp}^\top\big\|_{2,\infty}\notag\\
	&\hspace{1cm}\leq \left\|\mathcal{P}_{\left(\widetilde{\bm{U}}_{:,1: r'}\right)_{\perp}}\bm{U}_{:,1: r'}^\star\right\|_{2,\infty} + \big\|\widetilde{\bm{U}}_{:,1: r'}\big\|_{2,\infty}\left\|\big(\bm{U}_{:,1: r'}^\star\big)_{\perp}^\top\widetilde{\bm{U}}_{:,1: r'}\right\|\notag\\
	&\hspace{1cm}\lesssim \sqrt{\frac{\mu r}{m_1}}\frac{r\sqrt{m}_1\omega_{\sf max}\log m}{\sigma_{r'}^\star} + 2\sqrt{\frac{\mu r}{m_1}}\frac{r\sqrt{m}_1\omega_{\sf max}\log m}{\sigma_{r'}^\star}\notag\\
	&\hspace{1cm} \asymp \sqrt{\frac{\mu r}{m_1}}\frac{r\sqrt{m}_1\omega_{\sf max}\log m}{\sigma_{r'}^\star}.
\end{align}
By virtue of \eqref{ineq112} and \eqref{ineq111}, we arrive at
\begin{align}\label{ineq113}
	&\big\|\bm{U}_{:,1: r'}^{\sf oracle}\bm{U}_{:,1: r'}^{\sf oracle\top} - \bm{U}_{:,1: r'}^\star\bm{U}_{:,1: r'}^{\star\top}\big\|_{2,\infty}\notag\\
	&\quad \leq \big\|\bm{U}_{:,1: r'}^{\sf oracle}\bm{U}_{:,1: r'}^{\sf oracle\top} - \widetilde{\bm{U}}_{:,1: r'}\widetilde{\bm{U}}_{:,1: r'}^\top\big\|_{2,\infty} + \big\|\widetilde{\bm{U}}_{:,1: r'}\widetilde{\bm{U}}_{:,1: r'}^\top - \bm{U}_{:,1: r'}^\star\bm{U}_{:,1: r'}^{\star\top}\big\|_{2,\infty}\notag\\
	&\quad \lesssim \sqrt{\frac{\mu r^3}{m_1}}\left(\frac{r^2\sqrt{m_1}\omega_{\sf max}\log m}{\sigma_{r'}^\star} + \frac{r^2 \sqrt{m_1m_2}\omega_{\sf max}^2\log^2 m}{\sigma_{r'}^{\star2}}\right).
\end{align} 
\qed

%
\subsection{Proof of Lemma \ref{lm:space_estimate_expansion}}\label{proof:lm_space_estimate_expansion}
Denote by $\gamma_1$  the following counterclockwise contour on the complex plane:
\begin{align*}
	\gamma_1 = \bigg\{x + y\rm{i}&: x = \frac{\overline{\lambda}_{r_1} + \overline{\lambda}_{r_1+1}}{2}, -\frac{\overline{\lambda}_{r_1} - \overline{\lambda}_{r_1+1}}{2} \leq y \leq \frac{\overline{\lambda}_{r_1} + \overline{\lambda}_{r_1+1}}{2},\\
	&\quad \text{or } x = \overline{\lambda}_1 + \frac{\overline{\lambda}_{r_1} - \overline{\lambda}_{r_1+1}}{2}, -\frac{\overline{\lambda}_{r_1} - \overline{\lambda}_{r_1+1}}{2} \leq y \leq \frac{\overline{\lambda}_{r_1} + \overline{\lambda}_{r_1+1}}{2},\\
	&\quad \text{or } y = \pm\frac{\overline{\lambda}_{r_1} - \overline{\lambda}_{r_1+1}}{2}, \frac{\overline{\lambda}_{r_1} + \overline{\lambda}_{r_1+1}}{2} \leq x \leq \overline{\lambda}_1 + \frac{\overline{\lambda}_{r_1} - \overline{\lambda}_{r_1+1}}{2}\bigg\}.
\end{align*}
Then $\{\overline{\lambda}_i\}_{i=1}^{r_1}$ lie inside the contour $\gamma_1$ and $\{\overline{\lambda}_i\}_{i = r_1 + 1}^{n}$ (where $\overline{\lambda}_{i} = 0$ for $i \geq r+1$) reside outside $\gamma_1$. Moreover, for any $\eta \in \gamma_1$ and $1 \leq i \leq n$,
one has
\begin{align}\label{ineq47}
	\left|\eta - \overline{\lambda}_i\right| \leq \frac{\overline{\lambda}_{r_1} - \overline{\lambda}_{r_1+1}}{2}.
\end{align}
\paragraph{Step 1: decompose $\bm{U}_1\bm{U}_1^\top - \overline{\bm{U}}_1\overline{\bm{U}}_1^\top$.} First, we invoke a similar argument as in \citet[Theorem 1]{xia2021normal} to express  $\bm{U}_1\bm{U}_1^\top - \overline{\bm{U}}_1\overline{\bm{U}}_1^\top$ as an infinite sum. Denote by $\lambda_1 \geq \cdots \geq \lambda_{n}$ the eigenvalues of $\bm{M}$. Apply Weyl's inequality to obtain 
\begin{align*}
	\max_{1 \leq i \leq r}\left|\lambda_i - \overline{\lambda}_i\right| \leq \left\|\bm{Z}\right\| < \frac{\overline{\lambda}_{r_1} - \overline{\lambda}_{r_1+1}}{2}.
\end{align*}
As a result, we know that $\left\{\lambda_i\right\}_{i=1}^{r_1}$ are inside the contour $\gamma_1$, and $\left\{\lambda_i\right\}_{i=r_1+1}^{n}$ are outside the contour. Similar to Eqn.~(10) in \cite{xia2021normal}, one has
\begin{align}\label{eq:space_expansion_1}
	\bm{U}_1\bm{U}_1^\top = \frac{1}{2\pi\rm{i}}\oint_{\gamma_1}\left(\eta\bm{I} - \bm{M}\right)\rm{d}\eta.
\end{align}
We define $$\mathcal{R}_{\overline{\bm{M}}}(\eta) := \left(\eta\bm{I} - \overline{\bm{M}}\right)^{-1} = \sum_{i=1}^{n}\frac{1}{\eta - \overline{\lambda}_j}\overline{\bm{u}}_j\overline{\bm{u}}_j^\top.$$
In view of \eqref{ineq47}, for any $\eta \in \gamma_1$, we have
\begin{align*}
	\left\|\mathcal{R}_{\overline{\bm{M}}}(\eta)\right\| \leq \frac{2}{\overline{\lambda}_{r_1} - \overline{\lambda}_{r_1+1}}, 
\end{align*}
and consequently, 
\begin{align*}
	\left\|\mathcal{R}_{\overline{\bm{M}}}(\eta)\bm{Z}\right\| \leq \left\|\mathcal{R}_{\overline{\bm{M}}}(\eta)\right\|\left\|\bm{Z}\right\| \leq \frac{2\left\|\bm{Z}\right\|}{\overline{\lambda}_{r_1} - \overline{\lambda}_{r_1+1}} < 1.
\end{align*}
A similar argument in \citet[Eqn.~(13)]{xia2021normal} yields
\begin{align}\label{ineq52}
	\bm{U}_1\bm{U}_1^\top - \overline{\bm{U}}_1\overline{\bm{U}}_1^\top &= \sum_{k \geq 1}\frac{1}{2\pi\rm{i}}\oint_{\gamma_1}\left[\mathcal{R}_{\overline{\bm{M}}}(\eta)\bm{Z}\right]^k\mathcal{R}_{\overline{\bm{M}}}(\eta)\rm{d}\eta\notag\\
	&= \sum_{k \geq 1}\sum_{1 \leq j_1, \dots, j_{k+1} \leq n}\frac{1}{2\pi\rm{i}}\oint_{\gamma_1}\frac{\rm{d}\eta}{\left(\eta - \overline{\lambda}_{j_1}\right)\cdots\left(\eta - \overline{\lambda}_{j_{k+1}}\right)}\bm{P}_{\overline{\bm{u}}_{j_1}}\bm{Z}\bm{P}_{\overline{\bm{u}}_{j_2}}\bm{Z}\cdots\bm{P}_{\overline{\bm{u}}_{j_k}}\bm{Z}\bm{P}_{\overline{\bm{u}}_{j_{k+1}}}\notag\\
	&= \sum_{k \geq 1}\sum_{0 \leq j_1, \dots, j_{k+1} \leq r}\frac{1}{2\pi\rm{i}}\oint_{\gamma_1}\frac{\rm{d}\eta}{\left(\eta - \overline{\lambda}_{j_1}\right)\cdots\left(\eta - \overline{\lambda}_{j_{k+1}}\right)}\overline{\bm{P}}_{j_1}\bm{Z}\overline{\bm{P}}_{j_2}\bm{Z}\cdots\overline{\bm{P}}_{j_k}\bm{Z}\overline{\bm{P}}_{j_{k+1}}.
\end{align} 
Here, we define $\overline{\lambda}_0 = 0$ and the last line holds since $\overline{\lambda}_{i} = 0$ for all $i \geq r+1$ and $\overline{\bm{P}}_{0} = \overline{\bm{U}}_{\perp}\overline{\bm{U}}_{\perp}^\top = \sum_{i=r+1}^n\overline{\bm{u}}_i\overline{\bm{u}}_i^\top$.

\paragraph{Step 2: bounding $\|\bm{U}_1\bm{U}_1^\top - \overline{\bm{U}}_1\overline{\bm{U}}_1^\top\|_{2,\infty}$.} By virtue of \eqref{ineq52} and the triangle inequality, we see that: to prove \eqref{ineq:space_estimate_expansion}, it suffices to bound $|\frac{1}{2\pi\rm{i}}\oint_{\gamma_1}\frac{\rm{d}\eta}{(\eta - \overline{\lambda}_{j_1})\cdots(\eta - \overline{\lambda}_{j_{k+1}})}|$. We consider two scenarios: all of $j_1, \dots, j_{k+1}$ lie in the set $\{0\} \cup \{r_1+1, \dots, r_1\}$, and at least one of $j_1, \dots, j_{k+1}$ is in the set $\{1, \dots, r_1\}$.

\paragraph{Case 1: all of $j_1, \dots, j_{k+1}$ are either $0$ or larger than $r_1$.} In this case, none of $\overline{\lambda}_{j_1}, \dots, \overline{\lambda}_{j_{k+1}}$ is inside $\gamma_1$ and as a result, $f(\eta) = \frac{1}{(\eta - \overline{\lambda}_{j_1})\cdots(\eta - \overline{\lambda}_{j_{k+1}})}$ is analytic within and on $\gamma_1$. Cauchy's integral theorem tells us that  
\begin{align}\label{ineq53}
	\frac{1}{2\pi\rm{i}}\oint_{\gamma_1}\frac{\rm{d}\eta}{(\eta - \overline{\lambda}_{j_1})\cdots(\eta - \overline{\lambda}_{j_{k+1}})} = 0,
\end{align}
and thus we have
\begin{align}\label{ineq51}
	\frac{1}{2\pi\rm{i}}\oint_{\gamma_1}\frac{\rm{d}\eta}{\left(\eta - \overline{\lambda}_{j_1}\right)\cdots\left(\eta - \overline{\lambda}_{j_{k+1}}\right)}\overline{\bm{P}}_{j_1}\bm{Z}\overline{\bm{P}}_{j_2}\bm{Z}\cdots\overline{\bm{P}}_{j_k}\bm{Z}\overline{\bm{P}}_{j_{k+1}} = \bm{0}.
\end{align}
\paragraph{Case 2: at least one of $j_1, \dots, j_{k+1}$ is between $1$ and $r_1$.} Let
\begin{align}\label{eq:set_J}
	\mathcal{J} = \{j: 1 \leq j \leq r_1, \exists 1 \leq \ell \leq k+1 \text{ s.t. } j_{\ell} = j\},
\end{align}
and let $j_{\sf max}$ and $j_{\sf min}$ denote the largest and smallest elements in $\mathcal{J}$, respectively. We define the following counterclockwise rectangular contour:
\begin{align*}
	\gamma_2 = \bigg\{x + y\rm{i}&: x = \overline{\lambda}_{j_{\sf max}} - \frac{\overline{\lambda}_{r_1} - \overline{\lambda}_{r_1+1}}{2}, -\frac{\overline{\lambda}_{r_1} - \overline{\lambda}_{r_1+1}}{2} \leq y \leq \frac{\overline{\lambda}_{r_1} - \overline{\lambda}_{r_1+1}}{2},\\
	&\quad \text{or } x = \overline{\lambda}_{j_{\sf min}} + \frac{\overline{\lambda}_{r_1} - \overline{\lambda}_{r_1+1}}{2}, -\frac{\overline{\lambda}_{r_1} - \overline{\lambda}_{r_1+1}}{2} \leq y \leq \frac{\overline{\lambda}_{r_1} - \overline{\lambda}_{r_1+1}}{2},\\
	&\quad \text{or } y = \pm\frac{\overline{\lambda}_{r_1} - \overline{\lambda}_{r_1+1}}{2}, \overline{\lambda}_{j_{\sf max}} - \frac{\overline{\lambda}_{r_1} - \overline{\lambda}_{r_1+1}}{2} \leq x \leq \overline{\lambda}_{j_{\sf min}} + \frac{\overline{\lambda}_{r_1} - \overline{\lambda}_{r_1+1}}{2}\bigg\}.
\end{align*}
It is easy to verify that
\begin{align}\label{ineq48}
	\left|\eta - \lambda_{j_\ell}\right| \geq \frac{\overline{\lambda}_{r_1} - \overline{\lambda}_{r_1+1}}{2},~\qquad~\forall 1 \leq \ell \leq k+1, \eta \in \gamma_2
\end{align}
and
\begin{align}\label{ineq49}
	\frac{1}{2\pi\rm{i}}\oint_{\gamma_1}\frac{\rm{d}\eta}{\left(\eta - \overline{\lambda}_{j_1}\right)\cdots\left(\eta - \overline{\lambda}_{j_{k+1}}\right)} = \frac{1}{2\pi\rm{i}}\oint_{\gamma_2}\frac{\rm{d}\eta}{\left(\eta - \overline{\lambda}_{j_1}\right)\cdots\left(\eta - \overline{\lambda}_{j_{k+1}}\right)}.
\end{align} 
Moreover, the length of $\gamma_2$ is $$L(\gamma_2) = 2\left(\overline{\lambda}_{j_{\sf min}} - \overline{\lambda}_{j_{\sf max}}\right) + 4\left(\overline{\lambda}_{r_1} - \overline{\lambda}_{r_1+1}\right).$$ If $j_{\sf max} = j_{\sf min}$, i.e., there is only one element in $\mathcal{I}$, then applying the triangle inequality for contour integrals yields
\begin{align*}
	\left|\frac{1}{2\pi\rm{i}}\oint_{\gamma_2}\frac{\rm{d}\eta}{\left(\eta - \overline{\lambda}_{j_1}\right)\cdots\left(\eta - \overline{\lambda}_{j_{k+1}}\right)}\right| &\leq \frac{1}{2\pi}\sup_{\eta \in \gamma_2}\left|\frac{1}{\left(\eta - \overline{\lambda}_{j_1}\right)\cdots\left(\eta - \overline{\lambda}_{j_{k+1}}\right)}\right|\cdot L(\gamma_2)\\
	&\stackrel{\eqref{ineq48}}{\leq} \frac{1}{2\pi}\left(\frac{2}{\overline{\lambda}_{r_1} - \overline{\lambda}_{r_1+1}}\right)^{k+1}\cdot 4\left(\overline{\lambda}_{r_1} - \overline{\lambda}_{r_1+1}\right)\\
	&\leq \frac{4}{\pi}\left(\frac{2}{\overline{\lambda}_{r_1} - \overline{\lambda}_{r_1+1}}\right)^{k}.
\end{align*}
If $j_{\sf max} \neq j_{\sf min}$, then the triangle inequality tells us that for any $\eta \in \gamma_2$,
\begin{align}\label{ineq11}
	\max\left\{\left|\eta - \overline{\lambda}_{j_{\sf max}}\right|, \left|\eta - \overline{\lambda}_{j_{\sf min}}\right|\right\} \geq \frac{\left|\left(\eta - \overline{\lambda}_{j_{\sf max}}\right) - \left(\eta - \overline{\lambda}_{j_{\sf min}}\right)\right|}{2} = \frac{\overline{\lambda}_{j_{\sf min}} - \overline{\lambda}_{j_{\sf max}}}{2}.
\end{align}
In view of \eqref{ineq48}, \eqref{ineq11} and the basic inequality $\min\{\frac{a}{b}, \frac{c}{d}\} \leq \frac{a + c}{b + d}$ for $a, b, c, d > 0$, one has
\begin{align}\label{ineq15}
	\frac{1}{\left|\left(\eta - \overline{\lambda}_{j_{\sf max}}\right)\left(\eta - \overline{\lambda}_{j_{\sf min}}\right)\right|} &\leq \min\left\{\frac{4}{\left(\overline{\lambda}_{j_{\sf min}} - \overline{\lambda}_{j_{\sf max}}\right)\left(\overline{\lambda}_{r_1} - \overline{\lambda}_{r_1+1}\right)}, \left(\frac{2}{\overline{\lambda}_{r_1} - \overline{\lambda}_{r_1+1}}\right)^{2}\right\}\notag\\
	&\leq \left(\frac{2}{\overline{\lambda}_{r_1} - \overline{\lambda}_{r_1+1}}\right)\cdot \left(\frac{4}{\overline{\lambda}_{r_1} - \overline{\lambda}_{r_1+1} + \overline{\lambda}_{j_{\sf max}} - \overline{\lambda}_{j_{\sf min}}}\right)
\end{align}
for all $\eta \in \gamma_2$, and consequently, one has
\begin{align*}
	\left|\frac{1}{2\pi\rm{i}}\oint_{\gamma_2}\frac{\rm{d}\eta}{\left(\eta - \overline{\lambda}_{j_1}\right)\cdots\left(\eta - \overline{\lambda}_{j_{k+1}}\right)}\right| &\leq \frac{1}{2\pi}\sup_{\eta \in \gamma_2}\left|\frac{1}{\left(\eta - \overline{\lambda}_{j_1}\right)\cdots\left(\eta - \overline{\lambda}_{j_{k+1}}\right)}\right|\cdot L(\gamma_2)\\
	&\leq \frac{1}{2\pi}\left(\frac{2}{\overline{\lambda}_{r_1} - \overline{\lambda}_{r_1+1}}\right)\cdot \left(\frac{4}{\overline{\lambda}_{r_1} - \overline{\lambda}_{r_1+1} + \overline{\lambda}_{j_{\sf min}} - \overline{\lambda}_{j_{\sf min}}}\right)\\&\quad\cdot\left(\frac{2}{\overline{\lambda}_{r_1} - \overline{\lambda}_{r_1+1}}\right)^{k-1}\cdot\left(2\left(\overline{\lambda}_{j_{\sf min}} - \overline{\lambda}_{j_{\sf max}}\right) + 4\left(\overline{\lambda}_{r_1} - \overline{\lambda}_{r_1+1}\right)\right)\\
	&\leq \frac{8}{\pi}\left(\frac{2}{\overline{\lambda}_{r_1} - \overline{\lambda}_{r_1+1}}\right)^{k}.
\end{align*}
Therefore, it is guaranteed that
\begin{align}\label{ineq50}
	\left|\frac{1}{2\pi\rm{i}}\oint_{\gamma_2}\frac{\rm{d}\eta}{\left(\eta - \overline{\lambda}_{j_1}\right)\cdots\left(\eta - \overline{\lambda}_{j_{k+1}}\right)}\right| \leq \frac{8}{\pi}\left(\frac{2}{\overline{\lambda}_{r_1} - \overline{\lambda}_{r_1+1}}\right)^{k}.
\end{align}

Combining \eqref{ineq52}, \eqref{ineq51}, \eqref{ineq50} and the triangle inequality finishes the proof of \eqref{ineq:space_estimate_expansion}.
\paragraph{Step 3: bounding $\|(\overline{\bm{U}}_1\overline{\bm{U}}_1^\top - \bm{U}_1\bm{U}_1^\top)\overline{\bm{M}}\|_{2,\infty}$.} Next, we move on to control $\|(\overline{\bm{U}}_1\overline{\bm{U}}_1^\top - \bm{U}_1\bm{U}_1^\top)\overline{\bm{M}}\|_{2, \infty}$. By virtue of \eqref{ineq52}, we have
\begin{align}\label{eq:decomposition_2}
	&\big(\bm{U}_1\bm{U}_1^\top - \overline{\bm{U}}_1\overline{\bm{U}}_1^\top\big)\overline{\bm{M}}\notag\\
	&\quad= \sum_{k \geq 1}\sum_{0 \leq j_1, \dots, j_k \leq r}\sum_{j_{k+1} = 1}^{r}\frac{1}{2\pi\rm{i}}\oint_{\gamma_1}\frac{\rm{d}\eta}{\left(\eta - \overline{\lambda}_{j_1}\right)\cdots\left(\eta - \overline{\lambda}_{j_{k+1}}\right)}\overline{\bm{P}}_{j_1}\bm{Z}\overline{\bm{P}}_{j_2}\bm{Z}\cdots\overline{\bm{P}}_{j_k}\bm{Z}\overline{\bm{P}}_{j_{k+1}}\overline{\bm{M}}\notag\\
	&\quad= \sum_{k \geq 1}\sum_{0 \leq j_1, \dots, j_k \leq r}\sum_{j_{k+1} = 1}^{r}\frac{1}{2\pi\rm{i}}\oint_{\gamma_1}\frac{\overline{\lambda}_{j_{k+1}}\rm{d}\eta}{\left(\eta - \overline{\lambda}_{j_1}\right)\cdots\left(\eta - \overline{\lambda}_{j_{k+1}}\right)}\overline{\bm{P}}_{j_1}\bm{Z}\overline{\bm{P}}_{j_2}\bm{Z}\cdots\overline{\bm{P}}_{j_k}\bm{Z}\overline{\bm{P}}_{j_{k+1}}.
\end{align}
The second line and the third line make use of $\overline{\bm{U}}_{\perp}\overline{\bm{M}} = 0$ and $\overline{\bm{P}}_j\,\overline{\bm{M}} = \overline{\bm{u}}_j\,\overline{\bm{u}}_j^\top(\sum_{i=1}^{r}\overline{\lambda}_i\overline{\bm{u}}_i\,\overline{\bm{u}}_i^\top) = \overline{\lambda}_j$, respectively. In the rest of the proof, we will establish upper bounds for $|\frac{1}{2\pi\rm{i}}\oint_{\gamma_1}\frac{\overline{\lambda}_{j_{k+1}}\rm{d}\eta}{(\eta - \overline{\lambda}_{j_1})\cdots(\eta - \overline{\lambda}_{j_{k+1}})}|$ and justify the validity of \eqref{ineq:residual_expansion}.

\paragraph{Step 3.1: bounding $|\frac{1}{2\pi\rm{i}}\oint_{\gamma_1}\frac{\overline{\lambda}_{j_{k+1}}\rm{d}\eta}{(\eta - \overline{\lambda}_{j_1})\cdots(\eta - \overline{\lambda}_{j_{k+1}})}|$ for $k = 1$.} We consider two scenarios: (1) $1 \leq j_1, j_2 \leq r_1$ or  $j_1, j_2 \in \{r_1 +1, \dots, r\} \cup \{0\}$ and (2) only one of $j_1$ and $j_2$ falls into the set $\{1, \dots, r_1\}$.

\paragraph{Case 1: $1 \leq j_1, j_2 \leq r_1$ or  $j_1, j_2 \in \{r_1 +1, \dots, r\} \cup \{0\}$.} In this case, Cauchy's integral formula asserts that
\begin{align*}
	\frac{1}{2\pi\rm{i}}\oint_{\gamma_1}\frac{\rm{d}\eta}{\left(\eta - \overline{\lambda}_{j_1}\right)\left(\eta - \overline{\lambda}_{j_{2}}\right)} = 0, 
\end{align*}
and consequently, 
\begin{align}\label{ineq12}
	\left|\frac{1}{2\pi\rm{i}}\oint_{\gamma_1}\frac{\overline{\lambda}_{j_{2}}\rm{d}\eta}{(\eta - \overline{\lambda}_{j_1})(\eta - \overline{\lambda}_{j_{2}})}\right| = 0.
\end{align}

\paragraph{Case 2: exactly one of $j_i \in \{1, \dots, r_1\}$.} Apply Cauchy's integral formula to yield 
\begin{align*}
	\left|\frac{1}{2\pi\rm{i}}\oint_{\gamma_1}\frac{\rm{d}\eta}{\left(\eta - \overline{\lambda}_{j_1}\right)\left(\eta - \overline{\lambda}_{j_{2}}\right)}\right| = 	\left|\frac{1}{2\pi\rm{i}}\oint_{\gamma_1}\frac{\frac{1}{\eta - \overline{\lambda}_{j_1}}}{\eta - \overline{\lambda}_{j_{2}}}\rm{d}\eta\right| = \frac{1}{\left|\overline{\lambda}_{j_1} - \overline{\lambda}_{j_2}\right|}.
\end{align*}
Observing that the function $f(x) = \frac{1}{|x - 1|}$ is increasing on $(-\infty, 1)$ and decreasing on $(1, \infty)$, one has
\begin{align}\label{ineq13}
		\left|\frac{1}{2\pi\rm{i}}\oint_{\gamma_1}\frac{\overline{\lambda}_{j_{k+1}}\rm{d}\eta}{(\eta - \overline{\lambda}_{j_1})(\eta - \overline{\lambda}_{j_{2}})}\right| = \frac{\overline{\lambda}_{j_{2}}}{\left|\overline{\lambda}_{j_1} - \overline{\lambda}_{j_2}\right|} = \frac{1}{\left|\frac{\overline{\lambda}_{j_1}}{\overline{\lambda}_{j_2}} - 1\right|} \leq \frac{1}{\max\left\{\frac{\overline{\lambda}_{r_1}}{\overline{\lambda}_{r_1 + 1}} - 1, 1 - \frac{\overline{\lambda}_{r_1 + 1}}{\overline{\lambda}_{r_1}}\right\}} = \frac{\overline{\lambda}_{r_1}}{\overline{\lambda}_{r_1} - \overline{\lambda}_{r_1 + 1}}.
\end{align}

\paragraph{Step 3.2: bounding $|\frac{1}{2\pi\rm{i}}\oint_{\gamma_1}\frac{\overline{\lambda}_{j_{k+1}}\rm{d}\eta}{(\eta - \overline{\lambda}_{j_1})\cdots(\eta - \overline{\lambda}_{j_{k+1}})}|$ for $k > 1$.} If (1) all $j_1, \dots, j_{k+1}$ lie in the set $\{1, \dots, r_1\}$ or (2) all $j_1, \dots, j_{k+1}$ are all in the set $\{0\} \cup \{r_1 + 1, \dots, r\}$, then the function $$g(x) = \frac{\overline{\lambda}_{j_{k+1}}}{\left(\eta - \overline{\lambda}_{j_1}\right)\cdots\left(\eta - \overline{\lambda}_{j_{k+1}}\right)}$$ is analytic in and on $\gamma_1$ or outside on on $\gamma_1$. As a result, Cauchy's integral formula tells us that
\begin{align}\label{eq5}
	\left|\frac{1}{2\pi\rm{i}}\oint_{\gamma_1}\frac{\overline{\lambda}_{j_{k+1}}\rm{d}\eta}{\left(\eta - \overline{\lambda}_{j_1}\right)\cdots\left(\eta - \overline{\lambda}_{j_{k+1}}\right)}\right| = 0.
\end{align}
In the following proof, we assume that these two cases would not happen, i.e.,
\begin{align}\label{assump:set}
	1 \leq \left|\left\{j_1, \dots, j_{k+1}\right\} \cap \{1, \dots, r_1\}\right| \leq k.
\end{align}
Let $\gamma_3$ denote the following counterclockwise rectangular contour:
\begin{align*}
	\gamma_3 = \bigg\{x + y\rm{i}&: x = \frac{\overline{\lambda}_{j_{\sf max}} + \overline{\lambda}_{r_1+1}}{2}, -\frac{\overline{\lambda}_{j_{\sf max}} - \overline{\lambda}_{r_1+1}}{2} \leq y \leq \frac{\overline{\lambda}_{j_{\sf max}} - \overline{\lambda}_{r_1+1}}{2},\\
	&\quad \text{or } x = \overline{\lambda}_{j_{\sf min}} + \frac{\overline{\lambda}_{j_{\sf max}} - \overline{\lambda}_{r_1+1}}{2}, -\frac{\overline{\lambda}_{j_{\sf max}} - \overline{\lambda}_{r_1+1}}{2} \leq y \leq \frac{\overline{\lambda}_{j_{\sf max}} - \overline{\lambda}_{r_1+1}}{2},\\
	&\quad \text{or } y = \pm\frac{\overline{\lambda}_{j_{\sf max}} - \overline{\lambda}_{r_1+1}}{2}, \frac{\overline{\lambda}_{j_{\sf max}} + \overline{\lambda}_{r_1+1}}{2} \leq x \leq \overline{\lambda}_{j_{\sf min}} + \frac{\overline{\lambda}_{j_{\sf max}} - \overline{\lambda}_{r_1+1}}{2}\bigg\},
\end{align*}
where we recall that $j_{\sf min}$ (resp.~$j_{\sf max}$) is the smallest (resp.~largest) element in the set $\mathcal{J}$ defined in \eqref{eq:set_J}. Then one can check that
\begin{align}\label{ineq14}
	\left|\eta - \overline{\lambda}_{j_\ell}\right| \geq \frac{\overline{\lambda}_{j_{\sf max}} - \overline{\lambda}_{r_1+1}}{2},~\qquad~\forall 1 \leq \ell \leq k+1, \eta \in \gamma_3, 
\end{align}
and the length of $\gamma_3$ satisfies
\begin{align}\label{eq:length_gamma_3}
	L\left(\gamma_3\right) = 2\left(\overline{\lambda}_{j_{\sf min}} - \overline{\lambda}_{j_{\sf max}}\right) + 4\left(\overline{\lambda}_{j_{\sf max}} - \overline{\lambda}_{r_1+1}\right) = 2\overline{\lambda}_{j_{\sf min}} + 2\overline{\lambda}_{j_{\sf max}} - 4\overline{\lambda}_{r_1+1}.
\end{align}
In addition, we have
\begin{align}\label{eq4}
	\frac{1}{2\pi\rm{i}}\oint_{\gamma_1}\frac{\rm{d}\eta}{\left(\eta - \overline{\lambda}_{j_1}\right)\cdots\left(\eta - \overline{\lambda}_{j_{k+1}}\right)} = \frac{1}{2\pi\rm{i}}\oint_{\gamma_3}\frac{\rm{d}\eta}{\left(\eta - \overline{\lambda}_{j_1}\right)\cdots\left(\eta - \overline{\lambda}_{j_{k+1}}\right)}.
\end{align}
\paragraph{Case 1: $\overline{\lambda}_{j_{\sf min}} - \overline{\lambda}_{j_{\sf max}} \leq 3(\overline{\lambda}_{j_{\sf max}} - \overline{\lambda}_{r_1+1})$.} In this scenario, one has
\begin{align}\label{ineq18}
	\overline{\lambda}_{j_{\sf min}} - \overline{\lambda}_{r_1+1} = \overline{\lambda}_{j_{\sf min}} - \overline{\lambda}_{j_{\sf max}} + \overline{\lambda}_{j_{\sf max}} - \overline{\lambda}_{r_1+1} \leq 4\left(\overline{\lambda}_{j_{\sf max}} - \overline{\lambda}_{r_1+1}\right),
\end{align}
which further leads to
\begin{align*}
	L(\gamma_3) \leq 10\left(\overline{\lambda}_{j_{\sf max}} - \overline{\lambda}_{r_1+1}\right).
\end{align*}
In view of \eqref{ineq14}, \eqref{eq4}, \eqref{ineq18} and the previous inequality, one has
\begin{align}\label{ineq17}
	\left|\frac{1}{2\pi\rm{i}}\oint_{\gamma_1}\frac{\overline{\lambda}_{j_{k+1}}\rm{d}\eta}{\left(\eta - \overline{\lambda}_{j_1}\right)\cdots\left(\eta - \overline{\lambda}_{j_{k+1}}\right)}\right|
	&\leq \frac{1}{2\pi}\overline{\lambda}_{j_{k+1}}\left(\frac{2}{\overline{\lambda}_{j_{\sf max}} - \overline{\lambda}_{r_1+1}}\right)^{k+1}L\left(\gamma_3\right)\notag\\
	&\leq \frac{1}{2\pi}\overline{\lambda}_{j_{\sf min}}\left(\frac{2}{\overline{\lambda}_{j_{\sf max}} - \overline{\lambda}_{r_1+1}}\right)^{k+1}\cdot 10\left(\overline{\lambda}_{j_{\sf max}} - \overline{\lambda}_{r_1+1}\right)\notag\\
	&\leq \frac{80}{\pi}\frac{\overline{\lambda}_{j_{\sf min}}}{\overline{\lambda}_{j_{\sf min}} - \overline{\lambda}_{r_1+1}}\left(\frac{2}{\overline{\lambda}_{j_{\sf max}} - \overline{\lambda}_{r_1+1}}\right)^{k-1}\notag\\&\leq \frac{80}{\pi}\frac{\overline{\lambda}_{r_1}}{\overline{\lambda}_{r_1} - \overline{\lambda}_{r_1+1}}\left(\frac{2}{\overline{\lambda}_{r_1} - \overline{\lambda}_{r_1+1}}\right)^{k-1} = \frac{40}{\pi}\overline{\lambda}_{r_1}\left(\frac{2}{\overline{\lambda}_{r_1} - \overline{\lambda}_{r_1+1}}\right)^{k}.
\end{align}
The third inequality holds because of \eqref{ineq18}, whereas the last one applies the monotonicity of the function $f(x) = \frac{x}{x - \overline{\lambda}_{r_1+1}}$ for $x > \overline{\lambda}_{r_1+1}$ and the inequality $\overline{\lambda}_{j_{\sf min}} \geq \overline{\lambda}_{j_{\sf max}} \geq \overline{\lambda}_{r_1}$.
\paragraph{Case 2: $\overline{\lambda}_{j_{\sf min}} - \overline{\lambda}_{j_{\sf max}} > 3(\overline{\lambda}_{j_{\sf max}} - \overline{\lambda}_{r_1+1})$.}
Denote the following two disjoint sets
\begin{align}\label{eq:set_I_1}
	\mathcal{I}_1 = \left\{i: 1 \leq i \leq k+1, \overline{\lambda}_{j_i} \geq \frac{2}{3}\overline{\lambda}_{j_{\sf min}} + \frac{1}{3}\overline{\lambda}_{j_{\sf max}}\right\}
\end{align}
and
\begin{align}\label{eq:set_I_2}
	\mathcal{I}_2 = \left\{i: 1 \leq i \leq k+1, \overline{\lambda}_{j_i} \leq \frac{1}{3}\overline{\lambda}_{j_{\sf min}} + \frac{2}{3}\overline{\lambda}_{j_{\sf max}}\right\}.
\end{align}
By the definition of $j_{\sf min}$ and $j_{\sf max}$, we know that $\mathcal{I}_1$ and $\mathcal{I}_2$ are nonempty.
We consider the following three scenarios: (1) $\min\{|\mathcal{I}_1|, |\mathcal{I}_2|\} \geq 2$; (2) $|\mathcal{I}_1| = 1$ and (3) $|\mathcal{I}_2| = 1$.
\paragraph{Case 2.1: $\min\{|\mathcal{I}_1|, |\mathcal{I}_2|\} \geq 2$.} When it comes to this case, one can find four different indices $i_1, i_2, i_3, i_4$ such that $i_1, i_3 \in \mathcal{I}_1$, $i_2, i_4 \in \mathcal{I}_2$, $\overline{\lambda}_{j_{i_1}} = \overline{\lambda}_{j_{\sf min}}$ and $\overline{\lambda}_{j_{i_2}} = \overline{\lambda}_{j_{\sf max}}$. Then the triangle inequality tells us that
\begin{align*}
	\max\left\{\left|\eta - \overline{\lambda}_{j_{i_1}}\right|, \left|\eta - \overline{\lambda}_{j_{i_2}}\right|\right\} \geq \frac{1}{2}\left(\overline{\lambda}_{j_{i_1}} - \overline{\lambda}_{j_{i_2}}\right) = \frac{1}{2}\left(\overline{\lambda}_{j_{\sf min}} - \overline{\lambda}_{j_{\sf max}}\right)
\end{align*}
and
\begin{align*}
	\max\left\{\left|\eta - \overline{\lambda}_{j_{i_3}}\right|, \left|\eta - \overline{\lambda}_{j_{i_4}}\right|\right\} \geq \frac{1}{2}\left(\overline{\lambda}_{j_{i_3}} - \overline{\lambda}_{j_{i_4}}\right) \geq  \frac{1}{6}\left(\overline{\lambda}_{j_{\sf min}} - \overline{\lambda}_{j_{\sf max}}\right).
\end{align*}
Similar to \eqref{ineq15}, one can derive 
\begin{align*}
	&\left|\frac{1}{\left(\eta - \overline{\lambda}_{j_{i_1}}\right)\left(\eta - \overline{\lambda}_{j_{i_2}}\right)\left(\eta - \overline{\lambda}_{j_{i_3}}\right)\left(\eta - \overline{\lambda}_{j_{i_4}}\right)}\right|\\ &\quad = \frac{1}{\left|\eta - \overline{\lambda}_{j_{i_1}}\right|\left|\eta - \overline{\lambda}_{j_{i_2}}\right|}\frac{1}{\left|\eta - \overline{\lambda}_{j_{i_3}}\right|\left|\eta - \overline{\lambda}_{j_{i_4}}\right|}\\
	&\quad = \min\left\{\frac{4}{\left(\overline{\lambda}_{j_{\sf min}} - \overline{\lambda}_{j_{\sf max}}\right)\left(\overline{\lambda}_{j_{\sf max}} - \overline{\lambda}_{r_1+1}\right)}, \frac{4}{\left(\overline{\lambda}_{j_{\sf max}} - \overline{\lambda}_{r_1+1}\right)^2}\right\}\\&\qquad\cdot\min\left\{\frac{12}{\left(\overline{\lambda}_{j_{\sf min}} - \overline{\lambda}_{j_{\sf max}}\right)\left(\overline{\lambda}_{j_{\sf max}} - \overline{\lambda}_{r_1+1}\right)}, \frac{4}{\left(\overline{\lambda}_{j_{\sf max}} - \overline{\lambda}_{r_1+1}\right)^2}\right\}\\
	&\quad \leq \frac{8}{\left(\overline{\lambda}_{j_{\sf max}} - \overline{\lambda}_{r_1+1}\right)\left(\overline{\lambda}_{j_{\sf min}} - \overline{\lambda}_{r_1+1}\right)}\cdot  \frac{16}{\left(\overline{\lambda}_{j_{\sf max}} - \overline{\lambda}_{r_1+1}\right)\left(\overline{\lambda}_{j_{\sf min}} - \overline{\lambda}_{r_1+1}\right)}\\
	&\quad = \frac{128}{\left(\overline{\lambda}_{j_{\sf max}} - \overline{\lambda}_{r_1+1}\right)^2\left(\overline{\lambda}_{j_{\sf min}} - \overline{\lambda}_{r_1+1}\right)^2}.
\end{align*}
The penultimate line uses the basic inequality $\min\{a/b, c/d\} \leq (a + c)/(b + d)$ for $a, b, c, d > 0$. Putting the previous inequality, \eqref{ineq14}, \eqref{eq:length_gamma_3} and \eqref{ineq4} together, we reach
\begin{align}\label{ineq16}
	&\left|\frac{1}{2\pi\rm{i}}\oint_{\gamma_1}\frac{\overline{\lambda}_{j_{k+1}}\rm{d}\eta}{\left(\eta - \overline{\lambda}_{j_1}\right)\cdots\left(\eta - \overline{\lambda}_{j_{k+1}}\right)}\right|\notag\\ &\hspace{1cm}\leq \frac{1}{2\pi}\sup_{\eta \in \gamma_3}\left|\frac{\overline{\lambda}_{j_{k+1}}}{\left(\eta - \overline{\lambda}_{j_1}\right)\cdots\left(\eta - \overline{\lambda}_{j_{k+1}}\right)}\right|\cdot L(\gamma_3)\notag\\
	&\hspace{1cm}\leq \frac{1}{2\pi}\overline{\lambda}_{j_{k+1}}\frac{128}{\left(\overline{\lambda}_{j_{\sf max}} - \overline{\lambda}_{r_1+1}\right)^2\left(\overline{\lambda}_{j_{\sf min}} - \overline{\lambda}_{r_1+1}\right)^2}\left(\frac{2}{\overline{\lambda}_{j_{\sf max}} - \overline{\lambda}_{r_1+1}}\right)^{k+1-4}\cdot 4\left(\overline{\lambda}_{j_{\sf min}} - \overline{\lambda}_{r_1+1}\right)\notag\\
	&\hspace{1cm}\leq \frac{64}{\pi}\frac{\overline{\lambda}_{j_{\sf min}}}{\overline{\lambda}_{j_{\sf min}} - \overline{\lambda}_{r_1+1}}\left(\frac{2}{\overline{\lambda}_{j_{\sf max}} - \overline{\lambda}_{r_1+1}}\right)^{k-1}\notag\\
	&\hspace{1cm}\leq \frac{64}{\pi}\frac{\overline{\lambda}_{r_1}}{\overline{\lambda}_{r_1} - \overline{\lambda}_{r_1+1}}\left(\frac{2}{\overline{\lambda}_{j_{\sf max}} - \overline{\lambda}_{r_1+1}}\right)^{k-1}\notag\\
	&\hspace{1cm}\leq \frac{32}{\pi}\left(\frac{2}{\overline{\lambda}_{r_1} - \overline{\lambda}_{r_1+1}}\right)^{k}\overline{\lambda}_{r_1}.
\end{align}
The fourth line is due to $\overline{\lambda}_{j_{k+1}} \leq \overline{\lambda}_{j_{\sf min}}$ and the fifth line holds since $f(x) = \frac{x}{x - \overline{\lambda}_{r_1+1}}$ is a decreasing function on $(\overline{\lambda}_{r_1+1}, \infty)$.

\paragraph{Case 2.2: $|\mathcal{I}_1| = 1$.} 
Let us choose
\begin{align}\label{eq7}
	\ell \in \argmax_{i: 1 \leq i \leq k+1, i \notin \mathcal{I}_1}\overline{\lambda}_{j_i}.
\end{align}
We can see from the definition of $\mathcal{I}_1$ that
\begin{align*}
	\overline{\lambda}_{j_{\sf min}} - \overline{\lambda}_{j_\ell} \geq \overline{\lambda}_{j_{\sf min}} - \left(\frac{2}{3}\overline{\lambda}_{j_{\sf min}} + \frac{1}{3}\overline{\lambda}_{j_{\sf max}}\right) = \frac{1}{3}\left(\overline{\lambda}_{j_{\sf min}} - \overline{\lambda}_{j_{\sf max}}\right) > \overline{\lambda}_{j_{\sf max}} - \overline{\lambda}_{r_1+1}. 
\end{align*}
The last inequality is valid since $\overline{\lambda}_{j_{\sf min}} - \overline{\lambda}_{j_{\sf max}} > 3(\overline{\lambda}_{j_{\sf max}} - \overline{\lambda}_{r_1+1})$. We let $\gamma_4$ and $\gamma_5$ denote the following counterclockwise contours:
\begin{align*} 
	\gamma_4 = \bigg\{x + y\rm{i}&: x = \frac{\overline{\lambda}_{j_{\sf max}} + \overline{\lambda}_{r_1+1}}{2}, -\frac{\overline{\lambda}_{j_{\sf max}} - \overline{\lambda}_{r_1+1}}{2} \leq y \leq \frac{\overline{\lambda}_{j_{\sf max}} - \overline{\lambda}_{r_1+1}}{2},\\
	&\quad \text{or } x = \overline{\lambda}_{j_\ell} + \frac{\overline{\lambda}_{j_{\sf max}} - \overline{\lambda}_{r_1+1}}{2}, -\frac{\overline{\lambda}_{j_{\sf max}} - \overline{\lambda}_{r_1+1}}{2} \leq y \leq \frac{\overline{\lambda}_{j_{\sf max}} - \overline{\lambda}_{r_1+1}}{2},\\
	&\quad \text{or } y = \pm\frac{\overline{\lambda}_{j_{\sf max}} - \overline{\lambda}_{r_1+1}}{2}, \frac{\overline{\lambda}_{j_{\sf max}} + \overline{\lambda}_{r_1+1}}{2} \leq x \leq \overline{\lambda}_{j_\ell} + \frac{\overline{\lambda}_{j_{\sf max}} - \overline{\lambda}_{r_1+1}}{2}\bigg\}
\end{align*}
and
\begin{align*} 
	\gamma_5 = \bigg\{x + y\rm{i}&: x = \overline{\lambda}_{j_\ell} + \frac{\overline{\lambda}_{j_{\sf max}} - \overline{\lambda}_{r_1+1}}{2}, -\frac{\overline{\lambda}_{j_{\sf max}} - \overline{\lambda}_{r_1+1}}{2} \leq y \leq \frac{\overline{\lambda}_{j_{\sf max}} - \overline{\lambda}_{r_1+1}}{2},\\
	&\quad \text{or } x = \overline{\lambda}_{j_{\sf min}} + \frac{\overline{\lambda}_{j_{\sf max}} - \overline{\lambda}_{r_1+1}}{2}, -\frac{\overline{\lambda}_{j_{\sf max}} - \overline{\lambda}_{r_1+1}}{2} \leq y \leq \frac{\overline{\lambda}_{j_{\sf max}} - \overline{\lambda}_{r_1+1}}{2},\\
	&\quad \text{or } y = \pm\frac{\overline{\lambda}_{j_{\sf max}} - \overline{\lambda}_{r_1+1}}{2}, \overline{\lambda}_{j_\ell} + \frac{\overline{\lambda}_{j_{\sf max}} - \overline{\lambda}_{r_1+1}}{2} \leq x \leq \overline{\lambda}_{j_{\sf min}} + \frac{\overline{\lambda}_{j_{\sf max}} - \overline{\lambda}_{r_1+1}}{2}\bigg\}.
\end{align*}
For any complex number $\eta = x + y\rm{i}$ with $x = \overline{\lambda}_{j_\ell} + \frac{\overline{\lambda}_{j_{\sf max}} - \overline{\lambda}_{r_1+1}}{2}$ and $ -\frac{\overline{\lambda}_{j_{\sf max}} - \overline{\lambda}_{r_1+1}}{2} \leq y \leq \frac{\overline{\lambda}_{j_{\sf max}} - \overline{\lambda}_{r_1+1}}{2}$, we know that $g(\eta) = \frac{1}{\left(\eta - \overline{\lambda}_{j_1}\right)\cdots\left(\eta - \overline{\lambda}_{j_{k+1}}\right)} \neq 0$, and thus $g(\eta)$ is analytic on $\{\eta = x + y{\rm{i}}: x = \overline{\lambda}_{j_\ell} + \frac{\overline{\lambda}_{j_{\sf max}} - \overline{\lambda}_{r_1+1}}{2}, -\frac{\overline{\lambda}_{j_{\sf max}} - \overline{\lambda}_{r_1+1}}{2} \leq y \leq \frac{\overline{\lambda}_{j_{\sf max}} - \overline{\lambda}_{r_1+1}}{2}\}$. Applying Cauchy's integral formula yields 
\begin{align}\label{eq6}
	&\frac{1}{2\pi\rm{i}}\oint_{\gamma_3}\frac{\rm{d}\eta}{\left(\eta - \overline{\lambda}_{j_1}\right)\cdots\left(\eta - \overline{\lambda}_{j_{k+1}}\right)}\notag\\ &\hspace{1cm}= \frac{1}{2\pi\rm{i}}\oint_{\gamma_4}\frac{\rm{d}\eta}{\left(\eta - \overline{\lambda}_{j_1}\right)\cdots\left(\eta - \overline{\lambda}_{j_{k+1}}\right)} + \frac{1}{2\pi\rm{i}}\oint_{\gamma_5}\frac{\rm{d}\eta}{\left(\eta - \overline{\lambda}_{j_1}\right)\cdots\left(\eta - \overline{\lambda}_{j_{k+1}}\right)}.
\end{align}
Repeating a similar argument as for \eqref{ineq50} reveals that
\begin{align}\label{ineq19}
	\frac{1}{2\pi}\sup_{\eta \in \gamma_4}\left|\frac{1}{\prod_{1 \leq i \leq k+1, i \notin \mathcal{I}_1}\left(\eta - \overline{\lambda}_{j_i}\right)}\right|\cdot L\left(\gamma_4\right) \leq \frac{8}{\pi}\left(\frac{2}{\overline{\lambda}_{r_1} - \overline{\lambda}_{r_1+1}}\right)^{k-1}.
\end{align}
In addition, for any $\eta \in \gamma_4$, we know that its real part
\begin{align*}
	\text{Re}(\eta) &\leq \overline{\lambda}_{j_\ell} + \frac{\overline{\lambda}_{j_{\sf max}} - \overline{\lambda}_{r_1+1}}{2}\\ &\leq \frac{2}{3}\overline{\lambda}_{j_{\sf min}} + \frac{1}{3}\overline{\lambda}_{j_{\sf max}} + \frac{\overline{\lambda}_{j_{\sf max}} - \overline{\lambda}_{r_1+1}}{2}\\ &= \overline{\lambda}_{j_{\sf min}} - \frac{\overline{\lambda}_{j_{\sf min}} - \overline{\lambda}_{j_{\sf max}}}{3} + \frac{\overline{\lambda}_{j_{\sf max}} - \overline{\lambda}_{r_1+1}}{2}\\
	&\leq \overline{\lambda}_{j_{\sf min}} - \frac{\overline{\lambda}_{j_{\sf min}} - \overline{\lambda}_{j_{\sf max}}}{8} - \frac{5}{8}\left(\overline{\lambda}_{j_{\sf max}} - \overline{\lambda}_{r_1+1}\right) + \frac{\overline{\lambda}_{j_{\sf max}} - \overline{\lambda}_{r_1+1}}{2}\\
	&= \overline{\lambda}_{j_{\sf min}} - \frac{\overline{\lambda}_{j_{\sf min}} - \overline{\lambda}_{r_1+1}}{8},
\end{align*}
which further tells us that
\begin{align*}
	\left|\eta - \overline{\lambda}_{j_{\sf min}}\right| \geq \frac{\overline{\lambda}_{j_{\sf min}} - \overline{\lambda}_{r_1+1}}{8},~\qquad~\forall \eta \in \gamma_4.
\end{align*}
Putting \eqref{ineq19} and the previous inequality together, one has
\begin{align}\label{ineq20}
	&\left|\frac{1}{2\pi\rm{i}}\oint_{\gamma_4}\frac{\overline{\lambda}_{j_{k+1}}\rm{d}\eta}{\left(\eta - \overline{\lambda}_{j_1}\right)\cdots\left(\eta - \overline{\lambda}_{j_{k+1}}\right)}\right|\notag\\
	&\hspace{1cm}\leq \frac{1}{2\pi}\sup_{\eta \in \gamma_4}\left|\frac{1}{\prod_{1 \leq i \leq k, i \notin \mathcal{I}_1}\left(\eta - \overline{\lambda}_{j_i}\right)}\right|\cdot \sup_{\eta \in \gamma_4}\frac{\overline{\lambda}_{j_{k+1}}}{\left|\eta - \overline{\lambda}_{j_{\sf min}}\right|}\cdot L\left(\gamma_4\right)\notag\\
	&\hspace{1cm}\leq \frac{8}{\pi}\left(\frac{2}{\overline{\lambda}_{r_1} - \overline{\lambda}_{r_1+1}}\right)^{k-1}\cdot \frac{8\overline{\lambda}_{j_{\sf min}}}{\overline{\lambda}_{j_{\sf min}} - \overline{\lambda}_{r_1+1}}\notag\\
	&\hspace{1cm}\leq \frac{32}{\pi}\left(\frac{2}{\overline{\lambda}_{r_1} - \overline{\lambda}_{r_1+1}}\right)^{k}\overline{\lambda}_{r_1}.\end{align}
Here, the second and the third lines also make use of $|\mathcal{I}_1| = 1$. Note that for all $i \in \{1, \dots, k+1\}\backslash\mathcal{I}_1$, $\overline{\lambda}_{j_{i}}$ is not in or on $\gamma_5$. By virtue of Cauchy's integral formula, one has
\begin{align}\label{eq8}
	\frac{1}{2\pi\rm{i}}\oint_{\gamma_5}\frac{\rm{d}\eta}{\left(\eta - \overline{\lambda}_{j_1}\right)\cdots\left(\eta - \overline{\lambda}_{j_{k+1}}\right)} = \frac{1}{2\pi\rm{i}}\oint_{\gamma_5}\frac{\prod_{1 \leq i \leq k+1, i \notin \mathcal{I}_1}\frac{1}{\eta - \overline{\lambda}_{j_i}}}{\eta - \overline{\lambda}_{j_{\sf min}}}{\rm{d}\eta} = \prod_{1 \leq i \leq k+1, i \notin \mathcal{I}_1}\frac{1}{\overline{\lambda}_{j_{\sf min}} - \overline{\lambda}_{j_i}}.
\end{align}
Moreover, the definition of $\mathcal{I}_1$ tells us that
\begin{align*}
	\min_{i: 1 \leq i \leq k+1, i \notin \mathcal{I}_1}\left|\overline{\lambda}_{j_{\sf min}} - \overline{\lambda}_{j_{i}}\right| &\geq \frac{\overline{\lambda}_{j_{\sf min}} - \overline{\lambda}_{j_{\sf max}}}{3}\\ &= \frac{\overline{\lambda}_{j_{\sf min}} - \overline{\lambda}_{j_{\sf max}}}{4} + \frac{\overline{\lambda}_{j_{\sf min}} - \overline{\lambda}_{j_{\sf max}}}{12}\\ &\geq \left(\frac{\overline{\lambda}_{j_{\sf min}} - \overline{\lambda}_{j_{\sf max}}}{4} + \frac{\overline{\lambda}_{j_{\sf max}} - \overline{\lambda}_{r_1+1}}{4}\right) \vee \left(\overline{\lambda}_{j_{\sf max}} - \overline{\lambda}_{r_1+1}\right)\\
	&= \frac{\overline{\lambda}_{j_{\sf min}} - \overline{\lambda}_{r_1+1}}{4} \vee \left(\overline{\lambda}_{j_{\sf max}} - \overline{\lambda}_{r_1+1}\right).
\end{align*}
Combining \eqref{eq8} and the previous inequality, one has
\begin{align}\label{ineq21}
	\left|\frac{1}{2\pi\rm{i}}\oint_{\gamma_5}\frac{\overline{\lambda}_{j_{k+1}}\rm{d}\eta}{\left(\eta - \overline{\lambda}_{j_1}\right)\cdots\left(\eta - \overline{\lambda}_{j_{k+1}}\right)}\right|
	&\leq \prod_{1 \leq i \leq k+1, i \notin \mathcal{I}_1}\frac{1}{\overline{\lambda}_{j_{\sf min}} - \overline{\lambda}_{j_i}}\overline{\lambda}_{j_{\sf min}}\notag\\
	&\leq \frac{1}{\left(\overline{\lambda}_{j_{\sf max}} - \overline{\lambda}_{r_1+1}\right)^{k-1}}\frac{4\overline{\lambda}_{j_{\sf min}}}{\overline{\lambda}_{j_{\sf min}} - \overline{\lambda}_{r_1+1}}\notag\\
	&\leq \frac{1}{\left(\overline{\lambda}_{r_1} - \overline{\lambda}_{r_1+1}\right)^{k-1}}\frac{4\overline{\lambda}_{r_1}}{\overline{\lambda}_{r_1} - \overline{\lambda}_{r_1+1}}\notag\\
	&\leq \left(\frac{2}{\overline{\lambda}_{r_1} - \overline{\lambda}_{r_1+1}}\right)^k\overline{\lambda}_{r_1}.
\end{align}
Eqn. \eqref{eq4} together with \eqref{eq6}, \eqref{ineq20} and \eqref{ineq21} implies that
\begin{align}\label{ineq22}
	\left|\frac{1}{2\pi\rm{i}}\oint_{\gamma_1}\frac{\overline{\lambda}_{j_{k+1}}\rm{d}\eta}{\left(\eta - \overline{\lambda}_{j_1}\right)\cdots\left(\eta - \overline{\lambda}_{j_{k+1}}\right)}\right| \leq \frac{36}{\pi}\left(\frac{2}{\overline{\lambda}_{r_1} - \overline{\lambda}_{r_1+1}}\right)^k\overline{\lambda}_{r_1}.
\end{align}
\paragraph{Case 2.3: $|\mathcal{I}_2| = 1$.} In this case, we define 
\begin{align*}
	\ell' \in \argmin_{i: 1 \leq i \leq k+1, i \notin \mathcal{I}_2}\overline{\lambda}_{j_i}.
\end{align*}
Denote by $\gamma_6$ and $\gamma_7$ the following counterclockwise contours:
\begin{align*} 
	\gamma_4 = \bigg\{x + y\rm{i}&: x = \frac{\overline{\lambda}_{j_{\sf max}} + \overline{\lambda}_{r_1+1}}{2}, -\frac{\overline{\lambda}_{j_{\sf max}} - \overline{\lambda}_{r_1+1}}{2} \leq y \leq \frac{\overline{\lambda}_{j_{\sf max}} - \overline{\lambda}_{r_1+1}}{2},\\
	&\quad \text{or } x = \overline{\lambda}_{j_{\ell'}} - \frac{\overline{\lambda}_{j_{\sf max}} - \overline{\lambda}_{r_1+1}}{2}, -\frac{\overline{\lambda}_{j_{\sf max}} - \overline{\lambda}_{r_1+1}}{2} \leq y \leq \frac{\overline{\lambda}_{j_{\sf max}} - \overline{\lambda}_{r_1+1}}{2},\\
	&\quad \text{or } y = \pm\frac{\overline{\lambda}_{j_{\sf max}} - \overline{\lambda}_{r_1+1}}{2}, \frac{\overline{\lambda}_{j_{\sf max}} + \overline{\lambda}_{r_1+1}}{2} \leq x \leq \overline{\lambda}_{j_{\ell'}} - \frac{\overline{\lambda}_{j_{\sf max}} - \overline{\lambda}_{r_1+1}}{2}\bigg\}
\end{align*}
and
\begin{align*} 
	\gamma_5 = \bigg\{x + y\rm{i}&: x = \overline{\lambda}_{j_{\ell'}} - \frac{\overline{\lambda}_{j_{\sf max}} - \overline{\lambda}_{r_1+1}}{2}, -\frac{\overline{\lambda}_{j_{\sf max}} - \overline{\lambda}_{r_1+1}}{2} \leq y \leq \frac{\overline{\lambda}_{j_{\sf max}} - \overline{\lambda}_{r_1+1}}{2},\\
	&\quad \text{or } x = \overline{\lambda}_{j_{\sf min}} + \frac{\overline{\lambda}_{j_{\sf max}} - \overline{\lambda}_{r_1+1}}{2}, -\frac{\overline{\lambda}_{j_{\sf max}} - \overline{\lambda}_{r_1+1}}{2} \leq y \leq \frac{\overline{\lambda}_{j_{\sf max}} - \overline{\lambda}_{r_1+1}}{2},\\
	&\quad \text{or } y = \pm\frac{\overline{\lambda}_{j_{\sf max}} - \overline{\lambda}_{r_1+1}}{2}, \overline{\lambda}_{j_{\ell'}} - \frac{\overline{\lambda}_{j_{\sf max}} - \overline{\lambda}_{r_1+1}}{2} \leq x \leq \overline{\lambda}_{j_{\sf min}} + \frac{\overline{\lambda}_{j_{\sf max}} - \overline{\lambda}_{r_1+1}}{2}\bigg\}.
\end{align*}
Similar to \eqref{eq6}, one has
\begin{align}\label{eq9}
	&\frac{1}{2\pi\rm{i}}\oint_{\gamma_3}\frac{\rm{d}\eta}{\left(\eta - \overline{\lambda}_{j_1}\right)\cdots\left(\eta - \overline{\lambda}_{j_{k+1}}\right)}\notag\\ &\hspace{1cm}= \frac{1}{2\pi\rm{i}}\oint_{\gamma_6}\frac{\rm{d}\eta}{\left(\eta - \overline{\lambda}_{j_1}\right)\cdots\left(\eta - \overline{\lambda}_{j_{k+1}}\right)} + \frac{1}{2\pi\rm{i}}\oint_{\gamma_7}\frac{\rm{d}\eta}{\left(\eta - \overline{\lambda}_{j_1}\right)\cdots\left(\eta - \overline{\lambda}_{j_{k+1}}\right)}.
\end{align}
Repeating similar arguments as in \eqref{ineq20} and \eqref{ineq21} yields
\begin{subequations}
	\begin{align}
		\left|\frac{1}{2\pi\rm{i}}\oint_{\gamma_6}\frac{\overline{\lambda}_{j_{k+1}}\rm{d}\eta}{\left(\eta - \overline{\lambda}_{j_1}\right)\cdots\left(\eta - \overline{\lambda}_{j_{k+1}}\right)}\right| &\leq \left(\frac{2}{\overline{\lambda}_{r_1} - \overline{\lambda}_{r_1+1}}\right)^k\overline{\lambda}_{r_1},\label{ineq23a},\\
		\left|\frac{1}{2\pi\rm{i}}\oint_{\gamma_7}\frac{\overline{\lambda}_{j_{k+1}}\rm{d}\eta}{\left(\eta - \overline{\lambda}_{j_1}\right)\cdots\left(\eta - \overline{\lambda}_{j_{k+1}}\right)}\right| &\leq \frac{32}{\pi}\left(\frac{2}{\overline{\lambda}_{r_1} - \overline{\lambda}_{r_1+1}}\right)^k\overline{\lambda}_{r_1}\label{ineq23b}.
	\end{align}
\end{subequations}
Putting \eqref{eq4}, \eqref{eq9}, \eqref{ineq23a} and \eqref{ineq23b} together, one has
\begin{align}\label{ineq24}
	\left|\frac{1}{2\pi\rm{i}}\oint_{\gamma_1}\frac{\overline{\lambda}_{j_{k+1}}\rm{d}\eta}{\left(\eta - \overline{\lambda}_{j_1}\right)\cdots\left(\eta - \overline{\lambda}_{j_{k+1}}\right)}\right| \leq \frac{36}{\pi}\left(\frac{2}{\overline{\lambda}_{r_1} - \overline{\lambda}_{r_1+1}}\right)^k\overline{\lambda}_{r_1}.
\end{align}

In summary, we are guaranteed to have
\begin{align*}
		\left|\frac{1}{2\pi\rm{i}}\oint_{\gamma_1}\frac{\overline{\lambda}_{j_{k+1}}\rm{d}\eta}{\left(\eta - \overline{\lambda}_{j_1}\right)\cdots\left(\eta - \overline{\lambda}_{j_{k+1}}\right)}\right| \leq \frac{40}{\pi}\left(\frac{2}{\overline{\lambda}_{r_1} - \overline{\lambda}_{r_1+1}}\right)^k\overline{\lambda}_{r_1}.
\end{align*}
This together with \eqref{eq:decomposition_2} finishes the proof of \eqref{ineq:residual_expansion}.
\qed

\subsection{Proof of Lemma \ref{lm:noise_product}}\label{proof:lm_noise_product}
Throughout the subsection, we assume that $\mathcal{E}$ holds. 
\paragraph{Proof of \eqref{ineq7a}.} \eqref{ineq7a} clearly holds for $i = 0$ due to the definition of $\mu$. Now we consider the case $i \geq 1$. It is easy to verify that for any matrices $\bm{A}, \bm{B} \in \bbR^{m_1 \times m_1}$,
\begin{align}\label{eq3}
	\left(\bm{A} + \bm{B}\right)^{i} = \bm{B}^i + \sum_{j = 0}^{i-1}\bm{B}^j\bm{A}\left(\bm{A} + \bm{B}\right)^{i-j-1}.
\end{align}
This allows us to decompose $\bm{Z}_3^i\bm{U}^\star$ as follows:
\begin{align*}
	\bm{Z}_3^i\bm{U}^\star &= \left[\mathcal{P}_{\sf off\text{-}diag}\left(\bm{E}\bm{E}^\top - \bm{E}\bm{V}^\star\bm{V}^{\star\top}\bm{E}^\top\right)\right]^i\bm{U}^\star\\
	&= -\sum_{j=0}^{i-1}\left[\mathcal{P}_{\sf off\text{-}diag}\left(\bm{E}\bm{E}^\top\right)\right]^{j}\mathcal{P}_{\sf off\text{-}diag}\left( \bm{E}\bm{V}^\star\bm{V}^{\star\top}\bm{E}^\top\right)\left[\mathcal{P}_{\sf off\text{-}diag}\left(\bm{E}\bm{E}^\top - \bm{E}\bm{V}^\star\bm{V}^{\star\top}\bm{E}^\top\right)\right]^{i-j-1}\bm{U}^\star\\
	&\quad + \left[\mathcal{P}_{\sf off\text{-}diag}\left(\bm{E}\bm{E}^\top\right)\right]^{i}\bm{U}^\star\\
	&= -\sum_{j=0}^{i-1}\left[\mathcal{P}_{\sf off\text{-}diag}\left(\bm{E}\bm{E}^\top\right)\right]^{j}\bm{E}\bm{V}^\star\bm{V}^{\star\top}\bm{E}^\top\left[\mathcal{P}_{\sf off\text{-}diag}\left(\bm{E}\bm{E}^\top - \bm{E}\bm{V}^\star\bm{V}^{\star\top}\bm{E}^\top\right)\right]^{i-j-1}\bm{U}^\star\\
	&\quad + \sum_{j=0}^{i-1}\left[\mathcal{P}_{\sf off\text{-}diag}\left(\bm{E}\bm{E}^\top\right)\right]^{j}\mathcal{P}_{\sf diag}\left( \bm{E}\bm{V}^\star\bm{V}^{\star\top}\bm{E}^\top\right)\left[\mathcal{P}_{\sf off\text{-}diag}\left(\bm{E}\bm{E}^\top - \bm{E}\bm{V}^\star\bm{V}^{\star\top}\bm{E}^\top\right)\right]^{i-j-1}\bm{U}^\star\\
	&\quad + \left[\mathcal{P}_{\sf off\text{-}diag}\left(\bm{E}\bm{E}^\top\right)\right]^{i}\bm{U}^{\star}.
\end{align*}
In view of \eqref{ineq:power_V}, \eqref{ineq:power_U}, \eqref{ineq2a} and \eqref{ineq2b}, one can obtain the following upper bound for $\|\bm{Z}^i\bm{U}^\star\|_{2, \infty}$:
\begin{align}\label{ineq8}
	&\big\|\bm{Z}_3^i\bm{U}^\star\big\|_{2,\infty}\notag\\ &\quad\leq \sum_{j=0}^{i-1}\left\|\left[\mathcal{P}_{\sf off\text{-}diag}\left(\bm{E}\bm{E}^\top\right)\right]^{j}\bm{E}\bm{V}^\star\right\|_{2,\infty}\left\|\bm{V}^{\star\top}\bm{E}^\top\left[\mathcal{P}_{\sf off\text{-}diag}\left(\bm{E}\bm{E}^\top - \bm{E}\bm{V}^\star\bm{V}^{\star\top}\bm{E}^\top\right)\right]^{i-j-1}\bm{U}^\star\right\|\notag\\
	&\qquad + \sum_{j=0}^{i-1}\left\|\left[\mathcal{P}_{\sf off\text{-}diag}\left(\bm{E}\bm{E}^\top\right)\right]^{j}\mathcal{P}_{\sf diag}\left( \bm{E}\bm{V}^\star\bm{V}^{\star\top}\bm{E}^\top\right)\left[\mathcal{P}_{\sf off\text{-}diag}\left(\bm{E}\bm{E}^\top - \bm{E}\bm{V}^\star\bm{V}^{\star\top}\bm{E}^\top\right)\right]^{i-j-1}\bm{U}^\star\right\|\notag\\
	&\qquad + \left\|\left[\mathcal{P}_{\sf off\text{-}diag}\left(\bm{E}\bm{E}^\top\right)\right]^{i}\bm{U}^{\star}\right\|_{2,\infty}\notag\\
	&\quad \leq \sum_{j=0}^{i-1}\left\|\left[\mathcal{P}_{\sf off\text{-}diag}\left(\bm{E}\bm{E}^\top\right)\right]^{j}\bm{E}\bm{V}^\star\right\|_{2,\infty}\left\|\bm{E}\bm{V}^{\star}\right\|\left\|\mathcal{P}_{\sf off\text{-}diag}\left(\bm{E}\bm{E}^\top - \bm{E}\bm{V}^\star\bm{V}^{\star\top}\bm{E}^\top\right)\right\|^{i-j-1}\notag\\
	&\qquad + \sum_{j=0}^{i-1}\left\|\mathcal{P}_{\sf off\text{-}diag}\left(\bm{E}\bm{E}^\top\right)\right\|^{j}\left\|\bm{E}\bm{V}^\star\right\|_{2,\infty}^2\left\|\mathcal{P}_{\sf off\text{-}diag}\left(\bm{E}\bm{E}^\top - \bm{E}\bm{V}^\star\bm{V}^{\star\top}\bm{E}^\top\right)\right\|^{i-j-1}\notag\\
	&\qquad + \left\|\left[\mathcal{P}_{\sf off\text{-}diag}\left(\bm{E}\bm{E}^\top\right)\right]^{i}\bm{U}^{\star}\right\|_{2,\infty}\notag\\
	&\quad \leq \sum_{j=0}^{i-1}\left(C_3\sqrt{\mu r}\left(C_3\left(\sqrt{m_1m_2} + m_1\right)\omega_{\sf max}^2\log^2 m\right)^{j}\omega_{\sf max}\log m\right)\cdot C_5\sqrt{m_1}\omega_{\sf max}\log m\notag\\
	&\hspace{1.5cm}\cdot \left(3C_5\left(\sqrt{m_1m_2} + m_1\right)\omega_{\sf max}^2\log^2 m\right)^{i-j-1}\notag\\
	&\qquad + \sum_{j=0}^{i-1}\left(C_5\left(\sqrt{m_1m_2} + m_1\right)\omega_{\sf max}^2\log^2 m\right)^j\left(C_3\sqrt{\mu r}\omega_{\sf max}\log m\right)^2\cdot \left(3C_5\left(\sqrt{m_1m_2} + m_1\right)\omega_{\sf max}^2\log^2 m\right)^{i-j-1}\notag\\
	&\qquad + C_3\sqrt{\frac{\mu r}{m_1}}\left(C_3\left(\sqrt{m_1m_2} + m_1\right)\omega_{\sf max}^2\log^2 m\right)^{i}\notag\\
	&\quad \leq 3C_3\sqrt{\frac{\mu r}{m_1}}\left(C_3\left(\sqrt{m_1m_2} + m_1\right)\omega_{\sf max}^2\log^2 m\right)^{i},
\end{align}
provided that $C_3 \geq 6C_5$.
\paragraph{Proof of \eqref{ineq7b}.} When $i = 0$, \eqref{ineq7b} is a direct consequence of Lemma \ref{lm:power_V}. For $i \geq 1$, similar to \eqref{ineq8}, one has
\begin{align}\label{ineq9}
	&\big\|\bm{Z}_3^i\bm{E}\bm{V}^\star\big\|_{2,\infty}\notag\\ &\quad\leq \sum_{j=0}^{i-1}\left\|\left[\mathcal{P}_{\sf off\text{-}diag}\left(\bm{E}\bm{E}^\top\right)\right]^{j}\bm{E}\bm{V}^\star\right\|_{2,\infty}\left\|\bm{V}^{\star\top}\bm{E}^\top\left[\mathcal{P}_{\sf off\text{-}diag}\left(\bm{E}\bm{E}^\top - \bm{E}\bm{V}^\star\bm{V}^{\star\top}\bm{E}^\top\right)\right]^{i-j-1}\bm{E}\bm{V}^\star\right\|\notag\\
	&\qquad + \sum_{j=0}^{i-1}\left\|\left[\mathcal{P}_{\sf off\text{-}diag}\left(\bm{E}\bm{E}^\top\right)\right]^{j}\mathcal{P}_{\sf diag}\left( \bm{E}\bm{V}^\star\bm{V}^{\star\top}\bm{E}^\top\right)\left[\mathcal{P}_{\sf off\text{-}diag}\left(\bm{E}\bm{E}^\top - \bm{E}\bm{V}^\star\bm{V}^{\star\top}\bm{E}^\top\right)\right]^{i-j-1}\bm{E}\bm{V}^\star\right\|\notag\\
	&\qquad + \left\|\left[\mathcal{P}_{\sf off\text{-}diag}\left(\bm{E}\bm{E}^\top\right)\right]^{i}\bm{E}\bm{V}^\star\right\|_{2,\infty}\notag\\
	&\quad \leq \sum_{j=0}^{i-1}\left\|\left[\mathcal{P}_{\sf off\text{-}diag}\left(\bm{E}\bm{E}^\top\right)\right]^{j}\bm{E}\bm{V}^\star\right\|_{2,\infty}\left\|\bm{E}\bm{V}^{\star}\right\|^2\left\|\mathcal{P}_{\sf off\text{-}diag}\left(\bm{E}\bm{E}^\top - \bm{E}\bm{V}^\star\bm{V}^{\star\top}\bm{E}^\top\right)\right\|^{i-j-1}\notag\\
	&\qquad + \sum_{j=0}^{i-1}\left\|\mathcal{P}_{\sf off\text{-}diag}\left(\bm{E}\bm{E}^\top\right)\right\|^{j}\left\|\bm{E}\bm{V}^\star\right\|_{2,\infty}^2\left\|\mathcal{P}_{\sf off\text{-}diag}\left(\bm{E}\bm{E}^\top - \bm{E}\bm{V}^\star\bm{V}^{\star\top}\bm{E}^\top\right)\right\|^{i-j-1}\left\|\bm{E}\bm{V}^{\star}\right\|\notag\\
	&\qquad + \left\|\left[\mathcal{P}_{\sf off\text{-}diag}\left(\bm{E}\bm{E}^\top\right)\right]^{i}\bm{E}\bm{V}^{\star}\right\|_{2,\infty}\notag\\
	&\quad \leq \sum_{j=0}^{i-1}\left(C_3\sqrt{\mu r}\left(C_3\left(\sqrt{m_1m_2} + m_1\right)\omega_{\sf max}^2\log^2 m\right)^{j}\omega_{\sf max}\log m\right)\cdot \left(\sqrt{C_5}\sqrt{m_1}\omega_{\sf max}\log m\right)^2\notag\\
	&\hspace{1.5cm}\cdot \left(3C_5\left(\sqrt{m_1m_2} + m_1\right)\omega_{\sf max}^2\log^2 m\right)^{i-j-1}\notag\\
	&\qquad + \sum_{j=0}^{i-1}\left(C_5\left(\sqrt{m_1m_2} + m_1\right)\omega_{\sf max}^2\log^2 m\right)^j\left(C_3\sqrt{\mu r}\omega_{\sf max}\log m\right)^2\cdot \left(3C_5\left(\sqrt{m_1m_2} + m_1\right)\omega_{\sf max}^2\log^2 m\right)^{i-j-1}\notag\\&\hspace{2cm}\cdot\sqrt{C_5}\sqrt{m_1}\omega_{\sf max}\log m\notag\\
	&\qquad + C_3\sqrt{\mu r}\left(C_3\left(\sqrt{m_1m_2} + m_1\right)\omega_{\sf max}^2\log^2 m\right)^{i}\omega_{\sf max}\log m\notag\\
	&\quad \leq 3C_3\sqrt{\mu r}\left(C_3\left(\sqrt{m_1m_2} + m_1\right)\omega_{\sf max}^2\log^2 m\right)^{i}\omega_{\sf max}\log m,
\end{align}
provided that $C_3 \geq 6C_5$. In the last inequality, we used the assumption that $m_1 \gg \mu r$.

\paragraph{Proof of \eqref{ineq7c}.} We can directly use the same argument of \citet[Eqn. (121)]{zhou2023deflated} to prove \eqref{ineq7c}. We omit the details here for the sake of brevity.

\paragraph{Proof of \eqref{ineq7d}.} Putting \eqref{ineq1}, \eqref{ineq2a}, \eqref{ineq7a} and \eqref{ineq7b} together shows that: for all $0 \leq i \leq \log m$, we have
\begin{align}\label{ineq10}
	\left\|\bm{Z}_3^i\bm{Z}_1\right\|_{2, \infty} &= \left\|\bm{Z}_3^i\left(\bm{U}^{\star(2)}\bm{\Sigma}^{\star(2)}\bm{V}^{\star(2)\top}\bm{E} + \bm{E}\bm{V}^{\star(2)}\bm{\Sigma}^{\star(2)}\bm{U}^{\star(2)\top} + \bm{E}\bm{V}^{\star(2)}\bm{V}^{\star(2)\top}\bm{E}\right)\right\|_{2, \infty}\notag\\
	&\leq \left\|\bm{Z}_3^i\bm{U}^{\star(2)}\bm{\Sigma}^{\star(2)}\bm{V}^{\star(2)\top}\bm{E}^\top\right\|_{2,\infty} + \left\|\bm{Z}_3^i\bm{E}\bm{V}^{\star(2)}\bm{\Sigma}^{\star(2)}\bm{U}^{\star(2)\top}\right\|_{2,\infty} + \left\|\bm{Z}_3^i\bm{E}\bm{V}^{\star(2)}\bm{V}^{\star(2)\top}\bm{E}^\top\right\|_{2,\infty}\notag\\
	&\leq \big\|\bm{Z}_3^i\bm{U}^{\star(2)}\big\|_{2,\infty}\big\|\bm{\Sigma}^{\star(2)}\big\|\big\|\bm{E}\bm{V}^{\star(2)}\big\| + \big\|\bm{Z}_3^i\bm{E}\bm{V}^{\star(2)}\big\|_{2,\infty}\big\|\bm{\Sigma}^{\star(2)}\big\|  + \left\|\bm{Z}_3^i\bm{E}\bm{V}^{\star(2)}\right\|_{2,\infty}\big\|\bm{E}\bm{V}^{\star(2)}\big\|\notag\\
	&\leq \sigma_{\overline{r} + 1}^\star\big\|\bm{Z}_3^i\bm{U}^{\star}\big\|_{2,\infty}\big\|\bm{E}\bm{V}^{\star}\big\| + \sigma_{\overline{r} + 1}^\star\big\|\bm{Z}_3^i\bm{E}\bm{V}^{\star}\big\|_{2,\infty}  + \left\|\bm{Z}_3^i\bm{E}\bm{V}^{\star}\right\|_{2,\infty}\big\|\bm{E}\bm{V}^{\star}\big\|\notag\\
	&\leq \sigma_{\overline{r} + 1}^\star\cdot 3C_3\sqrt{\frac{\mu r}{m_1}}\left(C_3\left(\sqrt{m_1m_2} + m_1\right)\omega_{\sf max}^2\log^2 m\right)^{i}\cdot \sqrt{C}_5\sqrt{m}_1\omega_{\sf max}\log m\notag\\
	&\quad + \sigma_{\overline{r} + 1}^\star\omega_{\sf max}\log m\cdot 3C_3\sqrt{\mu r}\left(C_3\left(\sqrt{m_1m_2} + m_1\right)\omega_{\sf max}^2\log^2 m\right)^{i}\omega_{\sf max}\log m\notag\\
	&\quad + 3C_3\sqrt{\mu r}\left(C_3\left(\sqrt{m_1m_2} + m_1\right)\omega_{\sf max}^2\log^2 m\right)^{i}\omega_{\sf max}\log m\cdot \sqrt{C}_5\sqrt{m}_1\omega_{\sf max}\log m\notag\\
	&\leq C_2\sqrt{\mu r}\left(\sigma_{\overline{r} + 1}^\star + \sqrt{m}_1\omega_{\sf max}\log m\right)\left(C_3\left(\sqrt{m_1m_2} + m_1\right)\omega_{\sf max}^2\log^2 m\right)^{i}\omega_{\sf max}\log m,
\end{align}
provided that $C_2 \geq 9C_3\sqrt{C_5}$.
The fourth line makes use of the fact that $\|\bm{B}\| \leq \|\bm{A}\|$ and $\|\bm{B}\|_{2,\infty} \leq \|\bm{A}\|_{2,\infty}$ for any $\bm{A}$ and its submatrix $\bm{B}$.
\paragraph{Proof of \eqref{ineq7e}.} By virtue of \eqref{ineq1}, \eqref{ineq7a}, \eqref{ineq7c} and \eqref{ineq6}, we know that for all $0 \leq i \leq \log m$,
\begin{align}
	\left\|\bm{Z}_3^i\bm{Z}_2\right\|_{2, \infty} &\leq \left\|\bm{Z}_3^i\left(\mathcal{P}_{\widetilde{\bm{U}}^{(1)}}\bm{U}^{\star(2)}\big(\bm{\Sigma}^{\star(2)}\big)^2\bm{U}^{\star(2)\top}\mathcal{P}_{\big(\widetilde{\bm{U}}^{(1)}\big)_{\perp}} + \bm{U}^{\star(2)}\big(\bm{\Sigma}^{\star(2)}\big)^2\bm{U}^{\star(2)\top}\mathcal{P}_{\widetilde{\bm{U}}^{(1)}}\right)\right\|_{2, \infty}\notag\\
	&\leq \big\|\bm{Z}_3^i\widetilde{\bm{U}}^{(1)}\big\|_{2,\infty}\big\|\widetilde{\bm{U}}^{(1)\top}\bm{U}^{\star(2)}\big\|\big\|\bm{\Sigma}^{\star(2)}\big\|^2 + \big\|\bm{Z}_3^i\bm{U}^{\star(2)}\big\|_{2,\infty}\big\|\bm{\Sigma}^{\star(2)}\big\|^2\big\|\widetilde{\bm{U}}^{(1)\top}\bm{U}^{\star(2)}\big\|\notag\\
	&\leq \big\|\bm{Z}_3^i\widetilde{\bm{U}}^{(1)}\big\|_{2,\infty}\big\|\widetilde{\bm{U}}^{(1)\top}\left(\bm{U}_1^\star\right)_{\perp}\big\|\sigma_{\overline{r}+1}^{\star2} + \big\|\bm{Z}_3^i\bm{U}^{\star(2)}\big\|_{2,\infty}\big\|\widetilde{\bm{U}}^{(1)\top}\left(\bm{U}_1^\star\right)_{\perp}\big\|\sigma_{\overline{r}+1}^{\star2}\notag\\
	&\lesssim \sqrt{\frac{\mu r}{m_1}}\left(C_3\left(\sqrt{m_1m_2} + m_1\right)\omega_{\sf max}^2\log^2 m\right)^i\cdot \frac{\sqrt{m}_1\omega_{\sf max}\log m}{\sigma_{\overline{r}}^\star}\sigma_{\overline{r}+1}^{\star2}\notag\\
	&\leq \sqrt{\frac{\mu r}{m_1}}\left(C_3\left(\sqrt{m_1m_2} + m_1\right)\omega_{\sf max}^2\log^2 m\right)^i\cdot \sqrt{m}_1\omega_{\sf max}\sigma_{\overline{r}+1}^\star\log m\notag\\
	&= \sqrt{\mu r}\left(C_3\left(\sqrt{m_1m_2} + m_1\right)\omega_{\sf max}^2\log^2 m\right)^{i}\omega_{\sf max}\sigma_{\overline{r}+1}^\star\log m.
\end{align}
The third line holds since $\bm{U}^{\star(1)\top}\bm{U}^{\star(2)} = \bm{0}$, 
whereas the fifth line is due to the basic fact that $\sigma_{\overline{r}}^\star \geq \sigma_{\overline{r}+1}^\star$.
\qed
\section{Proof of Theorem \ref{thm:residual}}\label{proof:thm_residual}
\subsection{Several notation}
First, we introduce some notation that will be useful throughout the proof. We let
\begin{align}\label{def:G0}
	\bm{G}_{k+1}^{0} := \bm{G}_k,~\qquad~\forall 0 \leq k \leq k_{\sf max},
\end{align}
For any $0 \leq t \leq t_{k+1}$ and $0 \leq k \leq k_{\sf max}$, we define
\begin{align}\label{def:eigendecomposition_G}
	\bm{U}_{k+1}^t\bm{\Lambda}_{k+1}^t\bm{U}_{k+1}^{t\top} := \text{ the leading } r_k \text{ eigendecomposition of } \bm{G}_{k+1}^{t},
\end{align}
and denote
\begin{align}\label{def:G_general}
	\bm{G}_{k+1}^{t+1} := \mathcal{P}_{\sf off\text{-}diag}\left(\bm{G}_{k+1}^t\right) + \mathcal{P}_{\sf diag}\left(\bm{U}_{k+1}^t\bm{\Lambda}_{k+1}^t\bm{U}_{k+1}^{t\top}\right).
\end{align}
Recall that we can decompose $\bm{M}^{\sf oracle}$ into four terms:
\begin{align}\label{M_oracle_decompose}
	\bm{M}^{\sf oracle} = \widetilde{\bm{M}} + \bm{Z}_1 + \bm{Z}_2 + \bm{Z}_3 = \widetilde{\bm{M}} + \bm{Z},
\end{align}
where $\widetilde{\bm{M}}, \bm{Z}_1, \bm{Z}_2$ and $\bm{Z}_3$ are defined in \eqref{eq:decomposition_M_oracle}.

For notational convenience, we let 
\begin{align}\label{def:others}
	D_k^t = \left\|\bm{G}_k^t - \bm{M}^{\sf oracle}\right\|,~\qquad~F_k^t = \big\|\mathcal{P}_{\sf diag}\big(\bm{G}_k^t - \widetilde{\bm{M}}\big)\big\|,~\qquad~\text{and}~\qquad~L_k^t = \big\|\bm{G}_k^t - \widetilde{\bm{M}}\big\|.
\end{align}
In addition, we let
\begin{align}\label{def:tilde_U_k}
	\widetilde{\bm{U}}_k = \widetilde{\bm{U}}_{:, 1:r_k},~\qquad~\bm{U}^{\sf oracle}_k = \bm{U}^{\sf oracle}_{:, 1:r_k},~\qquad~\text{and}~\qquad~\bm{U}^\star_k = \bm{U}^\star_{:, 1:r_k},
\end{align}
where $\widetilde{\bm{U}}$ (resp.~$\bm{U}^{\sf oracle}$) is the rank-$r$ leading eigenspace of $\widetilde{\bm{M}}$ (resp.~$\bm{M}^{\sf oracle}$).
We also let $\mathcal{E}$ denote the following event:
\begin{align}\label{event:E}
	\mathcal{E} &= \{\eqref{ineq:power_V} \text{ and }\eqref{ineq:power_U} \text{ hold for }0 \leq k \leq \log n\} \cap \{\eqref{ineq2a}, \eqref{ineq2b}, \eqref{ineq2c}\text{ and }\eqref{ineq2d}\text{ hold}\}\notag\\ &\quad \cap \{\eqref{ineq:oracle_tilde}~\text{and}~\eqref{ineq:oracle}~\text{for all }~r' \in \mathcal{A}\}.
\end{align}
We know from Lemmas \ref{lm:power_V}, \ref{lm:power_U}, \ref{lm:event_collection}, Theorem \ref{thm:oracle_two_to_infty} and the union bound that
\begin{align}\label{ineq:event_probability1}
	\bbP\left(\mathcal{E}\right) \geq 1 - O\left(m^{-10}\right).
\end{align}
Throughout the rest of the proof, we assume that $\mathcal{E}$ occurs.

\subsection{Main steps for proving Theorem \ref{thm:residual}}

\paragraph{Step 1: a key property of $r_1$ selected in Algorithm \ref{algorithm:sequential_heteroPCA}.} First, we show that
\begin{align}\label{property:r_1}
	r_1 \in \mathcal{R}_1 \cap \mathcal{A},
\end{align}
where
\begin{align}\label{def:R_1}
	\mathcal{R}_1 := \left\{r' \leq r: \frac{\sigma_1\left(\bm{G}_0\right)}{\sigma_{r'}\left(\bm{G}_0\right)} \leq 4 \quad\text{and}\quad \sigma_{r'}\left(\bm{G}_0\right) - \sigma_{r' + 1}\left(\bm{G}_0\right) \geq \frac{1}{r}\sigma_{r'}\left(\bm{G}_0\right)\right\}.
\end{align}
and $\mathcal{A}$ is defined in \eqref{def:A}. Noting that $\widetilde{\bm{U}} = [\widetilde{\bm{U}}^{(1)}\ \widetilde{\bm{U}}^{(2)}]$ and putting \eqref{ineq2d} and \eqref{ineq:incoherence_tilde_U_2} together, one has
\begin{align}\label{ineq:incoherence_tilde_U}
	\big\|\widetilde{\bm{U}}\big\|_{2,\infty} = \big\|\big[\widetilde{\bm{U}}^{(1)}\ \widetilde{\bm{U}}^{(2)}\big]\big\|_{2,\infty} \leq \big\|\widetilde{\bm{U}}^{(1)}\big\|_{2,\infty} + \big\|\widetilde{\bm{U}}^{(2)}\big\|_{2,\infty} \leq 4\sqrt{\frac{\mu r}{m_1}}.
\end{align}
In view of \eqref{ineq:spectral_Z}, \eqref{ineq:incoherence_tilde_U} and the definition $\bm{G}_1^0 = \bm{G}_0 = \mathcal{P}_{\sf off\text{-}diag}(\bm{M}) = \mathcal{P}_{\sf off\text{-}diag}(\bm{M}^{\sf oracle}) = \mathcal{P}_{\sf off\text{-}diag}(\widetilde{\bm{U}}\widetilde{\bm{\Lambda}}\widetilde{\bm{U}}^\top + \bm{Z})$, we can derive
\begin{align}\label{ineq:L_1^0}
	L_1^0 &= \big\|\bm{G}_0 - \widetilde{\bm{M}}\big\|\notag\\ &= \big\|\mathcal{P}_{\sf diag}\big(\widetilde{\bm{U}}\widetilde{\bm{\Lambda}}\widetilde{\bm{U}}^\top\big) - \mathcal{P}_{\sf off\text{-}diag}\left(\bm{Z}\right)\big\|\notag\\
	&\leq \big\|\mathcal{P}_{\sf diag}\big(\widetilde{\bm{U}}\widetilde{\bm{\Lambda}}\widetilde{\bm{U}}^\top\big)\big\| + \left\|\mathcal{P}_{\sf off\text{-}diag}\left(\bm{Z}\right)\right\|\notag\\
	&\leq \big\|\widetilde{\bm{U}}\big\|_{2,\infty}^2\big\|\widetilde{\bm{\Lambda}}\big\| + 2\left\|\bm{Z}\right\|\notag\\
	&\leq 16\frac{\mu r}{m_1}\widetilde{\sigma}_1^2 + C_2\left(\sqrt{m_1}\omega_{\sf max}\log m\cdot\sigma_{\overline{r}+1}^{\star} + \left(\sqrt{m_1m_2} + m_1\right)\omega_{\sf max}^2\log^2 m\right).
\end{align}
This together with Weyl's inequality shows that, for all $i \in [m_1]$,
\begin{align}\label{ineq116}
	\big|\sigma_i\left(\bm{G}_0\right) - \widetilde{\sigma}_i^2\big| &= \big|\sigma_i\left(\bm{G}_0\right) - \sigma_{i}\big(\widetilde{\bm{M}}\big)\big|\notag\\
	&\leq 16\frac{\mu r}{m_1}\widetilde{\sigma}_1^2 + C_2\left(\sqrt{m_1}\omega_{\sf max}\log m\cdot\sigma_{\overline{r}+1}^{\star} + \left(\sqrt{m_1m_2} + m_1\right)\omega_{\sf max}^2\log^2 m\right).
\end{align}
Moreover, \eqref{ineq2a} tells us that
\begin{align*}
	\widetilde{\sigma}_i \leq \sigma_i^\star + \sqrt{C_5}\sqrt{m_1}\omega_{\sf max}\log m \leq \left(1 + \frac{1}{Cr^2}\right)\sigma_i^\star,~\qquad~\forall i \in \left[\overline{r}\right]
\end{align*}
for some large constant $C > 0$, where $\overline{r} = \max\mathcal{A}$ and the last inequality holds since $\sigma_{\overline{r}}^\star \geq C_0r\big[(m_1m_2)^{1/4} + rm_1^{1/2}\big]\omega_{\sf max}\log m$. Similarly, one can show that
\begin{align}\label{ineq115}
	\left(1 - \frac{1}{Cr^2}\right)\sigma_i^\star \leq \widetilde{\sigma}_i \leq \left(1 + \frac{1}{Cr^2}\right)\sigma_i^\star,~\qquad~\forall i \in \left[\overline{r}\right].
\end{align}
By virtue of \eqref{ineq115} and the assumption $\sigma_{\overline{r}}^\star \geq C_0r\big[(m_1m_2)^{1/4} + rm_1^{1/2}\big]\omega_{\sf max}\log m$, one has
\begin{align}\label{ineq120}
	&16\frac{\mu r}{m_1}\widetilde{\sigma}_1^2 + C_2\left(\sqrt{m_1}\omega_{\sf max}\log m\cdot\sigma_{\overline{r}+1}^{\star} + \left(\sqrt{m_1m_2} + m_1\right)\omega_{\sf max}^2\log^2 m\right)\notag\\
	&\hspace{1cm} \leq 16\frac{\mu r}{m_1}\widetilde{\sigma}_1^2 + C_2\left(\sqrt{m_1}\omega_{\sf max}\log m\cdot\sigma_{\overline{r}+1}^{\star} + \frac{1}{C_0^2r^2}\sigma_{\overline{r}}^{\star2}\right)\notag\\
	&\hspace{1cm} \leq 16\frac{\mu r}{m_1}\widetilde{\sigma}_1^2 + C_2\left(2\sqrt{m_1}\omega_{\sf max}\log m\cdot\widetilde{\sigma}_{\overline{r}+1} + \frac{4}{C_0^2r^2}\widetilde{\sigma}_{\overline{r}}^{2}\right)\notag\\
	&\hspace{1cm} \leq 16\frac{\mu r}{m_1}\widetilde{\sigma}_1^2 + C_2\left(\frac{2}{C_0r^2}\widetilde{\sigma}_{\overline{r}}^2 + \frac{4}{C_0^2r^2}\widetilde{\sigma}_{\overline{r}}^2\right)\notag\\
	&\hspace{1cm} \leq \frac{16c_1}{r^2}\widetilde{\sigma}_1^2 + \frac{1}{2Cr^2}\widetilde{\sigma}_{\overline{r}}^2\notag\\
	&\hspace{1cm} \leq \frac{1}{Cr^2}\widetilde{\sigma}_1^2,
\end{align}
provided that $C_0 \geq 8C\cdot C_2$ and $c_1 \leq \frac{1}{32C}$. Here, the fourth line comes from $\sigma_{\overline{r}}^\star \geq C_0r\big[(m_1m_2)^{1/4} + rm_1^{1/2}\big]\omega_{\sf max}\log m$ and \eqref{ineq115}, whereas the penultimate line uses the assumption $\mu \leq c_0m_1/r^3$. Combining \eqref{ineq116} and \eqref{ineq120} gives  
\begin{align}\label{ineq117}
	\sigma_1\left(\bm{G}_0\right) &\geq \widetilde{\sigma}_1^2 - \frac{1}{Cr^2}\widetilde{\sigma}_1^2 > \max\left\{\frac{1}{2}\widetilde{\sigma}_1^2, \frac{1}{2}\sigma_1^{\star2}, \tau\right\},
\end{align}
where the last inequality holds due to \eqref{ineq115} and $C_1^2/2 \geq C_{\tau}$.

Equipped with \eqref{ineq116}, \eqref{ineq115}, \eqref{ineq120} and \eqref{ineq117}, 
we can establish the property $r_1 \in \mathcal{R}_1$ using a similar argument as in the proof of \citet[Eqn. (62)]{zhou2023deflated}. Therefore, we only need to prove $r_1 \in \mathcal{A}$. In view of \eqref{ineq115}, \eqref{ineq117} and the definition of $\mathcal{R}_1$, we know that
\begin{align}\label{ineq118}
	\sigma_{r_1}\left(\bm{G}_0\right) \geq \frac{1}{4}\sigma_1\left(\bm{G}_0\right) \geq \frac{1}{8}\sigma_1^{\star2},
\end{align}
and consequently, 
\begin{align}\label{ineq119}
	 \sigma_{r_1}^{\star} &\geq \left(1 - \frac{C}{r^2}\right)\widetilde{\sigma}_{r_1} \geq \left(1 - \frac{C}{r^2}\right)\left[\sigma_{r_1}\left(\bm{G}_0\right) - \frac{1}{Cr^2}\widetilde{\sigma}_1^2\right]^{1/2} \geq \left(1 - \frac{C}{r^2}\right)\left(\frac{1}{8}\sigma_1^{\star2} - \frac{4}{Cr^2}\sigma_1^{\star2}\right)^{1/2}\notag\\& \geq \frac{1}{3}\sigma_1^{\star} > 2C_0r\big[(m_1m_2)^{1/4} + rm_1^{1/2}\big]\omega_{\sf max}\log m.
\end{align}
Furthermore, inequality \eqref{ineq116}, \eqref{ineq120} and \eqref{ineq118} combined imply that
\begin{align}\label{ineq122}
	\widetilde{\sigma}_{r_1}^2 \geq \sigma_{r_1}\left(\bm{G}_0\right) - \frac{1}{Cr^2}\widetilde{\sigma}_1^2 \geq \frac{1}{4}\sigma_{1}\left(\bm{G}_0\right) - \frac{1}{Cr^2}\widetilde{\sigma}_1^2 \geq \frac{1}{4}\left(\widetilde{\sigma}_1^2 - \frac{1}{Cr^2}\widetilde{\sigma}_1^2\right) - \frac{1}{Cr^2}\widetilde{\sigma}_1^2 \geq \frac{1}{5}\widetilde{\sigma}_1^2.
\end{align}
Inequalities \eqref{ineq116}, \eqref{ineq120}, \eqref{ineq122} and the triangle inequality taken together show that
\begin{align}\label{ineq121}
	\widetilde{\sigma}_{r_1}^2 - \widetilde{\sigma}_{r_1+1}^2 &\geq \sigma_{r_1}\left(\bm{G}_0\right) - \sigma_{r_1+1}\left(\bm{G}_0\right) - \big|\sigma_{r_1}\left(\bm{G}_0\right) - \widetilde{\sigma}_{r_1}^2\big| - \big|\sigma_{r_1+1}\left(\bm{G}_0\right) - \widetilde{\sigma}_{r_1+1}^2\big|\notag\\
	&\geq \frac{1}{r}\sigma_{r_1}\left(\bm{G}_0\right) - \frac{2}{Cr^2}\widetilde{\sigma}_{1}^2\notag\\
	&\geq \frac{1}{r}\widetilde{\sigma}_{r_1}^2 - \frac{1}{r}\big|\sigma_{r_1}\left(\bm{G}_0\right) - \widetilde{\sigma}_{r_1}^2\big| - \frac{2}{Cr^2}\widetilde{\sigma}_{1}^2\notag\\
	&\geq \frac{1}{r}\widetilde{\sigma}_{r_1}^2 - \frac{3}{Cr^2}\widetilde{\sigma}_{1}^2\notag\\
	&\geq \frac{1}{r}\widetilde{\sigma}_{r_1}^2 - \frac{15}{Cr^2}\widetilde{\sigma}_{r_1}^2\notag\\
	&\geq \frac{9}{10r}\widetilde{\sigma}_{r_1}^2.
\end{align}
Note that \eqref{ineq1} together with \eqref{ineq119} reveals that $r_1 \leq \overline{r}$, where $\overline{r}$ is the largest element in $\mathcal{A}$. Putting \eqref{ineq2a}, \eqref{ineq128}, \eqref{ineq115} and the previous inequality together, we arrive at 
\begin{align}\label{ineq123}
	\sigma_{r_1}^\star - \sigma_{r_1 + 1}^\star &\geq \widetilde{\sigma}_{r_1} - \widetilde{\sigma}_{r_1+1} - \left|\widetilde{\sigma}_{r_1} - \sigma_{r_1}^\star\right| - \left|\widetilde{\sigma}_{r_1+1} - \sigma_{r_1+1}^\star\right|\notag\\
	&\geq \frac{1}{\widetilde{\sigma}_{r_1} + \widetilde{\sigma}_{r_1+1}}\big(\widetilde{\sigma}_{r_1}^2 - \widetilde{\sigma}_{r_1+1}^2\big) - \sqrt{C_5}\sqrt{m_1}\omega_{\sf max}\log m - 2\sqrt{C_5}\sqrt{m_1}\omega_{\sf max}\log m\notag\\
	&\geq \frac{1}{2\widetilde{\sigma}_{r_1}}\cdot \frac{9}{10r}\widetilde{\sigma}_{r_1}^2 - \frac{1}{20r^2}\sigma_{r_1}^\star\notag\\
	&\geq \frac{9}{20r}\left(1 - \frac{1}{Cr}\right)\sigma_{r_1}^\star - \frac{1}{20r^2}\sigma_{r_1}^\star\notag\\
	&\geq \frac{1}{4r}\sigma_{r_1}^\star.
\end{align}
Inequality~\eqref{ineq119} taken together with \eqref{ineq123} validates $r_1 \in \mathcal{A}$, thus finishing the proof of \eqref{property:r_1}.

\paragraph{Step 2: bounding $D_1^t = \|\bm{G}_1^t - \bm{M}^{\sf oracle}\|$.} Now, we would like to deal with the quantities $\{D_1^t\}$. Recognizing that for all $t$ and $k$,
\begin{align*}
	\mathcal{P}_{\sf off\text{-}diag}\left(\bm{G}_k^t\right) = \mathcal{P}_{\sf off\text{-}diag}\left(\bm{M}^{\sf oracle}\right) =  \mathcal{P}_{\sf off\text{-}diag}\left(\bm{Y}\bm{Y}^\top\right),
\end{align*}
we can deduce that
\begin{align}\label{eq11}
	D_k^t = \left\|\mathcal{P}_{\sf diag}\left(\bm{G}_k^t - \bm{M}^{\sf oracle}\right)\right\|.
\end{align}
We would like to prove the following inequalities by induction:
\begin{subequations}
	\begin{align}
		F_1^t - 40\sqrt{\frac{\mu r^3}{m_1}}\left\|\bm{Z}\right\| - 20\sqrt{\frac{\mu r}{m_1}}\widetilde{\sigma}_{r_1 + 1}^2 &\leq \frac{1}{e^t}\left(F_1^0 - 40\sqrt{\frac{\mu r^3}{m_1}}\left\|\bm{Z}\right\| - 20\sqrt{\frac{\mu r}{m_1}}\widetilde{\sigma}_{r_1 + 1}^2\right),\label{ineq:induction_1}\\
		D_1^t &\leq F_1^t + 6C_3\sqrt{\frac{\mu r}{m_1}}\sqrt{m}_1\omega_{\sf max}\log m\cdot\sigma_{\overline{r} + 1}^{\star} + C_3^2\mu r\omega_{\sf max}^2\log^2 m,\label{ineq:induction_2}\\
		\left\|\bm{U}_1^t\bm{U}_1^{t\top} - \bm{U}_1^{\sf oracle}\bm{U}_1^{\sf oracle\top}\right\| &\leq 2\frac{D_1^t}{\lambda_{r_1}\left(\bm{M}^{\sf oracle}\right) - \lambda_{r_1+1}\left(\bm{M}^{\sf oracle}\right)} \leq \frac{1}{8},\label{ineq:induction_3}\\
		\left\|\bm{U}_1^t\right\|_{2,\infty} &\leq \left\|\bm{U}_1^t\bm{U}_1^{t\top} - \bm{U}_1^{\sf oracle}\bm{U}_1^{\sf oracle\top}\right\| + \left\|\bm{U}_1^{\sf oracle}\right\|_{2,\infty} \leq \frac{1}{4e}.\label{ineq:induction_4}
	\end{align}
\end{subequations}
\paragraph{Step 2.1: the base case ($t = 0$) for \eqref{ineq:induction_1}-\eqref{ineq:induction_4}.} Note that \eqref{ineq:induction_1} automatically holds when $t = 0$. Recalling that $\mathcal{P}_{\sf diag}(\bm{G}_1^0) = 0$, 
we can invoke \citet[Lemma 1]{zhang2022heteroskedastic} together with \eqref{ineq:incoherence_tilde_U} to obtain
\begin{align}\label{ineq:F_1^0}
	F_1^0 = \big\|\mathcal{P}_{\sf diag}\big(\widetilde{\bm{M}}\big)\big\| = \big\|\mathcal{P}_{\sf diag}\big(\widetilde{\bm{U}}\widetilde{\bm{\Lambda}}\widetilde{\bm{U}}^\top\big)\big\| \leq \big\|\widetilde{\bm{U}}\big\|_{2,\infty}^2\big\|\widetilde{\bm{\Lambda}}\big\| \leq 16\frac{\mu r}{m_1}\widetilde{\sigma}_1^2.
\end{align}
Furthermore, putting Lemma \ref{lm:power_V}, \eqref{eq:decomposition_M_oracle}, \eqref{ineq2d}, \eqref{ineq6} and \eqref{ineq:F_1^0} together, we have
\begin{align}\label{ineq:D_1^0}
	D_1^0 &= \left\|\mathcal{P}_{\sf diag}\left(\bm{M}^{\sf oracle}\right)\right\| = \big\|\mathcal{P}_{\sf diag}\big(\widetilde{\bm{M}} + \bm{Z}_1 + \bm{Z}_2 + \bm{Z}_3\big)\big\|\notag\\
	&\leq \big\|\mathcal{P}_{\sf diag}\big(\widetilde{\bm{M}}\big)\big\| + \left\|\mathcal{P}_{\sf diag}\left(\bm{Z}_1\right)\right\| + \left\|\mathcal{P}_{\sf diag}\left(\bm{Z}_2\right)\right\| + \left\|\mathcal{P}_{\sf diag}\left(\bm{Z}_3\right)\right\|\notag\\
	&\leq F_1^0+ \left\|\mathcal{P}_{\sf off-diag}\left(\bm{U}^{\star(2)}\bm{\Sigma}^{\star(2)}\bm{V}^{\star(2)\top}\bm{E}^\top + \bm{E}\bm{V}^{\star(2)}\bm{\Sigma}^{\star(2)}\bm{U}^{\star(2)\top} + \bm{E}\bm{V}^{\star(2)}\bm{V}^{\star(2)\top}\bm{E}^\top\right)\right\|\notag\\&\quad + 2\left\|\mathcal{P}_{\sf diag}\left(\mathcal{P}_{\widetilde{\bm{U}}^{(1)}}\bm{U}^{\star(2)}\big(\bm{\Sigma}^{\star(2)}\big)^2\bm{U}^{\star(2)\top}\mathcal{P}_{\big(\widetilde{\bm{U}}^{(1)}\big)_{\perp}}\right)\right\| + 0\notag\\
	&\leq F_1^0 + 2\big\|\bm{U}^{\star(2)}\big\|_{2,\infty}\big\|\bm{E}\bm{V}^{\star(2)}\big\|_{2,\infty}\big\|\bm{\Sigma}^{\star(2)}\big\| + \big\|\bm{E}\bm{V}^{\star(2)}\big\|_{2,\infty}^2 + 2\big\|\widetilde{\bm{U}}^{(1)}\big\|_{2,\infty}\big\|\widetilde{\bm{U}}^{(1)\top}\bm{U}^{\star(2)}\big\|\big\|\bm{\Sigma}^{\star(2)}\big\|^2\notag\\
	&\leq F_1^0 + 2\sqrt{\frac{\mu r}{m_1}}\cdot C_3\sqrt{\mu r}\omega_{\sf max}\log m\cdot \sigma_{\overline{r} + 1}^{\star} + \left(C_3\sqrt{\mu r}\omega_{\sf max}\log m\right)^2 + 4\sqrt{\frac{\mu r}{m_1}}\cdot C_3\frac{\sqrt{m}_1\omega_{\sf max}\log m}{\sigma_{\overline{r}}^\star}\sigma_{\overline{r} + 1}^{\star2}\notag\\
	&\leq F_1^0 + 6C_3\sqrt{\frac{\mu r}{m_1}}\sqrt{m}_1\omega_{\sf max}\log m\cdot\sigma_{\overline{r} + 1}^{\star} + C_3^2\mu r\omega_{\sf max}^2\log^2 m,
\end{align}
which validates \eqref{ineq:induction_2} for $t = 0$. Here, the last line holds since $\mu \leq c_0m_1/r^3$ and $\sigma_{\overline{r}}^\star \geq \sigma_{\overline{r} + 1}^{\star}$. Combining \eqref{ineq119}, \eqref{ineq122}, \eqref{ineq121} and \eqref{ineq123}, one has
\begin{align}\label{ineq124}
	\max\left\{\frac{\sigma_1^\star}{\sigma_{r_1}^\star}, \frac{\widetilde{\sigma}_1}{\widetilde{\sigma}_{r_1}}\right\} \leq 3~\qquad~\text{and}~\qquad~\min\left\{\frac{\widetilde{\sigma}_{r_1}^2 - \widetilde{\sigma}_{r_1+1}^2}{\widetilde{\sigma}_{r_1}^2}, \frac{\sigma_{r_1}^{\star2} - \sigma_{r_1+1}^{\star2}}{\sigma_{r_1}^{\star2}}\right\} \geq \frac{1}{4r}.
\end{align}
Moreover, by virtue of \eqref{ineq2a}, \eqref{ineq128} and the fact $r_1 \in \mathcal{A}$, we know that
\begin{align*}
	\big|\big(\widetilde{\sigma}_{r_1} - \widetilde{\sigma}_{r_1 + 1}\big) - \big(\sigma_{r_1}^\star - \sigma_{r_1 + 1}^\star\big)\big| &\leq \big|\widetilde{\sigma}_{r_1} - \sigma_{r_1}^\star\big| + \big|\widetilde{\sigma}_{r_1+1} - \sigma_{r_1+1}^\star\big|
	\leq \sqrt{C}_5\sqrt{m_1}\omega_{\sf max}\log m + 2\sqrt{C_5}\sqrt{m_1}\omega_{\sf max}\log m\\
	&\ll \frac{1}{r^2}\sigma_{r_1}^\star \lesssim \frac{1}{r}\left(\sigma_{r_1}^\star - \sigma_{r_1 + 1}^\star\right),
\end{align*}
where the last inequality comes from \eqref{ineq123}. This implies that
\begin{align*}
	\left(1 - \frac{1}{Cr}\right)\big(\sigma_{r_1}^\star - \sigma_{r_1 + 1}^\star\big) \leq \widetilde{\sigma}_{r_1} - \widetilde{\sigma}_{r_1 + 1} \leq \left(1 + \frac{1}{Cr}\right)\big(\sigma_{r_1}^\star - \sigma_{r_1 + 1}^\star\big).
\end{align*}
The previous inequality together with \eqref{ineq115}, \eqref{ineq:spectral_Z} and \eqref{ineq124} gives 
\begin{align*}
	\widetilde{\sigma}_{r_1}^2 - \widetilde{\sigma}_{r_1 + 1}^2 \asymp \sigma_{r_1}^{\star2} - \sigma_{r_1 + 1}^{\star2} \gg \left\|\bm{Z}\right\|.
\end{align*}
Recalling that $\bm{M}^{\sf oracle} = \widetilde{\bm{M}} + \bm{Z}$, one can invoke Weyl's inequality to obtain
\begin{align}\label{ineq125}
	\lambda_{r_1}\left(\bm{M}^{\sf oracle}\right) - \lambda_{r_1+1}\left(\bm{M}^{\sf oracle}\right) \asymp \widetilde{\sigma}_{r_1}^2 - \widetilde{\sigma}_{r_1 + 1}^2 \asymp \sigma_{r_1}^{\star2} - \sigma_{r_1 + 1}^{\star2} \gg \left\|\bm{Z}\right\|.
\end{align}
Note that $\bm{U}_1^0$ (resp.~$\bm{U}_1^{\sf oracle}$) is the rank-$r$ leading eigenspace of $\bm{G}_0$ (resp.~$\bm{M}^{\sf oracle}$). We know from the Davis-Kahan Theorem \cite[Theorem 2.7]{chen2021spectral}, \eqref{ineq:D_1^0}, \eqref{ineq124} and \eqref{ineq125} that 
\begin{align}\label{ineq126}
	\left\|\bm{U}_1^0\bm{U}_1^{0\top} - \bm{U}_1^{\sf oracle}\bm{U}_1^{\sf oracle\top}\right\| &\leq 2\frac{\left\|\bm{G}_0 - \bm{M}^{\sf oracle}\right\|}{\lambda_{r_1}\left(\bm{M}^{\sf oracle}\right) - \lambda_{r_1+1}\left(\bm{M}^{\sf oracle}\right)} = 2\frac{D_1^0}{\lambda_{r_1}\left(\bm{M}^{\sf oracle}\right) - \lambda_{r_1+1}\left(\bm{M}^{\sf oracle}\right)}\notag\\
	&\stackrel{\eqref{ineq125}}{\lesssim} \frac{\frac{\mu r}{m_1}\widetilde{\sigma}_1^2}{\widetilde{\sigma}_{r_1}^2 - \widetilde{\sigma}_{r_1 + 1}^2} + \frac{\sqrt{\frac{\mu r}{m_1}}\sqrt{m}_1\omega_{\sf max}\log m\cdot\sigma_{\overline{r} + 1}^{\star}}{\sigma_{r_1}^{\star2} - \sigma_{r_1 + 1}^{\star2}} + \frac{\mu r\omega_{\sf max}^2\log^2 m}{\sigma_{r_1}^{\star2} - \sigma_{r_1 + 1}^{\star2}}\notag\\
	&\stackrel{\eqref{ineq124}}{\lesssim} \frac{\frac{\mu r}{m_1}\widetilde{\sigma}_1^2}{\widetilde{\sigma}_{r_1}^2/r} + \frac{\sqrt{\frac{\mu r}{m_1}}\sqrt{m}_1\omega_{\sf max}\log m\cdot\sigma_{\overline{r} + 1}^{\star}}{\sigma_{r_1}^{\star2}/r} + \frac{\mu r\omega_{\sf max}^2\log^2 m}{\sigma_{r_1}^{\star2}/r}\notag\\
	&\stackrel{\eqref{ineq124}}{\lesssim} \frac{\mu r^2}{m_1} + \frac{\sqrt{\frac{\mu r}{m_1}}r\sqrt{m}_1\omega_{\sf max}\log m}{\sigma_{r_1}^{\star}} + \frac{\mu r\omega_{\sf max}^2\log^2 m}{\sigma_{r_1}^{\star2}/r}\notag\\
	&\ll \sqrt{\frac{\mu r}{m_1}} \leq \frac{1}{8},
\end{align}
which proves \eqref{ineq:induction_3} for $t = 0$. Here, the last inequality is due to $\mu \leq c_0m_1/r^3$ and $\sigma_{r_1}^\star \geq \sigma_{\overline{r}}^\star \geq C_0r\big[(m_1m_2)^{1/4} + rm_1^{1/2}\big]\omega_{\sf max}\log m$. Inequality \eqref{ineq:oracle} and the fact $r_1 \in \mathcal{A}$ together imply that
\begin{align}\label{ineq129}
	\left\|\bm{U}_1^{\sf oracle}\right\|_{2,\infty} = \left\|\bm{U}_1^{\sf oracle}\bm{U}_1^{\sf oracle\top}\right\|_{2,\infty} \leq \left\|\bm{U}_1^{\sf oracle}\bm{U}_1^{\sf oracle\top} - \bm{U}_1^\star\bm{U}_1^{\star\top}\right\|_{2,\infty} + \|\bm{U}_1^\star\|_{2,\infty} \leq 2\sqrt{\frac{\mu r^3}{m_1}}.
\end{align}
Combining \eqref{ineq126} and \eqref{ineq129}, one can further obtain that
\begin{align}\label{ineq130}
	\left\|\bm{U}_1^0\right\|_{2,\infty} = \left\|\bm{U}_1^0\bm{U}_1^{0\top}\right\|_{2,\infty} \leq \left\|\bm{U}_1^0\bm{U}_1^{0\top} - \bm{U}_1^{\sf oracle}\bm{U}_1^{\sf oracle\top}\right\| + \left\|\bm{U}_1^{\sf oracle}\right\|_{2,\infty} \leq 3\sqrt{\frac{\mu r^3}{m_1}} \leq \frac{1}{4e},
\end{align}
i.e., \eqref{ineq:induction_4} holds for $t = 0$.

\paragraph{Step 2.2: induction step ($t > 0$) for \eqref{ineq:induction_1}-\eqref{ineq:induction_4}.} Suppose that \eqref{ineq:induction_1}-\eqref{ineq:induction_4} hold for $t = t'$. We aim to show that they continue to hold for $t = t'+1$.

Recognizing that $\bm{U}_1^{t'}$ is the top-$r_{1}$ singular space of
\begin{align*}
	\bm{G}_1^{t'} = \mathcal{P}_{\widetilde{\bm{U}}_1}\widetilde{\bm{M}} + \left(\bm{G}_1^{t'} - \mathcal{P}_{\widetilde{\bm{U}}_1}\widetilde{\bm{M}}\right),
\end{align*}
we have
\begin{align}\label{ineq131}
	F_1^{t'+1} &= \big\|\mathcal{P}_{\sf diag}\big(\bm{G}_1^{t'+1} - \widetilde{\bm{M}}\big)\big\|\notag\\
	&= \left\|\mathcal{P}_{\sf diag}\big(\mathcal{P}_{\bm{U}_1^{t'}}\bm{G}_1^{t'} - \widetilde{\bm{M}}\big)\right\|\notag\\
	&\leq \left\|\mathcal{P}_{\sf diag}\left(\mathcal{P}_{\bm{U}_1^{t'}}\big(\bm{G}_1^{t'} - \widetilde{\bm{M}}\big)\right)\right\| + \left\|\mathcal{P}_{\sf diag}\left(\mathcal{P}_{\left(\bm{U}_1^{t'}\right)_{\perp}}\widetilde{\bm{M}}\mathcal{P}_{\widetilde{\bm{U}}}\right)\right\|\notag\\
	&\leq \left\|\bm{U}_1^{t'}\right\|_{2,\infty}\big\|\bm{G}_1^{t'} - \widetilde{\bm{M}}\big\| + \big\|\widetilde{\bm{U}}\big\|_{2,\infty}\left\|\big(\bm{U}_1^{t'}\big)_{\perp}\widetilde{\bm{M}}\right\|\notag\\
	&\leq \left\|\bm{U}_1^{t'}\right\|_{2,\infty}L_1^{t'} + 4\sqrt{\frac{\mu r}{m_1}}\left(\left\|\big(\bm{U}_1^{t'}\big)_{\perp}\mathcal{P}_{\widetilde{\bm{U}}_1}\widetilde{\bm{M}}\right\| + \left\|\mathcal{P}_{\widetilde{\bm{U}}_{:, r_1+1:r}}\widetilde{\bm{M}}\right\|\right)\notag\\
	&\leq \left\|\bm{U}_1^{t'}\right\|_{2,\infty}L_1^{t'} + 4\sqrt{\frac{\mu r}{m_1}}\left(2\left\|\bm{G}_1^{t'} - \mathcal{P}_{\widetilde{\bm{U}}_1}\widetilde{\bm{M}}\right\| + \left\|\mathcal{P}_{\widetilde{\bm{U}}_{:, r_1+1:r}}\widetilde{\bm{M}}\right\|\right)\notag\\
	&\leq \left\|\bm{U}_1^{t'}\right\|_{2,\infty}L_1^{t'} + 4\sqrt{\frac{\mu r}{m_1}}\left(2\left\|\bm{G}_1^{t'} - \widetilde{\bm{M}}\right\| + 3\left\|\mathcal{P}_{\widetilde{\bm{U}}_{:, r_1+1:r}}\widetilde{\bm{M}}\right\|\right)\notag\\
	&\leq \left(\left\|\bm{U}_1^{t'}\right\|_{2,\infty} + 8\sqrt{\frac{\mu r}{m_1}}\right)L_1^{t'} + 12\sqrt{\frac{\mu r}{m_1}}\widetilde{\sigma}_{r_1 + 1}^2.
\end{align}
Here, the second line holds since $\mathcal{P}_{\sf diag}(\bm{G}_1^{t'+1}) = \mathcal{P}_{\sf diag}(\bm{U}_1^{t'}\bm{\Lambda}_1^{t'}\bm{U}_1^{t'\top}) = \mathcal{P}_{\sf diag}(\mathcal{P}_{\bm{U}_1^{t'}}\bm{G}_1^{t'})$; the fourth line comes from \citet[Lemma 1]{zhang2022heteroskedastic}; the fifth line makes use of \eqref{ineq:incoherence_tilde_U}; the sixth line applies \citet[Lemma 8]{zhou2023deflated}; and the penultimate line invokes the triangle inequality. Note that
\begin{align*}
	L_1^{t'} &\leq F_1^{t'} + \big\|\mathcal{P}_{\sf off\text{-}diag}\big(\bm{G}_1^{t'+1} - \widetilde{\bm{M}}\big)\big\| = F_1^{t'} + \big\|\mathcal{P}_{\sf off\text{-}diag}\big(\bm{M}^{\sf oracle} - \widetilde{\bm{M}}\big)\big\|\\ &\leq F_1^{t'} + \left\|\mathcal{P}_{\sf off\text{-}diag}\left(\bm{Z}\right)\right\| \leq F_1^{t'} + \left\|\bm{Z}\right\| + \left\|\mathcal{P}_{\sf diag}\left(\bm{Z}\right)\right\| \leq F_1^{t'} + 2\left\|\bm{Z}\right\|.
\end{align*}
Inequality \eqref{ineq131} taken together with the previous inequality leads us to
\begin{align}\label{ineq132}
	F_1^{t'+1} &\leq \left(\left\|\bm{U}_1^{t'}\right\|_{2,\infty} + 8\sqrt{\frac{\mu r}{m_1}}\right)F_1^{t'} + 2\left(\left\|\bm{U}_1^{t'}\right\|_{2,\infty} + 8\sqrt{\frac{\mu r}{m_1}}\right)\left\|\bm{Z}\right\| + 12\sqrt{\frac{\mu r}{m_1}}\widetilde{\sigma}_{r_1 + 1}^2\notag\\
	&\stackrel{\eqref{ineq:induction_4}}{\leq} \left(\frac{1}{4e} + \frac{1}{4e}\right)F_1^{t'} + 2\left(\left\|\bm{U}_1^t\bm{U}_1^{t\top} - \bm{U}_1^{\sf oracle}\bm{U}_1^{\sf oracle\top}\right\| + \left\|\bm{U}_1^{\sf oracle}\right\|_{2,\infty} + 8\sqrt{\frac{\mu r}{m_1}}\right)\left\|\bm{Z}\right\| + 12\sqrt{\frac{\mu r}{m_1}}\widetilde{\sigma}_{r_1 + 1}^2\notag\\
	&\stackrel{\eqref{ineq:induction_3}~\text{and}~\eqref{ineq129}}{\leq} \frac{1}{2e}F_1^{t'} + 2\left(2\frac{D_1^{t'}}{\lambda_{r_1}\left(\bm{M}^{\sf oracle}\right) - \lambda_{r_1+1}\left(\bm{M}^{\sf oracle}\right)} + 2\sqrt{\frac{\mu r^3}{m_1}} +  8\sqrt{\frac{\mu r}{m_1}}\right)\left\|\bm{Z}\right\| + 12\sqrt{\frac{\mu r}{m_1}}\widetilde{\sigma}_{r_1 + 1}^2\notag\\
	&\stackrel{\eqref{ineq:induction_2}}{\leq} \frac{1}{2e}F_1^{t'} + 2\left(2\frac{F_1^{t'}}{\lambda_{r_1}\left(\bm{M}^{\sf oracle}\right) - \lambda_{r_1+1}\left(\bm{M}^{\sf oracle}\right)} + 10\sqrt{\frac{\mu r^3}{m_1}}\right)\left\|\bm{Z}\right\| + 12\sqrt{\frac{\mu r}{m_1}}\widetilde{\sigma}_{r_1 + 1}^2\notag\\
	&\hspace{1cm} + 4\frac{6C_3\sqrt{\frac{\mu r}{m_1}}\sqrt{m}_1\omega_{\sf max}\log m\cdot\sigma_{\overline{r} + 1}^{\star} + C_3^2\mu r\omega_{\sf max}^2\log^2 m}{\lambda_{r_1}\left(\bm{M}^{\sf oracle}\right) - \lambda_{r_1+1}\left(\bm{M}^{\sf oracle}\right)}\left\|\bm{Z}\right\|\notag\\
	&\stackrel{\eqref{ineq124}~\text{and}~\eqref{ineq125}}{\leq} \frac{1}{e}F_1^{t'} + 20\sqrt{\frac{\mu r^3}{m_1}}\left\|\bm{Z}\right\| + 12\sqrt{\frac{\mu r}{m_1}}\widetilde{\sigma}_{r_1 + 1}^2 + C_3^3\frac{\sqrt{\frac{\mu r}{m_1}}\sqrt{m}_1\omega_{\sf max}\log m\cdot\sigma_{\overline{r} + 1}^{\star} + \mu r\omega_{\sf max}^2\log^2 m}{\sigma_{r_1}^{\star2}/r}\left\|\bm{Z}\right\|\notag\\
	&\leq \frac{1}{e}F_1^{t'} + 21\sqrt{\frac{\mu r^3}{m_1}}\left\|\bm{Z}\right\| + 12\sqrt{\frac{\mu r}{m_1}}\widetilde{\sigma}_{r_1 + 1}^2,
\end{align}
where the last inequality is a consequence of $\sigma_{r_1}^{\star} \geq \sigma_{\overline{r}}^\star \geq C_0r\big[(m_1m_2)^{1/4} + rm_1^{1/2}\big]\omega_{\sf max}\log m$. Then one immediately has
\begin{align*}
	F_1^{t'+1} - 40\sqrt{\frac{\mu r^3}{m_1}}\left\|\bm{Z}\right\| - 20\sqrt{\frac{\mu r}{m_1}}\widetilde{\sigma}_{r_1 + 1}^2 &\leq \frac{1}{e}\left(F_1^{t'} - 40\sqrt{\frac{\mu r^3}{m_1}}\left\|\bm{Z}\right\| - 20\sqrt{\frac{\mu r}{m_1}}\widetilde{\sigma}_{r_1 + 1}^2\right)\\
	&\leq \frac{1}{e^{t'+1}}\left(F_1^{0} - 40\sqrt{\frac{\mu r^3}{m_1}}\left\|\bm{Z}\right\| - 20\sqrt{\frac{\mu r}{m_1}}\widetilde{\sigma}_{r_1 + 1}^2\right),
\end{align*}
which confirms that \eqref{ineq:induction_1} holds for $t = t'+1$.

In addition, we can prove \eqref{ineq:induction_2} for $t = t'+1$ by using the same argument as in \eqref{ineq:D_1^0}. Combining \eqref{ineq:induction_1}, \eqref{ineq:induction_2} for $t = t'+1$ and Weyl's inequality, we further have
\begin{align*}
	&\left\|\bm{U}_1^{t'+1}\bm{U}_1^{t'+1\top} - \bm{U}_1^{\sf oracle}\bm{U}_1^{\sf oracle\top}\right\|\\ &\quad\leq 2\frac{D_1^{t'+1}}{\lambda_{r_1}\left(\bm{M}^{\sf oracle}\right) - \lambda_{r_1+1}\left(\bm{M}^{\sf oracle}\right)}\\
	&\quad\leq 2\frac{F_1^{t'+1}}{\lambda_{r_1}\left(\bm{M}^{\sf oracle}\right) - \lambda_{r_1+1}\left(\bm{M}^{\sf oracle}\right)} + 2\frac{6C_3\sqrt{\frac{\mu r}{m_1}}\sqrt{m}_1\omega_{\sf max}\log m\cdot\sigma_{\overline{r} + 1}^{\star} + C_3^2\mu r\omega_{\sf max}^2\log^2 m}{\lambda_{r_1}\left(\bm{M}^{\sf oracle}\right) - \lambda_{r_1+1}\left(\bm{M}^{\sf oracle}\right)}\\
	&\quad= 2\frac{F_1^{t'+1} - 40\sqrt{\frac{\mu r^3}{m_1}}\left\|\bm{Z}\right\| - 20\sqrt{\frac{\mu r}{m_1}}\widetilde{\sigma}_{r_1 + 1}^2}{\lambda_{r_1}\left(\bm{M}^{\sf oracle}\right) - \lambda_{r_1+1}\left(\bm{M}^{\sf oracle}\right)} + 2\frac{6C_3\sqrt{\frac{\mu r}{m_1}}\sqrt{m}_1\omega_{\sf max}\log m\cdot\sigma_{\overline{r} + 1}^{\star} + C_3^2\mu r\omega_{\sf max}^2\log^2 m}{\lambda_{r_1}\left(\bm{M}^{\sf oracle}\right) - \lambda_{r_1+1}\left(\bm{M}^{\sf oracle}\right)}\\
	&\qquad + \frac{80\sqrt{\frac{\mu r^3}{m_1}}\left\|\bm{Z}\right\| + 40\sqrt{\frac{\mu r}{m_1}}\widetilde{\sigma}_{r_1 + 1}^2}{\lambda_{r_1}\left(\bm{M}^{\sf oracle}\right) - \lambda_{r_1+1}\left(\bm{M}^{\sf oracle}\right)}\\
	&\quad\stackrel{\eqref{ineq125}}{\lesssim} \frac{1}{e^{t'+1}}\frac{F_1^{0}}{\lambda_{r_1}\left(\bm{M}^{\sf oracle}\right) - \lambda_{r_1+1}\left(\bm{M}^{\sf oracle}\right)} + \frac{\sqrt{\frac{\mu r^3}{m_1}}\left\|\bm{Z}\right\| + \sqrt{\frac{\mu r}{m_1}}\widetilde{\sigma}_{r_1 + 1}^2}{\widetilde{\sigma}_{r_1}^2 - \widetilde{\sigma}_{r_1 + 1}^2}\\&\qquad2\frac{6C_3\sqrt{\frac{\mu r}{m_1}}\sqrt{m}_1\omega_{\sf max}\log m\cdot\sigma_{\overline{r} + 1}^{\star} + C_3^2\mu r\omega_{\sf max}^2\log^2 m}{\sigma_{r_1}^{\star2} - \sigma_{r_1 + 1}^{\star2}}\\
	&\quad\stackrel{\eqref{ineq:F_1^0}, \eqref{ineq124}~\text{and}~\eqref{ineq125}}{\ll} \sqrt{\frac{\mu r^3}{m_1}} \ll \frac{1}{8},
\end{align*}
which validates \eqref{ineq:induction_3} for $t = t'+1$.

Putting the previous inequality and \eqref{ineq129} together, one can prove that \eqref{ineq:induction_4} also holds for $t = t'+1$:
\begin{align*}
		\left\|\bm{U}_1^{t'+1}\right\|_{2,\infty} = \left\|\bm{U}_1^{t'+1}\bm{U}_1^{t'+1\top}\right\|_{2,\infty} \leq \left\|\bm{U}_1^{t'+1}\bm{U}_1^{t'+1\top} - \bm{U}_1^{\sf oracle}\bm{U}_1^{\sf oracle\top}\right\| + \left\|\bm{U}_1^{\sf oracle}\right\|_{2,\infty} \leq 3\sqrt{\frac{\mu r^3}{m_1}} \leq \frac{1}{4e}.
\end{align*}
Therefore, we have completed the proof of the induction step for \eqref{ineq:induction_1} - \eqref{ineq:induction_4}.

\paragraph{Step 3: bounding $D_k^t$ for $k > 1$} After establishing upper bounds on $\{D_1^t\}$, we now turn attention to the quantities $\{D_k^t\}_{k > 1}$. By setting $$t_1 \geq \log\left(C\frac{\sigma_1^{\star2}}{\sigma_{r_1+1}^{\star2} + \omega_{\sf max}^2}\right),$$
we can show that
\begin{align}\label{ineq133}
	F_2^0 = F_1^{t_1} &\stackrel{\eqref{ineq:induction_1}}{\leq} 40\sqrt{\frac{\mu r^3}{m_1}}\left\|\bm{Z}\right\| + 20\sqrt{\frac{\mu r}{m_1}}\widetilde{\sigma}_{r_1 + 1}^2 + \frac{1}{e^{t_1}}F_1^0\notag\\
	&\stackrel{\eqref{ineq:D_1^0}}{\leq} 40\sqrt{\frac{\mu r^3}{m_1}}\left\|\bm{Z}\right\| + 20\sqrt{\frac{\mu r}{m_1}}\widetilde{\sigma}_{r_1 + 1}^2 + \frac{\sigma_{r_1+1}^{\star2} + \omega_{\sf max}^2}{C\sigma_1^{\star2}}\cdot 16\frac{\mu r}{m_1}\widetilde{\sigma}_1^2\notag\\
	&\stackrel{\eqref{ineq:spectral_Z}~\text{and}~\eqref{ineq124}}{\leq} 40C_2\sqrt{\frac{\mu r^3}{m_1}}\left(\sqrt{m_1}\omega_{\sf max}\log m\cdot \sigma_{\overline{r}+1}^{\star} + \left(\sqrt{m_1m_2} + m_1\right)\omega_{\sf max}^2\log^2 m\right)\notag\\
	&\hspace{2cm} + 20\sqrt{\frac{\mu r}{m_1}}\widetilde{\sigma}_{r_1 + 1}^2 + \frac{1}{C}\sqrt{\frac{\mu r}{m_1}}\left(\omega_{\sf max}^2 + \sigma_{r_1+1}^{\star2}\right)\notag\\
	&\leq 41C_2\sqrt{\frac{\mu r^3}{m_1}}\left(\sqrt{m_1}\omega_{\sf max}\log m\cdot \sigma_{\overline{r}+1}^{\star} + \left(\sqrt{m_1m_2} + m_1\right)\omega_{\sf max}^2\log^2 m\right) + 21\sqrt{\frac{\mu r}{m_1}}\widetilde{\sigma}_{r_1 + 1}^2.
\end{align}
Here, the last line makes use of the following inequality:
\begin{align*}
	\sigma_{r_1+1}^{\star2} \stackrel{\eqref{ineq2a}~\text{and}~\eqref{ineq128}}{\leq} \left(\widetilde{\sigma}_{r_1 + 1} + 2\sqrt{C_5}\sqrt{m_1}\omega_{\sf max}\log m\right)^2
	\stackrel{\text{Cauchy-Schwarz}}{\leq} 2\widetilde{\sigma}_{r_1 + 1}^2 + 8C_5m_1\omega_{\sf max}^2\log^2 m.
\end{align*}
We define
\begin{align}\label{def:R_k}
	\mathcal{R}_k &:= \left\{r': \frac{\sigma_{r_{k-1}+1}\left(\bm{G}_0\right)}{\sigma_{r'}\left(\bm{G}_0\right)} \leq 4 \quad\text{and}\quad \sigma_{r'}\left(\bm{G}_0\right) - \sigma_{r' + 1}\left(\bm{G}_0\right) \geq \frac{1}{r}\sigma_{r'}\left(\bm{G}_0\right)\right\},
\end{align}
Choosing the numbers of iterations $\{t_i\}$ as in \eqref{iter1_matrix} and \eqref{iter2_matrix} and repeating similar arguments as in \eqref{property:r_1}, \eqref{ineq:induction_1} - \eqref{ineq:induction_4}, \eqref{ineq124}, \eqref{ineq125} and \eqref{ineq133}, we know that for all $1 \leq k \leq k_{\sf max}$ and $1 \leq t \leq t_k$, 
\begin{subequations}
	\begin{align}
		r_k &\in \mathcal{R}_k \cap \mathcal{A},\label{ineq:induction_general_1}\\
		\max\left\{\frac{\sigma_{r_{k-1}+1}^\star}{\sigma_{r_k}^\star}, \frac{\widetilde{\sigma}_{r_{k-1}+1}}{\widetilde{\sigma}_{r_k}}\right\} &\leq 3~\quad~\text{and}~\quad~\min\left\{\frac{\widetilde{\sigma}_{r_k}^2 - \widetilde{\sigma}_{r_k+1}^2}{\widetilde{\sigma}_{r_k}^2}, \frac{\sigma_{r_k}^{\star2} - \sigma_{r_k+1}^{\star2}}{\sigma_{r_k}^{\star2}}\right\} \geq \frac{1}{4r},\label{ineq:induction_general_2}\\
		\lambda_{r_k}\left(\bm{M}^{\sf oracle}\right) - \lambda_{r_k+1}\left(\bm{M}^{\sf oracle}\right) &\asymp \widetilde{\sigma}_{r_k}^2 - \widetilde{\sigma}_{r_k + 1}^2 \asymp \sigma_{r_k}^{\star2} - \sigma_{r_k + 1}^{\star2} \gg \left\|\bm{Z}\right\|,\label{ineq:induction_general_3}\\
		F_k^t - 40\sqrt{\frac{\mu r^3}{m_1}}\left\|\bm{Z}\right\| - 20\sqrt{\frac{\mu r}{m_1}}\widetilde{\sigma}_{r_k + 1}^2 &\leq \frac{1}{e^t}\left(F_k^0 - 40\sqrt{\frac{\mu r^3}{m_1}}\left\|\bm{Z}\right\| - 20\sqrt{\frac{\mu r}{m_1}}\widetilde{\sigma}_{r_k + 1}^2\right),\label{ineq:induction_general_4}\\
		D_k^t &\leq F_k^t + 6C_3\sqrt{\frac{\mu r}{m_1}}\sqrt{m}_1\omega_{\sf max}\log m\cdot\sigma_{\overline{r} + 1}^{\star} + C_3^2\mu r\omega_{\sf max}^2\log^2 m,\label{ineq:induction_general_5}\\
		\left\|\bm{U}_k^t\bm{U}_k^{t\top} - \bm{U}_k^{\sf oracle}\bm{U}_k^{\sf oracle\top}\right\| &\leq 2\frac{D_k^t}{\lambda_{r_k}\left(\bm{M}^{\sf oracle}\right) - \lambda_{r_k+1}\left(\bm{M}^{\sf oracle}\right)} \leq \frac{1}{8},\label{ineq:induction_general_6}\\
		\left\|\bm{U}_k^t\right\|_{2,\infty} &\leq \left\|\bm{U}_k^t\bm{U}_k^{t\top} - \bm{U}_k^{\sf oracle}\bm{U}_k^{\sf oracle\top}\right\| + \left\|\bm{U}_k^{\sf oracle}\right\|_{2,\infty} \leq \frac{1}{4e}.\label{ineq:induction_general_7}\\
		F_{k+1}^0 = F_k^{t_k} \leq 41C_2\sqrt{\frac{\mu r^3}{m_1}}\big(\sqrt{m_1}\omega_{\sf max}&\log m\cdot \sigma_{\overline{r}+1}^{\star} + \left(\sqrt{m_1m_2} + m_1\right)\omega_{\sf max}^2\log^2 m\big) + 21\sqrt{\frac{\mu r}{m_1}}\widetilde{\sigma}_{r_k + 1}^2.\label{ineq:induction_general_8}
	\end{align}
\end{subequations}
Taking $k = k_{\sf max}$ in \eqref{ineq:induction_general_8} yields that
\begin{align}\label{ineq134}
	F_{k_{\sf max}}^{t_{k_{\sf max}}} &\leq 41C_2\sqrt{\frac{\mu r^3}{m_1}}\big(\sqrt{m_1}\omega_{\sf max}\log m\cdot \sigma_{\overline{r}+1}^{\star} + \left(\sqrt{m_1m_2} + m_1\right)\omega_{\sf max}^2\log^2 m\big) + 21\sqrt{\frac{\mu r}{m_1}}\widetilde{\sigma}_{r_{k_{\sf max}} + 1}^2.
\end{align}
This together with \eqref{ineq:induction_general_5} implies that
\begin{align}\label{ineq136}
	D_{k_{\sf max}}^{t_{k_{\sf max}}} \leq 42C_2\sqrt{\frac{\mu r^3}{m_1}}\big(\sqrt{m_1}\omega_{\sf max}\log m\cdot \sigma_{\overline{r}+1}^{\star} + \left(\sqrt{m_1m_2} + m_1\right)\omega_{\sf max}^2\log^2 m\big) + 21\sqrt{\frac{\mu r}{m_1}}\widetilde{\sigma}_{r_{k_{\sf max}} + 1}^2.
\end{align}
Recall that $r_{k_{\sf max}}$ satisfies $r_{k_{\sf max}} = r$ or $\sigma_{r_{k_{\sf max}} + 1}(\bm{G}_{k_{\sf max}}) = \sigma_{r_{k_{\sf max}} + 1}(\bm{G}_{k_{\sf max}}^{t_{k_{\sf max}}}) \leq \tau$. 
\begin{itemize}[leftmargin=*]
	\item[1.] If $r_{k_{\sf max}} = r$, then it follows that $$\widetilde{\sigma}_{r_{k_{\sf max}} + 1}^2 = \widetilde{\sigma}_{r + 1}^2 = 0.$$ 
	\item[2.] If $\sigma_{r_{k_{\sf max}} + 1}(\bm{G}_{k_{\sf max}}) \leq \tau$, then Weyl's inequality and \eqref{ineq134} together show that
	\begin{align*}
		\widetilde{\sigma}_{r_{k_{\sf max}} + 1}^2 &= \sigma_{r_{k_{\sf max}} + 1}\big(\widetilde{\bm{M}}\big) \leq \sigma_{r_{k_{\sf max}} + 1}\left(\bm{G}_{k_{\sf max}}^{t_{k_{\sf max}}}\right) + \left\|\widetilde{\bm{M}} - \bm{G}_{k_{\sf max}}^{t_{k_{\sf max}}}\right\|\notag\\
		&= \sigma_{r_{k_{\sf max}} + 1}\left(\bm{G}_{k_{\sf max}}^{t_{k_{\sf max}}}\right) + F_{k_{\sf max}}^{t_{k_{\sf max}}}\notag\\
		&\leq \tau + 41C_2\sqrt{\frac{\mu r^3}{m_1}}\big(\sqrt{m_1}\omega_{\sf max}\log m\cdot \sigma_{\overline{r}+1}^{\star} + \left(\sqrt{m_1m_2} + m_1\right)\omega_{\sf max}^2\log^2 m\big) + 21\sqrt{\frac{\mu r}{m_1}}\widetilde{\sigma}_{r_{k_{\sf max}} + 1}^2\notag\\
		&\leq \tau + 41C_2\sqrt{\frac{\mu r^3}{m_1}}\big(\sqrt{m_1}\omega_{\sf max}\log m\cdot \sigma_{\overline{r}+1}^{\star} + \left(\sqrt{m_1m_2} + m_1\right)\omega_{\sf max}^2\log^2 m\big) + \frac{1}{2}\widetilde{\sigma}_{r_{k_{\sf max}} + 1}^2,
	\end{align*}
	which further gives 
	\begin{align}\label{ineq135}
		\widetilde{\sigma}_{r_{k_{\sf max}} + 1}^2 \leq 2\tau + 82C_2\sqrt{\frac{\mu r^3}{m_1}}\big(\sqrt{m_1}\omega_{\sf max}\log m\cdot \sigma_{\overline{r}+1}^{\star} + \left(\sqrt{m_1m_2} + m_1\right)\omega_{\sf max}^2\log^2 m\big) \leq 3\tau.
	\end{align}
\end{itemize}
Therefore, inequality \eqref{ineq135} is guaranteed to hold.

\paragraph{Step 4: proving \eqref{ineq:two_to_infty_crude}.}
We know from \eqref{ineq:induction_general_1} that $r_{k_{\sf max}} \in \mathcal{A}$. Also, \eqref{ineq:oracle} tell us that
\begin{align*}
	\big\|\bm{U}_{k_{\sf max}}^{\sf oracle}\bm{U}_{k_{\sf max}}^{\sf oracle\top} - \bm{U}_{k_{\sf max}}^{\star}\bm{U}_{k_{\sf max}}^{\star\top}\big\|_{2,\infty} &\lesssim \sqrt{\frac{\mu r^3}{m_1}}\left(\frac{r^2\sqrt{m_1}\omega_{\sf max}\log m}{\sigma_{r_{k_{\sf max}}}^\star} + \frac{r^2 \sqrt{m_1m_2}\omega_{\sf max}^2\log^2 m}{\sigma_{r_{k_{\sf max}}}^{\star2}}\right) \leq \sqrt{\frac{\mu r^3}{m_1}}.
\end{align*}
In view of \eqref{ineq136}, \eqref{ineq:induction_general_2}, \eqref{ineq:induction_general_3} and \eqref{ineq:induction_general_7}, we can demonstrate that
\begin{align}\label{ineq143}
	&\left\|\bm{U}_{k_{\sf max}}\bm{U}_{k_{\sf max}}^\top - \bm{U}_{k_{\sf max}}^{\sf oracle}\bm{U}_{k_{\sf max}}^{{\sf oracle}\top}\right\|\notag\\ &\quad\lesssim \frac{D_{k_{\sf max}}^{t_{k_{\sf max}}}}{\lambda_{r_{k_{\sf max}}}\left(\bm{M}^{\sf oracle}\right) - \lambda_{r_{k_{\sf max}}+1}\left(\bm{M}^{\sf oracle}\right)}\notag\\
	&\quad \lesssim \frac{\sqrt{\frac{\mu r^3}{m_1}}\big(\sqrt{m_1}\omega_{\sf max}\log m\cdot \sigma_{\overline{r}+1}^{\star} + \left(\sqrt{m_1m_2} + m_1\right)\omega_{\sf max}^2\log^2 m\big)}{\sigma_{r_{k_{\sf max}}}^{\star2} - \sigma_{r_{k_{\sf max}} + 1}^{\star2}} + \frac{\sqrt{\frac{\mu r}{m_1}}\widetilde{\sigma}_{r_{k_{\sf max}} + 1}^2}{\widetilde{\sigma}_{r_{k_{\sf max}}}^2 - \widetilde{\sigma}_{r_{k_{\sf max}} + 1}^2}\notag\\
	&\quad \lesssim \frac{\sqrt{\frac{\mu r^3}{m_1}}r\big(\sqrt{m_1}\omega_{\sf max}\log m\cdot \sigma_{\overline{r}+1}^{\star} + \left(\sqrt{m_1m_2} + m_1\right)\omega_{\sf max}^2\log^2 m\big)}{\sigma_{r_{k_{\sf max}}}^{\star2}} + \frac{\sqrt{\frac{\mu r^3}{m_1}}\widetilde{\sigma}_{r_{k_{\sf max}} + 1}^2}{\widetilde{\sigma}_{r_{k_{\sf max}}}^2}\notag\\
	&\quad \lesssim \sqrt{\frac{\mu r^3}{m_1}}.
\end{align}
Here, the last inequality holds since $\widetilde{\sigma}_{r_{k_{\sf max}}} \geq \widetilde{\sigma}_{r_{k_{\sf max}} + 1}$ and $\sigma_{r_{k_{\sf max}}}^{\star} \geq \sigma_{\overline{r}}^{\star} \geq C_0r\big[(m_1m_2)^{1/4} + rm_1^{1/2}\big]\omega_{\sf max}\log m$.
Combining the previous two inequalities yields
\begin{align}\label{ineq200}
	\left\|\bm{U}_{k_{\sf max}}\bm{U}_{k_{\sf max}}^\top - \bm{U}_{k_{\sf max}}^{\star}\bm{U}_{k_{\sf max}}^{\star\top}\right\|_{2, \infty} &\leq \big\|\bm{U}_{k_{\sf max}}^{\sf oracle}\bm{U}_{k_{\sf max}}^{\sf oracle\top} - \bm{U}_{k_{\sf max}}^{\star}\bm{U}_{k_{\sf max}}^{\star\top}\big\|_{2,\infty} + \left\|\bm{U}_{k_{\sf max}}\bm{U}_{k_{\sf max}}^\top - \bm{U}_{k_{\sf max}}^{\sf oracle}\bm{U}_{k_{\sf max}}^{{\sf oracle}\top}\right\|\notag\\
	&\lesssim \sqrt{\frac{\mu r^3}{m_1}},
\end{align}
which validates \eqref{ineq:two_to_infty_crude}.

\paragraph{Step 5: proving \eqref{ineq:two_to_infty_residual_matrix1} and \eqref{ineq:two_to_infty_residual_matrix2}.} 
Note that
\begin{align}\label{ineq:decomposition}
	&\big\|\big(\bm{U}_{k_{\sf max}}\bm{U}_{k_{\sf max}}^\top - \bm{U}_{k_{\sf max}}^\star\bm{U}_{k_{\sf max}}^{\star\top}\big)\bm{X}^\star\big\|_{2,\infty}\notag\\
	&\quad = \big\|\big(\bm{U}_{k_{\sf max}}\bm{U}_{k_{\sf max}}^\top - \bm{U}_{k_{\sf max}}^\star\bm{U}_{k_{\sf max}}^{\star\top}\big)\big(\bm{U}^{\star(1)}\bm{\Sigma}^{\star(1)}\bm{V}^{\star(1)\top} + \bm{U}^{\star(2)}\bm{\Sigma}^{\star(2)}\bm{V}^{\star(2)\top}\big)\big\|_{2,\infty}\notag\\
	&\quad = \big\|\big(\bm{U}_{k_{\sf max}}\bm{U}_{k_{\sf max}}^\top - \bm{U}_{k_{\sf max}}^\star\bm{U}_{k_{\sf max}}^{\star\top}\big)\bm{U}^{\star(1)}\bm{\Sigma}^{\star(1)}\bm{V}^{\star(1)\top}\big\|_{2,\infty}\notag\\
	&\qquad + \big\|\big(\bm{U}_{k_{\sf max}}\bm{U}_{k_{\sf max}}^\top - \bm{U}_{k_{\sf max}}^\star\bm{U}_{k_{\sf max}}^{\star\top}\big)\bm{U}^{\star(2)}\bm{\Sigma}^{\star(2)}\bm{V}^{\star(2)\top}\big\|_{2,\infty}\notag\\
	&\quad \leq \underbrace{\big\|\big(\bm{U}_{k_{\sf max}}\bm{U}_{k_{\sf max}}^\top - \widetilde{\bm{U}}_{k_{\sf max}}\widetilde{\bm{U}}_{k_{\sf max}}^{\top}\big)\bm{U}^{\star(1)}\bm{\Sigma}^{\star(1)}\bm{V}^{\star(1)\top}\big\|_{2,\infty}}_{=:\alpha_1}\notag\\
	&\qquad + \underbrace{\big\|\big(\widetilde{\bm{U}}_{k_{\sf max}}\widetilde{\bm{U}}_{k_{\sf max}}^{\top} - \bm{U}_{k_{\sf max}}^\star\bm{U}_{k_{\sf max}}^{\star\top}\big)\bm{U}^{\star(1)}\bm{\Sigma}^{\star(1)}\bm{V}^{\star(1)\top}\big\|_{2,\infty}}_{=:\alpha_2}\notag\\
	&\qquad + \underbrace{\big\|\bm{U}_{k_{\sf max}}\bm{U}_{k_{\sf max}}^\top - \bm{U}_{k_{\sf max}}^\star\bm{U}_{k_{\sf max}}^{\star\top}\big\|_{2,\infty}\sigma_{\overline{r} + 1}^\star}_{=:\alpha_3},
\end{align}
where the last line holds due to $\|\bm{U}^{\star(2)}\bm{\Sigma}^{\star(2)}\bm{V}^{\star(2)\top}\| = \|\bm{\Sigma}^{\star(2)}\| = \sigma_{\overline{r} + 1}^\star$. To control $\big\|\big(\bm{U}_{k_{\sf max}}\bm{U}_{k_{\sf max}}^\top - \bm{U}_{k_{\sf max}}^\star\bm{U}_{k_{\sf max}}^{\star\top}\big)\bm{X}^\star\big\|_{2,\infty}$, one only needs to bound $\alpha_1$, $\alpha_2$ and $\alpha_3$, respectively.
\paragraph{Step 5.1: bounding $\alpha_1$.} We first bound $\alpha_1 = \|(\bm{U}_{k_{\sf max}}\bm{U}_{k_{\sf max}}^\top - \widetilde{\bm{U}}_{k_{\sf max}}\widetilde{\bm{U}}_{k_{\sf max}}^{\top})\bm{U}^{\star(1)}\bm{\Sigma}^{\star(1)}\bm{V}^{\star(1)\top}\|_{2,\infty}$.
\paragraph{Step 5.1.1: bounding $\|(\bm{U}_{k_{\sf max}}\bm{U}_{k_{\sf max}}^\top - \bm{U}_{k_{\sf max}}^{\sf oracle}\bm{U}_{k_{\sf max}}^{{\sf oracle}\top})\widetilde{\bm{M}}\|_{2, \infty}$.} 
By virtue of Weyl's inequality and \eqref{ineq:spectral_Z}, one has
\begin{align}\label{ineq145}
	\lambda_{r_{k_{\sf max}}}\left(\bm{M}^{\sf oracle}\right) \leq \widetilde{\sigma}_{r_{k_{\sf max}}}^2 + \left\|\bm{Z}\right\| \asymp \widetilde{\sigma}_{r_{k_{\sf max}}}^2 \asymp \sigma_{r_{k_{\sf max}}}^{\star2}.
\end{align}
Recognizing that $\bm{G}_{k_{\sf max}}^{t_{k_{\sf max}}} = \bm{M}^{\sf oracle} + (\bm{G}_{k_{\sf max}}^{t_{k_{\sf max}}} - \bm{M}^{\sf oracle})$ and $\bm{U}_{k_{\sf max}}$ (resp.~$\bm{U}_{k_{\sf max}}^{\sf oracle}$) is the rank-$r_{k_{\sf max}}$ leading singular subspace of $\bm{G}_{k_{\sf max}}^{t_{k_{\sf max}}}$ (resp.~$\bm{M}^{\sf oracle}$), we invoke Lemma \ref{lm:1} to yield
\begin{align}\label{ineq144}
	&\left\|\left(\bm{U}_{k_{\sf max}}\bm{U}_{k_{\sf max}}^\top - \bm{U}_{k_{\sf max}}^{\sf oracle}\bm{U}_{k_{\sf max}}^{{\sf oracle}\top}\right)\bm{M}^{\sf oracle}\right\|_{2,\infty}\notag\\
	&\quad \leq \left\|\left(\bm{U}_{k_{\sf max}}\bm{U}_{k_{\sf max}}^\top - \bm{U}_{k_{\sf max}}^{\sf oracle}\bm{U}_{k_{\sf max}}^{{\sf oracle}\top}\right)\bm{M}^{\sf oracle}\right\|\notag\\
	&\quad \lesssim \frac{\lambda_{r_{k_{\sf max}}}\left(\bm{M}^{\sf oracle}\right)}{\lambda_{r_{k_{\sf max}}}\left(\bm{M}^{\sf oracle}\right) - \lambda_{r_{k_{\sf max}}+1}\left(\bm{M}^{\sf oracle}\right)}\big\|\bm{G}_{k_{\sf max}}^{t_{k_{\sf max}}} - \bm{M}^{\sf oracle}\big\|\notag\\
	&\quad \stackrel{\eqref{ineq:induction_general_3}~\text{and}~\eqref{ineq145}}{\lesssim} \frac{\sigma_{r_{k_{\sf max}}}^{\star2}}{\sigma_{r_{k_{\sf max}}}^{\star2} - \sigma_{r_{k_{\sf max}} + 1}^{\star2}}D_{k_{\sf max}}^{t_{k_{\sf max}}}\notag\\
	&\quad \stackrel{\eqref{ineq:induction_general_2}~\text{and}~\eqref{ineq136}}{\lesssim} r\left[\sqrt{\frac{\mu r^3}{m_1}}\big(\sqrt{m_1}\omega_{\sf max}\log m\cdot \sigma_{\overline{r}+1}^{\star} + \left(\sqrt{m_1m_2} + m_1\right)\omega_{\sf max}^2\log^2 m\big) + \sqrt{\frac{\mu r}{m_1}}\widetilde{\sigma}_{r_{k_{\sf max}} + 1}^2\right].
\end{align}
Combining \eqref{ineq:spectral_Z}, \eqref{ineq143} and \eqref{ineq144} leads to
\begin{align}\label{ineq146}
	&\left\|\left(\bm{U}_{k_{\sf max}}\bm{U}_{k_{\sf max}}^\top - \bm{U}_{k_{\sf max}}^{\sf oracle}\bm{U}_{k_{\sf max}}^{{\sf oracle}\top}\right)\widetilde{\bm{M}}\right\|_{2,\infty}\notag\\
	&\quad \leq \left\|\left(\bm{U}_{k_{\sf max}}\bm{U}_{k_{\sf max}}^\top - \bm{U}_{k_{\sf max}}^{\sf oracle}\bm{U}_{k_{\sf max}}^{{\sf oracle}\top}\right)\bm{M}^{\sf oracle}\right\|_{2,\infty} + \left\|\left(\bm{U}_{k_{\sf max}}\bm{U}_{k_{\sf max}}^\top - \bm{U}_{k_{\sf max}}^{\sf oracle}\bm{U}_{k_{\sf max}}^{{\sf oracle}\top}\right)\bm{Z}\right\|_{2,\infty}\notag\\
	&\quad \leq \left\|\left(\bm{U}_{k_{\sf max}}\bm{U}_{k_{\sf max}}^\top - \bm{U}_{k_{\sf max}}^{\sf oracle}\bm{U}_{k_{\sf max}}^{{\sf oracle}\top}\right)\bm{M}^{\sf oracle}\right\|_{2,\infty} + \left\|\bm{U}_{k_{\sf max}}\bm{U}_{k_{\sf max}}^\top - \bm{U}_{k_{\sf max}}^{\sf oracle}\bm{U}_{k_{\sf max}}^{{\sf oracle}\top}\right\|_{2,\infty}\left\|\bm{Z}\right\|\notag\\
	&\quad \lesssim r\left[\sqrt{\frac{\mu r^3}{m_1}}\left(\sqrt{m_1}\omega_{\sf max}\log m\cdot \sigma_{\overline{r}+1}^{\star} + \left(\sqrt{m_1m_2} + m_1\right)\omega_{\sf max}^2\log^2 m\right) + \sqrt{\frac{\mu r}{m_1}}\widetilde{\sigma}_{r_{k_{\sf max}} + 1}^2\right]\notag\\
	&\qquad + \sqrt{\frac{\mu r^3}{m_1}}\left(\sqrt{m_1}\omega_{\sf max}\log m\cdot \sigma_{\overline{r}+1}^{\star} + \left(\sqrt{m_1m_2} + m_1\right)\omega_{\sf max}^2\log^2 m\right)\notag\\
	&\quad \lesssim \sqrt{\frac{\mu r^5}{m_1}}\left(\sqrt{m_1}\omega_{\sf max}\log m\cdot \sigma_{\overline{r}+1}^{\star} + \left(\sqrt{m_1m_2} + m_1\right)\omega_{\sf max}^2\log^2 m\right) + \sqrt{\frac{\mu r^3}{m_1}}\widetilde{\sigma}_{r_{k_{\sf max}} + 1}^2.
\end{align}

\paragraph{Step 5.1.2: bounding $\|(\bm{U}_{k_{\sf max}}\bm{U}_{k_{\sf max}}^\top - \widetilde{\bm{U}}_{k_{\sf max}}\widetilde{\bm{U}}_{k_{\sf max}}^\top)\widetilde{\bm{M}}\|_{2,\infty}$.} Recalling that \eqref{ineq:spectral_Z} holds, we can invoke Lemma \ref{lm:space_estimate_expansion} to obtain
\begin{align}\label{ineq:expansion_tilde_U}
	&\big\|\big(\bm{U}_{k_{\sf max}}^{\sf oracle}\bm{U}_{k_{\sf max}}^{{\sf oracle}\top} - \widetilde{\bm{U}}_{k_{\sf max}}\widetilde{\bm{U}}_{k_{\sf max}}^{\top}\big)\widetilde{\bm{M}}\big\|_{2,\infty}\notag\\ &\quad\leq \frac{40}{\pi}\widetilde{\sigma}_{r_{k_{\sf max}}}^2\sum_{k \geq 1}\frac{2^{k}}{\big(\widetilde{\sigma}_{r_{k_{\sf max}}}^2 - \widetilde{\sigma}_{r_{k_{\sf max}}+1}^2\big)^k}\sum_{0 \leq j_1, \dots, j_{k+1} \leq r\atop \left(j_1, ..., j_{k+1}\right)^{\top} \neq \bm{0}_{k+1}}\big\|\widetilde{\bm{P}}_{j_1}\bm{Z}\widetilde{\bm{P}}_{j_2}\bm{Z}\cdots\bm{Z}\widetilde{\bm{P}}_{j_{k+1}}\big\|_{2,\infty}.
\end{align}
Here, we recall that $\widetilde{\bm{P}}_{j} = \widetilde{\bm{u}}_j\widetilde{\bm{u}}_j^\top$ for $1 \leq j \leq r$ and $\widetilde{\bm{P}}_{0} = \widetilde{\bm{U}}_{\perp}\widetilde{\bm{U}}_{\perp}^\top$. Repeat similar arguments as in \eqref{ineq112} to deduce that
\begin{align}\label{ineq147}
	&\big\|\big(\bm{U}_{k_{\sf max}}^{\sf oracle}\bm{U}_{k_{\sf max}}^{{\sf oracle}\top} - \widetilde{\bm{U}}_{k_{\sf max}}\widetilde{\bm{U}}_{k_{\sf max}}^{\top}\big)\widetilde{\bm{M}}\big\|_{2,\infty}\notag\\ &\quad
	\lesssim \sqrt{\frac{\mu r^3}{m_1}}\frac{r\left(\sqrt{m_1}\omega_{\sf max}\log m\cdot \sigma_{\overline{r}+1}^{\star} + \left(\sqrt{m_1m_2} + m_1\right)\omega_{\sf max}^2\log^2 m\right)}{\sigma_{r_{k_{\sf max}}}^{\star2}-\sigma_{r_{k_{\sf max}}+1}^{\star2}}\cdot\widetilde{\sigma}_{r_{k_{\sf max}}}^2\notag\\
	&\quad \stackrel{\eqref{ineq:induction_general_2}~\text{and}~\eqref{ineq115}}{\lesssim} \sqrt{\frac{\mu r^3}{m_1}}\frac{r\left(\sqrt{m_1}\omega_{\sf max}\log m\cdot \sigma_{\overline{r}+1}^{\star} + \left(\sqrt{m_1m_2} + m_1\right)\omega_{\sf max}^2\log^2 m\right)}{\sigma_{r_{k_{\sf max}}}^{\star2}/r}\cdot\sigma_{r_{k_{\sf max}}}^{\star2}\notag\\
	&\quad = \sqrt{\frac{\mu r^3}{m_1}}r^2\left(\sqrt{m_1}\omega_{\sf max}\log m\cdot \sigma_{\overline{r}+1}^{\star} + \left(\sqrt{m_1m_2} + m_1\right)\omega_{\sf max}^2\log^2 m\right).
\end{align}
Inequality~\eqref{ineq146} taken together with \eqref{ineq147} and the triangle inequality shows that
\begin{align}\label{ineq148}
	&\big\|\big(\bm{U}_{k_{\sf max}}\bm{U}_{k_{\sf max}}^\top - \widetilde{\bm{U}}_{k_{\sf max}}\widetilde{\bm{U}}_{k_{\sf max}}^{\top}\big)\widetilde{\bm{M}}\big\|_{2,\infty}\notag\\
	&\quad \lesssim \sqrt{\frac{\mu r^3}{m_1}}r^2\left(\sqrt{m_1}\omega_{\sf max}\log m\cdot \sigma_{\overline{r}+1}^{\star} + \left(\sqrt{m_1m_2} + m_1\right)\omega_{\sf max}^2\log^2 m\right) + \sqrt{\frac{\mu r^3}{m_1}}\widetilde{\sigma}_{r_{k_{\sf max}} + 1}^2.
\end{align}

\paragraph{Step 5.1.3: bounding $\alpha_1$.} Equipped with \eqref{ineq148}, we are now ready to bound $\alpha_1$. Recall that 
\begin{align*}
	\widetilde{\bm{M}} = \widetilde{\bm{U}}^{(1)}\big(\widetilde{\bm{\Sigma}}^{(1)}\big)^2\widetilde{\bm{U}}^{(1)\top} + \widetilde{\bm{U}}^{(2)}\big(\widetilde{\bm{\Sigma}}^{(2)}\big)^2\widetilde{\bm{U}}^{(2)\top},
\end{align*}
where $\widetilde{\bm{U}}^{(1)}$ and $\widetilde{\bm{\Sigma}}^{(1)}$ (resp.~$\widetilde{\bm{U}}^{(2)}$ and $\widetilde{\bm{\Sigma}}^{(2)}$) are defined in \eqref{svd1} (resp.~\eqref{eq:eigen_2}).
In view of \eqref{ineq143} and \eqref{ineq148}, one can obtain 
\begin{align}\label{ineq142}
	&\big\|\big(\bm{U}_{k_{\sf max}}\bm{U}_{k_{\sf max}}^\top - \widetilde{\bm{U}}_{k_{\sf max}}\widetilde{\bm{U}}_{k_{\sf max}}^{\top}\big)\widetilde{\bm{U}}^{(1)}\widetilde{\bm{\Sigma}}^{(1)}\big\|_{2,\infty}\notag\\
	&\quad \leq \left\|\big(\bm{U}_{k_{\sf max}}\bm{U}_{k_{\sf max}}^\top - \widetilde{\bm{U}}_{k_{\sf max}}\widetilde{\bm{U}}_{k_{\sf max}}^{\top}\big)\widetilde{\bm{U}}^{(1)}\big(\widetilde{\bm{\Sigma}}^{(1)}\big)^2\widetilde{\bm{U}}^{(1)\top}\right\|_{2,\infty}\left\|\big(\widetilde{\bm{\Sigma}}^{(1)}\big)^{-1}\right\|\notag\\
	&\quad \stackrel{(\romannumeral1)}{\lesssim} \left(\big\|\big(\bm{U}_{k_{\sf max}}\bm{U}_{k_{\sf max}}^\top - \widetilde{\bm{U}}_{k_{\sf max}}\widetilde{\bm{U}}_{k_{\sf max}}^{\top}\big)\widetilde{\bm{M}}\big\|_{2,\infty} + \left\|\big(\bm{U}_{k_{\sf max}}\bm{U}_{k_{\sf max}}^\top - \widetilde{\bm{U}}_{k_{\sf max}}\widetilde{\bm{U}}_{k_{\sf max}}^{\top}\big)\widetilde{\bm{U}}^{(2)}\big(\widetilde{\bm{\Sigma}}^{(2)}\big)^2\widetilde{\bm{U}}^{(2)\top}\right\|_{2,\infty}\right)\frac{1}{\sigma_{\overline{r}}^{\star}}\notag\\
	&\quad \leq \left(\big\|\big(\bm{U}_{k_{\sf max}}\bm{U}_{k_{\sf max}}^\top - \widetilde{\bm{U}}_{k_{\sf max}}\widetilde{\bm{U}}_{k_{\sf max}}^{\top}\big)\widetilde{\bm{M}}\big\|_{2,\infty} + \left\|\bm{U}_{k_{\sf max}}\bm{U}_{k_{\sf max}}^\top - \widetilde{\bm{U}}_{k_{\sf max}}\widetilde{\bm{U}}_{k_{\sf max}}^{\top}\right\|_{2,\infty}\big\|\widetilde{\bm{\Sigma}}^{(2)}\big\|^2\right)\frac{1}{\sigma_{\overline{r}}^{\star}}\notag\\
	&\quad \lesssim \Bigg(\sqrt{\frac{\mu r^3}{m_1}}r^2\left(\sqrt{m_1}\omega_{\sf max}\log m\cdot \sigma_{\overline{r}+1}^{\star} + \left(\sqrt{m_1m_2} + m_1\right)\omega_{\sf max}^2\log^2 m\right) + \sqrt{\frac{\mu r^3}{m_1}}\widetilde{\sigma}_{r_{k_{\sf max}} + 1}^2\notag\\&\hspace{1cm} + \sqrt{\frac{\mu r^3}{m_1}}\left(\frac{r^2\sqrt{m_1}\omega_{\sf max}\log m}{\sigma_{r_{k_{\sf max}}}^\star} + \frac{r^2 \sqrt{m_1m_2}\omega_{\sf max}^2\log^2 m}{\sigma_{r_{k_{\sf max}}}^{\star2}}\right)\widetilde{\sigma}_{\overline{r} + 1}^2\Bigg)\frac{1}{\sigma_{\overline{r}}^{\star}}\notag\\
	&\quad \stackrel{(\romannumeral2)}{\lesssim} \Bigg(\sqrt{\frac{\mu r^3}{m_1}}r^2\left(\sqrt{m_1}\omega_{\sf max}\log m\cdot \sigma_{\overline{r}+1}^{\star} + \left(\sqrt{m_1m_2} + m_1\right)\omega_{\sf max}^2\log^2 m\right) + \sqrt{\frac{\mu r^3}{m_1}}\widetilde{\sigma}_{r_{k_{\sf max}} + 1}^2\notag\\&\hspace{1cm} + \sqrt{\frac{\mu r^3}{m_1}}\left(\frac{r^2\sqrt{m_1}\omega_{\sf max}\log m}{\sigma_{r_{k_{\sf max}}}^\star} + \frac{r^2 \sqrt{m_1m_2}\omega_{\sf max}^2\log^2 m}{\sigma_{r_{k_{\sf max}}}^{\star2}}\right)\sigma_{\overline{r}}^{\star2}\Bigg)\frac{1}{\sigma_{\overline{r}}^{\star}}\notag\\
	&\quad \stackrel{(\romannumeral3)}{\lesssim} \sqrt{\frac{\mu r^3}{m_1}}r^2\left(\sqrt{m_1}\omega_{\sf max}\log m + \frac{\left(\sqrt{m_1m_2} + m_1\right)\omega_{\sf max}^2\log^2 m}{\sigma_{\overline{r}}^{\star}}\right) + \sqrt{\frac{\mu r^3}{m_1}}\frac{\widetilde{\sigma}_{r_{k_{\sf max}} + 1}^2}{\sigma_{\overline{r}}^{\star}}\notag\\
	&\quad \stackrel{(\romannumeral4)}{\lesssim} \sqrt{\frac{\mu r^3}{m_1}}\left(r^2\sqrt{m_1}\omega_{\sf max}\log m + r(m_1m_2)^{1/4}\omega_{\sf max}\log m\right) + \sqrt{\frac{\mu r^3}{m_1}}\frac{\tau}{\sigma_{\overline{r}}^{\star}}\notag\\
	&\quad \asymp \sqrt{\frac{\mu r^3}{m_1}}\left(r^2\sqrt{m_1}\omega_{\sf max}\log m + r(m_1m_2)^{1/4}\omega_{\sf max}\log m\right),
\end{align}
where $(\romannumeral1)$, $(\romannumeral2)$ and $(\romannumeral3)$ are consequences of $\widetilde{\sigma}_{r_{k_{\sf max}}} \stackrel{\eqref{ineq115}}{\asymp} \sigma_{r_{k_{\sf max}}}^\star \geq \sigma_{\overline{r}}^{\star} \asymp \widetilde{\sigma}_{\overline{r}} \gtrsim \max\{\sigma_{\overline{r}+1}^{\star}, \widetilde{\sigma}_{\overline{r}+1}\}$, 
and $(\romannumeral4)$ makes use of the fact $\sigma_{\overline{r}}^{\star} \geq C_0r[(m_1m_2)^{1/4} + rm_1^{1/2}]\omega_{\sf max}\log m$ and \eqref{ineq135}. Note that $\widetilde{\bm{U}}^{(1)}\widetilde{\bm{\Sigma}}^{(1)}\widetilde{\bm{W}}^{(1)\top}$ is the SVD of $\bm{U}^{\star(1)}\bm{\Sigma}^{\star(1)} + \bm{E}\bm{V}^{\star(1)}$. By virtue of the previous inequality, \eqref{ineq2a}, Theorem \ref{thm:oracle_two_to_infty} and the triangle inequality, we arrive at
\begin{align}\label{ineq149}
	\alpha_1 &= \big\|\big(\bm{U}_{k_{\sf max}}\bm{U}_{k_{\sf max}}^\top - \widetilde{\bm{U}}_{k_{\sf max}}\widetilde{\bm{U}}_{k_{\sf max}}^{\top}\big)\bm{U}^{\star(1)}\bm{\Sigma}^{\star(1)}\bm{V}^{\star(1)\top}\big\|_{2,\infty}\notag\\
	&= \big\|\big(\bm{U}_{k_{\sf max}}\bm{U}_{k_{\sf max}}^\top - \widetilde{\bm{U}}_{k_{\sf max}}\widetilde{\bm{U}}_{k_{\sf max}}^{\top}\big)\bm{U}^{\star(1)}\bm{\Sigma}^{\star(1)}\big\|_{2,\infty}\notag\\
	&\leq \big\|\big(\bm{U}_{k_{\sf max}}\bm{U}_{k_{\sf max}}^\top - \widetilde{\bm{U}}_{k_{\sf max}}\widetilde{\bm{U}}_{k_{\sf max}}^{\top}\big)\widetilde{\bm{U}}^{(1)}\widetilde{\bm{\Sigma}}^{(1)}\widetilde{\bm{W}}^\top\big\|_{2,\infty} + \big\|\big(\bm{U}_{k_{\sf max}}\bm{U}_{k_{\sf max}}^\top - \widetilde{\bm{U}}_{k_{\sf max}}\widetilde{\bm{U}}_{k_{\sf max}}^{\top}\big)\bm{E}\bm{V}^{\star(1)}\big\|_{2,\infty}\notag\\
	&\leq \big\|\big(\bm{U}_{k_{\sf max}}\bm{U}_{k_{\sf max}}^\top - \widetilde{\bm{U}}_{k_{\sf max}}\widetilde{\bm{U}}_{k_{\sf max}}^{\top}\big)\widetilde{\bm{U}}^{(1)}\widetilde{\bm{\Sigma}}^{(1)}\big\|_{2,\infty} + \big\|\bm{U}_{k_{\sf max}}\bm{U}_{k_{\sf max}}^\top - \widetilde{\bm{U}}_{k_{\sf max}}\widetilde{\bm{U}}_{k_{\sf max}}^{\top}\big\|_{2,\infty}\big\|\bm{E}\bm{V}^{\star(1)}\big\|\notag\\
	&\leq \sqrt{\frac{\mu r^3}{m_1}}\left(r^2\sqrt{m_1}\omega_{\sf max}\log m + r(m_1m_2)^{1/4}\omega_{\sf max}\log m\right)\notag\\
	&\quad + \sqrt{\frac{\mu r^3}{m_1}}\left(\frac{r^2\sqrt{m_1}\omega_{\sf max}\log m}{\sigma_{r_{k_{\sf max}}}^\star} + \frac{r^2 \sqrt{m_1m_2}\omega_{\sf max}^2\log^2 m}{\sigma_{r_{k_{\sf max}}}^{\star2}}\right)\cdot \sqrt{m_1}\omega_{\sf max}\log m\notag\\
	&\leq \sqrt{\frac{\mu r^3}{m_1}}\left(r^2\sqrt{m_1}\omega_{\sf max}\log m + r(m_1m_2)^{1/4}\omega_{\sf max}\log m\right) + \sqrt{\frac{\mu r^3}{m_1}}\sqrt{m_1}\omega_{\sf max}\log m\notag\\
	&\asymp \sqrt{\frac{\mu r^3}{m_1}}\left(r^2\sqrt{m_1}\omega_{\sf max}\log m + r(m_1m_2)^{1/4}\omega_{\sf max}\log m\right),
\end{align}
where the penultimate line follows since $\sigma_{\overline{r}}^{\star} \geq C_0r[(m_1m_2)^{1/4} + rm_1^{1/2}]\omega_{\sf max}\log m$.

\paragraph{Step 5.2: bounding $\alpha_2$.} In view of \eqref{ineq110}, we have
\begin{align}\label{eq14}
	&\mathcal{P}_{\left(\widetilde{\bm{U}}_{k_{\sf max}}\right)_{\perp}}\mathcal{P}_{\bm{U}_{k_{\sf max}}^\star}\bm{U}^{\star(1)}\bm{\Sigma}^{\star(1)}\bm{V}^{\star(1)\top}\notag\\ &\quad= \left(\mathcal{P}_{\left(\widetilde{\bm{U}}_{k_{\sf max}}\right)_{\perp}}\bm{U}_{k_{\sf max}}^\star\right)\bm{U}_{k_{\sf max}}^{\star\top}\bm{U}^{\star(1)}\bm{\Sigma}^{\star(1)}\bm{V}^{\star(1)\top}\notag\\
	&\quad= \Big[\widetilde{\bm{U}}_{:,r_{k_{\sf max}}+1: \overline{r}}^{(1)}\widetilde{\bm{\Sigma}}_{r_{k_{\sf max}}+1: \overline{r}, r_{k_{\sf max}}+1: \overline{r}}^{(1)}\widetilde{\bm{W}}_{:,r_{k_{\sf max}}+1: \overline{r}}^{(1)\top}\left(\bm{I}_{r_{k_{\sf max}}}\ \bm{0}\right)^\top\left(\bm{\Sigma}^{\star}_{1:r_{k_{\sf max}}, 1:r_{k_{\sf max}}}\right)^{-1}\notag\\&\qquad - \mathcal{P}_{\left(\widetilde{\bm{U}}_{:,1: r_{k_{\sf max}}}\right)_{\perp}}\bm{E}\bm{V}^{\star(1)}\left(\bm{I}_{r_{k_{\sf max}}}\ \bm{0}\right)^\top\left(\bm{\Sigma}^{\star}_{1:r_{k_{\sf max}}, 1:r_{k_{\sf max}}}\right)^{-1}\Big]\notag\\
	&\qquad\cdot \bm{\Sigma}^{\star}_{1:r_{k_{\sf max}}, 1:r_{k_{\sf max}}}\bm{V}_{:,1:r_{k_{\sf max}}}^{\star\top}\notag\\
	&\quad= \widetilde{\bm{U}}_{:,r_{k_{\sf max}}+1: \overline{r}}^{(1)}\widetilde{\bm{\Sigma}}_{r_{k_{\sf max}}+1: \overline{r}, r_{k_{\sf max}}+1: \overline{r}}^{(1)}\widetilde{\bm{W}}_{:,r_{k_{\sf max}}+1: \overline{r}}^{(1)\top}\left(\bm{I}_{r_{k_{\sf max}}}\ \bm{0}\right)^\top\bm{V}_{:,1:r_{k_{\sf max}}}^{\star\top}\notag\\
	&\qquad - \mathcal{P}_{\left(\widetilde{\bm{U}}_{:,1: r_{k_{\sf max}}}\right)_{\perp}}\bm{E}\bm{V}^{\star(1)}\left(\bm{I}_{r_{k_{\sf max}}}\ \bm{0}\right)^\top\bm{V}_{:,1:r_{k_{\sf max}}}^{\star\top}.
\end{align}
Here, the third line holds since $r_{k_{\sf max}} \leq \overline{r}$ and
\begin{align}\label{eq15}
	\bm{U}_{k_{\sf max}}^{\star\top}\bm{U}^{\star(1)}\bm{\Sigma}^{\star(1)}\bm{V}^{\star(1)\top} &= \bm{U}_{:,1:r_{k_{\sf max}}}^{\star\top}\bm{U}_{:,1:\overline{r}}^\star\bm{\Sigma}_{1:\overline{r}, 1:\overline{r}}^\star\bm{V}_{:,1:\overline{r}}^{\star\top}\notag\\
	&= \bm{U}_{:,1:r_{k_{\sf max}}}^{\star\top}\left(\bm{U}_{:,1:r_{k_{\sf max}}}^{\star}\bm{\Sigma}_{1:r_{k_{\sf max}}, 1:r_{k_{\sf max}}}^\star\bm{V}_{:,1:r_{k_{\sf max}}}^{\star\top} + \bm{U}_{:,r_{k_{\sf max}}+1:\overline{r}}^{\star}\bm{\Sigma}_{r_{k_{\sf max}}+1:\overline{r}, r_{k_{\sf max}}+1:\overline{r}}^\star\bm{V}_{:,r_{k_{\sf max}}+1:\overline{r}}^{\star\top}\right)\notag\\
	&= \bm{\Sigma}_{1:r_{k_{\sf max}}, 1:r_{k_{\sf max}}}^\star\bm{V}_{:,1:r_{k_{\sf max}}}^{\star\top}.
\end{align}
Repeating similar arguments as in \eqref{ineq107} and \eqref{ineq108}, one has 
\begin{subequations}
	\begin{align}
		\left\|\widetilde{\bm{U}}_{:,r_{k_{\sf max}}+1: \overline{r}}^{(1)}\widetilde{\bm{\Sigma}}_{r_{k_{\sf max}}+1: \overline{r}, r_{k_{\sf max}}+1: \overline{r}}^{(1)}\widetilde{\bm{W}}_{:,r_{k_{\sf max}}+1: \overline{r}}^{(1)\top}\left(\bm{I}_{r_{k_{\sf max}}}\ \bm{0}\right)^\top\bm{V}_{:,1:r_{k_{\sf max}}}^{\star\top}\right\|_{2,\infty} &\lesssim \sqrt{\frac{\mu r}{m_1}}r\sqrt{m}_1\omega_{\sf max}\log m,\label{ineq150a}\\
		\left\|\mathcal{P}_{\left(\widetilde{\bm{U}}_{:,1: r_{k_{\sf max}}}\right)_{\perp}}\bm{E}\bm{V}^{\star(1)}\left(\bm{I}_{r_{k_{\sf max}}}\ \bm{0}\right)^\top\bm{V}_{:,1:r_{k_{\sf max}}}^{\star\top}\right\|_{2,\infty} &\lesssim \sqrt{\frac{\mu r}{m_1}}\sqrt{m}_1\omega_{\sf max}\log m.\label{ineq150b}
	\end{align}
\end{subequations}
Combining \eqref{eq14}, \eqref{ineq150a}, \eqref{ineq150b} and the triangle inequality yields 
\begin{align}\label{ineq151}
	\left\|\mathcal{P}_{\left(\widetilde{\bm{U}}_{k_{\sf max}}\right)_{\perp}}\mathcal{P}_{\bm{U}_{k_{\sf max}}^\star}\bm{U}^{\star(1)}\bm{\Sigma}^{\star(1)}\bm{V}^{\star(1)\top}\right\|_{2,\infty} \lesssim \sqrt{\frac{\mu r^3}{m_1}}\sqrt{m}_1\omega_{\sf max}\log m.
\end{align}
Then we can bound $\alpha_2$ as follows:
\begin{align}\label{ineq152}
	\alpha_2 &= \left\|\big(\widetilde{\bm{U}}_{k_{\sf max}}\widetilde{\bm{U}}_{k_{\sf max}}^{\top} - \bm{U}_{k_{\sf max}}^\star\bm{U}_{k_{\sf max}}^{\star\top}\big)\bm{U}^{\star(1)}\bm{\Sigma}^{\star(1)}\bm{V}^{\star(1)\top}\right\|_{2,\infty}\notag\\
	&\leq \left\|\big(\widetilde{\bm{U}}_{k_{\sf max}}\widetilde{\bm{U}}_{k_{\sf max}}^{\top} - \bm{U}_{k_{\sf max}}^\star\bm{U}_{k_{\sf max}}^{\star\top}\big)\bm{U}_{k_{\sf max}}^\star\bm{U}_{k_{\sf max}}^{\star\top}\bm{U}^{\star(1)}\bm{\Sigma}^{\star(1)}\bm{V}^{\star(1)\top}\right\|_{2,\infty}\notag\\
	&\quad + \left\|\widetilde{\bm{U}}_{k_{\sf max}}\widetilde{\bm{U}}_{k_{\sf max}}^{\top}\left(\bm{U}_{k_{\sf max}}^\star\right)_{\perp}\left(\bm{U}_{k_{\sf max}}^\star\right)_{\perp}^{\star\top}\bm{U}^{\star(1)}\bm{\Sigma}^{\star(1)}\bm{V}^{\star(1)\top}\right\|_{2,\infty}\notag\\
	&= \left\|\mathcal{P}_{\left(\widetilde{\bm{U}}_{k_{\sf max}}\right)_{\perp}}\mathcal{P}_{\bm{U}_{k_{\sf max}}^\star}\bm{U}^{\star(1)}\bm{\Sigma}^{\star(1)}\bm{V}^{\star(1)\top}\right\|_{2,\infty}\notag\\
	&\quad + \left\|\widetilde{\bm{U}}_{k_{\sf max}}\widetilde{\bm{U}}_{k_{\sf max}}^{\top}\left(\bm{U}_{k_{\sf max}}^\star\right)_{\perp}\left(\bm{U}_{k_{\sf max}}^\star\right)_{\perp}^{\star\top}\bm{U}_{:,r_{k_{\sf max}}+1:\overline{r}}^{\star}\bm{\Sigma}_{r_{k_{\sf max}}+1:\overline{r}, r_{k_{\sf max}}+1:\overline{r}}^\star\bm{V}_{:,r_{k_{\sf max}}+1:\overline{r}}^{\star\top}\right\|_{2,\infty}\notag\\
	&\stackrel{\eqref{ineq151}}{\lesssim} \sqrt{\frac{\mu r^3}{m_1}}\sqrt{m}_1\omega_{\sf max}\log m + \big\|\widetilde{\bm{U}}_{k_{\sf max}}\big\|_{2,\infty}\big\|\widetilde{\bm{U}}_{k_{\sf max}}^{\top}\left(\bm{U}_{k_{\sf max}}^\star\right)_{\perp}\big\|\big\|\bm{\Sigma}_{r_{k_{\sf max}}+1:\overline{r}, r_{k_{\sf max}}+1:\overline{r}}^\star\big\|\notag\\
	&\stackrel{\eqref{ineq2d}~\text{and}~\eqref{ineq109b}}{\lesssim} \sqrt{\frac{\mu r^3}{m_1}}\sqrt{m}_1\omega_{\sf max}\log m + \sqrt{\frac{\mu r}{m_1}}\frac{r\sqrt{m}_1\omega_{\sf max}\log m}{\sigma_{r_{k_{\sf max}}}^\star}\sigma_{r_{k_{\sf max}} + 1}^\star\notag\\
	&\asymp \sqrt{\frac{\mu r^3}{m_1}}\sqrt{m}_1\omega_{\sf max}\log m.
\end{align}
\paragraph{Step 5.3: bounding $\alpha_3$.} By virtue of \eqref{ineq1} and \eqref{ineq200}, one has
\begin{align}\label{ineq153}
	\alpha_3 &= \big\|\bm{U}_{k_{\sf max}}\bm{U}_{k_{\sf max}}^\top - \bm{U}_{k_{\sf max}}^\star\bm{U}_{k_{\sf max}}^{\star\top}\big\|_{2,\infty}\sigma_{\overline{r} + 1}^\star
	\lesssim \sqrt{\frac{\mu r^3}{m_1}}\left(r\left[(m_1m_2)^{1/4} + rm_1^{1/2}\right]\omega_{\sf max}\log m\right).
\end{align}
\paragraph{Step 5.4: bounding $\|(\bm{U}_{k_{\sf max}}\bm{U}_{k_{\sf max}}^\top - \bm{U}_{k_{\sf max}}^\star\bm{U}_{k_{\sf max}}^{\star\top})\bm{X}^\star\|_{2,\infty}$ and $\|(\bm{I}_{n_1} - \bm{U}_{k_{\sf max}}\bm{U}_{k_{\sf max}}^\top)\bm{X}^\star\|_{2,\infty}$.} Putting \eqref{ineq:decomposition}, \eqref{ineq149}, \eqref{ineq152} and \eqref{ineq153} together, we arrive at
\begin{align}\label{ineq:two_to_infty_residual1}
	\big\|\big(\bm{U}_{k_{\sf max}}\bm{U}_{k_{\sf max}}^\top - \bm{U}_{k_{\sf max}}^\star\bm{U}_{k_{\sf max}}^{\star\top}\big)\bm{X}^\star\big\|_{2,\infty} \lesssim \sqrt{\frac{\mu r^3}{m_1}}\left(r^2\sqrt{m_1}\omega_{\sf max}\log m + r(m_1m_2)^{1/4}\omega_{\sf max}\log m\right).
\end{align}
Furthermore, we have
\begin{align*}
	\left\|\left(\bm{I}_{m_1} - \bm{U}_{k_{\sf max}}^\star\bm{U}_{k_{\sf max}}^{\star\top}\right)\bm{X}^\star\right\|_{2,\infty} 
	&= \left\|\bm{U}^{\star}\bm{\Sigma}^{\star}\bm{V}^{\star\top} - \bm{U}_{k_{\sf max}}^\star\bm{U}_{k_{\sf max}}^{\star\top}\bm{U}^{\star}\bm{\Sigma}^{\star}\bm{V}^{\star\top}\right\|_{2,\infty}\notag\\
	&= \left\|\bm{U}^{\star}\bm{\Sigma}^{\star}\bm{V}^{\star\top} - \bm{U}_{:,1:r_{k_{\sf max}}}^\star\bm{\Sigma}_{1:r_{k_{\sf max}}, 1:r_{k_{\sf max}}}^\star\bm{V}_{:,1:r_{k_{\sf max}}}^{\star\top}\right\|_{2,\infty}\notag\\
	&= \left\|\bm{U}_{:, r_{k_{\sf max}}+1:r}^\star\bm{\Sigma}_{r_{k_{\sf max}}+1:r, r_{k_{\sf max}}+1:r}^\star\bm{V}_{:, r_{k_{\sf max}}+1:r}^{\star\top}\right\|_{2,\infty}\notag\\
	&\leq \left\|\bm{U}^\star\right\|_{2,\infty}\left\|\bm{\Sigma}_{r_{k_{\sf max}}+1:r, r_{k_{\sf max}}+1:r}^\star\right\|\notag\\
	&\leq \sqrt{\frac{\mu r}{m_1}}\sigma_{r_{k_{\sf max}}+1}^\star\notag\\
	&\stackrel{\eqref{ineq2a}~\text{and}~\eqref{ineq128}}{\leq} \sqrt{\frac{\mu r}{m_1}}\big(\widetilde{\sigma}_{r_{k_{\sf max}}+1} + 2\sqrt{C_5}\sqrt{m}_1\omega_{\sf max}\log m\big)\notag\\
	&\stackrel{\eqref{ineq135}}{\lesssim} \sqrt{\frac{\mu r}{m_1}}\left(\sqrt{\tau} + \sqrt{m}_1\omega_{\sf max}\log m\right)\notag\\
	&\asymp \sqrt{\frac{\mu r}{m_1}}\left(r^2\sqrt{m_1}\omega_{\sf max}\log m + r(m_1m_2)^{1/4}\omega_{\sf max}\log m\right).
\end{align*}
This together with \eqref{ineq:two_to_infty_residual1} gives
\begin{align}\label{ineq:two_to_infty_residual2}
	\|(\bm{I}_{n_1} - \bm{U}_{k_{\sf max}}\bm{U}_{k_{\sf max}}^\top)\bm{X}^\star\|_{2,\infty} \lesssim \sqrt{\frac{\mu r^3}{m_1}}\left(r^2\sqrt{m_1}\omega_{\sf max}\log m + r(m_1m_2)^{1/4}\omega_{\sf max}\log m\right).
\end{align}

\section{Proof of Theorem \ref{thm:tuning_selection}}\label{proof:tuning_selection}
For notational convenience, we let $\bm{U}_i^\star = \bm{U}_{\bm{X}_i^\star} \in \mathcal{O}^{n_i, k_i}$ denote the left singular subspace of $\bm{X}_i^\star$ for all $i \in [3]$. Then we know from \eqref{ineq155} that 
\begin{align}\label{ineq193}
	\left\|\bm{U}_i^\star\right\|_{2,\infty} \leq \sqrt{\frac{k_i}{\beta n_i}},~\quad~\forall i \in [3].
\end{align}
In view of \citet[Lemma 7]{zhou2023deflated}, with probability exceeding $1 - O(n^{-10})$,
\begin{align}\label{ineq184}
	\left\|\mathcal{P}_{\sf off\text{-}diag}\left(\bm{E}_1\bm{E}_1^\top\right)\right\| \lesssim B^2\log^2n + \sqrt{n_1n_2n_3}\omega_{\sf max}^2\log n \leq \sqrt{n_1n_2n_3}\omega_{\sf max}^2\log n \ll n_2n_3\omega_{\sf max}^2.
\end{align}
For any $i \in [n_1]$, we know that
\begin{align}\label{ineq185}
	\left(\bm{E}_1\bm{E}_1^\top\right)_{i,i} = \sum_{j = 1}^{n_2}\sum_{\ell = 1}^{n_3}E_{i,j,\ell}^2 = \sum_{j = 1}^{n_2}\sum_{\ell = 1}^{n_3}\omega_{i,j,\ell}^2 + \sum_{j = 1}^{n_2}\sum_{\ell = 1}^{n_3}\left(E_{i,j,\ell}^2 - \omega_{i,j,\ell}^2\right).
\end{align}
If the noise is bounded, i.e., $E_{i,j,k} \leq B$ for all $(i, j, k) \in [n_1] \times [n_2] \times [n_3]$, then $\{E_{i,j,\ell}^2 - \omega_{i,j,\ell}^2\}_{i,j,\ell}$ are zero-mean, and
\begin{subequations}
	\begin{align*}
		\left|E_{i,j,\ell}^2 - \omega_{i,j,\ell}^2\right| &\leq 2B^2,\\
		\bbE\left[\left(E_{i,j,\ell}^2 - \omega_{i,j,\ell}^2\right)^2\right] &= \bbE\left[E_{i,j,\ell}^4\right] - \omega_{i,j,\ell}^4 \leq B^2\bbE\left[E_{i,j,\ell}^2\right] - \omega_{i,j,\ell}^4 \leq B^2\omega_{\sf max}^2.
	\end{align*}
\end{subequations}
In view of Bernstein's inequality and the union bound, one has, with probability exceeding $1 - O(n^{-10})$, for all $i \in [n_1]$,
\begin{align}\label{ineq186}
	\left|\sum_{j = 1}^{n_2}\sum_{\ell = 1}^{n_3}\left(E_{i,j,\ell}^2 - \omega_{i,j,\ell}^2\right)\right| \lesssim \sqrt{n_2n_3}B\omega_{\sf max}\sqrt{\log n} + B^2\log n \ll n_2n_3\omega_{\sf max}^2.
\end{align}
For the general case where the noise satisfies Assumption \ref{assump:noise}, using the truncation trick as in \citet[Section B.4.2]{zhou2023deflated}, one can show that \eqref{ineq186} also holds with probability exceeding $1 - O(n^{-10})$. Putting \eqref{ineq184}, \eqref{ineq185} and \eqref{ineq186} together, we know that with probability exceeding $1 - O(n^{-10})$,
\begin{align}\label{ineq188}
	\left\|\bm{E}_1\bm{E}_1^\top - {\sf diag}\left(\Bigg[\sum_{j = 1}^{n_2}\sum_{\ell = 1}^{n_3}\omega_{i,j,\ell}^2\Bigg]_{1 \leq i \leq n_1}\right)\right\| \ll n_2n_3\omega_{\sf max}^2.
\end{align}

For all $i \in [3]$, let $\bm{U}_i^\star \in \mathcal{O}^{n_i, k_i}$ denote the left singular subspace of $\bm{X}_i^\star$. The min-max principle for singular values reveals that
\begin{align}\label{ineq187}
	\sigma_{k_1 + 1}\left(\bm{Y}_1\right) &\geq \sigma_{k_1 + 1}\left(\bm{Y}_1\left(\mathcal{P}_{\bm{U}_{3\perp}^\star} \otimes \mathcal{P}_{\bm{U}_{2\perp}^\star}\right)\right)\notag\\
	&= \sigma_{k_1 + 1}\left(\bm{X}_1^\star\left(\mathcal{P}_{\bm{U}_{3\perp}^\star} \otimes \mathcal{P}_{\bm{U}_{2\perp}^\star}\right) + \bm{E}_1\left(\mathcal{P}_{\bm{U}_{3\perp}^\star} \otimes \mathcal{P}_{\bm{U}_{2\perp}^\star}\right)\right)\notag\\
	&= \sigma_{k_1 + 1}\left(\bm{E}_1\left(\mathcal{P}_{\bm{U}_{3\perp}^\star} \otimes \mathcal{P}_{\bm{U}_{2\perp}^\star}\right)\right)\notag\\
	&\geq \sigma_{k_1 + 1}\left(\bm{E}_1\right) - \left\|\bm{E}_1\left(\mathcal{P}_{\bm{U}_{3}^\star} \otimes \mathcal{P}_{\bm{U}_{2}^\star}\right)\right\| - \left\|\bm{E}_1\left(\mathcal{P}_{\bm{U}_{3\perp}^\star} \otimes \mathcal{P}_{\bm{U}_{2}^\star}\right)\right\| - \left\|\bm{E}_1\left(\mathcal{P}_{\bm{U}_{3}^\star} \otimes \mathcal{P}_{\bm{U}_{2\perp}^\star}\right)\right\|\notag\\
	&= \sqrt{\sigma_{k_1 + 1}\left(\bm{E}_1\bm{E}_1^\top\right)} - \left\|\bm{E}_1\left(\bm{U}_{3}^\star \otimes \bm{U}_{2}^\star\right)\right\| - \left\|\bm{E}_1\left(\bm{U}_{3\perp}^\star \otimes \bm{U}_{2}^\star\right)\right\| - \left\|\bm{E}_1\left(\bm{U}_{3}^\star \otimes \bm{U}_{2\perp}^\star\right)\right\|,
\end{align}
where the fourth line makes use of Weyl's inequality. We know from \eqref{ineq193} that
\begin{align*}
	\left\|\bm{U}_{3}^\star \otimes \bm{U}_{2}^\star\right\|_{2,\infty} &\leq \sqrt{\frac{k_2k_3}{\beta^2n_2n_3}},\\
	\left\|\bm{U}_{3\perp}^\star \otimes \bm{U}_{2}^\star\right\|_{2,\infty} &\leq \|\bm{U}_{2}^\star\|_{2,\infty} \leq \sqrt{\frac{k_2}{\beta n_2}},\\
	\left\|\bm{U}_{3}^\star \otimes \bm{U}_{2\perp}^\star\right\|_{2,\infty} &\leq \|\bm{U}_{3}^\star\|_{2,\infty} \leq \sqrt{\frac{k_3}{\beta n_3}}.
\end{align*}
Applying \citet[Lemma 5]{zhou2023deflated} yields that with probability at least $1 - O(n^{-10})$,
\begin{subequations}
	\begin{align}
		\left\|\bm{E}_1\left(\bm{U}_{3}^\star \otimes \bm{U}_{2}^\star\right)\right\| \lesssim B\sqrt{\frac{k_2k_3}{\beta^2n_2n_3}}\log n + \sqrt{n_1}\omega_{\sf max}\sqrt{\log n} \ll \sqrt{n_2n_3}\omega_{\sf max},\label{ineq189a}\\
		\left\|\bm{E}_1\left(\bm{U}_{3\perp}^\star \otimes \bm{U}_{2}^\star\right)\right\| \lesssim B\sqrt{\frac{k_2}{\beta n_2}}\log n + \sqrt{n_1 + k_2n_3}\omega_{\sf max}\sqrt{\log n} \ll \sqrt{n_2n_3}\omega_{\sf max},\label{ineq189b}\\
		\left\|\bm{E}_1\left(\bm{U}_{3}^\star \otimes \bm{U}_{2\perp}^\star\right)\right\| \lesssim B\sqrt{\frac{k_3}{\beta n_3}}\log n + \sqrt{n_1 + k_3n_2}\omega_{\sf max}\sqrt{\log n} \ll \sqrt{n_2n_3}\omega_{\sf max}.\label{ineq189c}
	\end{align}
\end{subequations}
Combining \eqref{ineq188}, \eqref{ineq189a} - \eqref{ineq189c} and Weyl's inequality, we obtain that with probability exceeding $1 - O(n^{-10})$,
\begin{align}\label{ineq190}
	\sigma_{k_1 + 1}\left(\bm{Y}_1\right) \leq \sigma_{k_1 + 1}\left(\bm{X}_1^\star\right) + \left\|\bm{E}_1\right\| = \left\|\bm{E}_1\right\| = \left(\left\|\bm{E}_1\bm{E}_1^\top\right\|\right)^{1/2} \leq 2\sqrt{n_2n_3}\omega_{\sf max}.
\end{align}
\begin{itemize}[leftmargin = *]
	\item[1.] If \eqref{assump:variance_same_order} holds, then \eqref{ineq188}, \eqref{ineq187} and \eqref{ineq189a} - \eqref{ineq189c} together show that with probability exceeding $1 - O(n^{-10})$,
	\begin{align}\label{ineq191}
		\sigma_{k_1 + 1}\left(\bm{Y}_1\right) &\geq \sqrt{cn_2n_3\omega_{\sf max}^2} - \frac{\sqrt{c}}{2}\sqrt{n_2n_3}\omega_{\sf max} \geq \frac{\sqrt{c}}{2}\sqrt{n_2n_3}\omega_{\sf max}.
	\end{align}
    The previous inequality together with \eqref{ineq190} shows that
    \begin{align*}
    	\sigma_{k_1 + 1}\left(\bm{Y}_1\right) \asymp \sqrt{n_2n_3}\omega_{\sf max}
    \end{align*}
with probability at least $1 - O(n^{-10})$. As a consequence, there exist two large enough constants $C_{\tau} > c_{\tau} > 0$ such that with probability at least $1 - O(n^{-10})$, 
\begin{align}\label{ineq192}
	c_{\tau}\left(n_1n_2n_3\right)^{1/2}\log^2 n \leq \tau/\omega_{\sf max}^2 &\leq C_{\tau}\left(n_1n_2n_3\right)^{1/2}\log^2 n .
\end{align}
    \item[2.] If Assumption 2 in Theorem \ref{thm:tuning_selection} holds, we choose $(\ell_1, \ell_2, \ell_3) \in [k_1] \times [k_2] \times [k_3]$ such that 
    \begin{align*}
    (\ell_1, \ell_2, \ell_3) \in \argmax_{i_1 \in [k_1], i_2 \in [k_2], i_3 \in [k_3]}S^\star_{i_1, i_2, i_3}\left(1 - S^\star_{i_1, i_2, i_3}\right) = \omega_{\sf max}^2.
    \end{align*}
   Then for any $(j_1, j_2, j_3) \in [n_1] \times [n_2] \times [n_3]$ with $z_{i, j_i}^\star = \ell_i$, $i \in [3]$,
   \begin{align*}
   	   \bbE\left[E_{j_1, j_2, j_3}^2\right] = \omega_{\sf max}^2.
   \end{align*}
   For any $i \in \{j \in [n_1]: z_{1, j}^\star = \ell_1\}$, 
   \begin{align*}
   	\sum_{j = 1}^{n_2}\sum_{\ell = 1}^{n_3}\omega_{i,j,\ell}^2 \geq \omega_{\sf max}^2\left|j \in [n_2]: z_{2, j}^\star = \ell_2\right|\cdot\left|j \in [n_3]: z_{3, j}^\star = \ell_3\right| \geq  \frac{\beta n_2}{k_2}\frac{\beta n_3}{k_3}\omega_{\sf max}^2 \asymp n_2n_3\omega_{\sf max}^2.
   \end{align*}
Since $|\{j \in [n_1]: z_{1, j}^\star = \ell_1\}| \geq \beta n_1/k_1 > k_1 + 1$, we can still show that \eqref{ineq191} holds with probability exceeding $1 - O(n^{-10})$. Repeating similar arguments as in \eqref{ineq192} shows that Theorem \ref{thm:tuning_selection} also holds.
\end{itemize}

\section{Technical lemmas}
\begin{lemma}\label{lm:1}
    Suppose that
    \begin{align}\label{eq13}
    	\bm{M} = \bm{M}^\star + \bm{E} \in \bbR^{n_1 \times n_2}
    \end{align}
and the SVDs of $\bm{M}^\star$ and $\bm{M}$ are given by
\begin{align*}
	\bm{M}^\star = \sum_{i=1}^{n_1}\sigma_i^\star\bm{u}_i^\star\bm{v}_i^{\star\top},~\quad~\text{and}~\quad~\bm{M}^{\star} = \sum_{i=1}^{n_1}\sigma_i\bm{u}_i\bm{v}_i^\top.
\end{align*}
Here, $\sigma_1^\star \geq \dots \geq \sigma_{n_1}^\star \geq 0$ (resp.~$\sigma_1 \geq \dots \geq \sigma_{n_1} \geq 0$) represent the singular values of $\bm{M}^\star$ (resp.~$\bm{M}$), $u_i^\star$ (resp.~$u_i$) denotes the left singular vector associated with the singular value $\sigma_i^\star$ (resp.~$\sigma_i$), and $v_i^\star$ (resp.~$v_i$) denotes the left singular vector associated with $\sigma_i^\star$ (resp.~$\sigma_i$). We let $\bm{U}^\star = [u_1^\star, \dots, u_r^\star] \in \bbR^{n_1 \times r}$ (resp.~$\bm{U} = [u_1, \dots, u_r] \in \bbR^{n_1 \times r}$) denote the rank-$r$ leading singular subspace of $\bm{M}^\star$ (resp.~$\bm{M}$).
If $$\sigma_r^\star - \sigma_{r+1}^\star > 2\|\bm{E}\|,$$ then we have
\begin{align*}
	\left\|\left(\bm{U}\bm{U}^\top - \bm{U}^\star\bm{U}^{\star\top}\right)\bm{M}^\star\right\| \leq \frac{4\sigma_r^\star\left\|\bm{E}\right\|}{\sigma_r^\star - \sigma_{r+1}^\star}.
\end{align*}
\end{lemma}
\begin{proof}\label{proof:lm_1}
	We define
	\begin{align*}
		\bm{\Sigma} &:= {\sf diag}\left(\sigma_1, \dots, \sigma_r\right),~\qquad\qquad~\bm{\Sigma}_{\perp} := {\sf diag}\left(\sigma_{r+1}, \dots, \sigma_{n_1}\right),\notag\\
		\bm{V} &:= \left[\bm{v}_1, \dots, \bm{v}_r\right] \in \bbR^{n_2 \times r},~\qquad~\bm{V}_{\perp} := \left[\bm{v}_{r+1}, \dots, \bm{v}_{n_2}\right] \in \bbR^{n_2 \times (n_2 - r)},
	\end{align*}
	and define $\bm{\Sigma}^\star, \bm{\Sigma}_{\perp}^\star, \bm{V}^\star, \bm{V}_{\perp}^\star$ similarly. 
	In view of \citet[Eqn. (2.27)]{chen2021spectral}, we have
	\begin{align*}
		\mathcal{P}_{\bm{U}_{\perp}}\mathcal{P}_{\bm{U}^\star}\bm{M}^\star &= \bm{U}_{\perp}\left(\bm{U}_{\perp}\bm{U}^\star\right)\left(\bm{U}^{\star\top}\bm{M}^\star\right)\\
		&= \bm{U}_{\perp}\left(\bm{\Sigma}_{\perp}\bm{V}_{\perp}^\top\bm{V}^\star\bm{\Sigma}^{\star-1} - \bm{U}_{\perp}^\top\bm{E}\bm{V}^\star\bm{\Sigma}^{\star-1}\right)\bm{\Sigma}^{\star}\bm{V}^{\star\top}\\
		&= \bm{U}_{\perp}\bm{\Sigma}_{\perp}\bm{V}_{\perp}^\top\bm{V}^\star\bm{V}^{\star\top} - \bm{U}_{\perp}^\top\bm{E}\bm{V}^\star\bm{V}^{\star\top}.
	\end{align*}
    Recognizing that
    \begin{align*}
    	\bm{U}\bm{U}^\top - \bm{U}^\star\bm{U}^{\star\top} &= \left(\bm{U}\bm{U}^\top\bm{U}^\star - \bm{U}^\star\right)\bm{U}^{\star\top} + \left(\bm{U}\bm{U}^\top - \bm{U}\bm{U}^\top\bm{U}^\star\bm{U}^{\star\top}\right)\\
    	&= -\mathcal{P}_{\bm{U}_{\perp}}\mathcal{P}_{\bm{U}^\star} + \bm{U}\bm{U}^\top\bm{U}_{\perp}^{\star}\bm{U}_{\perp}^{\star\top},
    \end{align*}
    we have
    \begin{align}
    		\left\|\left(\bm{U}\bm{U}^\top - \bm{U}^\star\bm{U}^{\star\top}\right)\bm{M}^\star\right\| &\leq \left\|\mathcal{P}_{\bm{U}_{\perp}}\mathcal{P}_{\bm{U}^\star}\bm{M}^\star\right\| + \left\|\bm{U}\bm{U}^\top\bm{U}_{\perp}^{\star}\bm{U}_{\perp}^{\star\top}\bm{M}^\star\right\|\notag\\
    		&\leq \left\|\bm{U}_{\perp}\bm{\Sigma}_{\perp}\bm{V}_{\perp}^\top\bm{V}^\star\bm{V}^{\star\top}\right\| + \left\|\bm{U}_{\perp}^\top\bm{E}\bm{V}^\star\bm{V}^{\star\top}\right\| + \left\|\bm{U}\bm{U}^\top\bm{U}_{\perp}^{\star}\bm{\Sigma}_{\perp}^{\star}\bm{V}_{\perp}^{\star\top}\right\|\notag\\
    		&\leq \left\|\bm{\Sigma}_{\perp}\right\|\left\|\bm{V}_{\perp}^\top\bm{V}^\star\right\| + \left\|\bm{E}\right\| + \left\|\bm{U}_{\perp}^{\star\top}\bm{U}\right\|\left\|\bm{\Sigma}_{\perp}^{\star}\right\|\notag\\
    		&\stackrel{(a)}{\leq} \sigma_{r+1}\frac{2\left\|\bm{E}\right\|}{\sigma_r^\star - \sigma_{r+1}^\star} + \left\|\bm{E}\right\| + \sigma_{r+1}^\star\frac{2\left\|\bm{E}\right\|}{\sigma_r^\star - \sigma_{r+1}^\star}\notag\\
    		&\stackrel{(b)}{\leq} \left(\sigma_{r+1}^\star + \left\|\bm{E}\right\|\right)\frac{2\left\|\bm{E}\right\|}{\sigma_r^\star - \sigma_{r+1}^\star} + \left\|\bm{E}\right\| + \sigma_{r+1}^\star\frac{2\left\|\bm{E}\right\|}{\sigma_r^\star - \sigma_{r+1}^\star}\notag\\
    		&\stackrel{(c)}{\leq} \frac{4\sigma_r^\star\left\|\bm{E}\right\|}{\sigma_r^\star - \sigma_{r+1}^\star}.
    \end{align}
    Here, (a) comes from \cite[Eqn. (2.26a)]{chen2021spectral}, (b) makes use of Weyl's inequality, and (c) holds due to the assumption $\sigma_r^\star - \sigma_{r+1}^\star > 2\|\bm{E}\|$.
\end{proof}
\begin{lemma}\label{lm:noise}
	Suppose that Assumption \ref{assump:noise} holds. 
	We let $\mathcal{A}$ denote the following set:
	\begin{align}\label{def:set_U}
		\mathcal{U} = \left\{\left(\bm{U}_1, \bm{U}_2, \bm{U}_3\right): \bm{U}_i \in \mathbb{R}^{n_i \times r_i}, \left\|\bm{U}_i\right\| \leq 1, \left\|\bm{U}_i\right\|_{2,\infty} \leq \sqrt{\frac{\mu_i r_i}{n_i}}, i \in [3]\right\}
	\end{align}
    and we define
    \begin{align*}
    	n = \max_{1 \leq i \leq 3} n_i,~\quad~\text{and}~\quad~r = \max_{1 \leq i \leq 3} r_i.
    \end{align*}
    If $n_1n_2n_3 \geq r^4n^2$, then
	with probability exceeding $1 - O(n^{-10})$, the following inequality holds:
    		\begin{align}
    		\sup_{\left(\bm{U}_1, \bm{U}_2, \bm{U}_3\right) \in \mathcal{U}}\left\|\bm{U}_1^\top\mathcal{M}_1\left(\bm{\mathcal{E}}\right)\left(\bm{U}_3 \otimes \bm{U}_2\right)\right\| &\lesssim \omega_{\sf max}\sqrt{n\mu_1\mu_2\mu_3r^3\log n},\label{ineq168}
    	\end{align}
    If, furthermore, the $E_{i,j,k}$'s are $\omega_{\sf max}$-sub-Gaussian, then with probability exceeding $1 - e^{-Cn}$, one has
    \begin{align}\label{ineq169}
    	\sup_{\left(\bm{U}_1, \bm{U}_2, \bm{U}_3\right) \in \mathcal{U}}\left\|\bm{U}_1^\top\mathcal{M}_1\left(\bm{\mathcal{E}}\right)\left(\bm{U}_3 \otimes \bm{U}_2\right)\right\| \leq \sqrt{nr}\omega_{\sf max}.
    \end{align}
\end{lemma}
\begin{proof}[Proof of Lemma \ref{lm:noise}]
	\eqref{ineq169} can be proved by simply combining \citet[Lemma A.2]{zhou2022optimal} (or Lemma 8.2 presented in its arxiv version) and the standard epsilon-net technique used in the proof of \citet[Lemma 5]{zhang2018tensor} and we omit the details here for the sake of brevity.
	
	\paragraph{Proving \eqref{ineq168}: the bounded noise case.}To prove \eqref{ineq168}, we first consider the bounded noise case, i.e., $|E_{i,j,k}| \leq B$ holds for all $(i,j,k) \in [n_1] \times [n_2] \times [n_3]$. 
	For any fixed $\left(\overline{\bm{U}}_1, \overline{\bm{U}}_2, \overline{\bm{U}}_3\right) \in \mathcal{U}$, note that 
	\begin{align*}
		\overline{\bm{U}}_1^\top\mathcal{M}_1\left(\bm{\mathcal{E}}\right)\left(\overline{\bm{U}}_3 \otimes \overline{\bm{U}}_2\right) = \sum_{i \in [n_1], j \in [n_2n_3]}\left(\mathcal{M}_1\left(\bm{\mathcal{E}}\right)\right)_{i,j}\left(\overline{\bm{U}}_1\right)_{i,:}^\top\left(\overline{\bm{U}}_3 \otimes \overline{\bm{U}}_2\right)_{j,:}
	\end{align*}
is a sum of independent zero-mean matrices. In addition, we have
\begin{align*}
	L &:= \max_{i \in [n_1], j \in [n_2n_3]}\left\|\left(\mathcal{M}_1\left(\bm{\mathcal{E}}\right)\right)_{i,j}\left(\overline{\bm{U}}_1\right)_{i,:}^\top\left(\overline{\bm{U}}_3 \otimes \overline{\bm{U}}_2\right)_{j,:}\right\| \leq B\prod_{i=1}^{3}\left\|\overline{\bm{U}}_i\right\|_{2,\infty} \leq B\sqrt{\frac{\mu_1\mu_2\mu_3r_1r_2r_3}{n_1n_2n_3}},\\
	V &:= \max\bigg\{\bigg\|\sum_{i \in [n_1], j \in [n_2n_3]}\bbE\left[\left(\mathcal{M}_1\left(\bm{\mathcal{E}}\right)\right)_{i,j}^2\right]\left\|\left(\overline{\bm{U}}_3 \otimes \overline{\bm{U}}_2\right)_{j,:}\right\|_2^2\left(\overline{\bm{U}}_1\right)_{i,:}^\top\left(\overline{\bm{U}}_1\right)_{i,:}\bigg\|,\\&\hspace{1.6cm}\bigg\|\sum_{i \in [n_1], j \in [n_2n_3]}\bbE\left[\left(\mathcal{M}_1\left(\bm{\mathcal{E}}\right)\right)_{i,j}^2\right]\left\|\left(\overline{\bm{U}}_1\right)_{i,:}\right\|_2^2\left(\overline{\bm{U}}_3 \otimes \overline{\bm{U}}_2\right)_{j,:}^\top\left(\overline{\bm{U}}_3 \otimes \overline{\bm{U}}_2\right)_{j,:}\bigg\|\bigg\}\\
	&\leq \max\left\{\omega_{\sf max}^2\left\|\left(\overline{\bm{U}}_3 \otimes \overline{\bm{U}}_2\right)\right\|_{\rm F}^2\left\|\overline{\bm{U}}_1\overline{\bm{U}}_1^\top\right\|, \omega_{\sf max}^2\left\|\overline{\bm{U}}_1\right\|_{\rm F}^2\left\|\left(\overline{\bm{U}}_3 \otimes \overline{\bm{U}}_2\right)\left(\overline{\bm{U}}_3 \otimes \overline{\bm{U}}_2\right)^\top\right\|\right\}\\
	&\lesssim \omega_{\sf max}^2\max\left\{r_2r_3, r_1\right\} \leq \omega_{\sf max}^2r^2.
\end{align*}
In view of the matrix Bernstein inequality, with probability exceeding $1 - e^{-Cnr\log n}$,
\begin{align}\label{ineq166}
	\left\|\overline{\bm{U}}_1^\top\mathcal{M}_1\left(\bm{\mathcal{E}}\right)\left(\overline{\bm{U}}_3 \otimes \overline{\bm{U}}_2\right)\right\| &\lesssim \sqrt{Vnr\log n} + Lnr\log n\notag\\ &\lesssim \omega_{\sf max}\sqrt{nr^3\log n} + \omega_{\sf max}\frac{\left(n_1n_2n_3\right)^{1/4}}{\log n} \sqrt{\frac{\mu_1\mu_2\mu_3r_1r_2r_3}{n_1n_2n_3}}nr\log n\notag\\
	&\leq \omega_{\sf max}\sqrt{n\mu_1\mu_2\mu_3r^3\log n}.
\end{align}
Here, the last line makes use of $n_1n_2n_3 \geq r^4n^2$.

 Repeating a similar argument as in \citet[Lemma 5.2]{vershynin2010introduction} yields that: there exists a set $\mathcal{B}_i \subset \mathcal{D}_i = \{\bm{x}: \bm{x} \in \bbR^{r_i}, \left\|\bm{x}\right\|_2 \leq \sqrt{\mu_ir_i/n_i}\}$ with cardinality at most $(1 + 8\sqrt{\mu_i r_i})^{r_i}$ such that for any $\bm{x} \in \mathcal{D}_i$, one can find $\bm{x}' \in \mathcal{B}_i$ such that 
 \begin{align*}
 	\|\bm{x} - \bm{x}'\|_2 \leq \frac{1}{4}\sqrt{\frac{1}{n_i}}.
 \end{align*}
 $\|\bm{x} - \bm{x}'\|_2 \leq c\sqrt{1/n_i}$. As a direct consequence, for any $\bm{U}_i \in \mathcal{U}_i := \left\{\bm{U} \in \mathbb{R}^{n_i \times r_i}: \left\|\bm{U}\right\| \leq 1, \left\|\bm{U}\right\|_{2,\infty} \leq \sqrt{\frac{\mu_i r_i}{n_i}}\right\}$, one can find $\bm{U}_i' \in \mathcal{F}_i := \{\bm{U}  \in \mathbb{R}^{n_i \times r_i}: \bm{U}_{j,:}^\top \in \mathcal{B}_i, \forall j \in [n_i]\}$ such that
 \begin{align*}
 	\left\|\bm{U}_i - \bm{U}_i'\right\|_{2,\infty} \leq \frac{1}{4}\sqrt{\frac{1}{n_i}}.
 \end{align*}
Let $\mathcal{U}_i'$ denote the following set:
\begin{align}\label{def:U_i'}
	\mathcal{U}_i' := \left\{\bm{U}' \in \bbR^{n_i \times r_i}: \bm{U}' \in \mathcal{F}_i, \inf_{\bm{U} \in \mathcal{U}_i}\left\|\bm{U} - \bm{U}'\right\|_{2,\infty} \leq \frac{1}{4}\sqrt{\frac{1}{n_i}}\right\}.
\end{align}
Then we can verify the following three properties:
\begin{subequations}
	\begin{align}
		\left|\mathcal{U}_i'\right| &\leq \left|\mathcal{F}_i\right| = \left|\mathcal{B}_i\right|^{n_i} \leq (1 + 8\sqrt{\mu_i r_i})^{n_ir_i} \leq n_i^{n_ir_i} \leq e^{n_ir_i\log n},\label{ineq163a}\\
		\forall \bm{U}_i &\in \mathcal{U}_i,~\quad~\exists \bm{U}_i' \in \mathcal{U}_i'~\quad~\text{s.t.}~\quad~\left\|\bm{U}_i - \bm{U}_i'\right\|_{2,\infty} \leq \frac{1}{4}\sqrt{\frac{1}{n_i}}~\quad~\text{and}~\quad~\left\|\bm{U}_i - \bm{U}_i'\right\| \leq \frac{1}{4},\label{ineq163b}\\
		\forall \bm{U}_i' &\in \mathcal{U}_i',~\quad~\exists \bm{U}_i \in \mathcal{U}_i~\quad~\text{s.t.}~\quad~\left\|\bm{U}_i - \bm{U}_i'\right\|_{2,\infty} \leq \frac{1}{4}\sqrt{\frac{1}{n_i}}~\quad~\text{and}~\quad~\left\|\bm{U}_i - \bm{U}_i'\right\| \leq \frac{1}{4}.\label{ineq163c}
	\end{align}
\end{subequations}
Here, the last two inequalities make use of $\|\bm{U}_i - \bm{U}_i'\| \leq \sqrt{n_i}\|\bm{U}_i - \bm{U}_i'\|_{2,\infty}$.

We define $$A = \sup_{\left(\bm{U}_1, \bm{U}_2, \bm{U}_3\right) \in \mathcal{U}}\left\|\bm{U}_1^\top\mathcal{M}_1\left(\bm{\mathcal{E}}\right)\left(\bm{U}_3 \otimes \bm{U}_2\right)\right\|~\quad~\text{and}~\quad~B = \sup_{\left(\bm{U}_1', \bm{U}_2', \bm{U}_3'\right) \in \mathcal{U}_1' \times \mathcal{U}_2' \times \mathcal{U}_3'}\left\|\bm{U}_1^{'\top}\mathcal{M}_1\left(\bm{\mathcal{E}}\right)\left(\bm{U}_3' \otimes \bm{U}_2'\right)\right\|.$$ For any $(\bm{U}_1, \bm{U}_2, \bm{U}_3) \in \mathcal{U} = \mathcal{U}_1 \times \mathcal{U}_2 \times \mathcal{U}_3$, we know from \eqref{ineq163b} and \eqref{ineq163c} that there exist $(\bm{U}_1', \bm{U}_2', \bm{U}_3') \in \mathcal{U}_1' \times \mathcal{U}_2' \times \mathcal{U}_3'$ such that
\begin{subequations}
	\begin{align}
		\bm{U}_i - \bm{U}_i' &\in \frac{1}{4}\mathcal{U}_i,~\quad~i \in [3],\label{ineq164a}\\
		\bm{U}_i' &\in \frac{5}{4}\mathcal{U}_i,~\quad~i \in [3].\label{ineq164b}
	\end{align}
\end{subequations}
The triangle inequality, \eqref{ineq164a} and \eqref{ineq164b} together show that
\begin{align*}
	\left\|\bm{U}_1^\top\mathcal{M}_1\left(\bm{\mathcal{E}}\right)\left(\bm{U}_3 \otimes \bm{U}_2\right)\right\| &\leq \left\|\left(\bm{U}_1 - \bm{U}_1'\right)^\top\mathcal{M}_1\left(\bm{\mathcal{E}}\right)\left(\bm{U}_3 \otimes \bm{U}_2\right)\right\| + \left\| \bm{U}_1^{'\top}\mathcal{M}_1\left(\bm{\mathcal{E}}\right)\left(\bm{U}_3 \otimes \left(\bm{U}_2 - \bm{U}_2'\right)\right)\right\|\notag\\
	&\quad + \left\| \bm{U}_1^{'\top}\mathcal{M}_1\left(\bm{\mathcal{E}}\right)\left(\left(\bm{U}_3 - \bm{U}_3'\right) \otimes \bm{U}_2'\right)\right\| + \left\|\bm{U}_1^{'\top}\mathcal{M}_1\left(\bm{\mathcal{E}}\right)\left(\bm{U}_3' \otimes \bm{U}_2'\right)\right\|\notag\\
	&\leq \frac{1}{4}A + \frac{1}{4}\cdot \frac{5}{4}A + \frac{1}{4}\cdot \frac{5}{4}\cdot \frac{5}{4}A + B\notag\\
	&\leq \frac{61}{64}A + B,
\end{align*}
which implies
\begin{align}\label{ineq165}
	A \leq \frac{64}{3}B.
\end{align}
Furthermore, \eqref{ineq166}, \eqref{ineq163a}, \eqref{ineq164b} and the union bound together imply that with probability exceeding $1 - e^{-Cnr\log n}\cdot \prod_ie^{n_ir_i\log n} \geq 1 - e^{-C'nr\log n}$,
\begin{align}\label{ineq167}
	B \lesssim \left(\frac{5}{4}\right)^3\omega_{\sf max}\sqrt{n\mu_1\mu_2\mu_3r^3\log n} \asymp \omega_{\sf max}\sqrt{n\mu_1\mu_2\mu_3r^3\log n}.
\end{align}
Putting \eqref{ineq165} and \eqref{ineq167} together, we arrive at
\begin{align*}
	\sup_{\left(\bm{U}_1, \bm{U}_2, \bm{U}_3\right) \in \mathcal{U}}\left\|\bm{U}_1^\top\mathcal{M}_1\left(\bm{\mathcal{E}}\right)\left(\bm{U}_3 \otimes \bm{U}_2\right)\right\| = A \lesssim \omega_{\sf max}\sqrt{n\mu_1\mu_2\mu_3r^3\log n}
\end{align*}
 with probability  exceeding $1 - e^{-C'nr\log n}$.

\paragraph{Proving \eqref{ineq168}: the general case.} For the general case where the noise matrix $\bm{E}$ satisfies Assumption \ref{assump:noise_matrix}, we can prove \eqref{ineq168} by repeating a similar argument as in \citet[Section B.4.2]{zhou2023deflated}.
\end{proof}

\bibliographystyle{apalike}
\bibliography{reference}

\begin{thebibliography}{}

\bibitem[Abbe, 2017]{abbe2017community}
Abbe, E. (2017).
\newblock Community detection and stochastic block models: recent developments.
\newblock {\em The Journal of Machine Learning Research}, 18(1):6446--6531.

\bibitem[Abbe et~al., 2015]{abbe2015exact}
Abbe, E., Bandeira, A.~S., and Hall, G. (2015).
\newblock Exact recovery in the stochastic block model.
\newblock {\em IEEE Transactions on information theory}, 62(1):471--487.

\bibitem[Abbe et~al., 2022]{abbe2022lp}
Abbe, E., Fan, J., and Wang, K. (2022).
\newblock An $\ell_p$ theory of {PCA} and spectral clustering.
\newblock {\em The Annals of Statistics}, 50(4):2359--2385.

\bibitem[Abbe et~al., 2020]{abbe2020entrywise}
Abbe, E., Fan, J., Wang, K., and Zhong, Y. (2020).
\newblock Entrywise eigenvector analysis of random matrices with low expected
  rank.
\newblock {\em Annals of statistics}, 48(3):1452.

\bibitem[Abbe and Sandon, 2015]{abbe2015community}
Abbe, E. and Sandon, C. (2015).
\newblock Community detection in general stochastic block models: Fundamental
  limits and efficient algorithms for recovery.
\newblock In {\em Annual Symposium on Foundations of Computer Science}, pages
  670--688. IEEE.

\bibitem[Agterberg and Zhang, 2022]{agterberg2022estimating}
Agterberg, J. and Zhang, A. (2022).
\newblock Estimating higher-order mixed memberships via the $\ell_{2, \infty}$
  tensor perturbation bound.
\newblock {\em arXiv preprint arXiv:2212.08642}.

\bibitem[Amini and Levina, 2018]{amini2018semidefinite}
Amini, A.~A. and Levina, E. (2018).
\newblock On semidefinite relaxations for the block model.

\bibitem[Anandkumar et~al., 2017]{anandkumar2017homotopy}
Anandkumar, A., Deng, Y., Ge, R., and Mobahi, H. (2017).
\newblock Homotopy analysis for tensor pca.
\newblock In {\em Conference on Learning Theory}, pages 79--104. PMLR.

\bibitem[Arous et~al., 2019]{arous2019landscape}
Arous, G.~B., Mei, S., Montanari, A., and Nica, M. (2019).
\newblock The landscape of the spiked tensor model.
\newblock {\em Communications on Pure and Applied Mathematics},
  72(11):2282--2330.

\bibitem[Arthur and Vassilvitskii, 2007]{arthur2007k}
Arthur, D. and Vassilvitskii, S. (2007).
\newblock K-means++: the advantages of careful seeding.
\newblock In {\em Proceedings of the eighteenth annual ACM-SIAM symposium on
  Discrete algorithms}, pages 1027--1035.

\bibitem[Bahmani et~al., 2012]{bahmani2012scalable}
Bahmani, B., Moseley, B., Vattani, A., Kumar, R., and Vassilvitskii, S. (2012).
\newblock Scalable k-means++.
\newblock {\em Proceedings of the VLDB Endowment}, 5(7):622--633.

\bibitem[Bandeira et~al., 2017]{bandeira2017tightness}
Bandeira, A.~S., Boumal, N., and Singer, A. (2017).
\newblock Tightness of the maximum likelihood semidefinite relaxation for
  angular synchronization.
\newblock {\em Mathematical Programming}, 163:145--167.

\bibitem[Bi et~al., 2018]{bi2018multilayer}
Bi, X., Qu, A., and Shen, X. (2018).
\newblock Multilayer tensor factorization with applications to recommender
  systems.
\newblock {\em The Annals of Statistics}, 46(6B):3308--3333.

\bibitem[Bi et~al., 2021]{bi2021tensors}
Bi, X., Tang, X., Yuan, Y., Zhang, Y., and Qu, A. (2021).
\newblock Tensors in statistics.
\newblock {\em Annual review of statistics and its application}, 8:345--368.

\bibitem[Boucheron et~al., 2013]{boucheron2013concentration}
Boucheron, S., Lugosi, G., and Massart, P. (2013).
\newblock {\em Concentration inequalities: A nonasymptotic theory of
  independence}.
\newblock Oxford university press.

\bibitem[Cai et~al., 2021]{cai2021subspace}
Cai, C., Li, G., Chi, Y., Poor, H.~V., and Chen, Y. (2021).
\newblock Subspace estimation from unbalanced and incomplete data matrices:
  $\ell_{2,\infty}$ statistical guarantees.
\newblock {\em The Annals of Statistics}, 49(2):944--967.

\bibitem[Cai et~al., 2022a]{cai2022nonconvex}
Cai, C., Li, G., Poor, H.~V., and Chen, Y. (2022a).
\newblock Nonconvex low-rank tensor completion from noisy data.
\newblock {\em Operations Research}, 70(2):1219--1237.

\bibitem[Cai et~al., 2022b]{cai2022uncertainty}
Cai, C., Poor, H.~V., and Chen, Y. (2022b).
\newblock Uncertainty quantification for nonconvex tensor completion:
  Confidence intervals, heteroscedasticity and optimality.
\newblock {\em IEEE Transactions on Information Theory}, 69(1):407--452.

\bibitem[Cai and Li, 2015]{cai2015robust}
Cai, T.~T. and Li, X. (2015).
\newblock Robust and computationally feasible community detection in the
  presence of arbitrary outlier nodes.
\newblock {\em The Annals of Statistics}, 43(3):1027--1059.

\bibitem[Cai and Zhang, 2018]{cai2018rate}
Cai, T.~T. and Zhang, A. (2018).
\newblock Rate-optimal perturbation bounds for singular subspaces with
  applications to high-dimensional statistics.
\newblock {\em The Annals of Statistics}, 46(1):60--89.

\bibitem[Celentano et~al., 2023]{celentano2023local}
Celentano, M., Fan, Z., and Mei, S. (2023).
\newblock Local convexity of the {TAP} free energy and {AMP} convergence for
  $z_2$-synchronization.
\newblock {\em The Annals of Statistics}, 51(2):519--546.

\bibitem[Chen and Yang, 2021]{chen2021cutoff}
Chen, X. and Yang, Y. (2021).
\newblock Cutoff for exact recovery of {Gaussian} mixture models.
\newblock {\em IEEE Transactions on Information Theory}, 67(6):4223--4238.

\bibitem[Chen and Cand{\`e}s, 2018]{chen2018projected}
Chen, Y. and Cand{\`e}s, E.~J. (2018).
\newblock The projected power method: An efficient algorithm for joint
  alignment from pairwise differences.
\newblock {\em Communications on Pure and Applied Mathematics},
  71(8):1648--1714.

\bibitem[Chen et~al., 2021a]{chen2021spectral}
Chen, Y., Chi, Y., Fan, J., and Ma, C. (2021a).
\newblock Spectral methods for data science: A statistical perspective.
\newblock {\em Foundations and Trends{\textregistered} in Machine Learning},
  14(5):566--806.

\bibitem[Chen et~al., 2020]{chen2020noisy}
Chen, Y., Chi, Y., Fan, J., Ma, C., and Yan, Y. (2020).
\newblock Noisy matrix completion: Understanding statistical guarantees for
  convex relaxation via nonconvex optimization.
\newblock {\em SIAM journal on optimization}, 30(4):3098--3121.

\bibitem[Chen et~al., 2019a]{chen2019spectral}
Chen, Y., Fan, J., Ma, C., and Wang, K. (2019a).
\newblock Spectral method and regularized {MLE} are both optimal for top-{K}
  ranking.
\newblock {\em Annals of statistics}, 47(4):2204.

\bibitem[Chen et~al., 2019b]{chen2019inference}
Chen, Y., Fan, J., Ma, C., and Yan, Y. (2019b).
\newblock Inference and uncertainty quantification for noisy matrix completion.
\newblock {\em Proceedings of the National Academy of Sciences},
  116(46):22931--22937.

\bibitem[Chen et~al., 2021b]{chen2021bridging}
Chen, Y., Fan, J., Ma, C., and Yan, Y. (2021b).
\newblock Bridging convex and nonconvex optimization in robust {PCA}: Noise,
  outliers and missing data.
\newblock {\em The Annals of Statistics}, 49(5):2948--2971.

\bibitem[Chen et~al., 2023]{chen2021convex}
Chen, Y., Fan, J., Wang, B., and Yan, Y. (2023).
\newblock Convex and nonconvex optimization are both minimax-optimal for noisy
  blind deconvolution under random designs.
\newblock {\em Journal of the American Statistical Association},
  118(542):858--868.

\bibitem[Chen et~al., 2016]{chen2016community}
Chen, Y., Kamath, G., Suh, C., and Tse, D. (2016).
\newblock Community recovery in graphs with locality.
\newblock In {\em International conference on machine learning}, pages
  689--698. PMLR.

\bibitem[Chi et~al., 2020]{chi2020provable}
Chi, E.~C., Gaines, B.~R., Sun, W.~W., Zhou, H., and Yang, J. (2020).
\newblock Provable convex co-clustering of tensors.
\newblock {\em The Journal of Machine Learning Research}, 21(1):8792--8849.

\bibitem[Chin et~al., 2015]{chin2015stochastic}
Chin, P., Rao, A., and Vu, V. (2015).
\newblock Stochastic block model and community detection in sparse graphs: A
  spectral algorithm with optimal rate of recovery.
\newblock In {\em Conference on Learning Theory}, pages 391--423. PMLR.

\bibitem[Cichocki et~al., 2015]{cichocki2015tensor}
Cichocki, A., Mandic, D., De~Lathauwer, L., Zhou, G., Zhao, Q., Caiafa, C., and
  Phan, H.~A. (2015).
\newblock Tensor decompositions for signal processing applications: From
  two-way to multiway component analysis.
\newblock {\em IEEE signal processing magazine}, 32(2):145--163.

\bibitem[De~Lathauwer et~al., 2000]{de2000best}
De~Lathauwer, L., De~Moor, B., and Vandewalle, J. (2000).
\newblock On the best rank-$1$ and rank-$(r_1, r_2,..., r_n)$ approximation of
  higher-order tensors.
\newblock {\em SIAM journal on Matrix Analysis and Applications},
  21(4):1324--1342.

\bibitem[Deng et~al., 2023]{deng2023correlation}
Deng, Y., Tang, X., and Qu, A. (2023).
\newblock Correlation tensor decomposition and its application in spatial
  imaging data.
\newblock {\em Journal of the American Statistical Association},
  118(541):440--456.

\bibitem[Deshpande et~al., 2017]{deshpande2017asymptotic}
Deshpande, Y., Abbe, E., and Montanari, A. (2017).
\newblock Asymptotic mutual information for the balanced binary stochastic
  block model.
\newblock {\em Information and Inference: A Journal of the IMA}, 6(2):125--170.

\bibitem[Florescu and Perkins, 2016]{florescu2016spectral}
Florescu, L. and Perkins, W. (2016).
\newblock Spectral thresholds in the bipartite stochastic block model.
\newblock In {\em Conference on Learning Theory}, pages 943--959. PMLR.

\bibitem[Fu and Dong, 2016]{fu20163d}
Fu, Y. and Dong, W. (2016).
\newblock {3D} magnetic resonance image denoising using low-rank tensor
  approximation.
\newblock {\em Neurocomputing}, 195:30--39.

\bibitem[Gao et~al., 2017]{gao2017achieving}
Gao, C., Ma, Z., Zhang, A.~Y., and Zhou, H.~H. (2017).
\newblock Achieving optimal misclassification proportion in stochastic block
  models.
\newblock {\em The Journal of Machine Learning Research}, 18(1):1980--2024.

\bibitem[Gao and Zhang, 2021]{gao2021exact}
Gao, C. and Zhang, A.~Y. (2021).
\newblock Exact minimax estimation for phase synchronization.
\newblock {\em IEEE Transactions on Information Theory}, 67(12):8236--8247.

\bibitem[Gu{\'e}don and Vershynin, 2016]{guedon2016community}
Gu{\'e}don, O. and Vershynin, R. (2016).
\newblock Community detection in sparse networks via grothendieck’s
  inequality.
\newblock {\em Probability Theory and Related Fields}, 165(3-4):1025--1049.

\bibitem[Hajek et~al., 2016a]{hajek2016achieving}
Hajek, B., Wu, Y., and Xu, J. (2016a).
\newblock Achieving exact cluster recovery threshold via semidefinite
  programming.
\newblock {\em IEEE Transactions on Information Theory}, 62(5):2788--2797.

\bibitem[Hajek et~al., 2016b]{hajek2016achievingE}
Hajek, B., Wu, Y., and Xu, J. (2016b).
\newblock Achieving exact cluster recovery threshold via semidefinite
  programming: Extensions.
\newblock {\em IEEE Transactions on Information Theory}, 62(10):5918--5937.

\bibitem[Han et~al., 2022a]{han2022exact}
Han, R., Luo, Y., Wang, M., and Zhang, A.~R. (2022a).
\newblock Exact clustering in tensor block model: Statistical optimality and
  computational limit.
\newblock {\em Journal of the Royal Statistical Society Series B},
  84(5):1666--1698.

\bibitem[Han et~al., 2022b]{han2022optimal}
Han, R., Willett, R., and Zhang, A.~R. (2022b).
\newblock An optimal statistical and computational framework for generalized
  tensor estimation.
\newblock {\em The Annals of Statistics}, 50(1):1--29.

\bibitem[Han et~al., 2023]{han2023eigen}
Han, X., Tong, X., and Fan, Y. (2023).
\newblock Eigen selection in spectral clustering: a theory-guided practice.
\newblock {\em Journal of the American Statistical Association},
  118(541):109--121.

\bibitem[Holland et~al., 1983]{holland1983stochastic}
Holland, P.~W., Laskey, K.~B., and Leinhardt, S. (1983).
\newblock Stochastic blockmodels: First steps.
\newblock {\em Social networks}, 5(2):109--137.

\bibitem[Hopkins et~al., 2015]{hopkins2015tensor}
Hopkins, S.~B., Shi, J., and Steurer, D. (2015).
\newblock Tensor principal component analysis via sum-of-square proofs.
\newblock In {\em Conference on Learning Theory}, pages 956--1006. PMLR.

\bibitem[Hu and Wang, 2023]{hu2023multiway}
Hu, J. and Wang, M. (2023).
\newblock Multiway spherical clustering via degree-corrected tensor block
  models.
\newblock {\em IEEE Transactions on Information Theory}.

\bibitem[Javanmard et~al., 2016]{javanmard2016phase}
Javanmard, A., Montanari, A., and Ricci-Tersenghi, F. (2016).
\newblock Phase transitions in semidefinite relaxations.
\newblock {\em Proceedings of the National Academy of Sciences},
  113(16):E2218--E2223.

\bibitem[Jegelka et~al., 2009]{jegelka2009approximation}
Jegelka, S., Sra, S., and Banerjee, A. (2009).
\newblock Approximation algorithms for tensor clustering.
\newblock In {\em International Conference on Algorithmic Learning Theory},
  pages 368--383. Springer.

\bibitem[Kannan and Vempala, 2009]{kannan2009spectral}
Kannan, R. and Vempala, S. (2009).
\newblock Spectral algorithms.
\newblock {\em Foundations and Trends{\textregistered} in Theoretical Computer
  Science}, 4(3--4):157--288.

\bibitem[Ke and Wang, 2022]{ke2022optimal}
Ke, Z.~T. and Wang, J. (2022).
\newblock Optimal network membership estimation under severe degree
  heterogeneity.
\newblock {\em arXiv preprint arXiv:2204.12087}.

\bibitem[Lei et~al., 2020]{lei2020consistent}
Lei, J., Chen, K., and Lynch, B. (2020).
\newblock Consistent community detection in multi-layer network data.
\newblock {\em Biometrika}, 107(1):61--73.

\bibitem[Lei and Rinaldo, 2015]{lei2015consistency}
Lei, J. and Rinaldo, A. (2015).
\newblock Consistency of spectral clustering in stochastic block models.
\newblock {\em The Annals of Statistics}, 43(1):215--237.

\bibitem[Lei, 2019]{lei2019unified}
Lei, L. (2019).
\newblock Unified $\ell_{2\rightarrow \infty}$ eigenspace perturbation theory
  for symmetric random matrices.
\newblock {\em arXiv preprint arXiv:1909.04798}.

\bibitem[Li et~al., 2023]{li2023approximate}
Li, G., Fan, W., and Wei, Y. (2023).
\newblock Approximate message passing from random initialization with
  applications to $z_2$ synchronization.
\newblock {\em Proceedings of the National Academy of Sciences (PNAS)},
  120(31).

\bibitem[Li and Wei, 2022]{li2022non}
Li, G. and Wei, Y. (2022).
\newblock A non-asymptotic framework for approximate message passing in spiked
  models.
\newblock {\em arXiv preprint arXiv:2208.03313}.

\bibitem[Li and Li, 2010]{li2010tensor}
Li, N. and Li, B. (2010).
\newblock Tensor completion for on-board compression of hyperspectral images.
\newblock In {\em 2010 IEEE International Conference on Image Processing},
  pages 517--520. IEEE.

\bibitem[Li et~al., 2021]{li2021convex}
Li, X., Chen, Y., and Xu, J. (2021).
\newblock Convex relaxation methods for community detection.
\newblock {\em Statistical science}, 36(1):2--15.

\bibitem[Li et~al., 2020]{li2020birds}
Li, X., Li, Y., Ling, S., Strohmer, T., and Wei, K. (2020).
\newblock When do birds of a feather flock together? $k$-means, proximity, and
  conic programming.
\newblock {\em Mathematical Programming}, 179:295--341.

\bibitem[Ling, 2022]{ling2022near}
Ling, S. (2022).
\newblock Near-optimal performance bounds for orthogonal and permutation group
  synchronization via spectral methods.
\newblock {\em Applied and Computational Harmonic Analysis}, 60:20--52.

\bibitem[Liu and Moitra, 2020]{liu2020tensor}
Liu, A. and Moitra, A. (2020).
\newblock Tensor completion made practical.
\newblock {\em Advances in Neural Information Processing Systems},
  33:18905--18916.

\bibitem[L{\"o}ffler et~al., 2021]{loffler2021optimality}
L{\"o}ffler, M., Zhang, A.~Y., and Zhou, H.~H. (2021).
\newblock Optimality of spectral clustering in the {Gaussian} mixture model.
\newblock {\em The Annals of Statistics}, 49(5):2506--2530.

\bibitem[Lounici, 2014]{lounici2014high}
Lounici, K. (2014).
\newblock High-dimensional covariance matrix estimation with missing
  observations.
\newblock {\em Bernoulli}, 20(3):1029--1058.

\bibitem[Lu and Zhou, 2016]{lu2016statistical}
Lu, Y. and Zhou, H.~H. (2016).
\newblock Statistical and computational guarantees of {Lloyd's} algorithm and
  its variants.
\newblock {\em arXiv preprint arXiv:1612.02099}.

\bibitem[Lyu and Xia, 2022]{lyu2022optimal}
Lyu, Z. and Xia, D. (2022).
\newblock Optimal clustering by lloyd algorithm for low-rank mixture model.
\newblock {\em arXiv preprint arXiv:2207.04600}.

\bibitem[Ma et~al., 2020]{ma2020implicit}
Ma, C., Wang, K., Chi, Y., and Chen, Y. (2020).
\newblock Implicit regularization in nonconvex statistical estimation: Gradient
  descent converges linearly for phase retrieval, matrix completion, and blind
  deconvolution.
\newblock {\em Foundations of Computational Mathematics}, 20(3):451--632.

\bibitem[Mai et~al., 2022]{mai2022doubly}
Mai, Q., Zhang, X., Pan, Y., and Deng, K. (2022).
\newblock A doubly enhanced em algorithm for model-based tensor clustering.
\newblock {\em Journal of the American Statistical Association},
  117(540):2120--2134.

\bibitem[Milligan and Cooper, 1986]{milligan1986study}
Milligan, G.~W. and Cooper, M.~C. (1986).
\newblock A study of the comparability of external criteria for hierarchical
  cluster analysis.
\newblock {\em Multivariate behavioral research}, 21(4):441--458.

\bibitem[Montanari and Sun, 2018]{montanari2018spectral}
Montanari, A. and Sun, N. (2018).
\newblock Spectral algorithms for tensor completion.
\newblock {\em Communications on Pure and Applied Mathematics},
  71(11):2381--2425.

\bibitem[Mossel et~al., 2014]{mossel2014belief}
Mossel, E., Neeman, J., and Sly, A. (2014).
\newblock Belief propagation, robust reconstruction and optimal recovery of
  block models.
\newblock In {\em Conference on Learning Theory}, pages 356--370. PMLR.

\bibitem[Mossel et~al., 2015]{mossel2015reconstruction}
Mossel, E., Neeman, J., and Sly, A. (2015).
\newblock Reconstruction and estimation in the planted partition model.
\newblock {\em Probability Theory and Related Fields}, 162:431--461.

\bibitem[Nasiri et~al., 2014]{nasiri2014fuzzy}
Nasiri, M., Rezghi, M., and Minaei, B. (2014).
\newblock Fuzzy dynamic tensor decomposition algorithm for recommender system.
\newblock {\em UCT Journal of Research in Science, Engineering and Technology},
  2(2):52--55.

\bibitem[Ndaoud, 2022]{ndaoud2022sharp}
Ndaoud, M. (2022).
\newblock Sharp optimal recovery in the two component {Gaussian} mixture model.
\newblock {\em The Annals of Statistics}, 50(4):2096--2126.

\bibitem[Richard and Montanari, 2014]{richard2014statistical}
Richard, E. and Montanari, A. (2014).
\newblock A statistical model for tensor {PCA}.
\newblock In {\em Advances in Neural Information Processing Systems}, pages
  2897--2905.

\bibitem[Rohe et~al., 2011]{rohe2011spectral}
Rohe, K., Chatterjee, S., and Yu, B. (2011).
\newblock Spectral clustering and the high-dimensional stochastic blockmodel.
\newblock {\em The Annals of Statistics}, 39(4):1878--1915.

\bibitem[Sidiropoulos et~al., 2017]{sidiropoulos2017tensor}
Sidiropoulos, N.~D., De~Lathauwer, L., Fu, X., Huang, K., Papalexakis, E.~E.,
  and Faloutsos, C. (2017).
\newblock Tensor decomposition for signal processing and machine learning.
\newblock {\em IEEE Transactions on signal processing}, 65(13):3551--3582.

\bibitem[Singer, 2011]{singer2011angular}
Singer, A. (2011).
\newblock Angular synchronization by eigenvectors and semidefinite programming.
\newblock {\em Applied and computational harmonic analysis}, 30(1):20--36.

\bibitem[Sun and Li, 2019]{sun2019dynamic}
Sun, W.~W. and Li, L. (2019).
\newblock Dynamic tensor clustering.
\newblock {\em Journal of the American Statistical Association},
  114(528):1894--1907.

\bibitem[Tong et~al., 2022]{tong2022scaling}
Tong, T., Ma, C., Prater-Bennette, A., Tripp, E., and Chi, Y. (2022).
\newblock Scaling and scalability: Provable nonconvex low-rank tensor
  estimation from incomplete measurements.
\newblock {\em Journal of Machine Learning Research}, 23(163):1--77.

\bibitem[Vershynin, 2010]{vershynin2010introduction}
Vershynin, R. (2010).
\newblock Introduction to the non-asymptotic analysis of random matrices.
\newblock {\em arXiv preprint arXiv:1011.3027}.

\bibitem[Vershynin, 2018]{vershynin2018high}
Vershynin, R. (2018).
\newblock {\em High-dimensional probability: An introduction with applications
  in data science}, volume~47.
\newblock Cambridge university press.

\bibitem[Von~Luxburg, 2007]{von2007tutorial}
Von~Luxburg, U. (2007).
\newblock A tutorial on spectral clustering.
\newblock {\em Statistics and computing}, 17:395--416.

\bibitem[Wang et~al., 2019]{wang2019three}
Wang, M., Fischer, J., and Song, Y.~S. (2019).
\newblock Three-way clustering of multi-tissue multi-individual gene expression
  data using semi-nonnegative tensor decomposition.
\newblock {\em The annals of applied statistics}, 13(2):1103.

\bibitem[Wang and Zeng, 2019]{wang2019multiway}
Wang, M. and Zeng, Y. (2019).
\newblock Multiway clustering via tensor block models.
\newblock {\em Advances in neural information processing systems}, 32.

\bibitem[Wozniak et~al., 2007]{wozniak2007neurocognitive}
Wozniak, J.~R., Krach, L., Ward, E., Mueller, B.~A., Muetzel, R., Schnoebelen,
  S., Kiragu, A., and Lim, K.~O. (2007).
\newblock Neurocognitive and neuroimaging correlates of pediatric traumatic
  brain injury: a diffusion tensor imaging (dti) study.
\newblock {\em Archives of Clinical Neuropsychology}, 22(5):555--568.

\bibitem[Wu et~al., 2019]{wu2019essential}
Wu, J., Lin, Z., and Zha, H. (2019).
\newblock Essential tensor learning for multi-view spectral clustering.
\newblock {\em IEEE Transactions on Image Processing}, 28(12):5910--5922.

\bibitem[Xia, 2021]{xia2021normal}
Xia, D. (2021).
\newblock Normal approximation and confidence region of singular subspaces.
\newblock {\em Electronic Journal of Statistics}, 15(2):3798--3851.

\bibitem[Xia et~al., 2021]{xia2021statistically}
Xia, D., Yuan, M., and Zhang, C.-H. (2021).
\newblock Statistically optimal and computationally efficient low rank tensor
  completion from noisy entries.
\newblock {\em The Annals of Statistics}, 49(1).

\bibitem[Xia et~al., 2022]{xia2022inference}
Xia, D., Zhang, A.~R., and Zhou, Y. (2022).
\newblock Inference for low-rank tensors—no need to debias.
\newblock {\em The Annals of Statistics}, 50(2):1220--1245.

\bibitem[Yan et~al., 2021]{yan2021inference}
Yan, Y., Chen, Y., and Fan, J. (2021).
\newblock Inference for heteroskedastic {PCA} with missing data.
\newblock {\em arXiv preprint arXiv:2107.12365}.

\bibitem[Yang and Ma, 2022]{yang2022optimal}
Yang, Y. and Ma, C. (2022).
\newblock Optimal tuning-free convex relaxation for noisy matrix completion.
\newblock {\em arXiv preprint arXiv:2207.05802}.

\bibitem[Yuan and Zhang, 2016]{yuan2016tensor}
Yuan, M. and Zhang, C.-H. (2016).
\newblock On tensor completion via nuclear norm minimization.
\newblock {\em Foundations of Computational Mathematics}, 16(4):1031--1068.

\bibitem[Zhang and Xia, 2018]{zhang2018tensor}
Zhang, A. and Xia, D. (2018).
\newblock Tensor {SVD}: Statistical and computational limits.
\newblock {\em IEEE Transactions on Information Theory}, 64(11):7311--7338.

\bibitem[Zhang et~al., 2022]{zhang2022heteroskedastic}
Zhang, A.~R., Cai, T.~T., and Wu, Y. (2022).
\newblock Heteroskedastic {PCA}: Algorithm, optimality, and applications.
\newblock {\em The Annals of Statistics}, 50(1):53--80.

\bibitem[Zhang and Zhou, 2020]{zhang2020non}
Zhang, A.~R. and Zhou, Y. (2020).
\newblock On the non-asymptotic and sharp lower tail bounds of random
  variables.
\newblock {\em Stat}, 9(1):e314.

\bibitem[Zhang, 2023]{zhang2023fundamental}
Zhang, A.~Y. (2023).
\newblock Fundamental limits of spectral clustering in stochastic block models.
\newblock {\em arXiv preprint arXiv:2301.09289}.

\bibitem[Zhang and Zhou, 2016]{zhang2016minimax}
Zhang, A.~Y. and Zhou, H.~H. (2016).
\newblock Minimax rates of community detection in stochastic block models.
\newblock {\em The Annals of Statistics}, 44(5):2252--2280.

\bibitem[Zhang and Zhou, 2022]{zhang2022leave}
Zhang, A.~Y. and Zhou, H.~H. (2022).
\newblock Leave-one-out singular subspace perturbation analysis for spectral
  clustering.
\newblock {\em arXiv preprint arXiv:2205.14855}.

\bibitem[Zhang et~al., 2020]{zhang2020denoising}
Zhang, C., Han, R., Zhang, A.~R., and Voyles, P.~M. (2020).
\newblock Denoising atomic resolution 4d scanning transmission electron
  microscopy data with tensor singular value decomposition.
\newblock {\em Ultramicroscopy}, 219:113123.

\bibitem[Zhong and Boumal, 2018]{zhong2018near}
Zhong, Y. and Boumal, N. (2018).
\newblock Near-optimal bounds for phase synchronization.
\newblock {\em SIAM Journal on Optimization}, 28(2):989--1016.

\bibitem[Zhou et~al., 2013]{zhou2013tensor}
Zhou, H., Li, L., and Zhu, H. (2013).
\newblock Tensor regression with applications in neuroimaging data analysis.
\newblock {\em Journal of the American Statistical Association},
  108(502):540--552.

\bibitem[Zhou and Chen, 2023]{zhou2023deflated}
Zhou, Y. and Chen, Y. (2023).
\newblock Deflated {HeteroPCA}: Overcoming the curse of ill-conditioning in
  heteroskedastic {PCA}.
\newblock {\em arXiv preprint arXiv:2303.06198}.

\bibitem[Zhou et~al., 2022]{zhou2022optimal}
Zhou, Y., Zhang, A.~R., Zheng, L., and Wang, Y. (2022).
\newblock Optimal high-order tensor {SVD} via tensor-train orthogonal
  iteration.
\newblock {\em IEEE Transactions on Information Theory}, 68(6):3991--4019.

\end{thebibliography}
\end{document}